\documentclass[11pt]{amsart}

\usepackage[utf8]{inputenc}
\usepackage{amsmath,amsfonts,amsthm,amssymb,amscd,mathabx}
\usepackage{epsfig}
\usepackage{upref}
\usepackage{bm}
\usepackage{calc}
\usepackage{caption}
\usepackage{subcaption}
\usepackage{hyperref}
\usepackage[usenames,dvipsnames]{xcolor}
\usepackage[normalem]{ulem}
\usepackage{wasysym}
\usepackage{mathrsfs}
\usepackage{enumitem}
\usepackage{setspace}
\usepackage{array}
\usepackage{cancel}
\usepackage{tikz-cd}
\usepackage{colonequals}
\usepackage[titletoc]{appendix}
\usepackage{overpic}
\usepackage{float} 
\usepackage{tikz}
\usepackage{extarrows}
\usepackage{euscript}
\usepackage{makecell}

\newtheorem{theorem}{Theorem}[section]

\newtheorem{corollary}[theorem]{Corollary}
\newtheorem{lemma}[theorem]{Lemma}
\newtheorem{proposition}[theorem]{Proposition}

\theoremstyle{definition}

\newtheorem{defx}[theorem]{Definition}
\newenvironment{definition}
{\pushQED{\qed}\defx}
{\popQED\enddefx}

\newtheorem{assumptx}[theorem]{Assumption}
\newenvironment{assumpt}
{\pushQED{\qed}\assumptx}
{\popQED\endassumptx}

\newtheorem{examplex}[theorem]{Example}
\newenvironment{example}
{\pushQED{\qed}\examplex}
{\popQED\endexamplex}

\newtheorem{remarkx}[theorem]{Remark}
\newenvironment{remark}
{\pushQED{\qed}\remarkx}
{\popQED\endremarkx}

\numberwithin{equation}{section}

\DeclareMathOperator{\base}{bs}
\DeclareMathOperator{\Ch}{CF}

\DeclareMathOperator{\Chain}{Ch}

\DeclareMathOperator{\codim}{codim}

\DeclareMathOperator{\Cont}{Cont}

\DeclareMathOperator{\Crit}{Crit}
\DeclareMathOperator{\crit}{crit}

\DeclareMathOperator{\diam}{diam}
\DeclareMathOperator{\Diff}{Diff}

\DeclareMathOperator{\dist}{dist}

\DeclareMathOperator{\End}{End}
\DeclareMathOperator{\Ed}{Ed}
\DeclareMathOperator{\Emb}{Emb}

\DeclareMathOperator{\fun}{fun}
\DeclareMathOperator{\fin}{fin}

\DeclareMathOperator{\Fl}{Fl}
\DeclareMathOperator{\gr}{\mathfrak{gr}}
\DeclareMathOperator{\Gr}{Gr}
\DeclareMathOperator{\grad}{grad}

\DeclareMathOperator{\Hess}{Hess}

\DeclareMathOperator{\Hom}{Hom}
\DeclareMathOperator{\Id}{Id}
\DeclareMathOperator{\lmod}{lmod}
\DeclareMathOperator{\lfmod}{lfmod}
\DeclareMathOperator{\im}{Im}
\DeclareMathOperator{\Ind}{Ind}
\DeclareMathOperator{\inte}{int}

\DeclareMathOperator{\model}{model}

\DeclareMathOperator{\Ob}{Ob}
\DeclareMathOperator{\opp}{opp}
\DeclareMathOperator{\pre}{pre}

\DeclareMathOperator{\rank}{rank}
\DeclareMathOperator{\re}{Re}
\DeclareMathOperator{\refe}{ref}
\DeclareMathOperator{\rmod}{rmod}
\DeclareMathOperator{\rfmod}{rfmod}
\DeclareMathOperator{\res}{res}
\DeclareMathOperator{\Res}{Res}

\DeclareMathOperator{\Stp}{Strp}

\DeclareMathOperator{\trun}{trun}

\DeclareMathOperator{\Ve}{Ve}

\DeclareMathOperator{\HFF}{HF}

\newcommand{\BD}{\mathbb{D}}
\newcommand{\BF}{\mathbb{F}}

\newcommand{\BK}{\mathbb{K}}
\newcommand{\Z}{\mathbb{Z}}
\newcommand{\C}{\mathbb{C}}
\newcommand{\CP}{\mathbb{CP}}
\newcommand{\HH}{\mathbb{H}}

\newcommand{\R}{\mathbb{R}}

\newcommand{\T}{\mathbb{T}}

\newcommand{\CA}{\mathcal{A}}
\newcommand{\CB}{\mathcal{B}}

\newcommand{\CF}{\mathcal{F}}
\newcommand{\CG}{\mathcal{G}}

\newcommand{\CL}{\mathcal{L}}
\newcommand{\CN}{\mathcal{N}}
\newcommand{\CQ}{\mathcal{Q}}

\newcommand{\CT}{\mathcal{T}}
\newcommand{\CY}{\mathcal{Y}}

\newcommand{\D}{\mathcal{D}}
\newcommand{\E}{\mathcal{E}}

\newcommand{\K}{\mathcal{K}}
\newcommand{\M}{\mathcal{M}}
\newcommand{\NR}{\mathcal{R}}
\newcommand{\Pa}{\mathcal{P}}
\newcommand{\Sch}{\mathcal{S}}
\newcommand{\SH}{\mathcal{H}}
\newcommand{\SO}{\mathcal{O}}

\newcommand{\V}{\mathcal{V}}
\newcommand{\W}{\mathcal{W}}
\newcommand{\X}{\mathcal{X}}
\newcommand{\MF}{\mathcal{MF}}

\newcommand{\sA}{\EuScript A}
\newcommand{\sB}{\EuScript B}
\newcommand{\sC}{\EuScript C}
\newcommand{\sD}{\EuScript D}
\newcommand{\sE}{\EuScript E}
\newcommand{\sF}{\EuScript F}
\newcommand{\sG}{\EuScript G}
\newcommand{\sI}{\EuScript I}
\newcommand{\sK}{\EuScript K}
\newcommand{\sL}{\EuScript L}
\newcommand{\sM}{\EuScript M}
\newcommand{\sN}{\EuScript N}
\newcommand{\sP}{\EuScript P}
\newcommand{\sQ}{\EuScript Q}
\newcommand{\sR}{\EuScript R}
\newcommand{\sS}{\EuScript S}
\newcommand{\sT}{\EuScript T}
\newcommand{\sU}{\EuScript U}
\newcommand{\sV}{\EuScript V}
\newcommand{\sX}{\EuScript X}
\newcommand{\sY}{\EuScript Y}

\newcommand{\F}{\mathfrak{F}}

\newcommand{\fa}{\mathfrak{a}}
\newcommand{\fb}{\mathfrak{b}}
\newcommand{\fc}{\mathfrak{c}}

\newcommand{\fn}{\mathfrak{n}}

\newcommand{\fo}{\mathfrak{o}}
\newcommand{\p}{\mathfrak{p}}

\newcommand{\w}{\mathfrak{w}}

\newcommand{\FC}{\mathfrak{C}}

\newcommand{\SC}{\mathscr{C}}

\newcommand{\cM}{\widecheck{\mathcal{M}}}

\newcommand{\bpartial}{\bar{\partial}}

\newcommand{\Step}{\textit{Step }}
\newcommand{\Case}{\textit{Case }}

\newcommand{\supp}{\text{supp}}
\newcommand{\half}{\frac{1}{2}}

\newcommand{\embed}{\hookrightarrow}
\newcommand{\pt}{\partial_t }
\newcommand{\ps}{\partial_s}

\newcommand{\px}{\partial_x }

\newcommand{\Th}{\mathbf{Th}}

\newcommand{\preceqdot}{\mathrel{\ooalign{$\preceq$\cr
			\hidewidth\raise0.15ex\hbox{$\cdot\mkern0.5mu$}\cr}}}

\newcommand{\sto}{\xrightarrow{s}}

\title{The Complex Gradient Flow Equation and Seidel's Spectral Sequence}
\author{Donghao Wang}

\date{\today}

\address{Simons Center for Geometry and Physics, Stonybrook, NY 11794, USA}
\email{dwang@scgp.stonybrook.edu}

\begin{document}

	\maketitle
	
		\[
	\textit{Dedicated to Professor Tomasz Mrowka on occasion of his 60th birthday}
	\]
	
		\begin{abstract} Following the proposals of Donaldson-Thomas, Haydys and Gaiotto-Moore-Witten, we give a construction of Fukaya-Seidel categories for a suitable class of Morse Landau-Ginzburg models using the complex gradient flow equation, which has the potential for generalization to some infinite dimensional examples. 
		
		In the course of this construction, we give an alternative proof to Seidel's spectral sequence for Lagrangian Floer cohomology, which can be viewed as a finite dimensional model for a potential bordered monopole Floer theory. The key observation is that under a neck-stretching limit, this complex gradient flow equation produces a natural geometric filtration on the Floer cochain complex. The resulting spectral sequence is then identified with Seidel's original one. 
		
	\end{abstract}

	\tableofcontents

\noindent\rule{\textwidth}{0.5pt}
\part{Introduction}

\section{Introduction}
\subsection{Motivation: towards an infinite dimensional Picard-Lefschetz theory}
It is the theme of Morse theory to understand the algebraic topology of a finite dimensional closed oriented manifold $N$ in terms of a Morse function $f: N\to \R$. There are several approaches towards this goal, with different flavors and scopes of generality. The one we are most interested in is via the so-called Morse-Smale-Witten complex \cite{W82}, which is freely generated by critical points of $f$ as a vector space and graded by the Morse index. Let $\nabla f$ denote the gradient vector field of $f$ with respect to a given Riemannian metric. If the flow of $-\nabla f$ satisfies certain generic conditions, then the differential map on this complex is defined by counting flowlines of $-\nabla f$ connecting critical points of adjacent Morse indices:
\begin{equation}\label{Intro.E.1}
p: \R_t\to N,\ \pt p(t)+\nabla f(p(t))=0.
\end{equation}

It was Floer's original idea \cite{F88,F89} that this approach can be generalized to some infinite dimensional manifolds such as the action functional on the loop space of a symplectic manifold, resolving the Arnold conjecture \cite{F89-Arnold} for a wide-ranging family. On the gauge theory side, Floer \cite{F88-Ins} exploited this idea to define the instanton Floer cohomology for integral homology 3-spheres which categorifies Casson's invariant and leads to a gluing theorem for Donaldson's invariants. Nowadays Floer cohomology is referred to any type of invariants defined as an infinite dimensional Morse cohomology and has been a powerful tool in tackling many important problems in low dimensional topology and symplectic topology. 

\medskip

As the complex analogue of Morse theory, it is the theme of Picard-Lefschetz theory to understand the symplectic topology of a complete K\"{a}hler manifold $M$ in terms of a holomorphic Morse function $W: M\to \C$ (called superpotential). Such a pair $(M,W)$ is called a Landau-Ginzburg model and is used extensively in mirror symmetry to formulate the so-called Calabi-Yau/Landau-Ginzburg correspondence \cite{VW89,M90,Witten93,FJR13}. A holomorphic Morse function exhibits very different properties from the real ones. If all critical points of $W$ are non-degenerate and finite, then they must concentrate on the same cohomological grading due to the holomorphicity of $W$. For a generic $e^{i\theta}\in S^1$, the Morse-Smale-Witten complex of $\re(e^{-i\theta}W)$ is well-defined but boring, as the differential map on this complex is always trivial. The correct symplectic invariant to be defined for $(M,W)$, due to the work of Seidel \cite{S08,Sei12}, is the so-called Fukaya-Seidel category of $(M,W)$, which in our context is a finite directed $A_\infty$-category whose objects are in bijection with the critical set of $W$, and as such is a more refined invariant than the Morse cohomology of $\re(e^{-i\theta}W)$. 


\medskip

The theme of this paper is to extend Floer's idea further to the case of holomorphic Morse functions and develop a prototype for an infinite dimensional Picard-Lefschetz theory. This allows us to transport existing techniques/results from symplectic topology (as finite dimensional toy models) to understand gauge theory (the infinite dimensional case). There are several instances in the literature where infinite dimensional Landau-Ginzburg models are considered.

\begin{example}\label{Intro.EX.1} Let $M$ be the space of $\bpartial$-operators on a given complex vector bundle $E$ on a Calabi-Yau 3-fold $\CY$, and $W$ the holomorphic Chern-Simons functional; a critical point of $W$ corresponds to an integrable operator giving rise to a genuine holomorphic structure $E$. This example motivates the construction of Donaldson-Thomas invariants \cite{Th00} as a counting of critical points of $W$. 
\end{example}
\begin{example}\label{Intro.EX.2} Let $M$ be the space of $SL(2,\C)$-connections on a given closed oriented 3-manifold $Z$, and $W$ the complex Chern-Simons functional; a critical point of $W$ corresponds to a flat $SL(2,\C)$-connection, which gives a representation of $\pi_1(Z)$ into $SL(2,\C)$. This example motivates the work of Abouzaid-Manolescu \cite{AM20} on a sheaf theoretic model of $SL(2,\C)$-Floer homology. 
\end{example}
\begin{example}\label{Intro.EX.3}Let $M$ be the Seiberg-Witten configuration space on a closed Riemann surface $\Sigma$, and $W$ the Dirac functional \cite{Wang202}. This example motivates the author's earlier work on monopole Floer cohomology for 3-manifolds with toroidal boundary \cite{Wang20, Wang203}. 
\end{example} 
\begin{example}\label{Intro.EX.4}Take a holomorphic symplectic manifold $\CN$ with a pair of holomorphic Lagrangian submanifolds $\CL_0, \CL_1$. Let $M$ be the space of paths $\gamma: [0,1]_s\to \CN$ with $\gamma(j)\in \CL_j, j=0,1$, and $W$ the complex action functional. This example motivates the ongoing project by Kontsevich-Soibelman on holomorphic Floer theory; see \cite{KS22}. 
\end{example}


The idea of developing Picard-Lefschetz theory in the infinite dimensional case originates with the work of Donaldson-Thomas \cite{DT96}, with more detailed plans proposed by Haydys \cite{Haydys15} and Gaiotto-Moore-Witten \cite{GMW15, GMW17}. The main obstacle in this line of thoughts involves generalizing Lagrangian Floer cohomology for infinite dimensional Lagrangian submanifolds \cite{NguyenI, NguyenII}. As the Fukaya-Seidel category of $(M,W)$ comprises only a very special class of Lagrangian submanifolds called thimbles, this obstacle may be circumvented by considering the complex analogue of \eqref{Intro.E.1} on $\C=\R_t\times\R_s$ with the real operator $\pt$ replaced by the complex operator $\pt+i\ps$ and $f$ by the Hamiltonian function $\im (e^{-i\theta}W)$:
\begin{equation}\label{Intro.E.2}
P: \R_t\times \R_s\to M,\ \pt P+J\ps P+\nabla \im (e^{-i\theta}W)(P(t,s))=0. 
\end{equation}

The $\theta=0$ version of \eqref{Intro.E.2} has been studied in the literature under many names: \textit{the complex gradient flow equation} \cite{DT96}, or the Witten equation \cite{FJR13} as used extensively in the quantum singularity theory (FJRW theory). In this paper, \eqref{Intro.E.2} will be called \textit{the $\theta$-instanton equation} following the convention of  \cite[Chapter 14]{GMW15} (note that they used the Greek letter $\zeta$ instead of $\theta$). The moral to take away from this paper is that \eqref{Intro.E.2} seems to be a more natural equation to consider in the realm of Landau-Ginzburg models, as opposed to the Cauchy-Riemann equation in the traditional symplectic literature, though \eqref{Intro.E.2} is only a perturbed version of the latter. As the first hint, this $\theta$-instanton equation \eqref{Intro.E.2} is the toy model to understand the 2-dimensional reduction of many gauge theoretic equations, as summarized in the list below. 
\[
 \eqref{Intro.E.2}  \text{ with }(M,W) \text{ as in }
\left\{
\begin{array}{ll}
\text{Example } \ref{Intro.EX.1} &\Rightarrow \text{  the Spin(7)-instanton equation on }\R^2\times \CY,\\
\text{Example } \ref{Intro.EX.2} &\Rightarrow \text{  the Haydys-Witten equations on }\R^2\times Z,\\
\text{Example } \ref{Intro.EX.3} &\Rightarrow \text{  the Seiberg-Witten equations on }\R^2\times \Sigma,\\
\text{Example } \ref{Intro.EX.4}  & \Rightarrow \text{  the Fueter equation on }\R^2\times [0,1].
\end{array}
\right.
\]

Donaldson-Thomas' original proposal \cite{DT96} is to construct a Fukaya-Seidel category for the holomorphic Chern-Simons functional by counting Spin(7)-instantons on $\R^2\times \CY$ and define more refined deformation invariants of the Calabi-Yau 3-fold $\CY$. The Donaldson-Thomas invariant is then interpreted as the number of objects of this directed $A_\infty$-category. The work of Haydys \cite{Haydys15} and Gaiotto-Moor-Witten \cite{GMW15,GMW17} is motivated by Witten's proposal \cite{Witten12} to define Khovanov homology for knots in a general 3-manifold by studying the Kapustin-Witten equations \cite{KW07}, which is the dimensional reduction of the Haydys-Witten equations on a 4-manifold.

\medskip

Following this line of thoughts, this paper is devoted to a detailed construction of Fukaya-Seidel categories for Landau-Ginzburg models $(M,W)$ satisfying a tameness condition (see Definition \ref{D1.1}; this condition is inspired by \cite{FJY18}), which has the potential for generalization to the infinite dimensional Example \ref{Intro.EX.3}; other examples remain difficult to tackle due to a compactness issue: a local compactness property (Lemma \ref{L1.6}) which is well-known for the Seiberg-Witten equations fails in general for other equations in this list \cite{Taubes13}. The construction of the Fukaya-Seidel category for the Dirac superpotential along with some concrete computations of examples will appear in a companion paper in the future.


\medskip

The present paper focuses on the finite dimensional case and explains in detail the geometric input required for such a construction. We remark that in the finite dimensional case there have been at least two approaches to construct Fukaya-Seidel categories for Morse Landau-Ginzburg models: either via vanishing cycles in the fiber of $W:M\to\C$ \cite{S08} or via the partially wrapped Fukaya category \cite{Sei12,Zac19, GPS20}. The geometric inputs required in these approaches, however, are not quite available in the infinite dimension case. On the other hand, our construction, though inspired by their works, is targeted specifically at the infinite dimensional generalization and is somewhat inconvenient for finite dimensional applications, as some formal structures are lost along this line. Nevertheless, as a concrete application, we give an alternative proof to Seidel's spectral sequence within our framework, which serves as a finite dimensional model for a potential bordered monopole Floer theory.

\subsection{Towards a bordered monopole Floer theory}\label{Intro.Sec.2} The Seiberg-Witten Floer cohomology is a powerful invariant introduced by Kronheimer-Mrowka \cite{Bible} for any closed oriented 3-manifold $Z$, which fits into a (3+1) TQFT and encodes important topological information about $Z$. This paper is motivated by an attempt to develop a bordered monopole Floer theory which computes a version of this Floer cohomology in terms of a splitting of $Z$. Suppose that $Z=Z_0 \cup_{\Sigma} Z_1$ is separated by a connected closed surface $\Sigma$, where $Z_0$ and $Z_1$ are 3-manifolds  with boundary $\partial Z_0\cong \Sigma\cong (-\partial Z_1)$. Let $M(\Sigma)$ denote the infinite dimensional K\"{a}hler manifold in Example \ref{Intro.EX.3} and $W_{\D}$ the Dirac superpotential. The Seiberg-Witten gauge theory is then tied to symplectic topology by a theorem of Nguyen. 

\begin{theorem}[\cite{NguyenI}]\label{Intro.T.13} The solution space of the 3-dimensional Seiberg-Witten equations on $Z_0$ (resp. $Z_1$) restricts to an infinite dimensional (possibly immersed) Lagrangian submanifold $\CL_0$ (reps. $\CL_1$) in $M(\Sigma)$ on which $\re W_{\D}$ is bounded above (resp. below). 
\end{theorem}

In this theorem one has to fix a cylindrical metric on $Z_0$ (resp. on $Z_1$). Let $g_R$ be the metric on $Z$ obtained by gluing $Z_0$, $Z_1$ with a long neck $[0,R]_s\times \Sigma$:
\[
(Z, g_R)=Z_0\cup [0,R]_s\times \Sigma\cup Z_1. 
\]
The dictionary between symplectic topology and the Seiberg-Witten gauge theory is then summarized by the table below. 

\medskip

\begin{center}
	\begin{tabular}{|c|c|}
		\hline
		Symplectic Topology & the Seiberg-Witten Gauge Theory\\ \hline
		$(M(\Sigma), W_{\D})$ & $\Sigma$  \\ \hline
		$\CL_0$&$Z_0$ \\ \hline
		$\CL_1$&$Z_1$ \\ \hline
		\makecell{the space of paths $\gamma: [0,R]_s\to M$\\ with $\gamma(0)\in \CL_0$ and $\gamma(R)\in \CL_1$ } &
		\makecell{	
			the Seiberg-Witten \\ configuration space on $(Z,g_R)$}\\ \hline
		\makecell{the symplectic action \\functional $\CA_{W,R}$ perturbed by $W$} & \makecell{the Chern-Simons-Dirac \\ functional on $(Z,g_R)$}\\ \hline
		\makecell{the complex gradient \\ flow equation $\eqref{Intro.E.2}$ on $\R_t\times [0,R]_s$} &\makecell{the Seiberg-Witten \\equations on $\R_t\times (Z,g_R)$}\\ \hline
		\makecell{the Floer cochain complex \\ 	$\Ch^*_R(\CL_0,\CL_1)$}
	 &  \makecell{the Seiberg-Witten Floer \\cochain complex of $(Z,g_R)$}\\ \hline
	\end{tabular}
\end{center}

\medskip

In contrast with the Atiyah-Floer conjecture \cite{Ati87} which assigns to each oriented surface $\Sigma$ the Fukaya category of a compact symplectic manifold, the picture we advertise here assigns to each $\Sigma$ the Fukaya-Seidel category of an infinite dimensional Landau-Ginzburg model. The main result of this paper suggests that under the neck-stretching limit as $R\to\infty$, \textbf{the energy filtration induced by the Chern-Simons-Dirac functional} on the Seiberg-Witten Floer cochain complex of $(Z,g_R)$ \textbf{has a particularly nice structure}, which may be understood concretely in terms of $A_\infty$-invariants associated to $Z_0$ and $Z_1$. We remark that the Dirac superpotential $W_{\D}$ is not Morse, and one has to perturb the Chern-Simons-Dirac functional by a closed 2-form on $Z_R$ in order to see this special structure, so only a variant of the reduced monopole Floer cohomology can be recovered through this approach. In next section we explain in some detail the geometric origin of this filtration and why Seidel's result is viewed as a toy model for this bordered theory.

\subsection{Main Results: Seidel's Spectral Sequence} From now on we focus on the finite dimensional case and review the basic setup. Throughout this paper we work with a base field $\BK$ of characteristic $2$. Let $(M,W)$ be any tame Landau-Ginzburg model in the sense of Definition \ref{D1.1}, so $M$ is exact with $c_1(TM, J_M)=0$. Let $L=\re W$ and $H=\im W$ denote the real and imaginary part of $W$ respectively. Assume that the critical set $\Crit(W)=\{x_1,\cdots, x_m\}$ is finite and ordered such that 
\begin{equation}\label{Intro.E.5}
H(x_1)<H(x_2)<\cdots <H(x_m). 
\end{equation}
Denote by $S_n$ (resp. $U_n$) the stable (resp. unstable) submanifold of $L$ associated to $x_n\in \Crit(W)$, then $W(S_n)\subset \C$ (resp. $W(U_n)$) is a ray emanating from the critical value $W(x_n)$ parallel to the real axis; see Figure \ref{Pic32} below. $S_n$ (resp. $U_n$) is also the Lefschetz thimble along this ray, which can be made into a graded exact Lagrangian submanifold. Then Seidel's theorem says the following.

\begin{theorem}[{\cite[Corollary 18.27]{S08}}]\label{Intro.T.7} For any pair of compact exact graded Lagrangian submanifolds $X,Y\subset M$, there is a spectral sequence:
	\begin{equation}\label{Intro.E.3}
	\bigoplus_{n=1}^m\HFF^*(U_n, Y)\otimes \HFF^*(X, S_n)\rightrightarrows \HFF^*(X,Y). 
	\end{equation}
\end{theorem} 

Seidel's original proof of Theorem \ref{Intro.T.7} relies on his deep generating theorem \cite[Corollary 18.25]{S08} for the Fukaya category of $(M,W)$, so \eqref{Intro.E.3} is only constructed algebraically (though this proof only works over a field of char $\neq2$,  we shall ignore this technical point in this exposition). The advantage of this approach is that all higher differentials of this spectral sequence can be computed step by step using his algebraic recipes, so \eqref{Intro.E.3} is viewed as a computational tool for the Floer cohomology group $\HFF^*(X,Y)$. In the sequel we shall refer to \eqref{Intro.E.3} as the algebraic spectral sequence. 

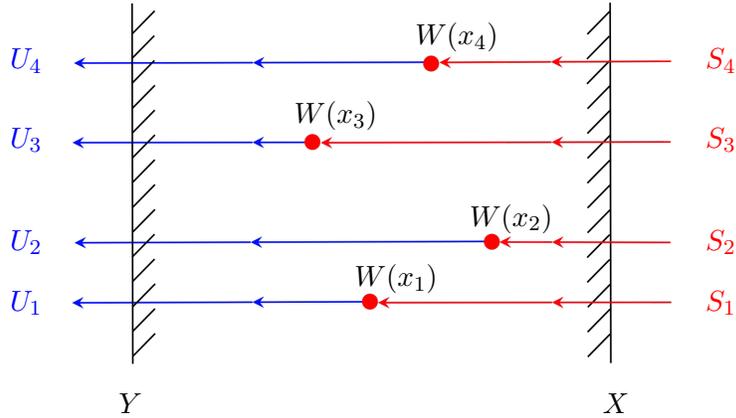
\begin{figure}[H]
	\centering
	\begin{overpic}[scale=.15]{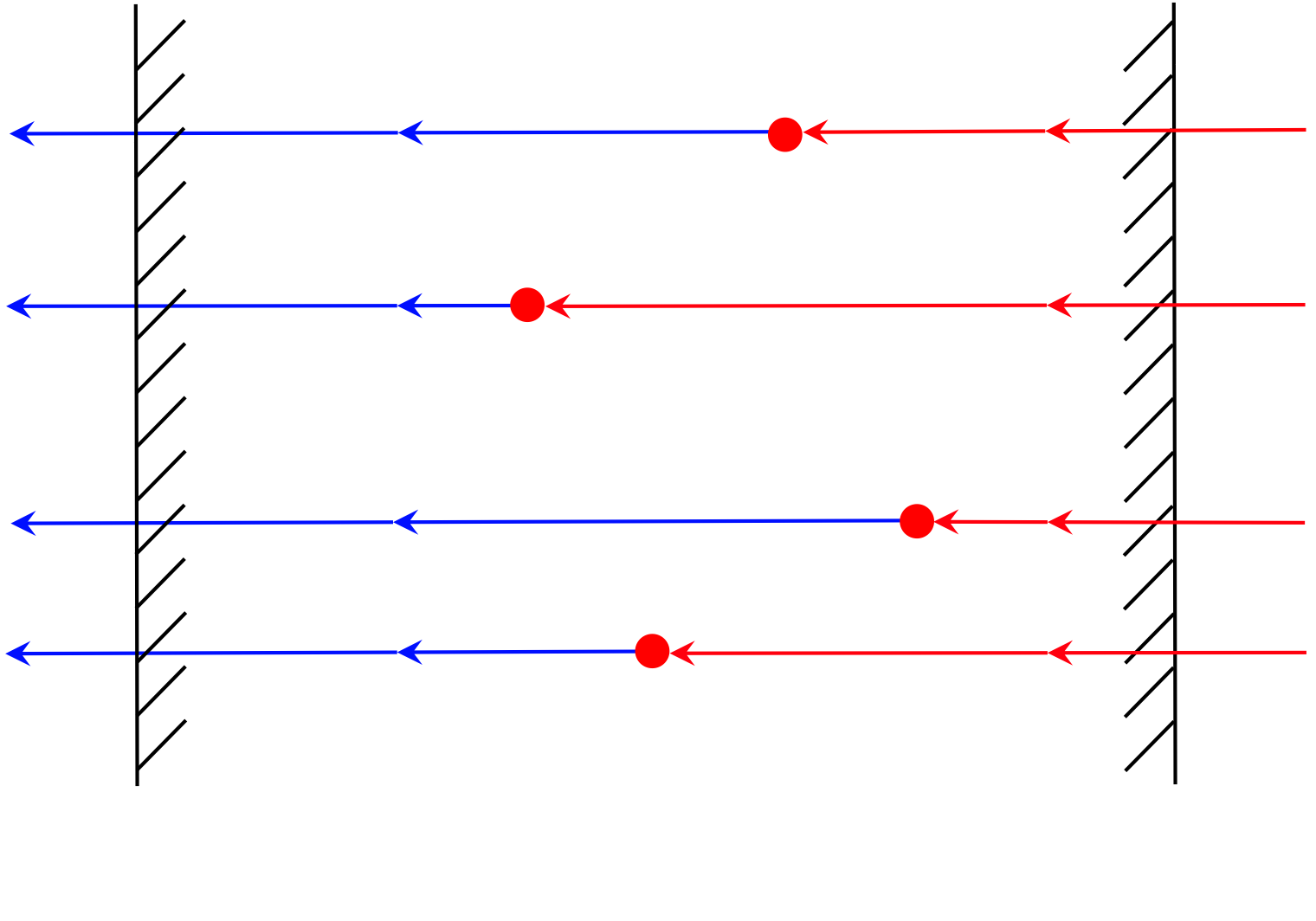}
		\put(-10,19){\color{blue}$U_1$}
		\put(-10,29){\color{blue}$U_2$}
		\put(-10,46){\color{blue}$U_3$}
		\put(-10,59){\color{blue}$U_4$}
		\put(105,19){\color{red}$S_1$}
		\put(105,29){\color{red}$S_2$}
		\put(105,46){\color{red}$S_3$}
		\put(105,59){\color{red}$S_4$}
		\put(47,23){$W(x_1)$}
		\put(66,33){$W(x_2)$}
		\put(37,50){$W(x_3)$}
		\put(57,63){$W(x_4)$}
		\put(8,2){$Y$}
		\put(88,2){$X$}
	\end{overpic}	
	\caption{The images of thimbles under $W:M\to \C$.}
\label{Pic32}
\end{figure}

We now explain the geometric origin of \eqref{Intro.E.3} from the complex gradient flow equation \eqref{Intro.E.2}. The Floer cohomology group $\HFF^*(X, Y)$ can be defined by counting $\theta$-instantons (with $\theta=0$) on an infinite strip $\R_t\times [0,R]_s$ with Lagrangian boundary conditions, where the stretching parameter $R>0$ is allowed to vary:
\begin{equation}\label{Intro.E.4}
\left\{
\begin{array}{rl}
P: \R_t\times [0, R]_s&\to M,\\
\pt P+J\ps P+\nabla H&=0,\\
P(t,R)&\in Y, \\
P(t, 0)&\in X.
\end{array}
\right.
\end{equation}
Denote the resulting Floer complex by $\Ch^*_R(X,Y)$. We may scale the domain to be $\R_t\times [0,1]_s$ so that \eqref{Intro.E.4} takes the form 
\begin{equation}\label{Intro.E.13}
\left\{
\begin{array}{rl}
\widetilde{P}: \R_t\times [0, 1]_s&\to M,\\
\pt \widetilde{P}+J\ps \widetilde{P}+R\cdot\nabla H&=0,\\
\widetilde{P}(t,1)&\in Y, \\
\widetilde{P}(t, 0)&\in X. 
\end{array}
\right.
\end{equation}
with $\widetilde{P}(t,s)\colonequals P(R\cdot t,R\cdot s), (t,s)\in \R_t\times [0,1]_s$. In the second case \eqref{Intro.E.13}, one may even set $R=0$ and return to the classical realm of $J$-holomorphic curves. 

\medskip

However, the secret of the spectral sequence \eqref{Intro.E.3} is revealed by taking $R\gg 1$ in \eqref{Intro.E.4}. To this end, we wish to understand the energy filtration on $\Ch^*_R(X,Y)$. A generator of $\Ch^*_R(X,Y)$ (also called a soliton) is a path $p_R: [0, R]_s\to M$ following the flow of $-\nabla L$, since $\nabla H=J\nabla L$:
\begin{equation}\label{Intro.E.19}
J\ps p_R+\nabla H=J(\ps p_R+\nabla L)=0 \text{ with }p_R(0)\in X \text{ and } p_R(R)\in Y.
\end{equation}
The Hamiltonian function $H$ is constant along $p_R$. For $R\gg 1$, any soliton $p_R$ must approximate a broken flowline of $-\nabla L$, and the condition \eqref{Intro.E.5} ensures that only one break is possible for this approximation (think of the red and blue curves associated to the same critical point in Figure \ref{Pic32} as a concatenation of such a broken flowline projected on $\C$). Thus the Floer complex $\Ch^*_R(X,Y)$ is a direct sum
\begin{equation}\label{Intro.E.6}
\Ch^*_R(X,Y)=\bigoplus_{n=1}^m V_R^n,\ R\gg 1
\end{equation}
where $V_R^n$ is the subspace generated by solitons approximating the $n$-th critical point $x_n\in \Crit(W)$ in the middle of the interval $[0,R]_s$. 

\medskip

Our next step is to understand the Floer differential $\partial$ acting on the direct sum \eqref{Intro.E.6}. $\Ch^*_R(X,Y)$ is the Morse-Smale-Witten complex of a perturbed symplectic action functional $\CA_{W,R}$ on the path space between $X$ and $Y$, and each soliton is a critical point of $\CA_{W,R}$. A short computation (Lemma \ref{GF.L.4}) shows that 
\begin{equation}\label{Intro.E.18}
\CA_{W,R}(p_R)=R\cdot H(x_n)+\SO(1), R\gg 1,
\end{equation}
if $p_R$ is a generator of $V_R^n$, so this value is blowing up linearly as $R\to\infty$ with the leading order term given by $H(x_n)$. Meanwhile, the Floer differential $\partial$ on $\Ch^*_R(X,Y)$ can only increase the value of $\CA_{W,R}$ and hence the value of $H(x_n)$. This implies that the differential map $\partial$ acting on \eqref{Intro.E.6} is a lower triangular matrix for $R\gg 1$. Thus the complex $\Ch_R^*(X,Y)$ is endowed with an decreasing filtration:
\[
\Ch_R(X,Y)=\Ch_R^{(1)}(X,Y)\supset \cdots\supset\Ch_R^{(m+1)}(X,Y)=\{0\},
\]
with 
\begin{equation}\label{Intro.E.11}
\Ch_R^{(n)}(X,Y)\colonequals\bigoplus_{j=n}^m V^j_R,\ 1\leq n\leq m+1, 
\end{equation}
and which induces a spectral sequence 
\begin{equation}\label{Intro.E.7}
\bigoplus_{n=1}^m H(V_R^n)\rightrightarrows H(\Ch^*_R(X,Y))\cong \HFF^*(X,Y).
\end{equation}

We refer to \eqref{Intro.E.7} as the geometric spectral sequence. Using a vertical gluing theorem (Section \ref{SecVT}), one verifies that $H(V_R^n)$ is isomorphic to the tensor product $\HFF_\natural^*(U_n, Y)\otimes\HFF_\natural^*(X, S_n)$, where $\HFF_\natural^*$ is a variant of $\HFF^*$ (see Section \ref{Intro.Sec5}) defined by counting $\theta$-instantons on the upper/lower half planes. Another application of the vertical gluing theorem proves that in fact $\HFF^*\cong \HFF_\natural^*$ (Section \ref{SecSS.9}), so the geometric spectral sequence \eqref{Intro.E.7} has the same $E_1$-page as the algebraic spectral sequence \eqref{Intro.E.3}. The main result of this paper then says that their higher pages are also isomorphic.

\begin{theorem}[The main result: the preliminary version; Theorem \ref{SS.T.8}]\label{Intro.T.6} The geometric spectral sequence \eqref{Intro.E.7} is isomorphic to Seidel's algebraic spectral sequence \eqref{Intro.E.3}. 
\end{theorem}

\begin{remark}\label{Intro.R.7} The spectral sequence \eqref{Intro.E.3} was first suggested by Donaldson based on the following observation \cite[Remark 18.28]{S08}: if one considers the product manifold $M^-\times M$, where the sign of the symplectic form is reversed on the first factor, as well as the Hamiltonian flow of the function $\widetilde{H}(x,y)=H(x)+H(y)$, then part of the diagonal $\Delta_M\subset M\times M$ may diverge within finite time, while the rest will approximate the disjoint union $\coprod_{n} S_n \times U_n$, suggesting a Lagrangian cobordism:
	\begin{equation}\label{Intro.E.8}
	\Delta_M\xrightarrow[\text{Hamiltonian Flow}]{\text{Lag. Cob.}}\coprod_{n=1}^m S_n \times U_n,
	\end{equation}
The general framework of Biran-Cornea \cite{BC13} will then produce a spectral sequence by pairing \eqref{Intro.E.8} with the product Lagrangian submanifold $X\times Y\subset M^-\times M$. However, this Lagrangian cobordism is never constructed explicitly in the literature. In our story this Hamiltonian flow is replaced by taking a large stretching parameter $R\gg 1$ in the equation \eqref{Intro.E.4}, and Biran-Cornea's framework is replaced by the geometric spectral sequence \eqref{Intro.E.7}.
\end{remark}

To prove Theorem \ref{Intro.T.6}, one has to construct Seidel's algebraic spectral sequence \eqref{Intro.E.3} within our framework. To this end, we construct in this paper using a variant of the $\theta$-instanton equation \eqref{Intro.E.2}:
	\begin{itemize}
	\item a finite directed $A_\infty$-category $\sA$ using all stable thimbles $S_1,\cdots, S_m$, $m=|\Crit(W)|$, called the Fukaya-Seidel category of $(M,W)$;
	\item a finite directed $A_\infty$-category $\sB$ using all unstable thimbles $U_m,\cdots, U_1$;
	\item a diagonal bimodule ${}_{\sA}\Delta_{\sB}$ with the property 
\begin{equation}\label{Intro.E.15}
\Delta(S_k, U_j)=\left\{
\begin{array}{cl}
\BK &\text{ if }j=k,\\
0 &\text{otherwise}. 
\end{array}
\right.
\end{equation}
\end{itemize}
For any compact exact graded Lagrangian submanifold $X$, we construct 
\begin{itemize}
\item a left $\sA$-module ${}_{\sA} X$ and a right $\sB$-module $X_{\sB}$. 
\end{itemize}

Theorem \ref{Intro.T.6} is then refined as follows.
\begin{theorem}[The main result: the complete version]\label{Intro.T.8} For any finite dimensional tame Landau-Ginzburg model $(M,W)$ $($so $M$ is exact with $2c_1(TM, J)=0)$ and any compact exact graded Lagrangian submanifolds $X,Y$, the following holds:
\begin{enumerate}[label=$(\arabic*)$]
	\item\label{T.7.1} $(\text{Theorem } \ref{FS.T.3})$ ${}_{\sA}\Delta_{\sB}$ induces a cohomological full and faithful embedding $\sA\embed \sQ_r\colonequals \rfmod(\sB)$, where $\sQ_r$ is the dg-category of right $\sB$-modules;
	\item\label{T.7.2} $(\text{Theorem }\ref{SS.T.2})$ for $R\gg 1$,  there is a quasi-isomorphism 
	\begin{equation}\label{Intro.E.9}
	\Ch^*_R(X, Y)\to \hom_{\sQ_r}(X_{\sB}, Y_{\sB}),
	\end{equation}
	which identifies the geometric filtration on the left with the algebraic filtration on the right;
	\item\label{T.7.3} $($\text{Corollary }\ref{SS.C.3}; the wall-crossing formula$)$ there is a quasi-isomorphism in the dg-category $\sQ_r$
	\begin{equation}\label{Intro.E.10}
	X_{\sB}\to \hom_{\sP_l}({}_{\sA}X, {}_{\sA}\Delta_{\sB})
	\end{equation}
	where $\sP_l$ denotes the dg-category of left $\sA$-modules. 
\end{enumerate}
\end{theorem}

Theorem \ref{Intro.T.8} \ref{T.7.1} and the property \eqref{Intro.E.15} imply that $(\sA,\sB)$ forms Koszul-duality pair, so one determines the other. The algebraic filtration on the complex $\hom_{\sQ_r}(X_{\sB}, Y_{\sB})$ is induced from an increasing filtration on $Y_{\sB}$ (see \cite[Remark 5.25]{S08} or Section \ref{SecAF.5}), which defines the algebraic spectral sequence \eqref{Intro.E.3}. Thus Theorem \ref{Intro.T.6} follows from Theorem \ref{Intro.T.8} \ref{T.7.2}.

\begin{remark} Let $\sF_c(M)$ denote the Fukaya category of compact exact graded Lagrangian submanifolds in $M$. Theorem \ref{Intro.T.8} \ref{T.7.2} may be used to show that there is a cohomologically full and faithful functor
\begin{equation}\label{Intro.E.12}
\sF_c(M)\to \sQ_r,
\end{equation}
reproving Seidel's result \cite[Corollary 18.25]{S08}. The author confesses that the analysis of this paper is not enough to construct this $A_\infty$-functor \eqref{Intro.E.12}. This failure may be remedied by assuming further that $(M,\omega_M)$ is a Liouville manifold (this is not a requirement of tameness) and by combining the compactness argument of this paper with that in \cite[Section 7]{S08}. Nevertheless, this is considered as a relatively minor drawback comparing to our potential infinite dimensional application.
\end{remark} 

Composing \eqref{Intro.E.9} with a quasi-inverse of \eqref{Intro.E.10}, we obtain that 

\begin{corollary}[{Corollary \ref{SS.C.5}}]\label{Intro.C.10} Under the assumptions of Theorem \ref{Intro.T.8}, we have a quasi-isomorphism for all $R>0$:
	\begin{equation}\label{Intro.E.14}
\Ch^*_R(X, Y)\to \hom_{\sQ_r}(\hom_{\sP_l}({}_{\sA}X, {}_{\sA}\Delta_{\sB}), Y_{\sB}).
	\end{equation}
\end{corollary}

\begin{remark} In light of Remark \ref{Intro.R.7}, one may think of the diagonal bimodule ${}_{\sA}\Delta_{\sB}$ as the algebraic invariant associated to $\Delta_M\subset M\times M$, and the right hand side of \eqref{Intro.E.14} should be viewed as the algebraic pairing between $\Delta_M$ and $X\times Y$. Thus the quasi-isomorphism \eqref{Intro.E.14} can be read as: the geometric pairing of $\Delta_M$ and $X\times Y$ is ``equal'' to their algebraic pairing.
\end{remark}

\subsection{The idea of Theorem \ref{Intro.T.8}} Although the analogue of Theorem \ref{Intro.T.8} is well-known in the literature, we emphasize that our strategy is new and independent of any existing proof of the generating theorem \cite{S08,GPS20}, which we now describe. As in the case of $\Ch^*_R(X,Y)$, the right $\sB$-module $X_{\sB}$ and the diagonal bimodule $_{\sA}\Delta_{\sB}$ can be deformed (up to quasi-isomorphisms) into a filtered $\sB$-module and a filtered $(\sA,\sB)$-bimodule respectively under a neck-stretching limit with $R\gg 1$: 
\begin{align*}
0=X_{\sB,R}^{(0)}\subset \cdots\subset X_{\sB,R}^{(n)}\subset X_{\sB,R}^{(n+1)}\subset\cdots\subset X_{\sB,R}^{(m)}&=X_{\sB,R}, \\
0=\Delta_R^{(0)}\subset \cdots \subset\Delta_R^{(n)}\subset \Delta_R^{(n+1)}\subset\cdots\subset \Delta_R^{(m)}&=\Delta_R.
\end{align*}
These increasing filtrations are induced by the symplectic action functionals and depend a priori on the stretching parameter $R\gg 1$ (one has to work with the $\theta$-instanton equation \eqref{Intro.E.2} with $\theta=\pi$ here). The $n$-th filtered submodule/sub-bimodule involves the first $n$ critical points $x_1,\cdots, x_n$ of $W$ in \eqref{Intro.E.5}. The chain map and $A_\infty$-homomorphism defined naturally from the geometry preserve these filtrations, so \eqref{Intro.E.9} and \eqref{Intro.E.10} are refined respectively as 
\begin{align*}
\Ch_R^{(n+1)}(X, Y)&\to \hom_{\sQ_r}(X_{\sB, R}/X_{\sB, R}^{(n)}, Y_{\sB}),\\
X_{\sB, R}^{(n)}&\to \hom_{\sP_l}({}_{\sA}X, \Delta_R^{(n)}),\ 0\leq n\leq m. 
\end{align*}

The idea behind this proof is the simple algebraic fact that a map between spectral sequences is an isomorphism between the $E_\infty$-page if and only if it is on the $E_1$-page. Thus one may pass to the associated graded complexes to verify that these maps are quasi-isomorphisms. In the latter case, only one critical point of $W$ is involved in the picture, and these maps can be understood concretely using a vertical gluing theorem by taking $R\to \infty$.

\begin{remark} Since $S_j\cap S_k=\emptyset$ for $j\neq k$, we have to perturb the stable and unstable thimbles by some small angles in the construction of $\sA, \sB$ and ${}_{\sA}\Delta_{\sB}$ and work with a perturbed version of \eqref{Intro.E.2}, called the $\alpha$-instanton equation, where the angle $\theta\in \R$ is allowed to change and hence replaced by a function $\alpha:\R_s\to \R$ that is constant when $|s|\gg1$: 
	\begin{equation}\label{Intro.E.16}
	P: \R_t\times \R_s\to M,\ \pt P+J\ps P+\nabla \im (e^{-i\alpha(s)}W)(P(t,s))=0. 
	\end{equation}
	With suitable asymptotic conditions at infinity, counting $\alpha$-instantons gives an alternative route to define the Floer cohomology for a pair of thimbles without using the Lagrangian boundary conditions, which refines the earlier proposals of Haydys and Gaiotto-Moore-Witten. 
	
	The bulk of this paper is devoted to a detailed construction of these $A_\infty$-invariants. We emphasize that with a few analytic results in place, Theorem \ref{Intro.T.8} follows rather formally from the existence of the geometric filtration \eqref{Intro.E.11}. In some sense, the complex gradient flow equation \eqref{Intro.E.2} is ``smart" enough to establish the algebraic spectral sequence in its own right.
\end{remark}

\subsection{Relation with Gaiotto-Moore-Witten's proposal: a vertical gluing theorem}\label{Intro.Sec5} In the last step of the proof of Theorem \ref{Intro.T.8}, we have to use an analytic gluing theorem to compute the $E_1$-pages of the spectral sequences and the map between them. We explain how this is done for the Floer complex $\Ch_R^*(X,Y)$. First, introduce a Floer cohomology of $(X, S_n)$ by counting $\theta$-instantons on the upper half planes (with $\theta=0$):
\begin{equation*}
\left\{
\begin{array}{rl}
P: \R_t\times [0,+\infty)_s&\to M,\\
\pt P+J\ps P+\nabla H&=0,\\
P(t, 0)&\in X,\\
\lim_{s\to \infty}P(t,s)&= x_n \text{ exponentially and uniformly in $t\in \R_t$},
\end{array}
\right.
\end{equation*}
where the convergence is understood as in \eqref{E1.7} below. A generator of the Floer complex $\Ch^*_\natural(X, S_n)$ is a path $p_X: [0,+\infty)_s\to M$ following the flow of $-\nabla L$ with $p_X(0)\in X\cap S_n$; the subscript $\natural$ is to distinguish $\Ch^*_\natural$ from the complex defined by counting $\theta$-instantons on a strip. Meanwhile, define a Floer cohomology of $(U_n, Y)$ by counting instantons on the lower half planes:
\begin{equation*}
\left\{
\begin{array}{rl}
P: \R_t\times (-\infty,0]_s&\to M,\\
\pt P+J\ps P+\nabla H&=0,\\
P(t, 0)&\in Y,\\
\lim_{s\to -\infty}P(t,s)&=x_n \text{ exponentially and uniformly in $t\in \R_t$}.
\end{array}
\right.
\end{equation*}

A generator of $\Ch^*_\natural(U_n, Y)$ is a path $p_Y: (-\infty, 0]_s\to M$ following the flow of $-\nabla L$ with $p_Y(0)\in U_n\cap Y$. The standard gluing theorem in Morse-Smale-Witten theory implies that any soliton $p_R\in V_R^n\subset \Ch^*_R(X,Y)$ (cf. \eqref{Intro.E.19}) is obtained by gluing some $p_X$ with $p_Y$ at $x_n$, and hence gives an isomorphism between vector spaces,
\begin{equation}\label{Intro.E.17}
\Ch^*_\natural(U_n, Y)\otimes \Ch^*_\natural(X, S_n)\to  V_R^n\subset \Ch^*_R(X,Y). 
\end{equation}

The upshot is that \eqref{Intro.E.17} is also an isomorphism between chain complexes for $R\gg 1$: one can also glue instantons vertically to identify the differential maps on both sides. This observation was first suggested by Gaiotto-Moore-Witten \cite{GMW15} and is verified in Section \ref{SecVT}. This gluing result is possible because by \eqref{Intro.E.18} the drop of the action functional $\CA_{W,R}$ along a Floer differential on $V_R^n$ is uniformly bounded as $R\to\infty$. Any instanton with this property on $\R_t\times [0,R]_s$ approximates the critical point $x_n\in \Crit(W)$ exponentially as $(t,s)$ tends to the middle line $\R_t\times \{\frac{R}{2} \}$ --- it is then compared with a solution when the neck is completely stretched, i.e., when $R=+\infty$. 

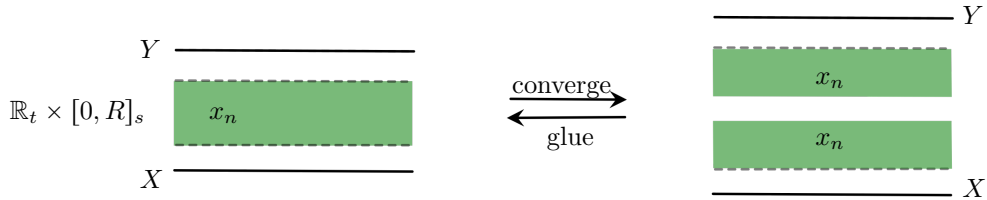
\begin{figure}[H]
	\centering
	\begin{overpic}[scale=.15]{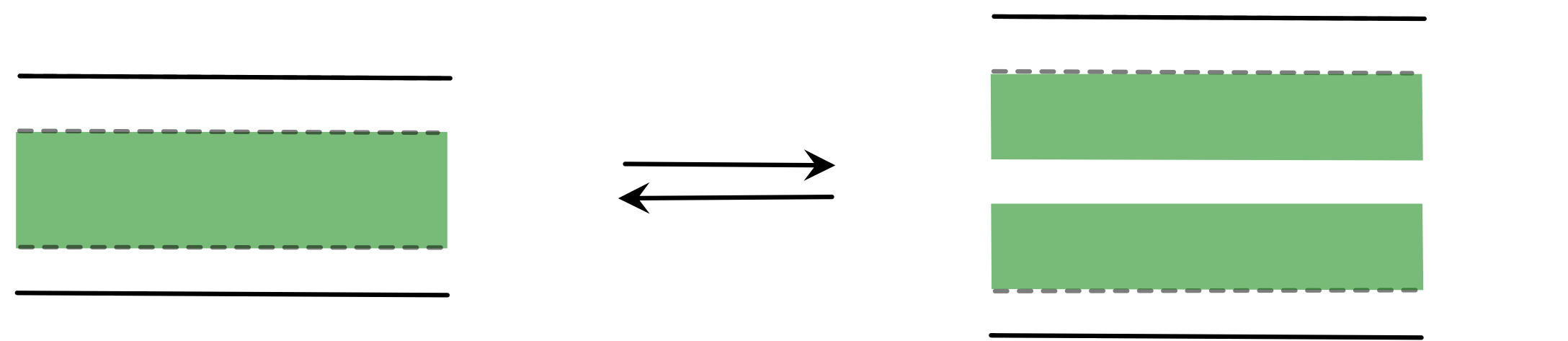}
		\put(-18,9){\small$\R_t\times [0, R]_s$}
		\put(-3,1){\small$X$}
		\put(5,9){\small$x_n$}
		\put(-3,16){\small$Y$}
		\put(92,20){\small$Y$}
		\put(92, 0){\small $X$}
		\put(75,13){\small $x_n$}
		\put(75, 6){\small $x_n$}
		\put(40,12){\small converge}
		\put(44,6){\small glue}
	\end{overpic}	
	\caption{Gluing Floer differentials vertically.}
	\label{Pic48}
\end{figure}

In fact, a general framework is proposed in \cite{GMW15} to understand all differentials on the Floer complex $\Ch^*_R(X,Y)$ as $R\to\infty$ which may go between different filtration levels with the drop of $\CA_{W,R}$ blowing up linearly. However, as we have followed Seidel's algebraic framework, their full proposal is not implemented in this paper.

\subsection{Future direction I: Non-compact Lagrangian submanifolds} The $A_\infty$-modules ${}_{\sA}X$ and $Y_{\sB}$ can be defined more generally for suitable non-compact $X,Y,$ so one expects Corollary \ref{Intro.C.10} to hold in greater generality. The optimal result we obtain is the following.

\begin{theorem}[In preparation]\label{Intro.T.14} Suppose that $(M,W)$ is tame, and $X,Y\subset M$ are any exact graded Lagrangian submanifold with bounded geometry and such that for some $C>0$,
	\begin{align*}
\re(e^{-i\beta_X} W)|_X<C \text{ for some }\beta_X\in [-\frac{\pi}{3},\frac{\pi}{3}],\\
\re(e^{-i\beta_Y} W)|_Y<C \text{ for some }\beta_Y\in [\frac{2\pi}{3},\frac{4\pi}{3}].
	\end{align*}
	Then the chain map \eqref{Intro.E.14} is a quasi-isomorphism. Figure \ref{Pic32} illustrates the case when $\beta_X=0$ and $\beta_Y=\pi$. 
\end{theorem}
\begin{remark} Theorem \ref{Intro.T.14} explains why different $A_\infty$-modules are used for $X,Y$ in Corollary \ref{Intro.C.10} --- the projections $W(X), W(Y)\subset \C$ are bounded above in different directions. In this context the Floer cohomology of $(X,Y)$ is not well-defined, and the complex $\Ch_R^*(X,Y)$ in \eqref{Intro.E.14} is defined by counting a special class of solitons and instantons, and whose invariance is only verified via Theorem \ref{Intro.T.14}. These solitons are required to approximate some critical points of $W$ when $R\gg 1$, so the decomposition \eqref{Intro.E.6} remains valid. In light of Nguyen's Theorem \ref{Intro.T.13}, Theorem \ref{Intro.T.14} is the more relevant version for our infinite dimensional application.
\end{remark}

\begin{remark} (Warning to the reader) Theorem \ref{Intro.T.14} is not proved in this Arxiv preprint, as it requires more sophisticated analysis to establish the compactness theorem. An appendix ($\sim$20 pages) is in preparation to address this in detail. The version proved in Section \ref{SecSS} is also for non-compact $X,Y$ but assumes in addition that $H|_{X}>C$ and $H|_{Y}<C$ (so the intersection $W(X)\cap W(Y)\subset \C$ remains compact), and its proof is identical to the case when $X,Y$ are both compact. 
\end{remark}

\begin{remark}\label{Intro.R.17} Although we have advertised this paper mainly through Seidel's spectral sequence, \eqref{Intro.E.14} is the gluing formula we wish to generalize for the Dirac superpotential in Example \ref{Intro.EX.3}. Returning to our discussion of the bordered monopole Floer theory in Section \ref{Intro.Sec.2}, as our Lagrangian submanifolds originate with  Nguyen's Theorem \ref{Intro.T.13}, the Lagrangian boundary condition will be removed completely in our final application in gauge theory. To develop a bordered theory along this line, one has to construct:
	\begin{itemize}
		\item directed $A_\infty$-categories $\sA$ and $\sB$ along with a diagonal bimodule ${}_{\sA}\Delta_{\sB}$ for the Dirac superpotential;
		\item a left $\sA$-module for $Z_0$ and a right $\sB$-module for $Z_1$. 
	\end{itemize}

The geometric filtration on the monopole Floer cochain complex of $Z_R=Z_0\cup [0,R]_s\times \Sigma\cup Z_1$ with $R\gg 1$ is induced by the Chern-Simons-Dirac functional in the same ways as in \eqref{Intro.E.6}. One would hope that a formula similar to \eqref{Intro.E.14} will recover this Floer cohomology along with this spectral sequence in terms of the $A_\infty$-invariants of $Z_0$ and $Z_1$, where the contribution of $Z_0$ and $Z_1$ is separated. In the genus $g(\Sigma)=1$ case, the underlying cochain complexes of these $A_\infty$-modules have been constructed in the author's earlier paper \cite{Wang20}. We hope more will be explored in the future.
\end{remark}

\subsection{Future direction II: a Picard-Lefschetz-Novikov theory} One technical point, however, complicates the discussion of Remark \ref{Intro.R.17}: the Dirac superpotential $W_{\D}$, once perturbed into a Morse function, is not single-valued on the quotient configuration space on $\Sigma$. In the finite dimensional case, this is saying that $W: M\to \T^2\cong \C/ (\Z\oplus\Z)$ is a holomorphic map onto a 2-torus with finitely many critical points. Thus a Picard-Lefschetz-Novikov theory is needed in order to develop a complete bordered theory along this line (as the complex analogue of the Morse-Novikov theory \cite{Nov81, Nov82} for $S^1$-valued Morse functions).

\medskip

As an intermediate step, we should first understand the case when $W: M\to \R\times S^1$ is valued in a cylinder. Suppose that $X,Y\subset M$ are compact exact graded Lagrangian submanifolds, then the Hamiltonian function $H=\im W: M\to S^1\cong \R/2\pi$ used in \eqref{Intro.E.4} is circle-valued, and hence the complex $\Ch^*_R(X,Y)$ is only defined over a Novikov ring $\BK[[U]]$, where the exponent of the $U$-variable keeps track of how many times an instanton goes around the circle factor (or equivalently one equips a local system on $M$ using $H: M\to S^1$ here).

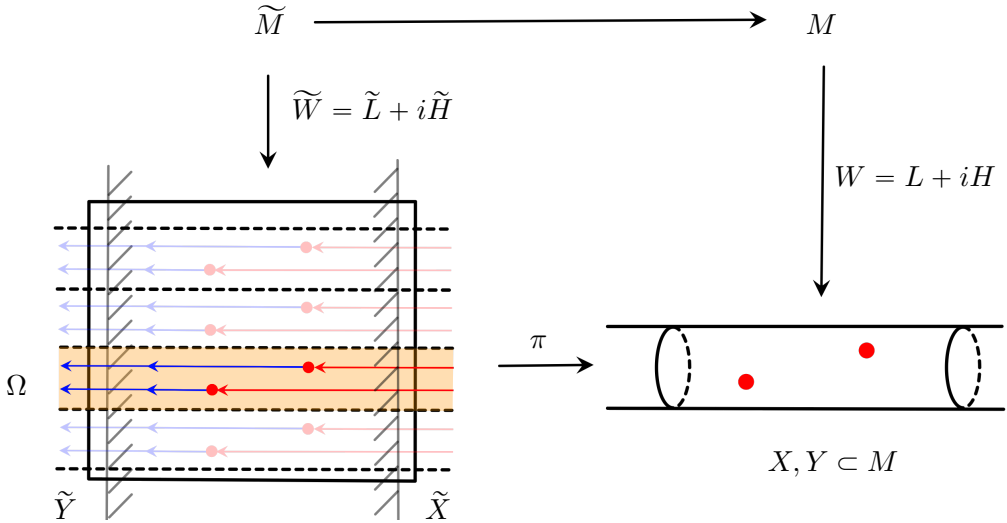
\begin{figure}[H]
	\centering
	\begin{overpic}[scale=.15]{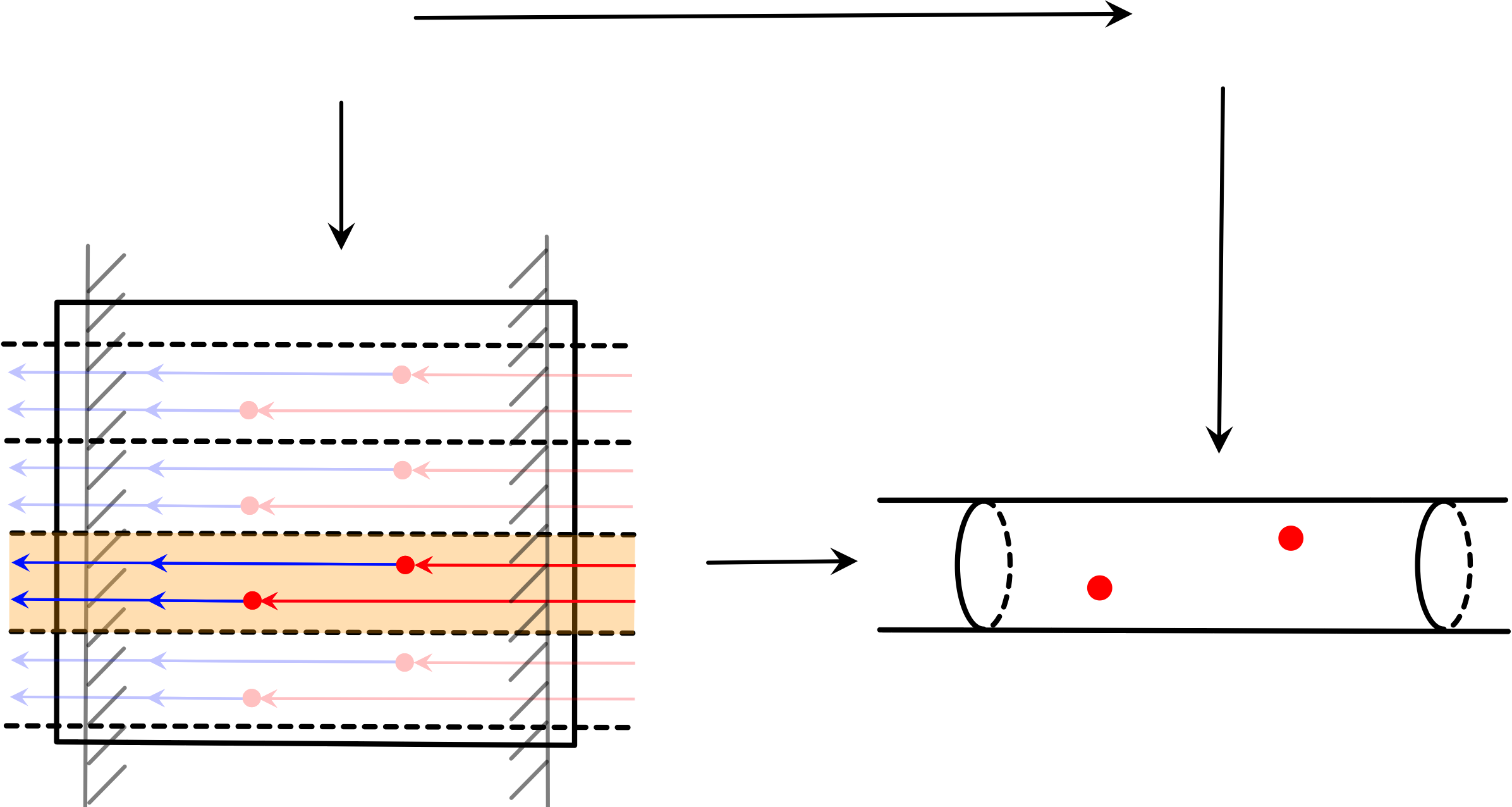}
		\put(21,51){$\widetilde{M}$}
		\put(79,51){$M$}
		\put(82,35){$W=L+iH$}
		\put(25,42){$\widetilde{W}=\widetilde{L}+i\widetilde{H}$}
		\put(75,5){$X,Y\subset M$}
		\put(39,0){$\widetilde{X}$}
		\put(0,0){$\widetilde{Y}$}
		\put(-5,13){$\Omega$}
		\put(50,18){$\pi$}
	\end{overpic}	
	\caption{A superpotential $W$ valued in $\R\times S^1$.}
	\label{Pic47}
\end{figure}

Consider the pull-back of $W: M\to \R\times S^1$ via the quotient map $\pi:\C\to \C/2\pi i\Z$, and denote this new fibration by $\widetilde{W}=\widetilde{L}+i\widetilde{H}:\widetilde{M}\to \C$. Let $\widetilde{X}$ (resp. $\widetilde{Y}$) be the lift of $X$ (resp. $Y$) in this covering space, $\Omega=\R\times [0,2\pi]\subset \C$ a fundamental domain of $\pi$, and $\Crit(\widetilde{W})=\{x_1,\cdots, x_m\}$ the set of critical points of $\widetilde{W}$ lying over $\Omega$, which are ordered by  
\[
\widetilde{H}(x_1)<\widetilde{H}(x_2)<\cdots< \widetilde{H}(x_m).
\]

To each critical point $x_n\in \Crit(\widetilde{W})$ is associated a pair of thimbles $(U_n, S_n)$, and the Floer cohomology group $\HFF_\natural^*(U_n, \widetilde{Y})$ and $\HFF_\natural^*(\widetilde{X}, S_n)$ are well-defined over $\BK$. The geometric spectral sequence \eqref{Intro.E.7} is then generalized to this case to give 
\[
\bigg(\bigoplus_{n=1}^m  \HFF_\natural^*(U_n, \widetilde{Y})\otimes_{\BK}\HFF_\natural^*(\widetilde{X}, S_n)\bigg)\otimes_\BK \BK[[U]]\rightrightarrows H(\Ch_R^*(X,Y)). 
\]

To the best of our knowledge, Seidel's generating theorem is not available yet in this context. The proof of Theorem \ref{Intro.T.8} or Theorem \ref{Intro.T.14} has a local nature, which may be generalized to understand any quotient subcomplex of $\Ch_R^*(X,Y)$ with finitely many filtration levels, truncating the coefficient ring from $\BK[[U]]$ into $\BK[U]/U^n$ for any $n\geq 1$. A complete description may be then obtained by algebraically taking $n\to\infty$. The author wishes to return to this topic in the near future.

\begin{remark} Another motivation for this development is for Example \ref{Intro.EX.2}: the complex Chern-Simons functional is a superpotential valued in $\R\times S^1$. Although the analysis of Haydys-Witten equations is difficult, it might be helpful to understand first the type of the algebraic invariants one should associate to any closed 3-manifold via this superpotential. 
\end{remark}

\subsection{Outline of the paper} This paper is organized as follows. 

\medskip

Part \ref{Part1} is devoted to the construction of Floer cohomology of thimbles by counting the $\alpha$-instantons \eqref{Intro.E.16} on $\R_t\times \R_s$. The analytic foundation of this Floer theory has been developed in the context of the Seiberg-Witten theory in \cite{Wang20,Wang20}, so some arguments might be sketchy here. The invariance of this Floer cohomology is then verified using the continuation method. 

Part \ref{Part2} is devoted to the construction of the Fukaya-Seidel category $\sA$, $\sB$ and ${}_{\sA}\Delta_{\sB}$. The idea is to attach lower half planes to the boundary of pointed disks and count instantons on these complete Riemann surfaces. Their invariance is then verified using categorical localization in Section \ref{SecFS}. The filtered version of ${}_{\sA}\Delta_{\sB}$ is constructed in Section \ref{SecGF}. The Koszul duality pattern between $\sA$ and $\sB$ is explained in Section \ref{SecPT}, which also proves Theorem \ref{Intro.T.8} \ref{T.7.1}. The vertical gluing theorem, which is the main analytic input of this work, is proved in Section \ref{SecVT}. 

Part \ref{Part3} is devoted to the proof of Seidel's spectral sequence, i.e., Theorem \ref{Intro.T.8} \ref{T.7.2} and \ref{T.7.3}. With all analytic tools developed in previous sections, they follow from a simple algebraic argument. Finally, some applications of the vertical gluing theorem is collected in Section \ref{SecSS}, including a special case of Seidel's long exact sequence \cite{S03} and an isomorphism of our Floer cohomology with the classical one defined by counting solutions on strips. 

\medskip

The author has tried to reconcile the conflict of interests for readers from different background by making this paper more or less self-contained. This explains partly why this paper is much longer than expected. Experts should feel free to skip the expository materials in these sections. 

\medskip
\textbf{Acknowledgment.} The author would like thank his thesis advisor, Tom Mrowka, for his guidance and constant encouragement throughout this project. The author would also like to thank Shaoyun Bai, Simon Donaldson, Andrew Hanlon, Andriy Haydys, Runjie Hu, Jianfeng Lin, Paul Seidel, Guangbo Xu and Yongquan Zhang for their suggestions on various technical points in this paper.  This work is partially supported by NSF through his thesis advisor's award DMS-2105512 and by a MathWorks fellowship. 
\newpage
\part{Thimbles}\label{Part1}

This part is devoted to the analytic theory of the $\alpha$-instanton equation and a detailed construction of Floer cohomology of thimbles without using Lagrangian boundary conditions, which is denoted by $\HFF_\natural^*$. In Section \ref{Sec1}, we introduce the notion of tame Landau-Ginzburg models following \cite{FJY18}, explain the setup and derive the energy estimate. In Section \ref{Sec2}, we establish the compact theorem of $\alpha$-instantons, both the local version and the global version. Section \ref{SecLG} is devoted to linear analysis and the canonical $\Z$-grading. In Section \ref{SecGRS}, a variant of the $\alpha$-instanton equation is introduced on a general Riemann surface. This allows us to construct the continuation map and verify the invariance of this Floer cohomology. The equivalence between $\HFF_\natural^*$ and the ordinary version $\HFF^*$ is established at the very end of this paper (Section \ref{SecSS.9}) using a vertical gluing theorem. 
\section{Floer Cohomology for Thimbles}\label{Sec1}

\subsection{Tame Landau-Ginzburg models} A Landau-Ginzburg Model $(M,W)$ is a non-compact complete K\"{a}hler manifold $(M, J_M, g_M, \omega_M)$ equipped with a holomorphic function $W: M\to \C$ called superpotential. Let $J=J_M$ denote the complex structure of $M$, $g_M$ the underlying Riemannian metric and $\omega_M$ the symplectic 2-form. Write $L=\re W$ and $H=\im W$. The Cauchy-Riemann equation $(dW)^{0,1}=0$ then says that
\begin{equation}\label{E1.1}
\nabla L+J\nabla H=0.
\end{equation}
 In most cases, the real direction is not special; one may ``rotate'' $W$ by multiplying some $e^{-i\theta}\in S^1$ and consider the gradient vector
\begin{equation}\label{E1.2}
\nabla \re(e^{-i\theta}W)=\cos\theta\nabla L+\sin\theta \nabla H=e^{J\theta }\nabla L. 
\end{equation} 

A Landau-Ginzburg model $(M,W)$ is called \textit{Morse} if all critical points of $L$ are non-generate. By \eqref{E1.2}, the critical set $\Crit(W)\colonequals \Crit(\re(e^{-i\theta}W))$ is independent of the choice of $e^{-i\theta}\in S^1$. To develop a Floer theory in this context, it is important to control the geometry of $(M,W)$ at infinity. A reasonable set of conditions has been proposed in \cite{FJY18}. Following their work, we introduce the notion of tame Landau-Ginzburg models.

\begin{definition}\label{D1.1} A Landau-Ginzburg model $(M,W)$ is called tame, if the following conditions hold:
	\begin{enumerate}[label=(A\arabic*)]
	\item\label{A1}  $(M,\omega_M)$ is an exact symplectic manifold, i.e., $\omega_M=d\lambda_M$ for a smooth 1-form $\lambda_M\in \Omega^1(M; \R)$ called primitive.
	\item\label{A2} $2c_1(TM, J_M)=0$;
	\item\label{A3} $(M, W)$ is Morse;
	\item\label{A4} $W|_{\Crit(W)}$ is injective, i.e., each singular fiber of $W$ contains a unique critical point;
	\item\label{A5} $(M, J_M, g_M,\omega_M)$ has bounded geometry, and there exists a proper function $\psi_M: M\to \R$ such that $\psi_M\geq 1$, $|\nabla \psi_M|\leq 1$ and for some $C_1, a_1>0$, we have 
		\begin{equation}\label{E1.1.3}
	C^{-1}_1\psi_M< |\nabla H|+1<C_1e^{a_1\psi_M}.
	\end{equation}	
	In particular, the function $|\nabla H|^2:M\to [0,\infty)$ is proper, and the critical set $\Crit(W)=|\nabla H|^{-1}(0)$ is finite, though the superpotential $W: M\to \C$ is not assumed to be proper.\qedhere
\end{enumerate}
\end{definition}

\begin{example}\label{EX1.2}
	Let $M=\C^n$ and $W=\half(\lambda_1z_1^2+\cdots+\lambda_nz_n^2)$ with each $\lambda_i>0$. Then the unique critical point is the origin.
\end{example}
\begin{example}\label{EX1.3}
	Let $M=\C$ and $W$ a polynomial of degree $d\geq 2$. $(M,W)$ is Morse if $\frac{\partial W}{\partial z}$ has only simple zeros. 
\end{example}
\begin{example} Definition \ref{D1.1} is inspired by the notion of \textit{regular tame exact Landau-Ginzburg systems} defined in \cite[Section 2]{FJY18}. The second condition \ref{A2} is to ensure that our Floer cohomology groups are $\Z$-graded; see Section \ref{SecLG} below. The condition \ref{A5} is to ensure that the Local Compactness Lemma \ref{L1.6} holds for the $\alpha$-instanton equation and is usually the most difficult one to verify in practice. Readers should feel free to replace \ref{A5} by any convenient criterion to fulfill this local property.
	
More examples of tame Landau-Ginzburg models are supplemented by \cite[Section 2]{FJY18}. Suppose that $M=\C^n$ is equipped with the standard K\"{a}hler metric, $W\in \C[z_1,\cdots, z_n]$ is a non-degenerate quasi-homogeneous polynomial plus a lower order one, and $\psi_M: M\to\R$ is the distance function to the origin, then the condition \eqref{E1.1.3} can be verified using \cite[Proposition 2.5]{FJY18}. 
	
	Alternatively, one may take $M=(\C^*)^n$ and equip each factor $\C^*\cong \C/2\pi i\Z$ with the standard metric on the cylinder.  Take $\psi_M$ to be the pull-back of a distance function on $\R^n$ via the projection map $M\to \R^n$. If $W$ is a convenient and non-degenerate Laurent polynomial in the sense of \cite[Section 2]{FJY18}, then the condition \eqref{E1.1.3} is verified by \cite[Proposition 2.8]{FJY18}. We refer interested readers to \cite{FJY18} for the precise definitions and necessary background from mirror symmetry. 
\end{example}

\begin{example} We also give a Landau-Ginzburg model which is not tame. Choose a polynomial $f:\C\to \C$ with $d$ distinct roots, and let $M$ be the $A_d$-type Milnor fiber 
	\[
	M=\{(y,z_1,\cdots, z_n): f(y)+z_1^2+\cdots+z_n^2=0\}\subset \C_y\times \C_z^n
	\]
	and $W\to \C_y$ the projection map onto the first factor. Then $(M,W)$ is not tame with the induced K\"{a}hler metric from $\C_y\times \C_z^d$, since $|\nabla H|^{-1}([0,\epsilon))$ is not compact for any $\epsilon>0$. It is not clear to the author what might be a good metric for this example; the scope of this paper is indeed constrained by the condition \ref{A5}.
\end{example}

In this paper, we shall \textbf{always} assume that $(M,W)$ is tame. For any critical point $q\in \Crit(W)$ and $\theta\in \R$, consider the stable submanifold of $\re(e^{-i\theta}W)$ at $q$: 
\[
\Lambda_{q,\theta}\colonequals \{x\in M: \exists p:[0,\infty)_s\to M,  \ps p+\nabla \re(e^{-i\theta}W)=0, p(0)=x, \lim_{s\to\infty} p=q\}. 
\]

If one thinks of $W:M\to \C$ as a projection map that defines a Lefschetz fibration with possibly non-compact fibers, then the image of $\Lambda_{q,\theta}$ is a ray $l_{q,\theta}$ that emanates from $W(q)$ at the angle $e^{i\theta}$, and $\Lambda_{q,\theta}$ is the Lefschetz thimble associated to the vanishing path $l_{q,\theta}$; see Figure \ref{Pic2} below. As it is difficult to make sense of the notion of monodromy and vanishing cycles in the infinite dimensional setting, we shall avoid this point of view in this paper. 

By \cite[Lemma 1.13]{S03} or \cite[Lemma 2.5]{E97}, $\Lambda_{q,\theta}$ is a  Lagrangian submanifold of $(M,\omega_M)$. Given any $q_0, q_1\in \Crit(W)$ and $\theta_0,\theta_1\in \R$ with $\theta_1<\theta_0<\theta_1+2\pi$, the Lagrangian Floer cohomology 
\[
\HFF^*(\Lambda_{q_0,\theta_1}, \Lambda_{q_1,\theta_1})
\]
can be defined by counting $J$-holomorphic strips with Lagrangian boundary conditions:
\begin{equation}\label{E1.20}
\left\{\begin{array}{rl}
P: \R_t\times [0,1]_s & \to M,\\
\pt P+ J\ps P&=0,\\
P(t,1)&\in \Lambda_{q_1,\theta_1},\\
P(t,0)&\in \Lambda_{q_0,\theta_0},\\
\end{array}
\right.
\end{equation}
or a suitable perturbed version of this so as to make the moduli spaces regular. However, generalizing this framework to the infinite dimensional setting is a challenging task \cite{NguyenI,NguyenII}. For the rest of this section, we describe an alternative route to this Floer cohomology so that this generalization becomes completely formal, while the analysis remains pretty much tractable.

\subsection{The $\alpha$-instanton equation on $\R_t\times \R_s$ and the boundary condition at infinity}\label{Sec1.2} For convenience, we write $\Lambda_j=\Lambda_{q_j,\theta_j}$ for $j=0,1$ and denote the Floer cohomology defined using this new approach by $\HFF_\natural^*(\Lambda_0,\Lambda_1)$ to distinguish it from $\HFF^*(\Lambda_0, \Lambda_1)$. We shall always work with a base field $\BK$ of characteristic $2$ to avoid the orientation issue. As usual, $\HFF^*_\natural(\Lambda_0, \Lambda_1)$ is interpreted as an infinite dimensional Morse cohomology. Choose a constant $R\geq \pi$ and a smooth function $\alpha: \R_s\to \R$ such that for some $0<\delta\ll1$,
\begin{equation}\label{E1.3}
\alpha(s)\equiv \left\{
\begin{array}{ll}
\theta_1&\text{if }s\geq R-\delta,\\
\theta_0-\pi&\text{if }s\leq \delta.
\end{array}
\right.
\end{equation}
 We require further that for some $0<\epsilon_{01}<1$ and $\beta\in (\theta_1-\frac{\pi}{2},\theta_0-\frac{\pi}{2})$, there holds
\begin{equation}\label{E1.4}
\re(e^{i(\beta-\alpha(s))})>\epsilon_{01}, \forall s\in \R_s.
\end{equation}
This forces the image of $\alpha$ to lie in an interval centered at $\beta$. Let $\SH$ denote the subspace of $C^\infty(M; \R)$ with finite $L^\infty_1$-norm. To make moduli spaces regular, we shall use a smooth 1-form $\delta H=\delta H_s ds\in \Omega^1(\R_s; \SH)$ supported on $[0,R]_s$ as a perturbation term and require that 
\begin{equation}\label{E1.21}
\int_{\R_s} ||\delta H_s||_{L^\infty(M)} ds<1 \text{ and } \int_{\R_s} ||\delta H_s||_{L^\infty_1(M)}^2 ds<1. 
\end{equation}
For technical reason to be explained shortly, $\HFF^*_\natural(\Lambda_0, \Lambda_1)$ can be defined only if the following condition holds:
\begin{equation}\label{E1.5}
\begin{array}{c}
\text{the rays $l_{q_0,\theta_0} $ and $l_{q_1,\theta_1}$ do not contain}\\
\text{any critical values of $W$ except their end points.}
\end{array}
\end{equation}
\begin{definition} For any $(\Lambda_0, \Lambda_1)$ satisfying \eqref{E1.5}, a Floer datum $\fa=(R,\alpha(s), \beta,\epsilon_{01},\delta H)$ is a quintuple satisfying all conditions above. 
\end{definition}

\begin{remark}\label{R1.5} The condition \eqref{E1.5} is to ensure a compactness property and is used only in the proof of Lemma \ref{L1.10} below. If the function $\alpha(s)$ is monotone, then \eqref{E1.5} can be relaxed as follows: if the rays $l_{q_0,\theta_0} $ and $l_{q_1,\theta_1}$ intersect at some $y\in \C$, then the segments $\overline{yW(q_0)}$ and $\overline{yW(q_1)}$ do not contain any critical values of $W$ except $W(q_0)$ and $W(q_1)$. See Figure \ref{Pic2} below.
\end{remark}

\begin{figure}[H]
	\centering
	\begin{overpic}[scale=.12]{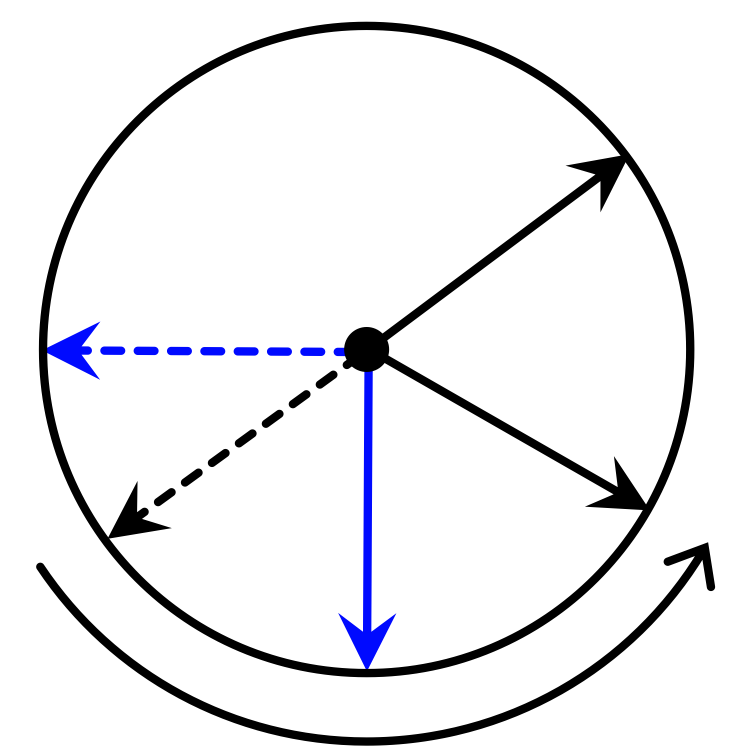}
	\put(90,85){$\theta_0$}
		\put(100,20){$\theta_1$}
			\put(-35,20){$\theta_0-\pi$}
			\put(55,20){$\beta$}
			\put(-40,50){$\beta-\frac{\pi}{2}$}
			\put(80,0){$\alpha(s)$}
	\end{overpic}	
	\caption{}
	\label{Pic1}
\end{figure}

For each critical point $q\in \Crit(W)$, choose a normal neighborhood $\SO(q)\subset M$ such that $\SO(q)$'s are mutually disjoint and the exponential map $\exp_q: T_qM\to M$ maps a small ball centered at the origin bijectively onto $\SO(q)$. Let $p_{\model}:\R_s\to M$ be a reference path such that $p_{\model}(s)\equiv  q_0$ when $s\leq 0$ and $\equiv q_1$ when $s\geq R$. A smooth $p:\R_s\to M$ is said to have finite $L^2_k$-distance with $p_{\model}, k\geq 1$, if there exist some $R_1>0$ such that 
\begin{itemize}
\item $p(-s)\in \SO(q_0)$ and $p(s)\in \SO(q_1)$ for all $s\geq R_1$;
\item $\exp_{q_1}^{-1}(p(s))\in L^2_k([R_1,+\infty), T_{q_1}M)$ and $\exp_{q_0}^{-1}(p(s))\in L^2_k((-\infty, -R_1], T_{q_0}M)$,
\end{itemize}
Since $L^2_k(\R_s)\embed C^0(\R_s)$ for $k\geq 1$, we must have $\lim_{s\to-\infty} p(s)=q_0$ and $\lim_{s\to\infty} p(s)=q_1$. Now consider the path space 
\begin{equation}\label{E1.19}
\Pa_k(\Lambda_0, \Lambda_1)\colonequals\{p\in C^\infty(\R_s; M): p \text{ has finite $L^2_k$-distance with }p_{\model} \}, k\geq 1.
\end{equation}

The action functional $\CA_{W,\fa}(p)$ is defined on this space by the formula
\begin{align}\label{ActionFunctional1}
\int_{\R_s}&-p^*\lambda_M+\delta H_s(p(s))ds\nonumber\\
&+\int_{\R_s} \big(\im(e^{-i\alpha(s)}W(p(s)))+\chi_{(-\infty,0]} \im (e^{-i\theta_0}W(q_0))-\chi_{[R,\infty]}\im (e^{-i\theta_1}W(q_1))\big)ds,
\end{align}
where $\chi_A$ is the indicator function of a subset $A\subset \R$. The function $\im(e^{-i\alpha(s)}W(p(s)))$ is in general not $L^1$-integrable on $\R_s$ -- it is necessary to subtract the limits at infinity to make sense of \eqref{ActionFunctional1}. Indeed, for all $s\gg R$,
\[
|W(p(s))-W(q_1)|<C|\exp_{q_1}^{-1}(p(s))|^2,\ |W(p(-s))-W(q_0)|<C|\exp_{q_0}^{-1}(p(-s))|^2.\ 
\]
Since $p$ has finite $L^2$-distance with $p_{\model}$, the integral \eqref{ActionFunctional1} is finite. The gradient vector field of $\CA_{W,\fa}$ is given by the formula
\[
\grad\CA_{W,\fa}(p)=J\ps p(s)+\nabla \big(\im(e^{-i\alpha(s)}W)+\delta H_s\big)\big( p(s)\big)\in C^\infty\cap L^2_{k-1}(\R_s, p^*TM). 
\]
Let $\FC(\Lambda_0, \Lambda_1;\fa)$ be the set of critical points of $\CA_{W,\fa}$. Any $p\in \FC(\Lambda_0, \Lambda_1;\fa)$ is called  \textit{an $\alpha$-soliton} and solves a pseudo-gradient flow equation:
\begin{equation}\label{E1.9}
\ps p(s)+\nabla \re(e^{-i\alpha(s)}W)-J\nabla(\delta H_s)=0,
\end{equation}
which we refer to as \textit{the $\alpha$-soliton equation} in this paper, though \eqref{E1.9} depends also on the perturbation 1-form $\delta H$. Since $W$ is Morse, any $\alpha$-soliton $p(s)$ converges exponentially in $C^\infty_{loc}$-topology to $q_0$ and $q_1$ as $s\to \pm\infty$ respectively. If $\delta H\equiv 0$, the projection of $p$ under $W:M\to\C$ is a smooth curve connecting $W(q_0)$ with $W(q_1)$ and is subject to the equation:
 \[
\ps (W(p(s)))=-e^{i\alpha(s)}|\nabla H\circ p(s)|^2.
\]
The condition \eqref{E1.4} has the following geometric interpretation, which is important for the energy estimate later on: the path $s\mapsto W(p(s))$ is monotone decreasing in the direction of $e^{i\beta}$. 
\begin{figure}[H]
	\centering
	\begin{overpic}[scale=.12]{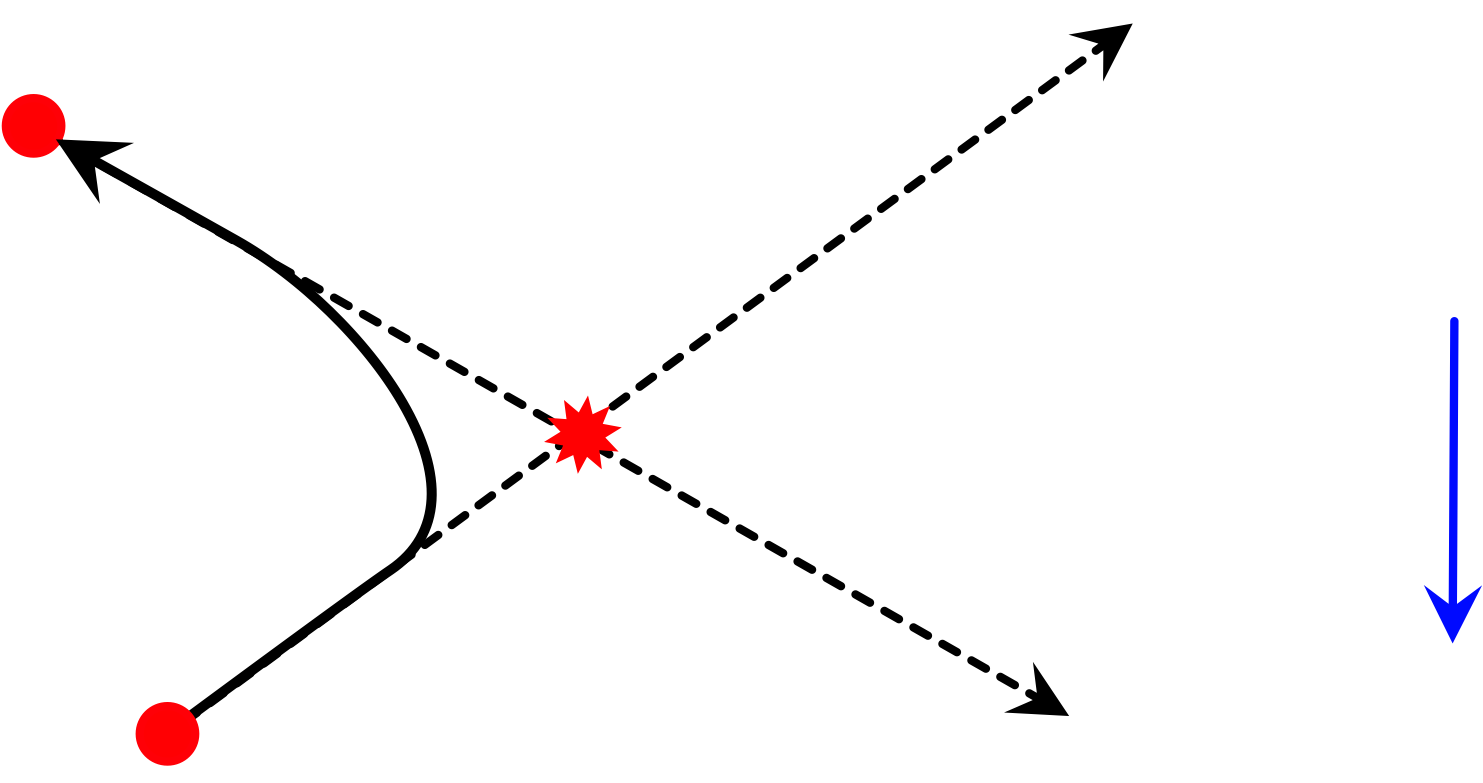}
		\put(-20,45){\small $W(q_1)$}
				\put(-10,0){\small $W(q_0)$}
				\put(5,20){\small $W(p(s))$}
				\put(50,45){$l_{q_0,\theta_0}$}
					\put(50,5){$l_{q_1,\theta_1}$}
					\put(80,53){$e^{i\theta_0}$}
						\put(75,0){$e^{i\theta_1}$}
							\put(105,20){$e^{i\beta}$}
							\put(48,22){$y=l_{q_0,\theta_0}\cap l_{q_1,\theta_1}$}
	\end{overpic}	
	\caption{}
	\label{Pic2}
\end{figure}
To construct the Floer cohomology of $(\Lambda_0, \Lambda_1)$, consider a smooth map $P: \R_t\times \R_s \to M$ solving \textit{the $\alpha$-instanton equation}:
\begin{equation}\label{E1.6}
\pt P+ J\big(\ps P+\nabla \re(e^{-i\alpha(s)}W)\big)+\nabla (\delta H_s)=0.
\end{equation}
To specify the boundary condition of $P$, fix $p_\pm(s)\in \FC(\Lambda_0, \Lambda_1;\fa)$  and consider a model map $P_{\model}: \R_t\times \R_s\to M$ such that 
\begin{equation}\label{E1.7}
\left\{\begin{array}{rlll}
P_{\model}(t,s)&\equiv p_-(s) &\text{ if } t<-1,&\\
P_{\model}(t,s)&\equiv p_+(s) &\text{ if } t>1,&\\
P_{\model}(\cdot ,s)&\to q_1&\text{ as }s\to \infty &\text{ uniformly exponentially,}\\
P_{\model}(\cdot ,s)&\to q_0&\text{ as }s\to -\infty &\text{ uniformly exponentially}.
\end{array}
\right.
\end{equation}
The convergence in \eqref{E1.7} is understood as follows. For some $\zeta>0$ and any $k,l\geq 0, s\gg R$, we have 
\[
\|\pt^k\ps^l \exp_{q_1}^{-1}(P_{\model}(\ \cdot\ ,s))\|_{L^\infty(\R_t)},\ \|\pt^k\ps^l \exp_{q_0}^{-1}(P_{\model}(\ \cdot\ ,-s))\|_{L^\infty(\R_t)}<e^{-\zeta |s|}.
\]
We require the solution $P$ to have finite $L^2_k(\R_t\times \R_s, M)$-distance with this model map $P_{\model}$ for some $k\geq 2$. In particular, the path $\{P(t,\cdot)\}_{t\in \R_t}$ is formally a downward gradient flowline for the action functional $\CA_{W,\fa}$. Denote the space of such solutions by $\M(p_-,p_+)$; translating the $t$-variable defines an $\R_t$-action on $\M(p_-,p_+)$. The quotient space $\cM(p_-, p_+)=\M(p_-,p_+)/\R$ is the moduli space of $\alpha$-instantons. If the perturbation 1-form $\delta H$ is chosen generically, the action functional $\CA_{W,\fa}$ becomes Morse at each $p\in  \FC(\Lambda_0, \Lambda_1;\fa)$. A compactness theorem then ensures that $\FC(\Lambda_0, \Lambda_1;\fa)$ is a finite discrete set of points. Moreover, for any $p_\pm$, the moduli space $\cM(p_-,p_+)$ is a smooth manifold and can be compactified by broken trajectories. The Morse-Smale-Witten complex of $\CA_{W,\fa}$ is freely generated by $\alpha$-solitons:
\[
\Ch^*_\natural(\Lambda_0, \Lambda_1;\fa)\colonequals \bigoplus_{p\in\FC (\Lambda_0, \Lambda_1;\fa)} \BK\cdot p. 
\]
with the differential map defined by counting $\alpha$-instantons:
\begin{equation}\label{E1.18}
\partial p_+\colonequals \sum_{p_-:\ \dim \cM(p_-,p_+)=0} \#\cM(p_-, p_+)\cdot p_-. 
\end{equation}

Then $\HFF^*_\natural(\Lambda_0, \Lambda_1;\fa)$ is defined as the homology of $(\Ch^*_\natural(\Lambda_0, \Lambda_1;\fa),\partial)$. Under some mild conditions, we will verify  that $\HFF^*_\natural(\Lambda_0, \Lambda_1;\fa)$ is isomorphic to the group $\HFF^*(\Lambda_0, \Lambda_1)$ defined using \eqref{E1.20} in Section \ref{SecSS.9}. In the next few sections, we provide the analytic foundation of this novel Floer cohomology including the basic compactness property of the $\alpha$-instanton equation, canonical gradings and the invariance of $\HFF_\natural^*(\Lambda_0, \Lambda_1;\fa)$ as we vary the Floer datum $\fa$.

\subsection{The energy estimate} The first step towards the compactification of $\cM(p_-,p_+)$ is the energy estimate for the $\alpha$-instanton equation \eqref{E1.6}. For any solution $P: \R_t\times \R_s\to M$, consider \textit{the energy density function}
\[
u=|\pt P|^2+|\ps P|^2+|\nabla H\circ P|^2: \R_t\times \R_s\to \R. 
\]
and for any open subset $\Omega\in \R_t\times\R_s$, define the energy of $P$ over $\Omega$ as
\[
\E_{an}(P;\Omega)=\half\int_\Omega u(t,s)dt\wedge ds.
\]

\begin{lemma}[The Energy Estimate I]\label{L1.5} Assuming \eqref{E1.4}, one can find a constant $\epsilon_{01}'>0$ such that for any $\alpha$-instanton $P$ subject to the boundary condition \eqref{E1.7} and $t_0<t_1$, we have 
\begin{align}\label{EnergyEquation1}
&\CA_{W,\fa}(P(t_0,\cdot))-\CA_{W,\fa}(P(t_1,\cdot))+2\int_{\R_s}\|\delta H_s\|_{L^\infty(M)}ds\\
+(t_1-t_0)&\bigg(\epsilon_{01}\re\big(e^{-i\beta}\big(W(q_0)-W(q_1)\big)\big)+\int_{\R_s} \|\delta H_s\|_{L^\infty_1(M)}^2ds\bigg)\geq \epsilon_{01}'\E_{an}(P; [t_0, t_1]_t\times \R_s). \nonumber
\end{align}
\end{lemma}
\begin{proof} Let $\fa_0=(R,\alpha(s),\beta,\epsilon,0)$ denote the Floer datum which is equal to $\fa$ except that the perturbation 1-form is trivial. Then the $\alpha$-instanton equation \eqref{E1.6} is cast into the form 
	\[
	\pt P+\grad \CA_{W,\fa_0}=-\nabla (\delta H_s). 
	\]
Take the square of both sides, then integrating over $[t_0,t_1]_t$ gives:
	\begin{align*}
	\CA_{W,\fa_0}(P(t_0,\cdot))-\CA_{W,\fa_0}(P(t_1,\cdot))=
\half \int_{[t_0,t_1]\times \R_s} |\pt P|^2+|\ps P+\nabla \re(e^{-i\alpha(s)}W)|^2-|\nabla (\delta H_s)|^2.
	\end{align*}
The last term on the right can be controlled in terms of $\|\delta H_s\|_{L^\infty_1(M)}$. Now we rewrite the second term as follows
\begin{align}\label{E1.8}
\half\int_{[t_0,t_1]\times \R_s} |\ps P|^2+|\nabla H|^2+2\big\langle \ps P, \nabla \re\big((e^{-i\alpha(s)}-\epsilon_{01} e^{-i\beta})W\big)\big\rangle \\
+\int_{[t_0,t_1]\times \R_s} \langle \ps P, \epsilon_{01} \nabla \re (e^{-i\beta}W)\rangle\nonumber.
\end{align}
The second line of \eqref{E1.8} can be integrated to obtain $\epsilon_{01}(t_1-t_0)\cdot\re (e^{-i\beta}(W(q_1)-W(q_0)))$. Since for all $s\in \R_s$, $\re(e^{i(\beta-\alpha(s))})>\epsilon_{01}$, we must have
	\[
|e^{-i\alpha(s)}-\epsilon_{01} e^{-i\beta}|<\sqrt{1-\epsilon^2_{01}}<1-\epsilon^2_{01}/2. 
	\]
Hence one may complete squares and lower bound the first line of \eqref{E1.8} by 
	\[
\half\int_{[t_0,t_1]\times \R_s} \frac{\epsilon^2_{01}}{2}|\ps P|^2+\frac{\epsilon^2_{01}}{2}|\nabla H\circ P|^2.
	\]
	Finally, for any path $p\in \Pa_k(\Lambda_0, \Lambda_1)$, the values $\CA_{W,\fa}(p)$ and $\CA_{W,\fa_0}(p)$ can differ at most by $\int_{\R_s}\|\delta H_s\|_{L^\infty(M)}ds$. This completes the proof of Lemma \ref{L1.5}. In fact, one may take $\epsilon_{01}'=\epsilon^2_{01}/2$. In the estimate \eqref{EnergyEquation1}, the two terms that involve the perturbation 1-forms $\delta H$ are controlled by \eqref{E1.21}.
\end{proof}

%

\section{Compactness}\label{Sec2}

\subsection{The local compactness lemma} Having established the energy estimate in Lemma \ref{L1.5}, the next step towards the compactification of $\cM(p_-, p_+)$ is to obtain the basic $C^0$-estimate for $\alpha$-instantons. 

\begin{lemma}[Local Compactness I: the Interior Case]\label{L1.6} Let $\Omega\subset \R_t\times \R_s$ be any bounded open subset and $C>0$. Suppose that $\xi: \Omega\to \C$ is any smooth function such that $|\xi(z)|<C$, and $H^t_z, H^s_z: M\to \R,\ z\in \Omega$ are families of Hamiltonian functions on $M$ such that 
	\[
	\|\nabla H^t_z\|_{L^\infty(M)}, \|\nabla H^s_z\|_{L^\infty(M)}<C
	\]
	for all $z\in \Omega$. Then any sequence of solutions $P_n:\Omega\to M$ to the Floer equation 
	\begin{equation}\label{E1.1.1}
	\big(\pt P_n-J\nabla H_z^t\big)+J\big(\ps P_n-\nabla \re (\overline{\xi(z)} W)-J\nabla H_z^s\big)=0
	\end{equation}
	with the uniform energy bound
	\begin{equation}\label{E1.1.2}
\E_{an}(P_n;\Omega)=\int_\Omega |d P_n|^2+|\nabla H\circ P_n|^2<C,
	\end{equation}
has a subsequence that converges in $C^\infty_{loc}$-topology on $\Omega$.  
\end{lemma}

\begin{remark} The condition \ref{A5} is used in this paper only via Lemma \ref{L1.6}. Readers should feel free to replace \ref{A5} by any convenient criterion to verify this local result. The analogue of Lemma \ref{L1.6} in the context of the Seiberg-Witten equations is \cite[Theorem 5.1.1]{Bible}.
\end{remark}

The proof of Lemma \ref{L1.6} relies on a diameter estimate for ``almost'' $J$-holomorphic curves in almost Hermitian manifolds with bounded geometry. 
\begin{proposition}\label{P1.1.3} Suppose that $(M, J, g_M,\omega_M)$ is an almost Hermitian manifold with bounded geometry $($$\omega_M$ is not necessarily closed$)$. Let $\bpartial_J$ denote the operator $\pt+J\ps$. Then for any $r>2$, there exists a function $\varphi:\R_+\to \R_+$ such that for any $C>0$ and for any map $P: B(0,1)\to M$, if 
	\[
	\int_{B(0,1)} |d P|^2+|\bpartial_J P|^r\leq C,
	\]
	then $\diam (P(B(0,\half)))\leq \varphi(C)$. 
\end{proposition}

If $P: B(0,1)\to M$ is genuinely $J$-holomorphic, then Proposition \ref{P1.1.3} is a consequence of  \cite[Lemma 1.1 \& Corollay 1.5]{IS00}. The estimates therein are robust enough to accommodate this general case. For the sake of completeness, we provide the proof of Proposition \ref{P1.1.3} in Appendix \ref{SecDE}.

\begin{proof}[Proof of Lemma \ref{L1.6}] It suffices to verify an a priori $C^0$-estimate for solutions on $\Omega$: for any relatively compact subset  $\Omega'\subset \Omega$, there exists $K=K(\Omega')\subset M$ compact such that if $P:\Omega\to M$ satisfies \eqref{E1.1.1} and \eqref{E1.1.2}, then $P(\Omega')\subset K$. The $C^\infty_{loc}$-convergence then follows from elliptic bootstrapping. Without loss of generality, assume that $\Omega=B(0,1)$ is the unit disk and $\Omega'=B(0,\half)$. Throughout this proof, we use $C_k, k\geq 1$ to denote constants independent of $P$.
	
Consider the function $f=\psi_M\circ P: B(0,1)\to M$, where $\psi_M: M\to \R$ is given by the condition \ref{A5}. Then $|df|=|\langle \nabla \psi_M, d P\rangle|\leq |d P|$ and $|f|<C_1(|\nabla H\circ P|+1)$. By the energy bound \eqref{E1.1.2}, $\|f\|_{L^2_1(B(0,1))}<C_2$ for some constant $C_2>0$. 

Fix some $r>2$. By Trudinger's inequality \cite[Proposition 4.2 \& (4.14)]{PDEIII}, we have $\|e^{a_1f}\|_{L^r(B(0,1))}<C_3$ for some $C_3>0$. Combined with \eqref{E1.1.3}, this implies that 
\[
\|\nabla H\circ P\|_{L^r(B(0,1))} <C_3C_1. 
\]
Since the gradients of $H^t_z, H^s_z, z\in \Omega$ are bounded uniformly in $L^\infty(M)$, we conclude using the equation \eqref{E1.1.1} that 
\[
\int_{B(0,1)}|\bpartial_JP|^r<C_4
\]
for some $C_4>0$. Note that $\|d P\|_{L^2(B(0,1))}^2\leq \E_{an}(P; \Omega)<C$. Now we use Proposition \ref{P1.1.3} to conclude that the diameter of $P(B(0,\half))$ is $<\varphi(C_4+C)$. Finally, since the function $\psi_M: M\to \R$ is proper, the estimate $\|\psi_M\circ P\|_{L^2(B(0,\half))}<C_2$ shows that $P(B(0,\half))$ is confined in a compact subset of $M$. 
\end{proof}

\subsection{Small Energy Estimates} The $\alpha$-instanton equation  \eqref{E1.6} recovers the $\theta$-instanton equation 
\begin{equation}\label{E1.23}
\pt P+J\ps P+\nabla\im(e^{-i\theta}W)=0
\end{equation}
with $\theta=\theta_1$ when $s\geq R$ and $=\theta_0-\pi$ when $s\leq 0$. The Local Compactness Lemma \ref{L1.6} can be refined in this special case to control the $L^2_k$-norm of $P$ in terms of the energy of $P$ provided that this energy sufficiently small. 

\begin{lemma}[Small Energy Estimates]\label{L1.8} For any bounded open subset $\Omega\subset \R_t\times \R_s$, any relatively compact subset $\Omega'\subset \Omega$ and $k\geq 1$, there exist $C_k, \epsilon_1>0$ satisfying the following property. If $P: \Omega\to M$ solves the standard equation \eqref{E1.23} for some $\theta\in \R$ with $\E_{an}(P;\Omega)<\epsilon_1$, then the image $P(\Omega')\subset\SO(q)$ for some $q\in \Crit(W)$, and for all $k\geq 1$,
	\begin{equation}\label{E1.27}
	\|\exp_q^{-1}\circ P\|_{L^2_k(\Omega', T_qM)}^2<C_k \E_{an}(P;\Omega).
	\end{equation}
\end{lemma}

\begin{proof} With loss of generality, assume $\theta=0,\ \Omega=B(0,1)$ and $\Omega'=B(0,\half)$. Since $H: M\to \C$ is Morse, for any $q\in \Crit(W)$, there is a smaller neighborhood $\SO'(q)\subset \SO(q)$ and $C_q>0$ such that for any $x\in \SO'(q)$,
	\begin{equation}\label{E1.24}
	|\exp_q^{-1}(x)|<C_q|\nabla H(x)|. 
	\end{equation}
	If $\E_{an}(P;\Omega)=0$, then $P\equiv q$ for some $q\in \Crit(W)$. Hence by Lemma \ref{L1.6} there exists $\epsilon_1>0$ such that $\E_{an}(P;\Omega)<\epsilon_1$ implies that $P(B(0,\frac{3}{4}))\subset \SO'(q)$ for some $q\in \Crit(W)$. When $k=1$, the estimate \eqref{E1.27} follows from \eqref{E1.24}. The rest of the proof is by elliptic bootstrapping. When $k\geq 2$, choose a sequence of radii $r_k=\half<r_{k-1}<\cdots<r_1=\frac{3}{4}$. We may assume that this critical point $q$ has been fixed, and we will not distinguish $\SO(q)$ with its preimage in $T_qM$. Then G{a}rding's inequality implies that for any $1\leq l\leq k-1$,
	\begin{align}\label{E1.25}
\|P\|_{L^2_{l+1}(B(0,r_{l+1}))}&\leq C\|P\|_{L^2_{l}(B(0,r_{l}))}+C\|(\pt+J\ps)P\|_{L^2_{l}(B(0,r_{l}))}\nonumber\\
&=C\|P\|_{L^2_{l}(B(0,r_{l})}+C\|\nabla H\|_{L^2_{l}(B(0,r_{l}))}\nonumber\\
&\leq C\|P\|_{L^2_{l}(B(0,r_{l}))}+C\sum_{l_1+\cdots +l_n=l-1} \|\nabla^{n}(\nabla H)(\nabla^{l_1}P, \cdots, \nabla^{l_l}P)\|_{L^2(B(0,r_l))}.
	\end{align}
	In the last summation, there are two distinct cases: if $n=1$, then there is a unique entry involving $\nabla^{l-1}P$; if $n\geq 2$, then $l_j<l-1$, and one may use the Sobolev inequality $\|\nabla^{l_j}P\|_{L^{2n}}\leq C'\|P\|_{L^2_l}$. Since $\nabla H$ is smooth, $\nabla H$ along with its higher derivatives are all bounded on $\SO(q)$. The norm $\|P\|_{L^2_{l}(B(0,r_l)}$ can be made small using the induction hypothesis, so any higher order terms involving this norm can be absorbed into the linear one. We deduce from \eqref{E1.25} that 
	\begin{align*}
	\|P\|_{L^2_{l+1}(B(0,r_{l+1}))}&\leq C'\|P\|_{L^2_l(B(0,r_{l}))}+ \text{ high order terms of }\|P\|_{L^2_l(B(0,r_l))}\leq C''\|P\|_{L^2_l(B(0,r_l))}.
	\end{align*}
The constant $C''$ may depend on $1\leq l\leq k-1$. This completes the proof of Lemma \ref{L1.8}.
\end{proof}

\subsection{The global version} We now state and prove the global version of Lemma \ref{L1.6}, from which one can easily compactify the moduli space $\cM(p_-, p_+)$ by broken trajectories following the standard argument as in \cite[Section 16]{Bible}. The full proof of this compactification is omitted in this paper.
\begin{proposition}\label{P1.7} Let $I\subset \R_t$ be any finite time interval and  $P_n: \R_t\times \R_s\to M$  any sequence of solutions to the $\alpha$-instanton equation \eqref{E1.6} subject to the boundary condition \eqref{E1.7}. Then one can find a subsequence of $\{P_n|_{I\times \R_s}\}$ that converges in $L^2_k(I\times \R_s, M)$-topology for any $k\geq 1$.
\end{proposition}

Although the $L^2_k(I\times \R_s; M)$-norm of a map $P_n|_{I\times \R_s}$ is not defined, it makes sense to measure the $L^2_k$-norm for their difference under the boundary condition \eqref{E1.7}. More concretely, the $L^2_k$-convergence in Proposition \ref{P1.7} means that for some $R_1>R$ and any $P_n$ in this subsequence, $P_n$ converges smoothly on $I\times [-R_1,R_1]_s$. When $s\geq R_1$, $P_n(t,s)\in \SO(q_1)$ and $\exp^{-1}_{q_1}\circ P_n$ converges in $L^2_k(I\times [R_1,+\infty)_s; T_{q_1}M)$ for all $k\geq 1$. A similar convergence holds when $s\leq -R_1$ and with $q_1$ replaced by $q_0$. 

In general, the total energy $\E_{an}(P;\R_t\times \R_s)=\half \int_{\R^2} u$ is not finite. For a counterexample, consider $P(t,s)=p_-(s)$ but $p_-$ is not constant. To deduce Proposition \ref{P1.7}, it is important to understand the distribution of  the energy density function $u$ on $\R_t\times \R_s$. The most ideal result is the following:

\begin{proposition}[Exponential Decay in the Spatial Direction]\label{P1.8} There exists constant $C,\zeta>0$ such that for any solution $P:\R_t\times \R_s\to M$ to the $\alpha$-instanton equation \eqref{E1.6} subject to the boundary condition \eqref{E1.7}, we have the estimate
	\[
	u(t,s)\leq Ce^{-\zeta |s|}. 
	\]
\end{proposition}

We first prove Proposition \ref{P1.7} assuming Proposition \ref{P1.8}.

\begin{proof}[Proof of Proposition \ref{P1.7} (Sketch)] Take a finite time interval $I'$ such that $I\subset \inte(I')$. By the small energy estimate from Lemma \ref{L1.8} and Proposition \ref{P1.8}, there exists some $R_1>0$ such that $P_n(t,s)\in \SO(q_0)$ for all $s<-R_1$ and $\in \SO(q_1)$ for all $s>R_1$. Fix some $k\geq 2$ and $\epsilon>0$. Then there exists another $R_2=R_2(k,\epsilon)>R_1$ such that $\|\exp_{q_0}^{-1}\circ P_n\|_{L^2_k}$ and $\|\exp_{q_1}^{-1}\circ P_n\|_{L^2_k}$ are bounded uniformly by $\epsilon$ on the tail $I'\times \{|s|>R_2\}$. Over the rectangle $I'\times [-R_2,R_2]_s$, the $L^2_k$-convergence of a subsequence is provided by the Local Compactness Lemma \ref{L1.6}. In this way, we find a converging subsequence in $L^2_k$-topology for a fixed $k$. Now we use the diagonal argument to find a converging subsequence in $L^2_k$  for all $k\geq 1$.  
\end{proof}

The proof of Proposition \ref{P1.8} requires some preliminary results from \cite{Wang202}. As the complete proof has been carried in detail in the context of the Seiberg-Witten equations in \cite[Section 6]{Wang20}, we shall only present the main ideas here. The first step towards Proposition \ref{P1.8} is the uniform convergence:

\begin{lemma}[Uniform Decay in the Spatial Direction]\label{L1.9} There exists a function $\eta: [0,\infty)\to [0,\infty)$ with $\lim_{s\to \infty} \eta(s)\to 0$ and such that for any solution $P:\R_t\times \R_s\to M$ to the $\alpha$-instanton equation \eqref{E1.6} subject to the boundary condition \eqref{E1.7}, we have $u(t,s)\leq \eta(|s|)$. 
\end{lemma}

To prove Lemma \ref{L1.9}, we have to analyze the time intervals on which the drop of the action functional $\CA_{W,\fa}$ is small. The next lemma relies essentially on the condition \eqref{E1.5}. 

\begin{lemma}\label{L1.10} For any $\delta>0$, there exists constants $\epsilon(\delta), R(\delta)>0$ with the following property. For any interval $I=[t_0, t_1]_t\subset \R_t$ with $|t_1-t_0|=2$ and any solution $P:\R_t\times \R_s\to M$ to the $\alpha$-instanton equation \eqref{E1.6} subject to the boundary condition \eqref{E1.7}, if $t\in I,\ |s|>R(\delta)$ and $\CA_{W,\fa}(P(t_0,\cdot))-\CA_{W,\fa}(P(t_1,\cdot))<\epsilon(\delta)$, then $u(t,s)<\delta$.
\end{lemma}
\begin{proof}[Proof of Lemma \ref{L1.10} (Sketch)] If not, we can find a sequence of solutions $P_n$ and intervals $I_n=[t_{0,n},t_{1,n}]_t$ such that
	\[
	 \CA_{W,\fa}(P_n(t_{0,n},\cdot))-\CA_{W,\fa}(P_n(t_{1,n},\cdot))\to 0. 
	\]
	 and for some $t_n\in I_n$ and $|s_n|\to\infty$, $u(t_n, s_n)\geq \delta$ for all $n$. By translating the time variable $t$, we may assume that $I_n=[-1,1]_t$ is independent of this sequence. The Local Compactness Lemma \ref{L1.6} allows us to find a subsequence that converges to a ``broken solution'' of \eqref{E1.9}:
	\[
	q_0=q_0'\xrightarrow{p_0}q_1'\xrightarrow{p_1}q_2'\xrightarrow{p_2}\cdots\xrightarrow{p_l}q_{l+1}'=q_1. 
	\]
	with $q_k'\in \Crit(W)$. This means that for some special $0\leq j\leq l$, $p_j:\R_s\to M$ is a solution of \eqref{E1.9} with $\lim_{s\to \infty} p_j(s)=q_{j+1}'$ and $\lim_{s\to -\infty} p_j(s)=q_{j}'$. For any $k\neq j$, we have 
	\[
	\left\{
	\begin{array}{rl}
\ps p_k(s)+\nabla\re(e^{-i\theta'}W)&=0,\\
\lim_{s\to \infty} p_k(s)&=q_{k+1}',\\
\lim_{s\to -\infty} p_k(s)&=q_{k}'\\
	\end{array}
	\right.
	\] 
	with $\theta'\equiv \theta_0-\pi$ if $k<j$ and $\equiv \theta_1$ if $k>j$. This convergence is in the following sense: each $p_k$ defines a map $[-1,1]_t\times \R_s\to M$ which is constant in time. For $k=j$, $
	P_n\to p_j 
$ in $C^\infty_{loc}$-topology as $n\to\infty$. In general, there exists a sequence of numbers 
\[
s^{(0)}_n<\cdots<s^{(j-1)}_n<s^{(j)}_n=0<s^{(j+1)}_n<s^{(j+2)}_n<\cdots <s^{(l)}_n
\]
such that $P_n(t,s-s^{(k)}_n)\to p_k(t,s)$ in $C^\infty_{loc}$-topology and $\lim_{n\to\infty} |s^{(k)}_n-s^{(k-1)}_n|=\infty$ for all $k$.

\medskip
	
	The projection of such a broken solution under $W:M\to \C$ is a smooth curve connecting $W(q_0)$ with $W(q_1)$ that passes through each $W(q_k')$; see Figure \ref{Pic3} below. The condition \eqref{E1.5} then implies that each $p_k$ must be constant unless $k=j$. This contradicts the assumption that $u_n(t_n,s_n)>\delta$ with $|s_n|\to\infty$. Notice that if the function $\alpha(s)$ is monotone, then one may work with the weaker condition in Remark \ref{R1.5} for the last step.
\end{proof}

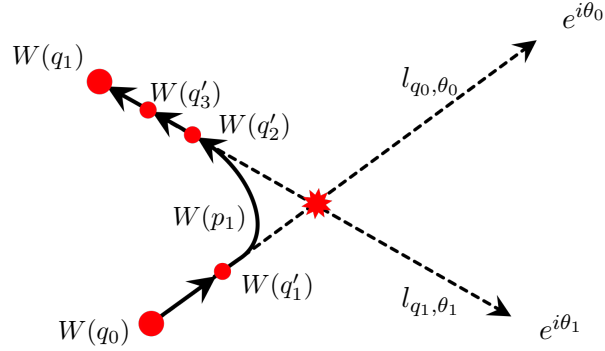
\begin{figure}[H]
	\centering
	\begin{overpic}[scale=.15]{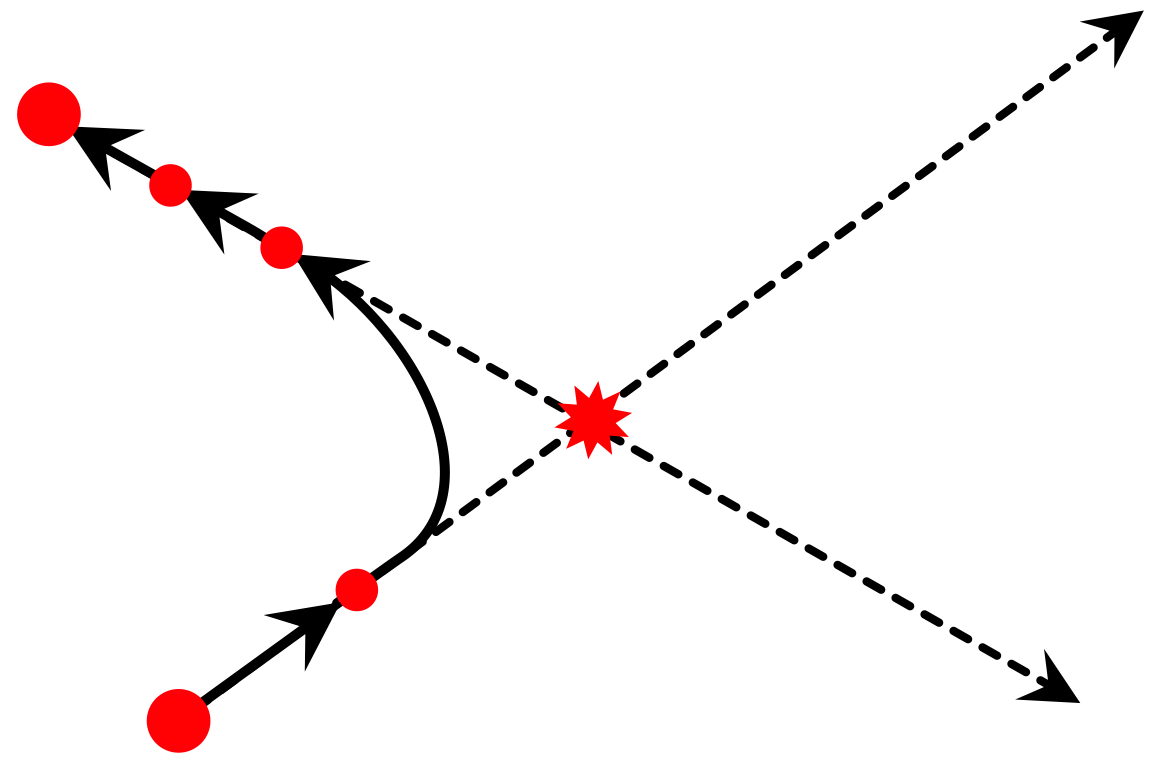}
		\put(-15,60){\small $W(q_1)$}
		\put(-5,0){\small $W(q_0)$}
			\put(35,10){\small $W(q_1')$}
				\put(30,45){\small $W(q_2')$}
					\put(15,52){\small $W(q_3')$}
		\put(20,25){\small $W(p_1)$}
		\put(70,55){$l_{q_0,\theta_0}$}
		\put(70,7){$l_{q_1,\theta_1}$}
		\put(105,68){$e^{i\theta_0}$}
		\put(100,0){$e^{i\theta_1}$}
	\end{overpic}	
	\caption{The limit is a broken solution with $j=1$ and $l=3$.}
	\label{Pic3}
\end{figure}

To deduce Lemma \ref{L1.9} from Lemma \ref{L1.10}, we have to rule out a bubbling phenomenon: for a sequence of solutions $P_n$, some amount of energy may slide off in the spatial direction as $s\to \infty$ and converge to a $\theta$-instanton $P_\star: \R_t\times \R_s\to M$ with $\theta=\theta_0-\pi$ or $\theta_1$ and whose total energy $\E_{an}(P_\star;\R_t\times \R_s)$ is finite. Such a solution is called  point-like in the literature \cite[Section 14.1]{GMW15} and can be ruled out by \cite[Lemma 2.7]{Wang202}. In practice, a stronger statement is needed for the proof:

\begin{lemma}[{\cite[Lemma 2.5 \& 2.7]{Wang202}}]\label{L1.15} There exists some constant $\epsilon>0$ with the following property. Let $P_\star: \R_t\times \R_s\to M$ be any $\theta$-instanton, i.e., a solution to the equation \eqref{E1.23} with a constant angle $\theta$. If the energy density function $u_\star$ of $P_\star$ satisfies the pointwise estimate $u_\star(z)<\epsilon$ for any $|z|\gg 1$, then $P_\star$ is constant, i.e., $P_\star\equiv q$ for some $q\in \Crit(W)$.
\end{lemma}
\begin{proof}[Proof of Lemma \ref{L1.15} (Sketch)] In this case, one can deduce the exponential decay of $u_\star(z)$ as $|z|\to \infty$ from \cite[Lemma 2.5]{Wang202}, so in fact the total energy $\E_{an}(P_\star; \R_t\times \R_s)$ is finite. Then one applies \cite[Lemma 2.7]{Wang202} to show that $P_\star$ is constant. The analogue of this lemma in the context of gauged Landau-Ginzburg models is proved in \cite[Theorem 5.1 \& 6.1]{Wang202}. 
\end{proof}

\begin{proof}[Proof of Lemma \ref{L1.9} (Sketch)] For any $\delta>0$, we have to find some $R_1>0$ such that $u(t,s)<\delta$ for all $t\in \R_t$ and $|s|>R_1$. To this end, we cover $\R_t$ by $\bigcup I_n$ with $I_n\colonequals [n-1,n+1]_t$. We say an interval $I_n$ is good if 
	\[
	\CA_{W,\fa}(P(n-1,\cdot ))-	\CA_{W,\fa}(P(n+1,\cdot ))<\epsilon(\delta)
	\]
	where $\epsilon(\delta)$ is the constant in Lemma \ref{L1.10}, and bad otherwise. Then the desired estimate holds for all $t$ in some good interval and $|s|>R(\delta)$. The number of bad intervals is uniformly bounded (independent of $P$). If the desired $R_1$ does not exist for some of them, then the bubbling phenomenon would occur, but this is ruled out using Lemma \ref{L1.15}. Hence, the desired constant $R_1$ exists also for bad intervals. This completes the proof of Lemma \ref{L1.9}. The analogue of this lemma in the context of the Seiberg-Witten equations is \cite[Theorem 6.3]{Wang20}.
\end{proof}

The passage from the uniform decay in Lemma \ref{L1.9} to the exponential decay in Proposition \ref{P1.8} is the content of \cite[Theorem 6.1]{Wang202}. We highlight a few key arguments below and refer the readers to \cite{Wang202} for the complete proof.

\begin{proof}[Proof of Proposition \ref{P1.8} (Sketch)] The plan is to apply the maximum principle on the upper half plane \cite[Corollary A.2]{Wang202} and use a Bochner-type formula \cite[Lemma 6.4]{Wang202} for the energy density function $u$. To verify the exponential decay of $u$, it suffices to work over the region $|s|\geq R$ on which the function $\alpha(s)$ is constant. In this case, we have 
	\begin{align}\label{E1.26}
0=\half \Delta |d P|^2+|\Hess P|^2&+|\Hess H(\pt P)|^2+|\Hess H(\ps P)|^2\nonumber\\
 &+\text{ cubic and quartic terms in $\pt P$ and $\ps P$}.
	\end{align}
	When $0\leq u(z)\ll 1$, the map $P(z)$ is close to some critical point $q\in \Crit(W)$; then the Morse condition of $W$ implies that
	\[
	|\Hess H(V)|>\zeta|V|,\ V=\pt P \text{ or }\ps P
	\]
	for some $\zeta>0$, and we may absorb the higher order terms of $d P$ in \eqref{E1.26} by the quadratic term when $|d P|\ll 1$. Now we use Lemma \ref{L1.9} to deduce that for some $R_1>R$, 
	\[
	0\geq (\Delta+\zeta^2)|d P|^2,
	\]
	on $\R_t\times \{|s|\geq R_1\}$. Finally, we apply the maximum principle to conclude. For this argument to work, it is essential to have $\alpha(s)$ constant when $|s|\gg 1$. The analogue of this result in the context of the Seiberg-Witten equations is \cite[Theorem 6.4]{Wang20}. 
\end{proof}

\section{Linear Analysis and Gradings}\label{SecLG}

The required Fredholm theory to set up the moduli space $\cM(p_-, p_+)$ has been developed in the more sophisticated context of the Seiberg-Witten equations in \cite[Part 4]{Wang20}. In this section, we recollect a few basic facts from \cite{Wang20} with emphasis on the finite dimensional problem addressed in this paper. While the Lagrangian submanifold $\Lambda_{q,\theta}$ depends only on the angle $e^{i\theta}\in S^1$, a real-valued lift $\theta\in \R$ of this angle gives an extra grading on $\Lambda_{q,\theta}$, making our Floer cohomology group $\HFF^*_\natural(\Lambda_0, \Lambda_1; \fa)$ canonically $\Z$-graded. As the construction is parallel to the case of Lagrangian boundary conditions, we shall only point out the necessary adaptation and refer readers to \cite[Section 11]{S08} for the complete treatment.

\subsection{Graded Lagrangian submanifolds} Given a Hermitian vector space $(V, J_V)$ of $\dim_\C V=\fn$, let $\Gr(V)=U(\fn)/O(\fn)$ denote the Lagrangian Grassmannian of unoriented linear Lagrangian subspaces. A quadratic volume form on $(V, J_V)$ is an isomorphism of complex lines 
$
\eta_V^2:(\bigwedge^{top}_\C V)^{\otimes 2}\to \C,
$
 which defines a squared phase map on $\Gr(V)$ by the formula:
\begin{equation}\label{LG.E.1}
\xi_V: \Gr(V)\to S^1,\ \xi_V(\Pi)=\frac{\eta_V(v_1\wedge v_2\wedge\cdots \wedge v_{\fn})^2}{|\eta_V(v_1\wedge v_2\wedge\cdots \wedge v_{\fn})|^2}.
\end{equation}
where $v_1,\cdots, v_{\fn}$ is any basis of $\Pi\in \Gr(V)$. Denote by $\Gr^\#(V)$ the pull-back bundle of $\R\to S^1$ via $\xi_V$.  Since $(\xi_V)_*:\pi_1(\Gr(V))\to \pi_1(S^1)$ is an isomorphism, $\Gr^\#(V)$ is the universal cover of $\Gr(V)$. A \textit{graded} Lagrangian subspace of $V$ is an element $\Pi^\#=(\Pi, \xi^\#)$ in $\Gr^\#(V)$, where $\Pi\in \Gr(V)$ is a linear Lagrangian subspace and $\xi^\#\in\R$ satisfies $\exp(2\pi i\xi^\#)=\xi_V(\Pi)$. Since $\Gr^\#(V)$ is simply connected, any pair of graded Lagrangian submanifolds $(\Pi^\#_0, \Pi^\#_1)$ is connected by a unique path (up to homotopy) in $\Gr^\#(V)$. If the underlying Lagrangian subspaces intersect transversely $\Pi_0\pitchfork\Pi_1$, then one can associate an integer, called the Maslov index of $(\Pi_0^\#, \Pi_1^\#)$
\begin{equation}\label{LG.E.2}
i(\Pi^\#_0, \Pi^\#_1)\in \Z,
\end{equation}
which is the intersection number of this path with a suitable hypersurface in $\Gr^\#(V)$; see \cite[Section (11g)]{S08}.

\medskip

By Definition \ref{D1.1} \ref{A2}, $(M, J_M)$ is a K\"{a}hler manifold with $2c_1(TM, J_M)=0$; so one may pick a quadratic complex volume form $\eta_M^2$, i.e., a smooth non-vanishing section of $\K_M^2=(\bigwedge^{top}_\C TM)^{\otimes -2}$. Denote by $\Gr(TM)\to M$ the bundle of Lagrangian Grassmannnians and by  $\xi_M: \Gr(TM)\to S^1$  the squared phase map associated to $\eta_M$, which is defined pointwise by \eqref{LG.E.1}. A graded Lagrangian submanifold $(X, \xi_X^\#)$ is a Lagrangian submanifold $X$ with a function $\xi_X^\#: X\to \R$ (the grading of $X$) such that $\exp(2\pi i\xi_X^\#(x))=\xi_M(T_xL)$ for all $x\in L$. The obstruction of such a grading $\xi_X^\#$ is the Maslov class $\mu_X\in H^1(X,\Z)$ classified by the map $X\to S^1,\ x\mapsto \xi_M(T_xL)$. 

\medskip

Since $W: M\to \C$ is Morse, the second derivative of  $d_2W\colonequals \nabla dW: T_q M\otimes T_qM\to \C$ is a non-degenerate symmetric $J$-bilinear form at every $q\in \Crit(W)$ and hence defines a quadratic complex volume form $\eta_q^2$ on $T_qM$ by the formula:
\[
(v_1,\cdots, v_{\fn})\mapsto \det(d^2W(v_j, v_k))_{1\leq j,k\leq \fn}, v_j\in T_q M,\ \fn=\dim_\C TM. 
\]
To find the local form of $\eta_q^2$, take $(M,W)$ as in Example \ref{EX1.2}; then $\eta_q^2=\lambda_1\cdots\lambda_{\fn} (dz_1\wedge \cdots dz_{\fn})^{\otimes 2}$. We require that each $\eta_M^2|_{T_qM}=\eta_q^2$ for all $q\in \Crit(W)$. With this convention being said, a grading of $\Lambda_{q,\theta}$ is determined by the integral lift $\theta\in \R$. A direct computation in this local model shows that $
\xi_M(T_q\Lambda_{q, \theta})=\exp(i\fn\theta).$ Then we declare that
\begin{equation}\label{LG.E.4}
\xi^\#_{\Lambda_{q, \theta}}(T_q\Lambda_{q, \theta})=\frac{\fn\theta}{2\pi}.
\end{equation} 
Since $\Lambda_{q, \theta}$ is contractible, this specifies a grading on $\Lambda_{q,\theta}$.

\subsection{Hessians} 

For any smooth $p\in \Pa_k(\Lambda_0, \Lambda_1), k\geq 1$, the Hessian of $\CA_{W,\fa}$ at $p$ is given by the formula
\begin{align}\label{LG.E.3}
\Hess_p \CA_{W,\fa}: L^2_1(\R_s; p^*TM)&\to L^2(\R_s; p^*TM)\nonumber\\
v(s)&\mapsto J\frac{D}{ds}v(s)+\Hess_{p(s)}\big(\im(e^{-i\alpha(s)}W) +\delta H_s\big)\big(v(s)\big),
\end{align}
which is clearly $L^2$-self-adjoint. Since $p(s)$ decays exponentially as $s\to \pm\infty$ to $q_0$ and $q_1$ respectively, the image of $p(s)$ in $M$ is a finite path connecting $q_0$ with $q_1$. One may trivialize the bundle $p^*TM$ for $|s|\gg R$ using the Levi-Civita connection and think of $v(s)$ as a section in $T_{q_0}M$ if $s<0$ and in $T_{q_1} M$ if $s>R$. Hence $\Hess_p \CA_{W,\fa}$ takes the form 
\[
J\ps-\Hess_{q_0} \im(e^{-i\theta_0}W)+ \text{ a compact operator }
\]
on the interval $(-\infty, 0]_s$ and respectively
\[
J\ps+\Hess_{q_1} \im(e^{-i\theta_1}W)+ \text{ a compact operator }
\]
on $[R,+\infty)_s$. The model problem to understand the Fredholm property of $\Hess_p \CA_{W,\fa}$ is the following. Given a Hermitian vector space $(V, J_V)$, consider an invertible self-adjoint operator $D_V: V\to V$ that anti-commutes with $J_V$. Since $J_VD_V+D_VJ_V=0$, the spectrum of $D_V$ is symmetric about the origin. A direct computation shows that the operator
\begin{equation}\label{GL.E.3}
L^2_1(\R_s;V)\to L^2(\R_s; V),\ v(s)\mapsto \ps v(s)+D_V\big(v(s)\big),
\end{equation}
is invertible. The spectrum of \eqref{GL.E.3} is purely essential and is given by $
(-\infty, -\lambda_V]\cup [\lambda_V, +\infty)
$ where $\lambda_V$ is the first positive eigenvalue of $D_V$. Returning to the discussion of Floer cohomology, we shall take $(V, J_V)=(T_{q_j}M, J)$ and $D_V=(-1)^{j-1}\Hess_{q_j}\im (e^{-\theta_j}W)$ for $j=0,1$. Then a parametrix-patching argument shows that 
$\Hess_p \CA_{W,\fa}$ is Fredholm, so its essential spectrum is disjoint from the origin; in fact, this essential spectrum is given by
\[
(-\infty, -\lambda_1]\cup [\lambda_1,+\infty)
\]
where $\lambda_1$ is the first positive eigenvalue of $\Hess_{q_0}H$ or $\Hess_{q_1}H$ depending on which is smaller (note that $\lambda_1$ is independent of $\theta\in \R$). One can still make sense of the spectral flow for a path of operators like $\Hess_p \CA_{W,\fa}$. For details, see \cite[Section 11]{Wang20}.

\subsection{A short review} Before we turn to the canonical grading on $\HFF_\natural^*(\Lambda_0,\Lambda_1;\fa)$, we review briefly how this is done in the classical case. Let $(X,Y)$ be any pair of graded Lagrangian submanifolds of $M$ and $I_R\colonequals[-R,R]_s$. Consider the path space 
\[
\Pa(X, Y)_R=\{p:I_R\to M:\ p(-R)\in X,\ p(R)\in Y\}.
\]
For any $p\in \Pa_R(X,Y)$ and any family of Hamiltonian functions $\delta H=\delta H_sds\in \Omega^1(I_R; \SH)$, consider the $L^2$-self adjoint operator:
\begin{align}\label{GL.E.4}
\Hess_p\CA_{\delta H}:L^2_1(I_R, \partial I_R; p^*TM)&\to L^2(I_R; p^*TM),\\
 v(s)&\mapsto J\frac{D}{ds} v(s)+\Hess_{p(s)} (\delta H_s)\big(v(s)\big).\nonumber
\end{align}
where $v(s)$ is an $L^2_1$-section  of $p^*TM$ with $v(-R)\in T_{p(-R)}X$ and $v(R)\in T_{p(R)}Y$. This operator is the Hessian of a suitable action functional $\CA_{\delta H}$ on the path space $\Pa(X,Y)_R$. A pair $\p=(p, \delta H)$ is called \textit{non-degenerate} if the operator \eqref{GL.E.4} is invertible. If $2c_1(TM, J)=0$, then there is a grading function \cite[Section (12b)]{S08}:
\[
\overline{\gr}:\Pa(X,Y)_R\times \Omega^1(I_R; \SH)\dashrightarrow \Z
\]
defined on the subspace of all non-degenerate pairs and subject to the following axioms:
\begin{enumerate}[label=(A-\Roman{*})]
\item\label{A-I} (Index Axiom) Let $\p_\pm=(p_\pm, \delta H_\pm)$ be non-degenerate and $\p(t): [-1,1]_t\to \Pa(X,Y)_R\times \Omega^1(I_R;\SH)$ any smooth path (if exist) connecting $\p_\pm$ and constant when $0\leq |t-1|\ll 1$. Thus $\p(t)$ defines a smooth map $P:Z_R\colonequals\R_t\times I_R\to M$ along with a 1-form $\delta H'=\delta H'_s ds\in \Omega^1(Z_R; \SH)$ such that $(P, \delta H')(t,\cdot)\equiv \p_-$ if $t\leq -1$ and $\equiv \p_+$ if $t\geq 1$. Then the Fredholm operator
\begin{align}\label{LG.E.5}
\D_P: L^2_1(Z_R, \partial Z_R; P^*TM)&\to L^2(Z_R; P^*TM)\\
v&\mapsto \frac{D}{dt}v+J\frac{D}{ds}v+\Hess_P (\delta H_{s}')(v).\nonumber
\end{align}
has index $=\overline{\gr}(\p_-)-\overline{\gr}(\p_+)$ (which is also the spectral flow of \eqref{GL.E.3} along the path $\p(t)$).

\item\label{A-II} (Normalization Axiom) Suppose that the path $p\in \Pa(X, Y)_R$ follows the Hamiltonian flow of some $\delta H\in \Omega^1(\R_s; \SH)$, then the linearization of this flow transports $T_{p(-R)}X\subset T_{p(-R)}M$ into a graded Lagrangian subspace of $T_{p(R)}M$, denoted by $\Pi_X^\#$. Let  $\Pi_Y^\#=T_{p(1)}Y$ with the induced grading from $Y$. If $\p=(p, \delta H)$ is non-degenerate, then $\Pi_X^\#$ and $\Pi_Y^\#$ intersect transversely, and
$
\overline{\gr}(\p)=i(\Pi_X^\#,\Pi_Y^\#)
$
is the Maslov index \eqref{LG.E.2}.
\end{enumerate}

The grading function $\overline{\gr}$ is determined uniquely by these axioms: the Index Axiom \ref{A-I} determines $\overline{\gr}$ up to a global $\Z$-action; this ambiguity is then fixed by \ref{A-II}. 
\subsection{Canonical gradings}\label{SecLG.4} We follow the same scheme to define the canonical grading on $\HFF_\natural^*(\Lambda_0, \Lambda_1;\fa)$ and look for a connection with the grading function $\overline{\gr}$ in the classical case; the normalization axiom will be adapted accordingly. Let $\CQ_R$ denote the space of Floer data whose first entry is equal to some fixed $R\geq \pi$. A pair $\p=(p,\fa)\in \Pa_k(\Lambda_0,\Lambda_1)\times \CQ_R$ is called \textit{non-degenerate} if the operator $\Hess_p \CA_{W,\fa}$ defined by \eqref{GL.E.3} is invertible. Then the grading function 
\[
\gr:\Pa_k(\Lambda_0, \Lambda_1)\times \CQ_R\dashrightarrow\Z
\]
defined on the subspace of non-degenerate pairs is subject to the following axioms:
\begin{enumerate}[label=(A'-\Roman{*})]
	\item\label{A=I}(Index Axiom) Let $\p_\pm=(p_\pm, \fa_\pm)$ be non-degenerate and $\p(t): [-1,1]_t\to \Pa_k(\Lambda_0,\Lambda_1)\times\CQ_R$ be any smooth path (if exist) connecting $\p_\pm$ and constant when $0\leq |t-1|\ll 1$. Thus $\p(t)$ defines a smooth map $P:\R_t\times \R_s\to M$, a smooth function $\alpha': \R_t\times \R_s\to \R$ and a smooth 1-form $\delta H'=\delta H'_s ds\in \Omega^1(\R_t\times \R_s, \SH)$ which is supported on $Z_R$ such that for $\pm t\geq 1$, $(P, \alpha', \delta H')$ is constant in time and determined by $\p_\pm=(p_\pm, \fa_\pm)$. Then the Fredholm operator 
	\begin{align*}
	\D_P: L^2_1(\R_t\times \R_s; P^*TM)&\to L^2(\R_t\times \R_s; P^*TM)\\
	v&\mapsto \frac{D}{dt}v+J\frac{D}{ds}v+\Hess_P\big(\im(e^{-i\alpha'}W)+\delta H'_s\big)(v).
	\end{align*}
	has index $=\gr(\p_-)-\gr(\p_+)$ (which is interpreted as a spectral flow). 
	\item\label{A=II}(Normalization Axiom) Suppose that for some $R_1\geq R$, $p(s)\equiv q_0$ when $s\leq -R_1$ and $\equiv q_1$ when $s\geq R_1$, then the restriction of $p$ on $[-R_1, R_1]_s$ determines a path $p_{R_1}\in \Pa(\Lambda_0, \Lambda_1)_{R_1}$. Moreover, any Floer datum $\fa=(R,\alpha(s),\beta, \epsilon_{01},\delta H)$ defines a smooth 1-form 
	\[
	\delta H_{R_1}^{\fa}\colonequals\im(e^{-i\alpha(s)}W)ds+\delta H\in \Omega^1(Z_R; \SH). 
	\]
	If $\p=(p,\fa)$ is non-degenerate, so is $\p_{R_1}\colonequals (p_{R_1}, \delta H_{R_1}^{\fa})$. Then $\gr(\p)=\overline{\gr}(\p_{R_1})$. 
\end{enumerate}

\begin{remark} There are two remarks in order. First, for any path $p:\R_s\to M$ as in \ref{A=II}, the operator $\Hess_p\CA_{W,\fa}$ takes the stand form \[
	J\ps+\Hess_{q_1}\im(e^{-i\theta_1}W)=J(\ps+ \Hess_{q_1}\re(e^{-i\theta_1}W)).
	\]
	when $s\geq R_1$. If $v(s)\in \ker \Hess_p \CA_{W,\fa}$, then $v(R_1)=T_{q_1}\Lambda_1\in T_{q_1}M$, which is the negative spectrum subspace of $\Hess_{q_1}\re(e^{-i\theta_1}W))$. The same argument shows that $v(-R_1)\in T_{q_0}\Lambda_0$. Thus $\p=(p,\fa)$ is non-degenerate if and only if $\p_{R_1}$ is non-degenerate. 
	
	Secondly, the grading $\overline{\gr}(\p_{R_1})$ is independent of the choice of $R_1$. Indeed, one may always pull back the pair $\p_{R_1}=(p_{R_1}, \delta H^\fa_{R_1})$ to the unit interval $I_1=[-1,1]_s$. By changing $R_1$ one obtains a continuous family of invertible $L^2$-self-adjoint operators defined on $I_1$ with the same domain; therefore they have the same grading by the Index Axiom \ref{A-I}. 
\end{remark}

\begin{proposition}\label{LG.P.2} The grading function $\gr$ that satisfies Axioms \ref{A=I} and \ref{A=II} exists and is unique. 
\end{proposition}

\begin{proof} The uniqueness of $\gr$ is straightforward. To prove the existence,  suppose that $\p_\pm=(p_\pm,\fa_\pm)$ are non-degenerate and for some $R_1\geq R$, $p_\pm(s)\equiv  q_0$ when $s\leq -R_1$ and $\equiv q_1$ when $s\geq R_1$, then their gradings are determined by Axiom \ref{A=II}. We have to show that for this class of non-degenerate pairs Index Axiom \ref{A=I} holds. To this end, consider a smooth map $P:\R_t\times \R_s\to M$ such that 
	\[
P(t,s)=\left\{\begin{array}{ll}
q_0 &\text{ if }s\leq -R_1,\\
q_1 & \text{ if }s\geq R_1,\\
p_-(s)& \text{ if }s\leq t,\\
p_+(s)& \text{ if }s\geq t,
\end{array}
\right.
	\]
	and choose smooth interpolations for $\alpha_\pm$ and $\delta H_\pm$ respectively to define the operator $\D_P$. Moreover, we can truncate these data on the strip $Z_{R_1}=\R_t\times [-R_1, R_1]_s$ and define another operator $\D_{P_{R_1}}$ by \eqref{LG.E.5}. By the Index Axiom \ref{A-I}, we have 
	\[
\Ind\D_{P_{R_1}}=\overline{\gr}(\p_{+,R_1})-\overline{\gr}(\p_{-,R_1})=\gr(\p_+)-\gr(\p_-).
\]
 It remains to verify that $\Ind \D_{P_{R_1}}=\Ind \D_P$. The easiest way to see this is by the excision principle, for which we have to introduce four auxiliary operators. To ease our notations, denote $\Pi^\pm_j$ by the positive (resp. negative) spectrum subspace of $\Hess_{q_1}\re(e^{-i\theta_j}W)$; then $T_{q_j}\Lambda_j=\Pi^-_j\subset T_{q_j}M$, $j=0,1$. Let 
\[
Z^-_{R_1}=\R_t\times (-\infty,R_1]_s \text{ and } Z^+_{R_1}=\R_t\times [-R_1,+\infty)_s.
\]
Given a symplectic vector bundle $V\to Z_{R_1}$ and Lagrangian sub-bundles  $F_\pm \subset V|_{\R_t\times \{\pm R_1\}}$, we use
\[
L^2_1(Z_{R_1}, \partial Z_{R_1}; V, F_-, F_+)
\]
to denote the space of $L^2_1$-sections $v$ of $V$ such that $v(t, \pm R_1)\in F_\pm$ for all $t\in \R_t$. We adopt a similar convention for sections over $Z^\pm_{R_1}$ in which case only one boundary condition is needed. With that being said, consider the operators:
\begin{align*}
&\D_1:L^2_1(Z_{R_1},  \partial Z_{R_1};T_{q_1}M, \Pi_1^+, \Pi_1^-)\to L^2(Z_{R_1}; T_{q_1}M),\\
&\D_1^+: L^2_1(Z_{R_1}^+, \partial Z_{R_1}^+;T_{q_1}M, \Pi_1^+)\to L^2(Z_{R_1^+};T_{q_1}M),\\
&\D_1^+=\D_1=\pt+J\big(\ps+\Hess_{q_1}\re (e^{-i\theta_1}W)\big),
\end{align*}
and 
\begin{align*}
&\D_0:L^2_1(Z_{R_1},  \partial Z_{R_1};T_{q_0}M, \Pi_0^-, \Pi_0^+)\to L^2(Z_{R_1}; T_{q_0}M),\\
&\D_0^-: L^2_1(Z_{R_1}^-, \partial Z_{R_1}^-;T_{q_0}M, \Pi_0^+)\to L^2(Z_{R_1}^-;T_{q_0}M),\\
&\D_0^-=\D_0\colonequals\pt+J\big(\ps-\Hess_{q_0}\re(e^{-i\theta_0}W)\big).
\end{align*}
These operators are invertible -- they are all cast into the form $\pt+\D'$ for some invertible $L^2$-self-adjoint operator $\D'$. The pattern for the boundary conditions is summarized as in Figure \ref{Pic37} below. Note that $\D_P=\D_{P_{R_1}}=\D_1$ on $\R_t\times [R, R_1]_s$ and $=\D_0$ on $\R_t\times [-R_1,-R]_s$. Then the excision principle shows that for all $R_1\gg R$,
\[
\Ind \D_{P_{R_1}}+\Ind \D_1^++\Ind \D_0^-=\Ind \D_P+\Ind\D_1+\Ind\D_0,
\]
and therefore $\Ind \D_{P_{R_1}}=\Ind \D_P$. A concrete proof of this excision formula can be found in \cite[Section A.3]{Wang20} or in the proof of the vertical gluing theorem in Section \ref{SecVT}; see Lemma \ref{VT.L.4}. We leave the details to interested readers.
\end{proof}

\begin{remark} This excision principle is extremely useful as it reduces any index computation on such non-compact Riemann surfaces to the classical case of Lagrangian boundary conditions. This idea will be used repeatedly in the rest of the paper.
\end{remark}

\begin{figure}[H]
	\centering
	\begin{overpic}[scale=.12]{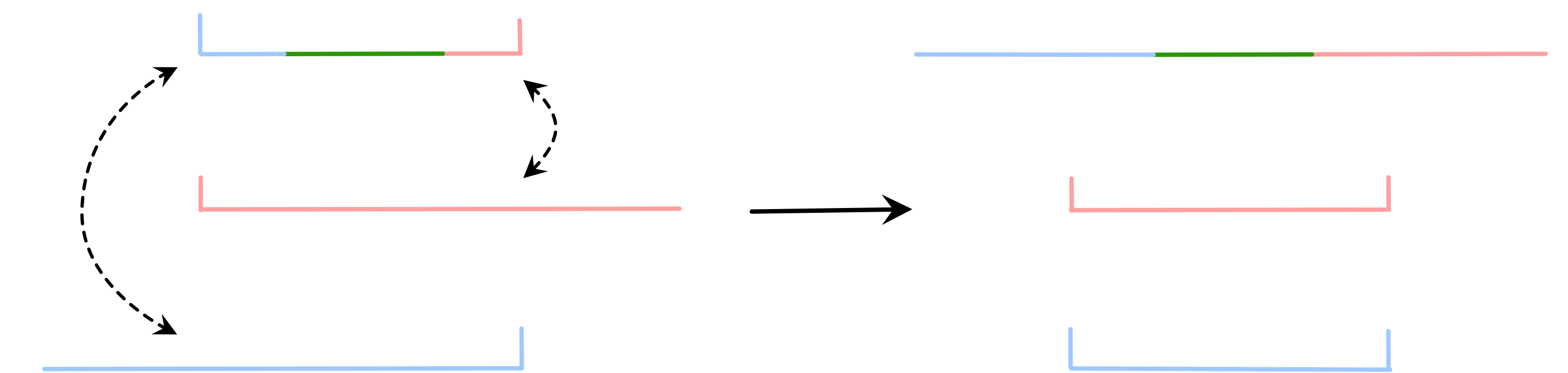}
		\put(49,12){\small Excise}
		\put(22,22){\small $\D_{P_{R_1}}$}
			\put(22,12){\small $\D_1^+$}
				\put(22,2){\small $\D_0^-$}
								\put(77,22){\small $\D_P$}
									\put(77,12){\small $\D_1$}
										\put(77,2){\small $\D_0$}
										\put(7,21){\small $\Pi_0^-$}
										\put(35,21){\small $\Pi_1^-$}
										\put(7,10){\small $\Pi_1^+$}
													\put(35,0){\small $\Pi_0^+$}
														\put(62,10){\small $\Pi_1^+$}
																					\put(91,10){\small $\Pi_1^-$}
																						\put(62,0){\small $\Pi_0^-$}
																					\put(91,0){\small $\Pi_0^+$}
	\end{overpic}	
	\caption{The Excision Principle.}
	\label{Pic37}
\end{figure}
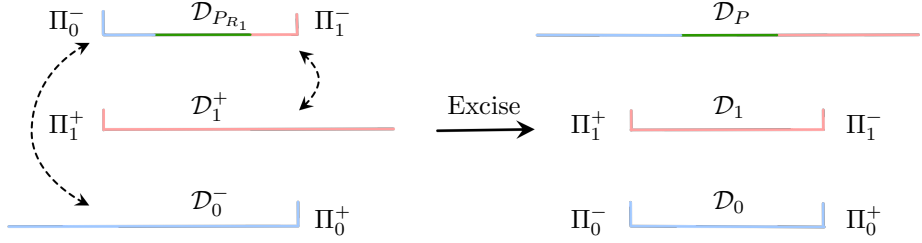
\begin{remark} For any $\alpha$-soliton $p\in \FC(\Lambda_0, \Lambda_1;\fa)$, we actually have $p(s)\in \Lambda_0$ when $s\leq 0$ and $\in \Lambda_1$ when $s\geq R$. If $\p=(p, \fa)$ is non-degenerate, then the linearized Hamiltonian flow along the path $p(s)$ transports $T_{p(-R)}\Lambda_0\subset T_{p(-R)}M$ into a graded Lagrangian subspace of $T_{p(R)}M$ transverse to $T_{p(R)}\Lambda_1$. A normalization property like Axiom \ref{A-II} still holds in this case. Although it is very enlightening to have this in mind, this property is not very useful for the purpose of this work and will not be proved here. 
\end{remark}


\subsection{Perturbation and Transversality}\label{Sec1.6} A Floer datum $\fa=(R, \alpha(s), \beta,\epsilon, \delta H)$ is called \textit{admissible} if any $\alpha$-soliton $p\in \FC(\Lambda_0, \Lambda_1;\fa)$ is non-degenerate, i.e., the Hessian of $\CA_{W,\fa}$ at $p$ is invertible, and for any $\alpha$-solitons $p_\pm$ and $P\in \M(p_-, p_+)$, the linearization of \eqref{E1.6} is a surjective Fredholm operator from $L^2_1\to L^2$. The support of $\delta H\in \Omega^1(\R_s;\SH)$ is always confined in $[0,R]_s$, but one can always make $\fa$ admissible by choosing $\delta H$ generically. This is due to a unique continuation property: if two solutions $P_0, P_1$ of \eqref{E1.6} are equal on $\{t_0\}\times [0,R]_s$ for some $t_0\in \R_t$, then they are equal on the whole space. By taking derivatives of \eqref{E1.6} repeatedly, one can show that under this assumption $P_0, P_1$ along with all their higher derivatives would agree at the point $(t_0,0)$, then this statement follows from Aronszajn's theorem \cite[Theorem 2.3.4]{MS12}.

\section{The Floer Equation on General Riemann Surfaces}\label{SecGRS}
\subsection{Phase functions, the energy estimate and the Floer equation} Having constructed the Floer cohomology $\HFF^*_\natural(\Lambda_0,\Lambda_1;\fa)$, our next goal is to understand its dependence on the Floer datum $\fa$. This requires generalizing the $\alpha$-instanton equation \eqref{E1.6} to a general Riemann surface and derive the energy estimate. Recall that any smooth function $K\in \SH\colonequals C^\infty(M;\R)$ induces a Hamiltonian vector field $X_K\in C^\infty(M; TM)$ with $dK=\omega_M(\cdot, X_K)$. The Poisson bracket
$
\{K_1, K_2\}\colonequals \omega_M(X_{K_1}, X_{K_2})
$
makes $C^\infty(M;\R)$ into a Lie algebra, and the map 
\begin{equation}\label{E1.10}
(C^\infty(M;\R),\{\cdot,\cdot\})\to (C^\infty(M; TM), [\cdot, \cdot]),\  K\mapsto X_K
\end{equation}
is a Lie algebra homomorphism. We identify $\C$ as a subspace of $C^\infty(M ;\R)$ by sending 
\begin{equation}\label{E1.11}
x\mapsto \re(\bar{x} W);
\end{equation}
so $1\mapsto L$ and $i\mapsto H$. Given any compact Riemann surface $S$ with boundary, a complex 1-form $\kappa\in\Omega^1(S; \C)$ becomes a $C^\infty(M;\R)$-valued 1-form under the map \eqref{E1.11}. By  \eqref{E1.10}, $\kappa$ induces a Hamiltonian vector field $
X_{\kappa}\in \Omega^1(S; C^\infty(M; TM))$. We summarize the relevant maps as follows:
\[
\begin{tikzcd}
TS \arrow[r,"\kappa"]\arrow[rrr,"X_{\kappa}", bend left] & \C\arrow[r, "\eqref{E1.11}"] &(C^\infty(M;\R),\{\cdot,\cdot\}) \arrow[r,"\eqref{E1.10}"] & (C^\infty(M; TM), [\cdot, \cdot]).
\end{tikzcd}
\]
\begin{definition} \textit{A phase function} is a smooth map $\Xi=(a,b): S\to \R^2\cong \C$ with $\det D\Xi\colonequals \Xi^*\omega_{\R^2}/dvol_S\leq 0$. Its induced 1-form $\kappa_\Xi$ is defined by the formula
	\begin{equation}\label{E1.17}
	\kappa_\Xi=(d_Sb)\cdot L+(-d_S a)\cdot H=-\im (\overline{d_S\Xi}\cdot W)\in \Omega^1(S; C^\infty(M; \R)).
	\end{equation}
	 The curvature of $d+\kappa_\Xi$ is computed as 
	\begin{align}\label{E1.13}
	F_{d+\kappa_\Xi}&=d_S\kappa_\Xi+\half \{\kappa_\Xi\wedge\kappa_\Xi\}=d_S a\wedge d_Sb\cdot \{L, H\}=|\nabla H|^2\cdot \Xi^*\omega_{\R^2}\\
	&=|\nabla H|^2(\det D\Xi)\cdot dvol_S\in \Omega^2(S; C^\infty(M; \R)),\nonumber
	\end{align}
	which is non-positive everywhere on $M$. The Hamiltonian vector field of $\kappa_\Xi$ is given by the formula
	\[
	X_{\kappa_\Xi}=(d_Sa)\cdot \nabla L+(d_S b)\cdot \nabla H=\nabla \re(\overline{d_S\Xi}\cdot W)\in \Omega^1(S; C^\infty(M; TM)). \qedhere
	\]
\end{definition}

The Floer equation defined for a smooth map $P: S\to M$ reads as follows:
\begin{equation}\label{E1.12}
(dP-X_{\kappa_\Xi+\delta\kappa}-X_{\delta H})^{0,1}=0. 
\end{equation}
where $\delta\kappa\in\Omega^1(S; \C)$ is an correction term and $\delta H\in \Omega^1(S;\SH)$ is a perturbation 1-form. Recall that $\SH$ is the subspace of $C^\infty(M;\R)$ with finite $L^\infty_1$-norm. Define \textit{the energy density function} of $P$ over $S$ as 
\[
u=|\nabla P|^2+|\nabla H\circ P|^2: S\to \R^+.
\]
\begin{lemma}[The Energy Estimate II]\label{L1.12} Suppose the phase function $\Xi: S\to \R^2$ and the correction 1-form $\delta\kappa\in \Omega^1(S; \C)$ satisfy the pointwise bound
	\begin{equation}\label{PointwiseEstimate}
	-\det D\Xi-|\delta\kappa^{0,1}|^2\geq  \epsilon_S
	\end{equation}
	for some constant  $\epsilon_S>0$. Then there exists some $\epsilon_S'>0$ such that for any solution $P: S\to M$ to the Floer equation \eqref{E1.12}, we have 
	\begin{equation}\label{E1.15}
	\int_S P^*\omega_M-\int_{\partial S} (\Id, P)^*\kappa_\Xi\geq \epsilon_S' \int_S u\cdot dvol_S-4\int_S \|\delta H\|_{L^\infty_1}^2,
	\end{equation}
where $\|\delta H\|_{L^\infty_1}$ denotes the norm of $\delta H$ as a map $TS\to L^\infty_1(M;\R)$. 
\end{lemma}
\begin{proof} Take any $z\in S$ and let $\{\pt, \ps\}$ be an oriented orthonormal basis of $T_zS$. Then the equation $(dP-X_{\kappa_\Xi+\delta\kappa})^{0,1}(\pt)$ is cast into the form
	$
	A+ JB-C=D
	$ with
\begin{align*}
A&=\pt P-X_{\kappa_\Xi}(\pt), & B&=\ps P-X_{\kappa_\Xi}(\ps),\\
 C&=X_{\delta\kappa}(\pt)+JX_{\delta\kappa}(\ps), & D&= X_{\delta H}(\pt)+JX_{\delta H}(\ps).
\end{align*}
In particular, 
\begin{equation}\label{E1.14}
|D|^2=|A+JB-C|^2=|A-C|^2+|JB-C|^2+2\langle A, JB\rangle-|C|^2. 
\end{equation}

Let $(\Id, P):S\to S\times M$ be the graph of $P$. If one thinks of $\kappa_\Xi\in C^\infty(S\times M; T^*S)$ and $F_{d+\kappa_\Xi}\in C^{\infty}(S\times M; \Lambda^2 T^*S)$ as differential forms on $S\times M$, then the computation in \cite[Lemma 8.1.6]{MS12} says that 
\[
\langle A, JB\rangle\cdot dvol_S=-P^*\omega_M+d_S\big((\Id, P)^*\kappa_\Xi\big)-(\Id, P)^*F_{d+\kappa_\Xi}. 
\]

Combined with \eqref{E1.13} and \eqref{E1.14}, we deduce that 
\begin{align*}
&P^*\omega_M-d_S((\Id,P)^*\kappa_\Xi)\\
=&\half (|A-C|^2+ |JB-C|^2- |C|^2-|D|^2)\cdot dvol_S-(\Id, P)^*F_{d+\kappa_\Xi}\\
=&\big(\half |A-C|^2+\half |JB-C|^2+|\nabla H\circ  P|^2 (-\det D\Phi-|\delta \kappa^{0,1}|^2)-|D|^2\big)\cdot dvol_S\\
\geq & \big(\half |A-C|^2+\half |JB-C|^2+\epsilon_S|\nabla H\circ  P|^2-|D|^2 \big)\cdot dvol_S\\
\geq & \big(\epsilon_S' u-4\|\delta H\|^2_{L^\infty_1}\big)\cdot dvol_S
\end{align*}
for some $\epsilon_S'>0$. Now we integrate both sides over the surface $S$. 
\end{proof}
\begin{figure}[H]
	\centering
	\begin{overpic}[scale=.10]{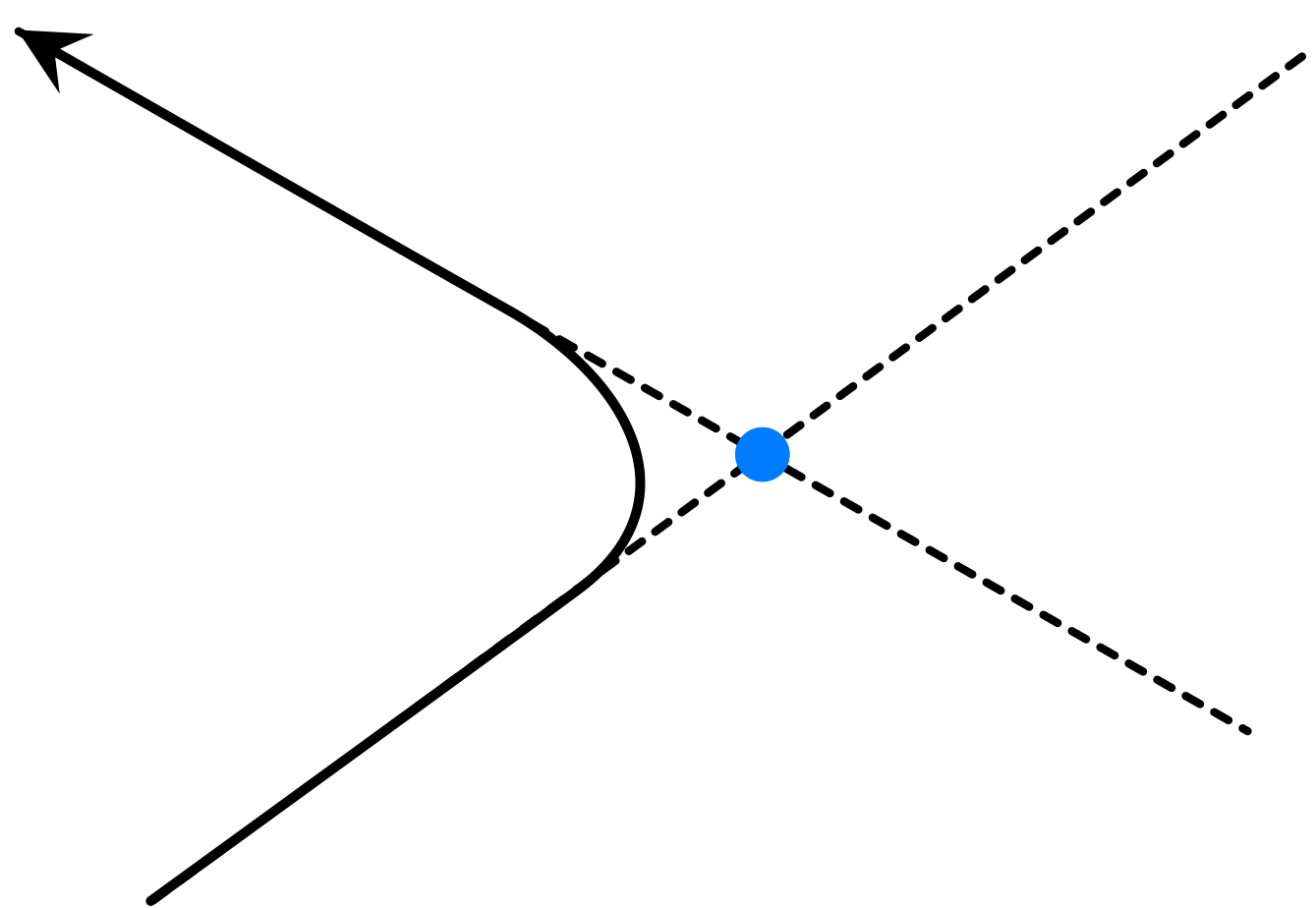}
			\put(65,32){\small $\leftarrow$ the origin}
		\put(25,30){\small $\gamma(s)$}
			\put(-17,67){\small $-e^{i\theta_1}$}
				\put(-10,-5){\small $-e^{i\theta_0}$}

	\end{overpic}	
	\caption{The characteristic curve $\gamma(s)$ for $\alpha(s)$.}
	\label{Pic4}
\end{figure}
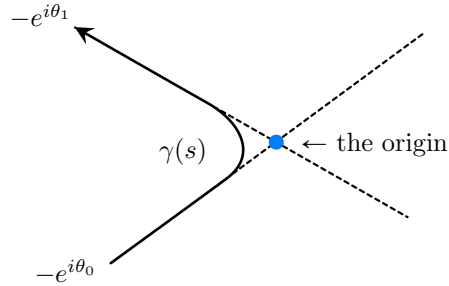

\begin{definition} \textit{A phase pair} $(\Xi,\delta\kappa)$ on a Riemann surface $S$ is a phase function $\Xi: S\to \C$ along with a correction term $\delta \kappa$ satisfying the estimate \eqref{PointwiseEstimate}.
\end{definition}

\begin{example}\label{EX2.2} For any Floer datum $\fa=(R,\alpha(s),\beta,\epsilon_{01}, \delta H)$ one can associate \textit{a canonical phase pair} as follows. Take $S=[t_0,t_1]\times [-R_1,R_1]\subset \R_t\times \R_s$ to start. We say that $\gamma: \R_s\to \C$ is \textit{a characteristic curve} for the function $\alpha(s)$ if $\ps \gamma(s)=-e^{i\alpha(s)}$. This curve is not unique; one possible normalization is to require that
	\begin{itemize}
		\item $\gamma(s)\in \{-re^{i\theta_1}: r\geq 0\}$ when $s\gg 1$ and $\in \{-re^{i\theta_0}: r\geq 0\}$ when $s\ll -1$.
	\end{itemize}
Now consider the phase function $
	\Xi(t,s)=-i\epsilon_{01}e^{i\beta}\cdot t+\gamma(s),$ then 
	\begin{align*}
	-\det D\Xi&=\epsilon_{01}\cos(\beta-\alpha(s)),&\kappa_\Xi&=\im ((-i\epsilon_{01}e^{-i\beta}dt+e^{-i\alpha(s)}ds)\cdot W).
	\end{align*}
Let $\delta\kappa=\im (i\epsilon_{01}e^{-i\beta}dt\cdot W)$, then we have
\begin{align*}
X_{\kappa_\Xi+\delta \kappa}(\pt)&=0, &X_{\kappa_\Xi+\delta \kappa}(\ps)&=-\nabla \re(e^{-i\alpha(s)}W),
\end{align*}
and the Floer equation \eqref{E1.12} recovers the $\alpha$-instanton equation \eqref{E1.6} on $\R_t\times \R_s$. In this case, since $|\delta\kappa^{0,1}|^2= \epsilon_{01}^2/2$, and our condition \eqref{E1.4} ensures that $-\det D\Xi\geq\epsilon_{01}^2$, we may set $\epsilon_S=\epsilon_{01}^2/2$. The left hand side of \eqref{E1.15} reduces to 
\[
\int_{\partial S} P^*\lambda_M+\im \big((i\epsilon_{01}e^{-i\beta}dt-e^{-i\alpha(s)}ds)\cdot (W\circ P)\big).
\]
This recovers the left hand side of \eqref{EnergyEquation1} as $R_1\to \infty$, if one ignores the terms coming from $\delta H$. The proofs of Lemma \ref{L1.5} and Lemma \ref{L1.12} are indeed identical in this special case. 
\end{example}

\begin{figure}[H]
	\centering
	\begin{overpic}[scale=.15]{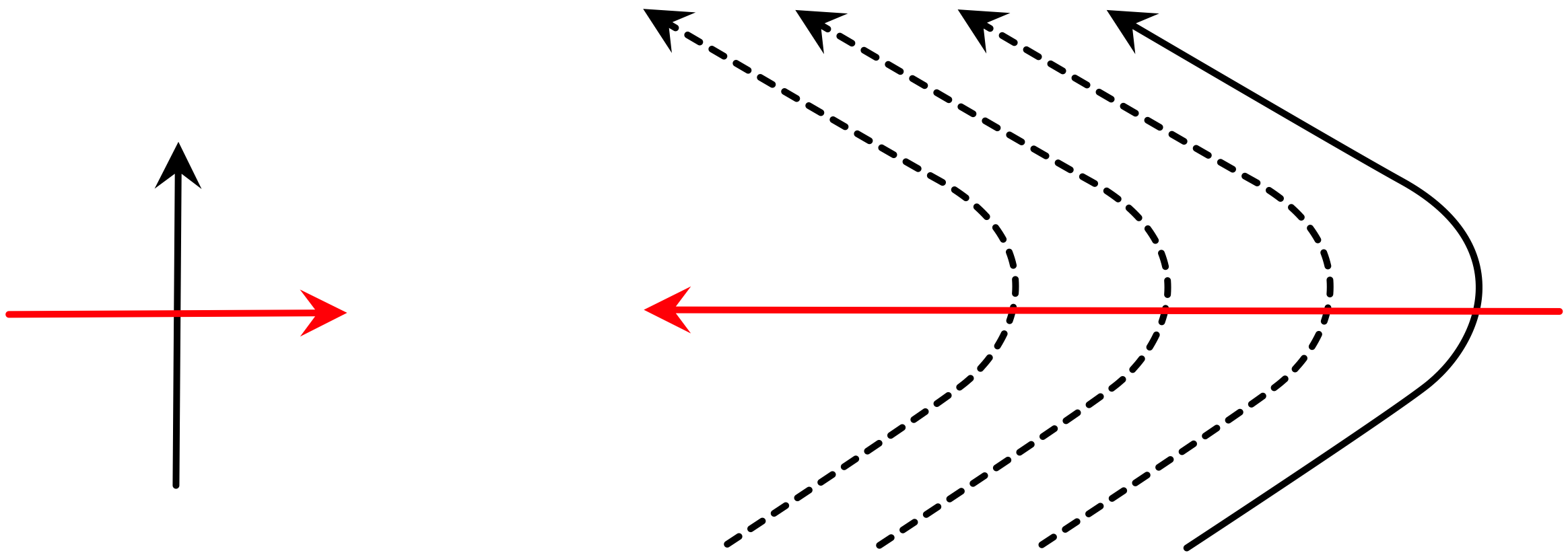}
		\put(25,15){\large $\xrightarrow{\hspace{0.5cm}\Xi\hspace{0.5cm}}$}
		\put(85,30){\small $\gamma(s)$}
		\put(40,12){\small $-i\epsilon_{01}e^{i\beta}\cdot t$}
	\end{overpic}	
	\caption{The phase function $\Xi$ associated to a Floer datum.}
	\label{Pic5}
\end{figure}
\begin{remark} In the classical development of Floer theory for Lefschetz fibrations, the connection $d+\kappa_\Xi$ being negatively curved proves important for any compactness results \cite[P.69]{Sei19}. In our case, the image $W\circ P(S)$ does not admit any a priori control, and the coefficients of $\kappa_\Xi$ are forced to be closed (and so exact if the surface $S$ is contractible); the negativity of $F_{d+\kappa_\Xi}$ arises from the relation $\{L, H\}=|\nabla H|^2$ and the fact that the phase function $\Xi: S\to \C$ is orientation reversing. 
\end{remark}

\subsection{Continuation maps}\label{Subsec:Continuation} Given critical points $q_0, q_1\in \Crit(W)$ and angles $\theta_0, \theta_1\in \R$ with $\theta_1<\theta_0<\theta_1+2\pi$, one can find different interpolations 
\[
\alpha^-(s), \alpha^+(s): \R_s\to \R
\]
satisfying the boundary condition \eqref{E1.3} and such that \eqref{E1.4} holds for possibly different $(\beta^\pm, \epsilon^\pm_{01})$:
\[
\re(e^{i(\beta^\pm-\alpha^\pm(s))})> \epsilon_{01}^\pm>0. 
\] 
Let $\fa^\pm=(R^\pm, \alpha^\pm(s),\beta^\pm, \epsilon^\pm_{01}, \delta H_s^\pm)$ be any admissible Floer data associated to $(\Lambda_0, \Lambda_1)$ where $\Lambda_j=\Lambda_{q_j,\theta_j}, j=0,1$. Using the Floer equation \eqref{E1.12}, we shall construct a continuation map: 
\begin{equation}\label{ContinuationMaps}
\Cont_{\fa_+,\fa_-}: \HFF_\natural^*(\Lambda_0,\Lambda_1;\fa^+)\to  \HFF_\natural^*(\Lambda_0,\Lambda_1;\fa^-)
\end{equation}
such that 
\begin{equation}\label{E2.20}
\Cont_{\fa',\fa'}\circ \Cont_{\fa,\fa'}=\Cont_{\fa,\fa''}
\end{equation}
for any triple $(\fa, \fa',\fa'')$, and $\Cont_{\fa,\fa}=\Id$. This proves that the Floer cohomology $\HFF_\natural^*(\Lambda_0,\Lambda_1;\fa)$ is independent of the choice of the Floer datum $\fa$ up to canonical isomorphisms. 

Replacing the first entry of a Floer datum by a larger number will only change the functional $\CA_{W,\fa}$ by a constant; so we may assume that $R^+=R^-$ to start.  For simplicity we will take $R^\pm=\pi$ in the sequel. 

 Fix a choice of characteristic curves $\gamma^\pm(s)$ for $\alpha^\pm(s)$. We have to specify the phase pair ($\Xi,\delta\kappa$) for the Riemann surface $S=\R_t\times \R_s$ to define this continuation map, and we wish $(\Xi,\delta\kappa)$ agrees with the canonical ones associated to $\fa^\pm$ when $|t|\gg 1$. For any $K>0$, consider the space $\Emb_{K}(\fa^+,\fa^-)$ of smooth embeddings
\[
\Xi^{\dagger}: \R_t\times [0,\pi]_s\to \C
\]
that satisfies the following properties:
\begin{itemize}
\item for some $c_{\Xi^\dagger}\in \C$, we have  \begin{equation}\label{E2.14}
\Xi^\dagger(t,s)=\left\{
\begin{array}{ll}
-i\epsilon_{01}^-e^{i\beta^-}\cdot t+\gamma^-(s) &\text{ on } (-\infty, 0]_t\times [0,\pi]_s,\\
-i\epsilon_{01}^+e^{i\beta^+}\cdot t+\gamma^+(s)+c_{\Xi^\dagger}&\text{ on } [K, \infty)_t\times [0,\pi]_s;
\end{array}
\right. 
\end{equation}
\item for some $0<\delta\ll \pi$,
 \begin{equation}\label{E2.15}
 \Xi^\dagger(t,s)=\left\{
 \begin{array}{ll}
g_1(t)-e^{i\theta_1}\cdot (s-\pi) &\text{ on }\R_t\times (\pi-\delta, \pi]_s,\\
g_0(t)+e^{i\theta_0}\cdot s&\text{ on } \R_t\times [0,\delta)_s,
 \end{array}
 \right. 
 \end{equation}
 where $g_1(t)\colonequals \Xi^\dagger (t,\pi)$ and $g_0(t)\colonequals \Xi^\dagger (t,0)$;\\
 \item for any $t\in \R_t$, 
\begin{equation}\label{E2.16}
-\im (e^{-i\theta_j}\pt g_j(t))-\half |\pt g_j(t))|^2>0, j=0,1.
\end{equation}
\end{itemize}

We shall simple write $\Emb_{K}=\Emb_{K}(\fa^+,\fa^-)$ when the Floer data $\fa^\pm$ are clear from the context. By \eqref{E2.14}, $\Xi^\dagger$ is determined by its restriction on $[-1, K+1]_t\times [0,\pi]_s$. 
We equip $\Emb_{K}$ with the smooth topology as a subspace of $C^\infty([-1, K+1]_t\times [0,\pi]_s,\C)$. For any $K_1<K_2$, we have a natural inclusion map $\Emb_{K_1}\to \Emb_{K_2}$. The proof of the next lemma is deferred to the end of this section.
\begin{lemma}\label{L2.5} If we equip the direct limit $
	\Emb\colonequals \varinjlim \Emb_{K}$ with the inductive topology, then $	\Emb$ is weakly contractible.
\end{lemma}

By \eqref{E2.15}, any $\Xi^\dagger\in \Emb_K$ extends to an orientation-reversing diffeomorphism $\Xi: \R^2\to \C$ by setting (see Figure \ref{Pic6} below)
\begin{equation}\label{E2.8}
\Xi(t,s)=\left\{
\begin{array}{ll}
g_1(t)-e^{i\theta_1}\cdot (s-\pi)&\text{ if } s\geq \pi,\\
\Xi^\dagger(t,s) & \text{ if }s\in [0,\pi]_s,\\
g_0(t)+e^{i\theta_0}\cdot s&\text{ if } s\leq 0.
\end{array}
\right. 
\end{equation}
Finally, away from the rectangle $[0, K]_t\times [0,\pi]_s$, define the correction 1-form $\delta \kappa$ by the formula:
\begin{equation}\label{E2.9}
\delta\kappa(t,s)=\left\{
\begin{array}{cl}
\im(i\epsilon_{01}^-e^{-i\beta^-}dt\cdot W)&\text{ on } (-\infty, 0]_t\times \R_s,\\
\im(i\epsilon_{01}^+e^{-i\beta^+}dt\cdot W)&\text{ on } [K,+\infty)_t\times \R_s,\\
\im (\overline{\pt g_1(t)}dt\cdot W)&\text{ on } \R_t\times [\pi,+\infty)_s,\\
\im (\overline{\pt g_0(t)}dt\cdot W)&\text{ on } \R_t\times (-\infty, 0]_s,
\end{array}
\right. 
\end{equation}

\begin{definition} Given any Floer data $\fa^\pm=(\pi, \alpha^\pm,\beta^\pm, \epsilon_{01}^\pm, \delta H^\pm)$ for the thimbles $(\Lambda_0, \Lambda_1)$, a continuation datum is a quadruple $\fc=(K, \Xi,\delta\kappa, \delta H)$, where
	\begin{itemize}
\item $K>0$ and the phase pair $(\Xi, \delta\kappa)$ is constructed as above using $\Xi^\dagger\in \Emb_K$ and such that for some $\epsilon_\fc>0$,
\begin{equation}\label{E2.17}
-\det D\Xi-|\delta \kappa^{0,1}|^2>\epsilon_\fc \text{ for all }(t,s)\in\R_t\times\R_s;
\end{equation}
\item the perturbation 1-form $\delta H\in \Omega^1(\R_t\times \R_s;\SH)$ is supported on $\R_t\times [0,\pi]_s$, and $\delta H=\delta H^+$ when $t\leq 0$ and $=\delta H^-$ when $t\geq K$.
	\end{itemize}

Since $(\Xi, \delta\kappa)$ agrees with the phase pair associated to $\fa^\pm$ when $t\leq 0$ or $\geq K$ (up to a translation by the constant $c_{\Xi^\dagger}\in \C$ in \eqref{E2.14}), the condition \eqref{E2.16} implies that the pointwise bound \eqref{E2.17} holds away from the rectangle $[0,K]_t\times [0,\pi]_s$ for some $\epsilon_\fc>0$. Hence, \eqref{E2.17} is fulfilled for a possibly smaller $\epsilon_\fc$ if the extension of $\delta \kappa$ on $[0, K]_t\times [0, \pi]_s$ is chosen carefully enough.
\end{definition}
\begin{figure}[H]
	\centering
	\begin{overpic}[scale=.15]{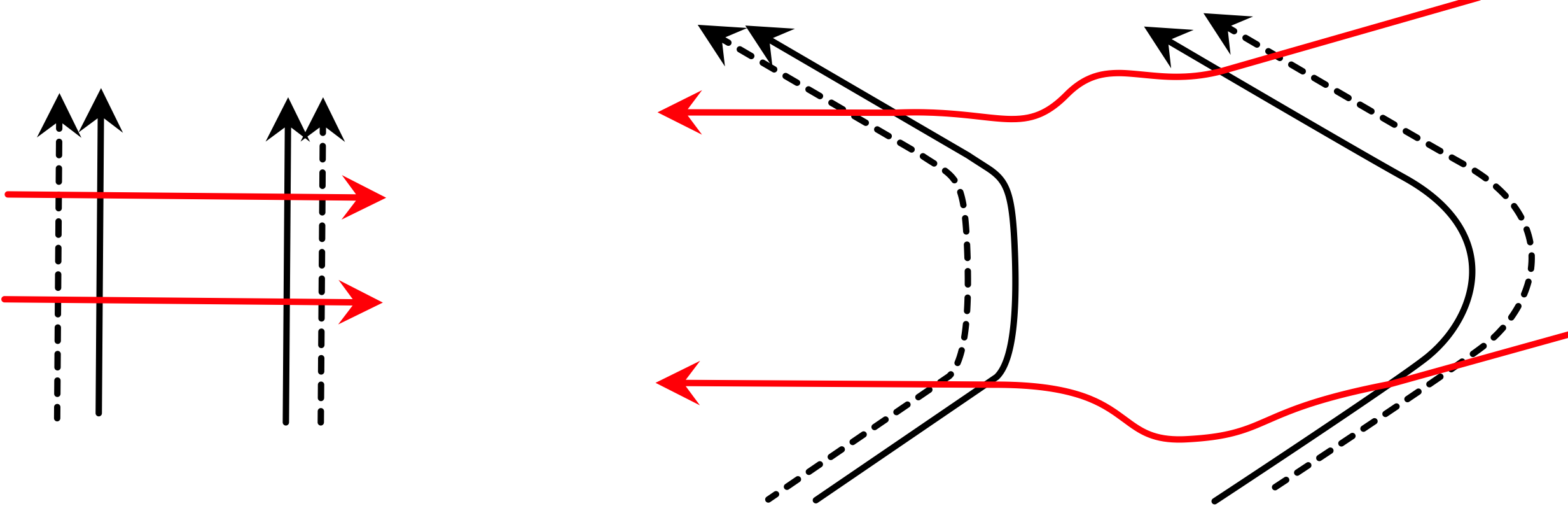}
		\put(30,15){\large $\xrightarrow{\hspace{0.5cm}\Xi\hspace{0.5cm}}$}
		\put(86,15){\small $\gamma^-(s)$}
			\put(66,15){\small $\gamma^+(s)$}
			\put(62,28){\small $g_1(t)$}
				\put(65,3){\small $g_0(t)$}
		\put(40,10){\small $-i\epsilon_{01}^+e^{i\beta^+}\cdot t$}
			\put(100,11){\small $-i\epsilon_{01}^-e^{i\beta^-}\cdot t$}
			\put(6,2){\small $0$}
			\put(17,2){\small $K$}
			\put(-3,12){\small $0$}
				\put(-3,19){\small $\pi$}
	\end{overpic}	
	\caption{The phase function $\Xi: \R_t\times \R_s\to \C$ for a continuation datum.}
	\label{Pic6}
\end{figure}
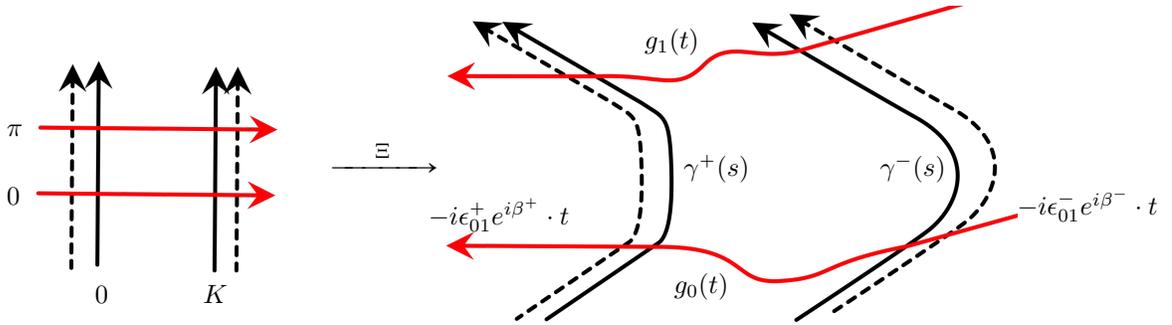

Away from the rectangle $[0, K]_t\times [0,\pi]_s$, the Floer equation \eqref{E1.12}  associated to any continuation datum $\fc$ takes the form
\[
\pt P+J\big(\ps P+\nabla \re(e^{-i\tilde{\alpha}(t,s)}W
)\circ P\big)+\nabla (\delta H_s)=0
\]
 with 
\begin{equation}\label{E2.12}
\tilde{\alpha}(t,s)=\left\{
\begin{array}{ll}
\alpha^-(s)&\text{ if } t\leq 0,\\
\alpha^+(s)&\text{ if } t\geq K,\\
\theta_1&\text{ if } s\geq \pi,\\
\theta_0-\pi &\text{ if }s\leq 0. 
\end{array}
\right. 
\end{equation}

The boundary condition for a solution of \eqref{E1.12} is dictated by the same model map $P_{\model}: \R_t\times \R_s\to M$ as in \eqref{E1.7} but this time with 
\[
p_\pm(s)\in \FC(\Lambda_0, \Lambda_1;\fa^\pm).
\]

To derive the energy estimate, we apply Lemma \ref{L1.12} to the rectangle $[0,K]_t\times [-R_1,R_1]_s$ and let $R_1\to \infty$ to obtain:
\begin{align}\label{E2.13}
&\CA_{W,\fa^-}(P(0,\cdot ))-\CA_{W,\fa^+}(P(K,\cdot ))+ \text{ some terms involving }\delta H
\\
&+\im(W(q_0)\cdot \overline{g_0(K)-g_0(0)})-\im(W(q_1)\cdot \overline{g_1(K)-g_1(0)})>\epsilon_\fc'\int_{[0,K]\times \R_s} u.\nonumber
\end{align}
The constant $\epsilon_{\fc}'>0$ depends only on the continuation datum $\fc$.  Since $\tilde{\alpha}$ is constant when $|s|\geq \pi$, the proof of Proposition \ref{P1.8} remains valid in this case. Hence one may form the moduli space $\M(p_-,p_+; \fc)$ (which does not carry an $\R_t$-action) and prove the compactness theorem as usual. The transversality is achieved by choosing $\delta H$ generically, so there is a well-defined chain map:
\begin{align*}
 \Ch_\natural^*(\Lambda_0,\Lambda_1;\fa^+)&\to  \Ch_\natural^*(\Lambda_0,\Lambda_1;\fa^-)\\
 p_+&\mapsto \sum_{p_-:\ \dim \M(p_-,p_+)=0} \#\M(p_-, p_+;\fc)\cdot p_-.
\end{align*}

The induced map on cohomology is the continuation map $\Cont_{\fa^+,\fa^-}$ in \eqref{ContinuationMaps}, which is independent of the choice of $\fc$, since the space $\Emb=\varinjlim \Emb_K$ is path-connected by Lemma \ref{L2.5}. For any Floer data $(\fa,\fa',\fa'')$, we have a composition map:
 \begin{equation}\label{E2.22}
  \Emb_{K_2}(\fa',\fa'')\times  \Emb_{K_1}(\fa,\fa')\to  \Emb_{K_1+K_2}(\fa,\fa'')
 \end{equation}
sending a pair $(\Xi_2^\dagger, \Xi_1^\dagger)$ to the embedding
 \[
 \Xi^\dagger(t,s)\colonequals\left\{
 \begin{array}{ll}
 \Xi_2^\dagger (t,s) &\text{ if }t\leq K_2,\\
 \Xi_1^\dagger (t-K_2,s)+c_{\Xi^\dagger_2}+(-i\epsilon_{01}'e^{i\beta'})\cdot K_2&\text{ if }t\geq K_2,
 \end{array}
 \right.
 \]
 with $c_{\Xi^\dagger}\colonequals c_{\Xi_1^\dagger}+c_{\Xi_2^\dagger}+(-i\epsilon_{01}'e^{i\beta'})\cdot K_2$. This allows us to compose the cobordism data and verify the composition law \eqref{E2.20}. Thus, we have proved 
\begin{proposition}\label{P:Invariance1} The Floer cohomology $\HFF_\natural^*(\Lambda_0,\Lambda_1;\fa)$, which we shall denote by  $\HFF_\natural^*(\Lambda_0,\Lambda_1)$ in the sequel, is independent of the Floer data $\fa$ up to canonical isomorphisms.
\end{proposition}

Proposition \ref{P:Invariance1} allows us to compute $\HFF_\natural^*(\Lambda_0,\Lambda_1)$ in some trivial cases.
\begin{lemma}\label{lemma:Vanishing} Recall that the projection $W(\Lambda_{q_j,\theta_j})$ is a ray $l_{q_j, \theta_j}$ emanating from $W(q_j), j=0,1$ $($see Figure \ref{Pic2}$)$.  If the rays $l_{q_0,\theta_0}$ and $l_{q_1,\theta_1}$ do not intersect in $\C$, then $\HFF_\natural^*(\Lambda_0,\Lambda_1)$ is trivial. 
\end{lemma}
\begin{proof} Consider a sequence of smooth monotone functions $\alpha^n(s)$ ``approximating" the discontinuous function  \[
\alpha^\dagger(s)=	\left\{\begin{array}{ll}
\theta_1 & \text{ if }s>0,\\
 \theta_0-\pi & \text{ if }s\leq 0, \\
	\end{array}
	\right.
	\]
For instance, we require $\ps \alpha^n$ to be supported on $[-\frac{1}{n},\frac{1}{n}]_s$. Since $\alpha^n(s)$ is monotone, we can choose the same pair $(\beta,\epsilon_{01})$ in \eqref{E1.4} for all $n$. Let $\fa^n=(\pi, \alpha^n, \beta,\epsilon_{01}, \delta H^n\equiv 0)$, and pick an $\alpha^n$-soliton $p^n\in \FC(\Lambda_0,\Lambda_1; \fa^n)$, if this space is non-empty for all $n$. We have uniform energy estimates for this sequence $(p_n)$ by Lemma \ref{L1.5}. By the compactness theorem, $(p^n)$ contains a subsequence converging a solution of \eqref{E1.9} with $\alpha(s)=\alpha^\dagger(s)$ and $\delta H=0$. But for this discontinuous $\alpha^\dagger$, such a soliton (consider the evaluation at $s=0$) are in bijection with points in $\Lambda_0\cap \Lambda_1$, which is empty if $l_{q_0,\theta_0}$ and $l_{q_1,\theta_1}$ do not intersect.
\end{proof}
\begin{lemma}\label{lemma:CanonicalGenerator}For $q_0=q_1=q\in \Crit(W)$ and any $\theta_1<\theta_0<\theta_1+2\pi$, the group $\HFF_\natural^*(\Lambda_0,\Lambda_1)$ is an 1-dimensional vector space generated canonically by the constant map $e_q$ at $q$, which has $\gr(e_q)=0$.
\end{lemma}
\begin{proof} Let $\fa$ be any Floer datum with $\delta H=0$. Then for any $p(s)\in \FC(\Lambda_0,\Lambda_1;\fa)$, the energy estimate from Lemma \ref{L1.5} shows that $\int_{\R_s} |\ps p|^2+|\nabla H\circ p|^2=0$; so a soliton $p(s)\equiv q$ must be constant. Since $q_0=q_1=q$, $T_q\Lambda_0$ and $T_q\Lambda_1$ lives in the same symplectic vector space $T_qM$. Then one uses the Normalization Axiom \ref{A=II} from Section \ref{SecLG.4} and \cite[Example 11.20]{S08} to show that
	 \[
	 \gr(e_q)=\overline{\gr}(e_q')=i(T_q\Lambda_0, T_q\Lambda_1)=n\bigg(\lfloor \frac{\theta_1-\theta_0}{2\pi}\rfloor  +1\bigg)=0.
	 \]
	 where $e_q'$ is the restriction of $e_q$ on any finite interval, and each $\Lambda_j$ is graded as in \eqref{LG.E.4}.
	
Now given any Floer data $\fa^\pm$ with trivial perturbation 1-forms, we also set $\delta H=0$ in the cobordism data $\fc$. Then the left hand side of the energy estimate \eqref{E2.13} is equal to
	\begin{align*}
\CA_{W,\fa^-}(p_-)-\CA_{W,\fa^+}(p_+)+\im(W(q)\cdot \overline{g_0(K)-g_0(0)-g_1(K)+g_1(0)}).
	\end{align*}
	By \eqref{E2.14}\eqref{E2.15} and the fact that $p_\pm(s)\equiv q$, this sum is zero. Hence, the Floer equation defining the continuation map has only a constant solution $P(t,s)\equiv q$. To see that the linearized operator $\D$ at this constant solution is invertible, it suffices to work with the Landau-Ginzburg model $(T_qM, v\mapsto \half \langle \Hess_q W (v), v\rangle_{g_M})$ (this is Example \ref{EX1.2}) for which the linearized problem is identical to the non-linear one. The same energy argument shows that the Hessians of $\CA_{W,\fa^\pm}$ at $p_\pm$ are invertible and $\D$ is injective. To see that $\D$ has Fredholm index $0$, one uses the Index Axiom \ref{A=I}. This proves that $\D$ is also surjective.
\end{proof}

\begin{lemma}[Poincar\'{e} duality]\label{L2.11} For any $\theta_1<\theta_0<\theta_1+2\pi$ and $q_0, q_1\in \Crit(W)$, there is a grading-preserving isomorphism:
	\[
	\HFF_\natural^*( \Lambda_{q_1\theta_1},\Lambda_{q_0,\theta_0-2\pi})\cong D\HFF_\natural^*(\Lambda_{q_0,\theta_0}, \Lambda_{q_1,\theta_1})
	\]
	where $D$ denotes the dual of a graded vector space.
\end{lemma}

\begin{proof} For any admissible Floer datum $\fa=(R, \alpha(s),\beta, \epsilon,\delta H)$ for the pair $(\Lambda_{q_0,\theta_0}, \Lambda_{q_1,\theta_1})$, consider the involution $\tau(s)=R-s, s\in \R_s$ and the Floer datum  $\fa'=(R, \alpha'(s),\beta', \epsilon,\delta H')$ with $\alpha'(s)=\alpha(\tau(s))-\pi$, $\beta'=\beta-\pi$ and $\delta H'=\tau^* \delta H$. Then $\fa'$ is admissible for the pair $( \Lambda_{q_1\theta_1},\Lambda_{q_0,\theta_0-2\pi})$, and the complex $\Ch_\natural^*( \Lambda_{q_1\theta_1},\Lambda_{q_0,\theta_0-2\pi};\fa')$ is the dual of $\Ch_\natural^*(\Lambda_{q_0,\theta_0}, \Lambda_{q_1,\theta_1};\fa)$. The fact that this map is grading-preserving follows from the normalization Axiom \ref{A-II} and \ref{A=II}. 
\end{proof}

\subsection{The embedding space} Finally, we prove Lemma \ref{L2.5}. This proof will be used again when we construct the Fukaya-Seidel category of $(M,W)$ in Section \ref{SecFS}, so we explain this in some detail. 

\begin{proof}[Proof of Lemma \ref{L2.5}] Let $D^2\subset \R^2$ be the unit disk and $f: D^2\to \R^2$ any immersion. A simple argument using winding numbers shows that 
	\begin{equation}\label{E2.21}
\text{	$f$ is an embedding if and only if $f|_{S^1}$ is. }
	\end{equation}
We shall use this criterion frequently to decide whether $f$ is an embedding. Now consider the space $\Emb_K^*$ of embeddings 
	\[
	\Xi^\dagger: \R_t\times [0,\pi]_s\to \C
	\]
	that satisfies \eqref{E2.14} and a weaker version of \eqref{E2.16}:
	\begin{equation}\label{E2.18}
	-\im(e^{-i\theta_j}\pt g_j(t))>0, j=0,1, t\in \R_t,
	\end{equation}
	where $g_0(t)=\Xi^\dagger(t,0)$ and $g_1(t)=\Xi^\dagger(t,\pi)$. This embedding space $\Emb_K^*$ is more convenient than $\Emb_K$: the condition \eqref{E2.18} says that the curve $g_t(t)$ is transverse to the direction $e^{i\theta_j}, j=0,1$, which is a property invariant under a self-diffeomorphism of $[0,K]_t\times [0, \pi]_s$ preserving the boundary. We have a digram of natural inclusions for any $K< K'$:
	\[
	\begin{tikzcd}
\Emb_K\arrow[r] \arrow[d]&\Emb_{K'}\arrow[d]\\
\Emb_K^*\arrow[r] & \Emb_{K'}^*. 
	\end{tikzcd}
	\]

	For any continuous map $\varphi: S^n\to \Emb_K$, $n\geq 0$, we show that $\varphi$ is null-homotopic as a map $S^n\to \Emb_{K_2}$ for some $K_2>K$. We write $\Xi_r^\dagger=\varphi(r):\R_t\times [0,\pi]_s\to \C, r\in S^n$  and $c_r=c_{\Xi^\dagger_r}\in \C$ for the constant in \eqref{E2.14}. The function $r\mapsto c_r$ is continuous, so $c_{\max}\colonequals\max |c_r|<\infty$. 
	This null-homotopy is constructed in three steps:
	
	\medskip
	
	\Step 1. For some $K_1\gg K$, $\varphi$ is homotopic to a family in $\Emb_{K_1}$ with $c_r\equiv 0, r\in S^n$. 
	
	It suffices to deform the embedding $\Xi_r^\dagger$ on the rectangle $[K, +\infty)_t\times [0,\pi]_s$ using a linear homotopy. Pick a cutoff function $\chi_1: \R\to [0,1]$ such that $\chi_1(x)=0$ for $x\leq 0$ and $=1$ for $x\geq 1$. For any $r\in S^n$ and $y\in [0,1]$, consider the map
	\[
	\Xi_{r,y}^\dagger (t,s)=-y\cdot c_r\cdot \chi_1(\frac{t-K}{K_1-K})+\Xi^\dagger_r(t,s),\ (t,s)\in \R_t\times [0,\pi]_s.
	\]
Then $	\Xi_{r,0}^\dagger=	\Xi_r^\dagger$. For any $K\leq t\leq K_1$, 
\[
\big|\pt \Xi_{r,y}^\dagger (t,s)-(-i\epsilon^+_{01}e^{i\beta^+})\big|=\bigg|\frac{yc_r}{K_1-K}\cdot (\px\chi_1)(\frac{t-K}{K_1-K})\bigg|\leq \frac{c_{\max} \|\px\chi_1\|_\infty}{K_1-K}.
\]
Using \eqref{E2.21}, one verifies that $\Xi^\dagger_{r,y}$ is an embedding satisfying \eqref{E2.14}\eqref{E2.15}\eqref{E2.16} for any $K_1\gg K$; then 
\[
\varphi_y: S^n\to \Emb_{K_1},\ r\mapsto \Xi_{r,y}^\dagger,\ y\in [0,1],
\]
is the desired homotopy.

\medskip

\Step 2. From now on, we assume that $c_r\equiv 0, r\in S^n$. In this case, $\Xi^\dagger_r(t,s)$ is constant in $r$ when $t\leq 0$ or $\geq K$, and we show that $\varphi$ is null-homotopic in the larger space $\Emb_K^*$. 

Since $S^n$ is compact, the property \eqref{E2.15} holds for this family with a uniform constant $0<\delta_{\min}\ll 1$, i.e.,
\[
\ps \Xi^\dagger_r(t,s)=\left\{
\begin{array}{rcr}
-e^{i\theta_1}&\text{ if }&\pi-\delta_{\min}\leq s\leq \pi,\\
e^{i\theta_0}&\text{ if } &0\leq s\leq \delta_{\min},
\end{array}
\right. 
\]
for all $r\in S^n$. We first show that $\varphi$ is homotopic in $\Emb_K^*$ to a family $\varphi_2$ which is also constant in $r$ when $s\in [0,\frac{\delta_{\min}}{2}]\cup [\pi-\frac{\delta_{\min}}{2}, \pi]$. This is done by pushing the embeddings $\Xi_r^\dagger$ ``outwards"; see Figure \ref{Pic36} below. To illustrate, pick a base point $*\in S^n$ and a smooth function $f: \R_t\to \R^+$ such that $\supp f=[0, K]_t$, so $f>0$ on $(0, K)_t$. Consider the path
\begin{equation}
g_0^*: \R_t\to \C,\ t\mapsto\Xi^\dagger_*(t, 0)-f(t)e^{i\theta_0},
\end{equation}
which is still transverse to the direction $e^{i\theta_0}$, as well as the embedding
\begin{align*}
\Xi_{0,*}^\dagger: \R_t\times [0,\pi]_s&\to \C\\
(t,s)&\mapsto g_0^*(t)+e^{i\theta_0}\cdot s.
\end{align*}
Choose another cutoff function $\chi_2:[0,\pi]_s\to [0,1]$ such that $\chi_2'\leq 0$,  $\chi_2(s)\equiv 1$ on $[0,\frac{\delta_{\min}}{2}]$ and $\equiv 0$ on $[\delta_{\min}, \pi]$. Then consider the linear homotopy 
\[
\Xi_{r,y}^\dagger(t,s)=(1-y\chi_2(s))\cdot \Xi^\dagger_r(t,s)+y\chi_2(s) \cdot \Xi^\dagger_{0,*}(t,s),\ r\in S^n,\ y\in [0,1].
\]
Note that $\Xi_{r,y}(t,s)=\Xi_r(r,s)$ for all $s\in [\delta_{\min}, \pi]$ and $y\in [0,1]$. For $s\in [0,\delta_{\min}]$, we compute
\begin{align*}
\pt \Xi_{r,y}^\dagger(t,s)&=(1-y\chi_2(s))\cdot \pt\Xi^\dagger_r(t,s)+y\chi_2(s) \cdot\pt \Xi^\dagger_{0,*}(t,s),\\
\ps \Xi_{r,y}^\dagger(t,s)&=e^{i\theta_0}+y\chi_2'(s)\cdot  (\Xi_{0,*}^\dagger(t,s)-\Xi^\dagger_r(t,s))\\
&=e^{i\theta_0}+y\chi_2'(s)\cdot  (g_0^*(t)-\Xi^\dagger_r(t,0))\\
&=e^{i\theta_0}\cdot (1-y\chi_2'(s)f(t))+y\chi_2'(s)\cdot  (\Xi^\dagger_*(t,0)-\Xi^\dagger_r(t,0)).
\end{align*}
Since $\chi_2'(s)\leq 0$, $1-y\chi_2'(s)f(t)\geq 0$. Moreover, $\pt\Xi_{r,y}(t,s)$ is transverse to the direction $e^{i\theta_0}$ for all $(t,s)\in \R_t\times [0,\delta_{\min}]_s$. Since $f(t)\neq 0$ on $(0,K)$, if $f$ is multiplied by a large constant $\gg 0$, then $\Xi_{r,y}^\dagger$ becomes an immersion (and so an embedding by \eqref{E2.21}) on $\R_t\times [0,\delta_{\min}]_s$ for any $r\in S^n$ and $y\in [0,1]$. We apply the same trick and construct another homotopy for the other boundary component $\R_t\times \{\pi\}$ so that at the end of the second homotopy we obtain a family $(\Xi^\dagger_{r, 2})_{r\in S^n}$ which is constant whenever $s\in [0,\frac{\delta_{\min}}{2}]\cup [\pi-\frac{\delta_{\min}}{2}, \pi]$ or $t\in (-\infty, 0]\cup [K,\infty)$. Finally, we construct the null-homotopy of $\varphi$ using the fact that the group $\Diff(D^2,\partial D^2)$ of self-diffeomorphisms rel boundary is trivial \cite{Smale59}. 

\medskip

\Step 3. By \Step 2, we have a continuous map $\psi: (D^{n+1}, S^n)\to (\Emb_K^*, \Emb_K)$ with $\psi|_{S^n}=\varphi$. We show that for some $K_2\gg K$, $\psi$ can be pushed into $\Emb_{K_2}$ inside $\Emb_{K_2}^*$ by precomposing with a family of self-diffeomorphisms of $[0,K_2]_t\times [0,\pi]_s$. More precisely, one can find a continuous family of self-diffeomorphisms 
\begin{equation}\label{E2.19}
\Phi_{r, y}: \R_t\times [0,\pi]_s\to \R_t \times [0,\pi]_s, r\in D^{n+1}, y\in [0,1]
\end{equation}
such that
\begin{itemize}
\item $\Phi_{r,y}=\Id$ if $t\not\in [0,K_2]$ or  $s\in [\frac{\delta_{\min}}{2},\ \pi-\frac{\delta_{\min}}{2}]$;
\item $\Phi_{r,y}=\Id$ when $r\in S^n$ or $y=0$;
\item $\Phi_{r,y}$ preserves the $s$-coordinate, i.e., $\Phi_{r,y}(t,s)=(\tilde{\Phi}_{r,y,s}(t),s)$ for a family $\tilde{\Phi}_{r,y,s}$ of self-diffeomorphism of $[0,K_2]_t$ rel boundary;
\item for $y=1$, $\Xi^\dagger_r\circ \Phi_{r,1}$ satisfies the stronger condition \eqref{E2.16}.
\end{itemize}

Note that the weaker condition \eqref{E2.18} is invariant under this operation. It is always possible to achieve \eqref{E2.16} by reparameterizing the time variable. For instance, take $K_2=K\epsilon^{-2}$ for some $0<\epsilon\ll 1$, and consider a diffeomorphism $\tilde{\Phi}:[0, K_2]_t\to [0, K_2]_t$ such that 
\[
\pt\tilde{\Phi}=\left\{\begin{array}{ll}
1 & t\in [0,\epsilon]\cup [K_2-\epsilon, K_2]_t,\\
\epsilon & t\in [2\epsilon, K\epsilon^{-1}],\\
\leq 1+2\epsilon & \text{ otherwise}.
\end{array}
\right.
\]
Then for $\epsilon\ll 1$, the path $\Xi^\dagger_r( \tilde{\Phi}(t),s)$ satisfies \eqref{E2.16} for all $r\in D^{n+1}$ and $s\in \{0,\pi\}$. Then the existence of the family $\Phi_{r, y}$ in \eqref{E2.19} follows from the contractibility of $\Diff(D^1,\partial D^1)$.

 Finally, \eqref{E2.15} is achieved by precomposing $\Xi^\dagger_r\circ \Phi_{r,y}$ with another family of self-diffeomorphisms of $[0, K_2]_t\times [0, s]_s$ fixing the boundary. This completes the proof of Lemma \ref{L2.5}. 

(The key observation here is that given \eqref{E2.18}, \eqref{E2.15} and \eqref{E2.16} can be arranged by reparameterizing the surface near the boundary).
\end{proof}

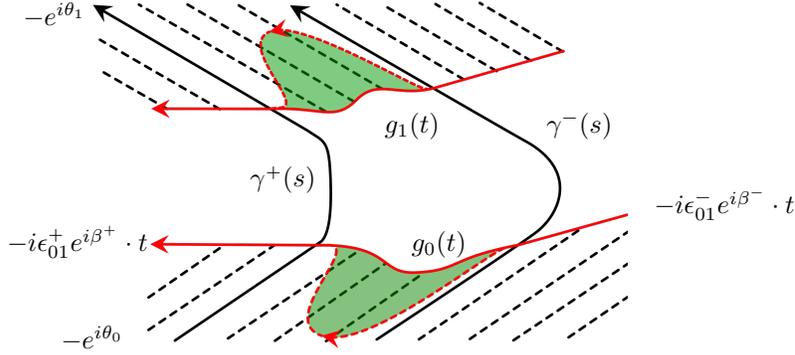
\begin{figure}[H]
	\centering
	\begin{overpic}[scale=.12]{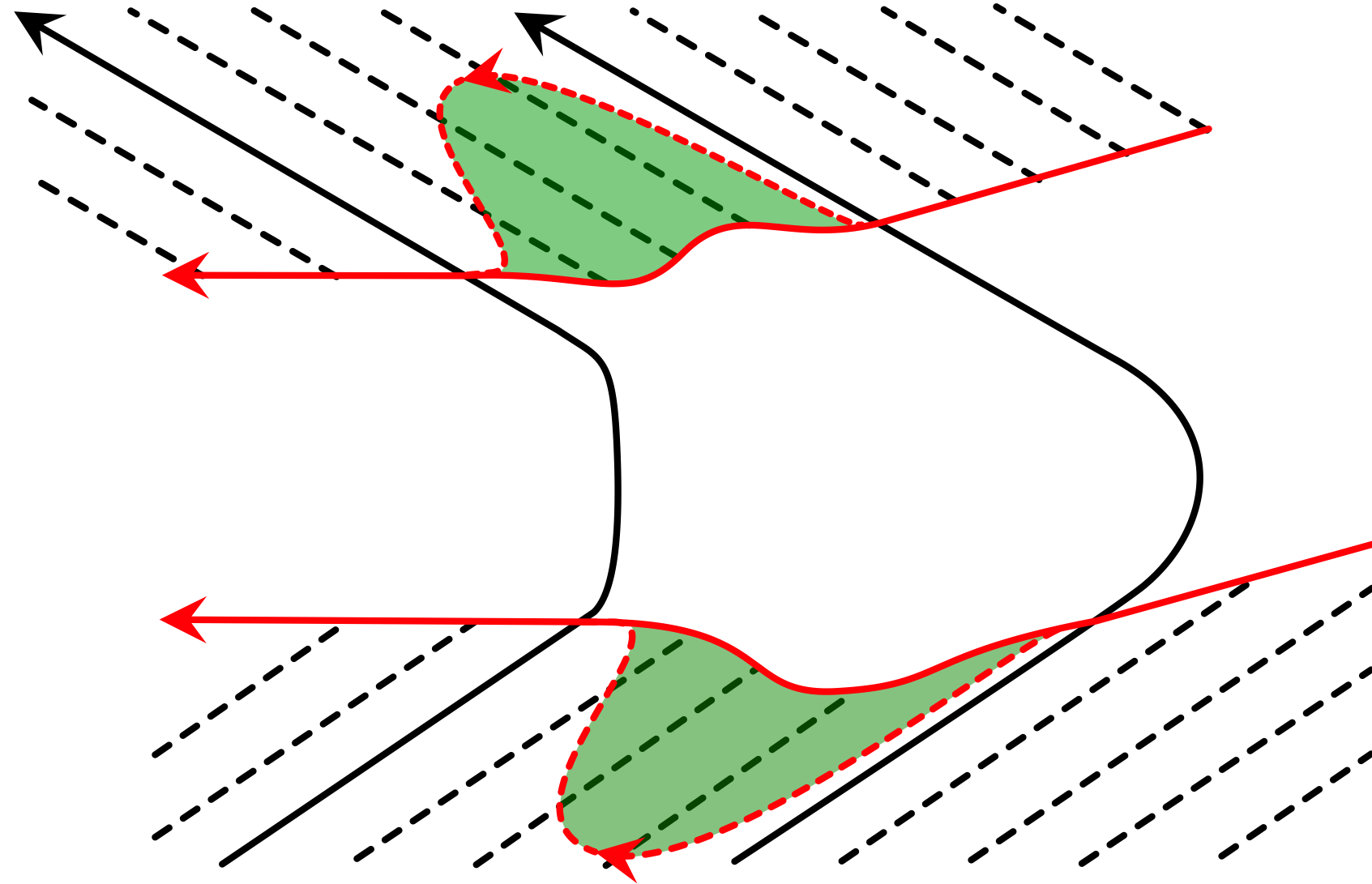}
		\put(85,40){\small $\gamma^-(s)$}
	\put(30,30){\small $\gamma^+(s)$}
	\put(55,40){\small $g_1(t)$}
	\put(60,18){\small $g_0(t)$}
	\put(-15,18){\small $-i\epsilon_{01}^+e^{i\beta^+}\cdot t$}
	\put(105,25){\small $-i\epsilon_{01}^-e^{i\beta^-}\cdot t$}
	\put(-12,60){\small $-e^{i\theta_1}$}
		\put(-5,0){\small $-e^{i\theta_0}$}
	\end{overpic}	
	\caption{Push the curves $g_0(t)$ and $g_1(t)$ outwards.}
	\label{Pic36}
\end{figure}

\newpage

\part{Fukaya-Seidel Categories}\label{Part2}

In this part we take up the task of constructing the Fukaya-Seidel category of a tame Landau-Ginzburg model. As we shall avoid Lagrangian boundary condition whenever possible, higher operations in the Fukaya-Seidel category are defined by considering the Floer equation \eqref{E1.12} on Riemann surfaces with planar ends --- to each boundary component of a pointed disk is attached a lower half plane $\HH^-=\R_t\times \R^-_s$, where $\R^\pm_s \colonequals\{\pm s\geq 0 \}$ denotes the closed half real line. The analysis of the Floer equation \eqref{E1.12} on such a surface is analogous to the case of $\alpha$-instantons on $\R_t\times \R_s$. 

To make this idea work, we have to specify a family of metrics on pointed disks which is isometric to $\R_t\times [0, \delta)_s$ near each boundary component. This is done by choosing suitable meromorphic quadratic differentials, and we summarize relevant theories in Section \ref{SecQD}. Section \ref{SecFS} is devoted to the construction of the Fukaya-Seidel category $\sA$, the Koszul dual category $\sB$ and the diagonal bimodule ${}_{\sA}\Delta_{\sB}$. Their invariance is verified using categorical localization. In Section \ref{SecGF} we explain the origin of the geometric filtration through a neck-stretching limit and upgrade ${}_{\sA}\Delta_{\sB}$ into a filtered bimodule. Section \ref{SecPT} is devoted to the proof of the Koszul duality, i.e., Theorem \ref{Intro.T.8} \ref{T.7.1}. The vertical gluing theorem, which is the main analytic input of this work, is proved in Section \ref{SecVT}.

\medskip

Since this paper is also targeted at people who specialize in gauge theory, we begin this part with a short review of $A_\infty$-categories in Section \ref{SecAP}.  The standard reference of this subject is \cite[Chapter 1]{S08}. Since our base field $\BK$ is of characteristic $2$, we shall ignore the signs in $A_\infty$-associativity relations. Experts should feel free to skip this section completely.

\section{Algebraic Preliminaries}\label{SecAP}  

\subsection{$A_\infty$-categories and $A_\infty$-functors}

An $A_\infty$-category $\sA$ consists of a set of objects $\Ob\sA$, a $\Z$-graded vector space $\hom_{\sA}(X_0, X_1)$  for every pair of objects $(X_0, X_1)$, and composition maps of every order $d\geq 1$,
\[
\mu_{\sA}^d: \hom_{\sA}(X_{d-1},X_d)\otimes\cdots\otimes\hom_{\sA}(X_0,X_1)\to \hom_{\sA}(X_0, X_d)[2-d].
\]
By $[k]$, we mean the degree of a graded vector space is shifted down by $k\in \Z$. These maps are required to satisfy the $A_\infty$-associativity relations:
\begin{equation}\label{AF.E.1}
\sum_{m,n}\mu_{\sA}^{d-m+1}(a_d,\cdots, a_{m+n+1}, \mu_{\sA}^{m}(a_{m+n},\cdots, a_{n+1}),a_n\cdots, a_1)=0,\ d\geq 1.
\end{equation}
In this paper, all $A_\infty$-categories are assumed to be strictly unital, i.e., there exists an element $e_X\in \hom_{\sA}(X,X)$ for all objects $X\in \Ob\sA$ such that 
\begin{equation}\label{AF.E.9}
\mu_{\sA}^1(e_X)=0,\ \mu_{\sA}^2(\cdot, e_X)=\mu_{\sA}^2(e_X,\cdot)=\Id \text{ and }\mu_{\sA}^d(\cdots, e_X,\cdots)=0, d\geq 3.
\end{equation}
 The first order map $\mu_{\sA}^1$ makes $\hom_{\sA}(X_0,X_1)$ into a chain complex for every pair $(X_0, X_1)$, and $\mu_{\sA}^2$ is a chain map of degree 0. The cohomological category $H(\sA)$ associated to $\sA$ consists of the same set of objects, with the morphism spaces replaced by $\Hom_{H(\sA)}(X_0, X_1)\colonequals H(\hom_{\sA}(X_0,X_1), \mu_{\sA}^1)$ and the composition maps by $[\mu^2_{\sA}]$. A dg category is considered as an $A_\infty$-category with $\mu^d_{\sA}\equiv 0$, $d\geq 3$. 

An $A_\infty$-functor between two $A_\infty$-categories $\sA,\sB$ consists of a map $\sF:\Ob \sA\to \Ob\sB$ and mutilinear maps of every order $d\geq 1$,
\[
\sF^d: \hom_{\sA}(X_{d-1},X_d)\otimes\cdots\otimes\hom_{\sA}(X_0,X_1)\to \hom_{\sB}(\sF X_0, \sF X_d)[1-d], 
\]
which satisfy the polynomial equations:
\begin{align}\label{AF.E.2}
\sum_{r}&\sum_{s_1,\cdots, s_r} \mu_{\sB}^r(\sF^{s_r}(a_d,\cdots, a_{d-s_r+1}),\cdots \sF^{s_1}(a_{s_1},\cdots, a_1))\\
&=\sum_{m,n}\sF^{d-m+1}(a_d,\cdots, a_{m+n+1}, \mu_{\sA}^{m}(a_{m+n},\cdots, a_{n+1}),a_n\cdots, a_1).\nonumber
\end{align}

A (strictly unital) $A_\infty$-functor $\sF:\sA\to \sB$ preserves strict units, i.e., $\sF^1(e_{X})=e_{\sF X}$ and $\sF^d(\cdots, e_X,\cdots)=0$ for all objects $X\in \Ob\sA$ and $d\geq 2$. By \eqref{AF.E.2}, $\sF^1: \hom_{\sA}(X_0, X_1)\to \hom_{\sB}(\sF X_0, \sF X_1)$ is a chain map, which induces an ordinary linear functor between their cohomological categories $
H(\sF): H(\sA)\to H(\sB).$ An $A_\infty$-functor $\sF: \sA\to \sB$ is called a quasi-isomorphism (resp. a quasi-equivalence) if $H(\sF)$ is an isomorphism (resp. equivalence) of categories. The functor category $\fun(\sA, \sB)$ is an $A_\infty$-category consisting of $A_\infty$-functors between $\sA$ and $\sB$ as objects, while cycles of $\hom_{\fun(\sA, \sB)}(\sF_0, \sF_1)$ are considered as natural transformations between $A_\infty$-functors $\sF_0, \sF_1:\sA\to \sB$. For the precise definition, see \cite[Section (1d)]{S08}. We record an elementary lemma derived from the $A_\infty$-relations \eqref{AF.E.2}.
\begin{lemma}\label{AF.L.1} For any objects $X_0,X_1, X_2\in \Ob \sA$ and any $A_\infty$-functor $\sF: \sA\to \sB$, the following diagrams are commutative up to chain-homotopy:
	\begin{equation}\label{AF.E.3}
	\begin{tikzcd}
\hom_{\sA}(X_1,X_2)\otimes\hom_{\sA}(X_0,X_1)\arrow[r,"\sF^1\otimes \Id"] \arrow[d,"\mu_{\sA}^2"]& \hom_{\sB}(\sF X_1, \sF X_2)\otimes \hom_{\sA}(X_0, X_1)\arrow[d,"{\mu_{\sB}^2(\cdot,\sF^1(\cdot))}"]\\
\hom_{\sA}(X_0, X_2) \arrow[r,"\sF^1"] &\hom_{\sB}(\sF X_0, \sF X_2),
	\end{tikzcd}
	\end{equation} 

		\begin{equation}\label{AF.E.4}
\begin{tikzcd}
\hom_{\sA}(X_1,X_2)\arrow[r,"\sF^1"] \arrow[d,"{a_1\mapsto \mu_{\sA}^2(a,\cdot)}"]& \hom_{\sB}(\sF X_1, \sF X_2)
\arrow[d,"{b\mapsto \mu_{\sB}^2(b,\sF^1(\cdot))}"]\\
\hom_{\sA}(X_0, X_2) \otimes D\hom_{\sA}(X_0,X_1) \arrow[r,"\sF^1\otimes \Id"]&
\hom_{\sB}(\sF X_0, \sF X_2)\otimes 	D\hom_{\sA}(X_0,X_1),
	\end{tikzcd}
\end{equation} 
	where the second diagram \eqref{AF.E.4} is obtained from the first \eqref{AF.E.3} by replacing $\hom_{\sA}(X_0,X_1)$ by its vector dual $D\hom_{\sA}(X_0,X_1)$.
\end{lemma}

\subsection{$A_\infty$-modules and $A_\infty$-bimodules}\label{SecAF.2} Denote by $\Chain$ the dg category of chain complexes of $\BK$-vector spaces, and let $\sA, \sB$ be $A_\infty$-categories. An $A_\infty$-bimodule $\sN$ over $(\sA, \sB)$ is a $A_\infty$-bifunctor $\sA\times \sB^{\opp}\to \Chain$. More concretely, $\sN$ assigns to every pair of objects $X\in \Ob \sA$ and $Y\in \Ob \sB$ a graded $\BK$-vector space $\sN(Y, X)$ along with maps 
\begin{align}\label{AF.E.5}
\mu_{\sN}^{r|1|s}: &\hom_{\sA} (X_{r-1}, X_r)\otimes \cdots\otimes \hom_{\sA}(X_0, X_1)\otimes \sN(Y_s, X_0)\\
&\otimes \hom_{\sB}(Y_{s-1}, Y_s)\otimes \cdots\otimes \hom_{\sB}(Y_0, Y_1)\to \sN(Y_0, X_r), r, s\geq 0,\nonumber
\end{align}
which satisfy the $A_\infty$-associativity relations:
\begin{align}\label{AF.E.6}
\sum_{r'<r, s'<s} &\mu_{\sN}^{r|1|s'}(a_r,\cdots,a_{r-r'+1},\mu_{\sN}^{r-r'|1|s-s'}(a_{r-r'},\cdots, a_1, x, b_s,\cdots,b_{s'+1}), b_{s'},\cdots, b_1)\nonumber\\
&+\sum_{n+m\leq r} \mu_{\sN}^{r-m|1|s}(a_r,\cdots, a_{n+m+1}, \mu^m_{\sA}(a_{n+m},\cdots, a_{n+1}), a_n,\cdots, a_1, x, b_s,\cdots, b_1)\\
&+\sum_{n+m\leq s} \mu_{\sN}^{r|1|s-m}(a_r,\cdots a_1, x, b_s,\cdots, b_{n+m+1}, \mu^m_{\sB}(b_{n+m},\cdots, b_{n+1}),b_n,\cdots b_1)=0. \nonumber
\end{align}

We further impose a unitality condition: for any objects $X\in \Ob \sA$ and $Y\in \Ob\sB$, 
\begin{align}\label{AF.E.7}
\mu^{1|1|0}_{\sN}(e_X, \cdot)&=\mu_{\sN}^{0|1|1}(\cdot, e_Y)=\Id,\\
\mu^{r|1|s}_{\sN}(\cdots, e_X,\cdots, x,\cdots)&=\mu^{r|1|s}_{\sN}(\cdots, x,\cdots, e_Y,\cdots)=0 \text{ for }r+s\geq 2. \nonumber
\end{align}

There are two additional points of view to think of $A_\infty$-bimodules. For any $(\sA,\sB)$-bimodule $\sN$, construct an $A_\infty$-category $\sE_{\sN}$ by setting $\Ob \sE_{\sN}=\Ob \sB\sqcup \Ob \sA$,
\begin{align}\label{AF.E.12}
\hom_{\sE_{\sN}} (X_0, X_1)&=\hom_{\sA}(X_0, X_1), &\hom_{\sE_{\sN}}(Y_0, Y_1)&=\hom_{\sB}(Y_0, Y_1), \\
\hom_{\sE_{\sN}}(Y_0, X_0)&=\sN(Y_0, X_0),& \hom_{\sE_{\sN}}(X_0, Y_0)&=\{0\},\nonumber 
\end{align}
and $\mu_{\sE_{\sN}}^{d}=\mu_{\sA}^d, \mu_{\sB}^d$ or $\mu_{\sN}^{r|1|s}, r+1+s=d$, depending on the entries. Then the relations $\eqref{AF.E.6}$ and $\eqref{AF.E.7}$ are equivalent to saying that $(\sE_{\sN},\mu_{\sE_{\sN}}^d, d\geq 1)$ is an $A_\infty$-category in the sense of \eqref{AF.E.1} and \eqref{AF.E.9}. 

Consider the $A_\infty$-category with a single object $*$ and $\hom_{\BK}(*, *)=\BK e_*$, which we denote by $\BK$. A right $\sB$-module (resp. left $\sA$-module) is simply a $(\sA,\sB)$-bimodule $\sM$ with $\sA=\BK$ (resp. $\sB=\BK$), and we write $\sM(Y)\colonequals \sM(Y,*)$ (resp. $\sM(X)\colonequals \sM(*, X)$). More concretely, a right $A_\infty$-module over $\sB$ is an $A_\infty$-functor $\sM: \sB^{\opp}\to \Chain$ which assigns to every object $Y\in \Ob \sB$ a graded vector space $\sM(Y)$ together with composition maps of every order $d\geq 1$,
\[
\mu_{\sM}^d=\mu_{\sM}^{0|1|d-1}: \sM(X_{d-1})\otimes \hom_{\sA}(X_{d-2}, X_{d-1})\otimes\cdots\otimes \hom_{\sA}(X_0, X_1)\to \sM(X_0),
\]
 satisfying a suitable variant of \eqref{AF.E.6} and \eqref{AF.E.7}; see \cite[Section (1j)]{S08}. The space of right $\sB$-modules form a dg category, which we denote by $\sQ=\rmod(\sB)\colonequals\fun(\sB, \Chain)$. For any $\sM_0,\sM_1\in \Ob \sQ$, a morphism $t=(t^d)_{d\geq 1}\in \hom_{\sQ}(\sM_0, \sM_1)$ of degree $|t|$, also called a pre-module homomorphism, consists of a sequence of maps:
 \[
 t^d: \sM_0(Y_{d-1})\otimes\hom_{\sB}( Y_{d-1}, Y_{d-2})\otimes\cdots \otimes \hom_{\sB}(Y_0, Y_1)\to \sM_1(Y_0)[|t|-d+1]. 
 \]
 
 The strict unitality requires $t^d(x,a_{d-1},\cdots, a_1), d\geq 2$ to vanish if one of $a_i$'s is the unit. The differential on $\hom_{\sQ}(\sM_0, \sM_1)$ is defined by the formula:
 \begin{align}\label{AF.E.10}
(\mu_{\sQ}^1 t)^d(x,&b_{d-1},\cdots, b_1)=\sum_{n} \mu_{\sM_1}^{n+1}\big( t^{d-n}(x,b_{d-1},\cdots, b_{n+1}), b_n,\cdots, b_1\big)\nonumber\\
&+\sum_{n} t^{n+1} \big(\mu_{\sM_0}^{d-n}(b, a_{d-1},\cdots, b_{n+1}),b_n,\cdots, b_1\big)\nonumber\\
&+\sum_{m+n<d} t^{d-m+1}\big(x,b_{d-1},\cdots, b_{n+m+1}, \mu_{\sB}^m(b_{n+m},\cdots, b_{n+1}),b_n\cdots, b_1\big),
 \end{align}
while the higher compositions are defined by:
 \[
 \big(\mu_{\sQ}^2(t_2,t_1)\big)^d(x, b_{d-1},\cdots, b_1)=\sum_{n} t_2^{n+1}\big(t_1^{d-n}(x, b_{d-1},\cdots, b_{n+1}),b_n,\cdots, b_1\big),\ \mu_{\sQ}^d=0, d\geq 3. 
 \]
 
A right $\sB$-module homomorphism $t: \sM_0\to \sM_1$ is a cycle of $\mu_{\sQ}^1$, whose first order map $t^1$ induces an ordinary module homomorphism $H(t): H(\sM_0)\to H(\sM_1)$. 

\medskip

Returning to the discussion of $A_\infty$-bimodules, an $(\sA,\sB)$-bimodule $\sN$ defines an $A_\infty$-functor $r_{\sN}$ which assigns to each object $X\in \Ob\sA$ the right $\sB$-module $\sN(-,X)$. The mutilinear maps on the morphism spaces 
\[
r_{\sN}^{d}: \hom_{\sA}(X_{d-1},X_d)\otimes \cdots\otimes \hom_{\sA}(X_0, X_1)\to \hom_{\sQ}(\sN(-,X_0),\sN(-,X_d))
\]
are precisely given by $\mu_{\sN}^{d|1|*}$. Then the relations \eqref{AF.E.6} are equivalent to the polynomial equations \eqref{AF.E.3} for $r_{\sN}$.  Similarly, let $\lmod(\sA)\colonequals \fun(\sA,\Chain)$ denote the dg-category of left $\sA$-modules. Then the $(\sA, \sB)$-bimodule $\sN$ also defines an $A_\infty$-functor 
\[
l_{\sN}: \sB\to \lmod(\sA)^{\opp}.
\]

Consider $\sA$ as the diagonal $(\sA, \sA)$-bimodule with
$
\sA(X_0, X_1)\colonequals \hom_{\sA}(X_0, X_1), X_0, X_1\in \Ob \sA.
$ Then we obtain the Yoneda embedding functors:
\begin{align*}
r_{\sA}: \sA\to \rmod(\sA),\ l_{\sA}: \sA\to \lmod(\sA)^{\opp},
\end{align*}
which are cohomologically full and faithful. More generally, we have 

\begin{lemma}[{\cite[Lemma 2.12]{S03}}]\label{AF.L7.2} For $X\in \Ob \sA$, denote by $\sX^r$ (resp. $\sX^l$) the Yoneda image of $X$ under $r_{\sA}$ (resp. $l_{\sA}$). Then for any right $\sA$-module $\sM$ and any left $\sA$-module $\sM'$, the Yoneda maps,
\begin{align*}
\sM(X) &\to \hom_{\rmod(\sA)} (\sX^r, \sM), & \sM'(X) &\to \hom_{\lmod(\sA)} (\sX^l,\sM'),\\
x&\mapsto (\mu_{\sM}^{d+1}(x,\cdots))_{d\geq 1}, & x&\mapsto (\mu_{\sM'}^{d+1}(\cdots,x))_{d\geq 1}.
\end{align*} 
are quasi-isomorphisms. 
\end{lemma}

In this paper, we shall use the extended $A_\infty$-category $\sE_{\sN}$ to formulate the functoriality (and invariance) of $A_\infty$-bimodules. An $A_\infty$-bifunctor between an $(\sA, \sB)$-bimodule $\sN$ and an $(\sA',\sB')$-bimodule $\sN'$ is simply $A_\infty$-functor $\sG: \sE_{\sN}\to \sE_{\sN'}$ such that $\sG(\Ob \sA)\subset \Ob \sA'$ and $\sG(\Ob\sB)\subset \Ob\sB'$. In other words, any $A_\infty$-bifunctor $\sG$ consists of $A_\infty$-functors $\sG_{\sA}\colonequals\sG|_{\sA}$ and $\sG_{\sB}\colonequals\sG|_{\sB}$ together with maps
\begin{align*}
\sG^{r|1|s}: &\hom_{\sA} (X_{r-1}, X_r)\otimes \cdots \hom_{\sA}(X_0, X_1)\otimes \sN(Y_s, X_0)\\
&\otimes \hom_{\sB}(Y_{s-1}, Y_s)\otimes \cdots\otimes \hom_{\sB}(Y_0, Y_1)\to \sN'(\sG_{\sA} Y_0, \sG_{\sB}X_r), r, s\geq 0.
\end{align*}
satisfying certain relations. The next lemma allows us to think of $\sG$ in terms of the $A_\infty$-functors $r_\sN$ and $r_{\sN'}$.
\begin{lemma}\label{AF.L7.3} For any $A_\infty$-bifunctor $\sG$, the bilinear maps $\sG^{r|1|s}, r,s\geq 0$ define an $A_\infty$-natural transformation $r_{\sG}\in\fun_{(\sA, \sQ)}(r_{\sN}, \sG_{\sB}^*\circ r_{\sN'}\circ \sG_{\sA})$. Conversely, any bifunctor $\sG$ consists of $A_\infty$-functors $\sG_{\sA}: \sA\to \sA'$, $\sG_{\sB}:\sB\to \sB'$ and such an $A_\infty$-natural transformation. 
\[
\begin{tikzcd}
\sA \arrow[r, "r_{\sN}"]\arrow[d, "\sG_{\sA}"] \arrow[rd,phantom, "\Downarrow {r_\sG}"]& \sQ\colonequals \rmod(\sB)\\
\sA'\arrow[r, "r_{\sN'}"]& \sQ'\colonequals\rmod(\sB')\arrow[u,"\sG_{\sB}^*"].
\end{tikzcd}
\]
In particular, $\sG$ is a quasi-isomorphism if and only if $\sG_{\sA}, \sG_{\sB}$ are quasi-isomorphisms and $H(r_{\sG})$ is a natural isomorphism between the functors $H(r_{\sN})$ and $H(\sG_{\sB}^*\circ r_{\sN'}\circ \sG_{\sA}): H(\sA)\to H(\sQ)$. 
\end{lemma}

\subsection{Mapping cones} By \cite[Section (3s)]{S08}, the dg-category $\sQ=\rmod(\sB)$ of right $\sB$-modules is triangulated. For any $A_\infty$-homomorphism $t=(t^d)_{d\geq 1}:\sM_0\to\sM_1$ of degree zero, the mapping cone $\sC=\sC one(t)$ is defined to be 
\begin{align*}
\sC(Y)&=\sM_0(Y)[1]\oplus \sM_1(Y),\\
&\mu_{\sC}^d\big((x_0,x_1), a_{d-1},\cdots, a_1\big)\\
&=\big(\mu_{\sM_0}^d(x_0,a_{d-1},\cdots, a_1), \mu_{\sM_1}^d (x_1,a_{d-1},\cdots, a_1)+t^d(x_0,a_{d-1},\cdots, a_1)\big)
\end{align*}

We have obvious inclusion and projection pre-homomorphisms:
\begin{equation}\label{AF.E.16}
\iota_0:\sM_0\to \sC,\ \pi_0: \sC\to \sM_0,\ \iota_1:\sM_1\to \sC,\ \pi_1:\sC\to\sM_1,
\end{equation}
with $\deg \iota_0=-1, \deg \pi_0=1$ and $\deg \iota_1=\deg \pi_1=0$,
which fit into the digram below:
\begin{equation}\label{AF.E.11}
\begin{tikzcd}
\sM_0\arrow[rr,"t"] \arrow[rd,bend right, "\iota_0"']&& \sM_1\arrow[ld,"\iota_1"']\arrow[ll,bend right,"0"']\\
& \sC\arrow[lu,"\pi_0"']\arrow[ru,bend right,"\pi_1"']& 
\end{tikzcd}
\end{equation}
\begin{lemma}[{\cite[{Lemma 3.35}]{S08}}]\label{AF.L7.6} A direct computation shows that 
	\[
	\mu^1_{\sQ}(\iota_1)=\mu^1_{\sQ}(\pi_0)=0,\  \mu^1_{\sQ}(\iota_0)=\mu_{\sQ}^2(\iota_1,t),\ \mu^1_{\sQ}(\pi_1)=\mu_{\sQ}^2(t,\pi_0),\ 0=\mu^2_{\sQ}(\pi_0, \iota_1),
	\]
	and moreover, $e_{\sC}=\mu^2_{\sQ}(\iota_1,\pi_1)+\mu^2_{\sQ}(\iota_0,\pi_0)$. By the distinguished triangulation detection lemma \cite[Lemma 3.7]{S08}, the diagram \eqref{AF.E.11} descends to an exact triangle in the cohomological category of $\sQ$.  
\end{lemma}

\subsection{Directed $A_\infty$-categories} In this paper, we shall primarily work with finite directed $A_\infty$-categories, meaning that $(\Ob \sA, \prec)$ is a finite set equipped with a total order $\prec$ such that
\[
\hom_{\sA} (X_0, X_1)=\left\{\begin{array}{cl}
0 & \text{ if } X_1\prec X_0,\\
\BK\cdot e_{X_0} & \text{ if }X_0=X_1,\\
\text{finite dimensional over }\BK & \text{ if } X_0\prec X_1.
\end{array}
\right.
\]

For finite directed $A_\infty$-categories, we only consider $A_\infty$-modules (resp. $A_\infty$-bimodules) with \textit{finite dimensional} morphism spaces. Let $\Chain_{\fin}$ denote the dg category of \textit{finite dimensional} chain complexes of $\BK$-vector spaces. Let $\sP_l=\lfmod(\sA)\colonequals
\fun(\sA, \Chain_{\fin})$ (resp. $\sP_r=\rfmod(\sA)\colonequals
\fun(\sA^{\opp}, \Chain_{\fin}))$ denote the dg category of finite left $\sA$-modules (resp. right $\sA$-modules). Consider the duality functor
\[
\sD: \Chain_{\fin}\to \Chain_{\fin}^{\opp},
\]
which assigns to a finite chain complex $V$ its dual complex $DV[-\fn]$ whose degree is shifted up by an integer $\fn$. This defines a pair of quasi-isomorphisms of dg categories, denoted also by $\sD$,
\begin{align}\label{AF.E.24}
\sD: \sP_l^{\opp}&\to \sP_r, & \sD: \sP_r&\to \sP_l^{\opp},\\
\sM&\mapsto \sD\sM, & \sN&\mapsto \sD\sN,\nonumber
\end{align}
which allows us to go back and forth between finite right $\sA$-modules and finite left $\sA$-modules. This degree shifting is conventional, and in applications one takes $\fn=\dim_\C (TM, J_M)$; cf. Section \ref{SecFL.5}. For any left $\sA$-module $\sM$, $\sD\sM$ assigns to each object $X\in \sA$ the dual complex $D\sM(X)[-n]$ along with the dualized composition maps $\mu_{\sD\sM}^d\colonequals D\mu_{\sM}^d$:
\begin{equation}\label{AF.E.13}
 D\sM(X_0)[-\fn]\otimes\hom_{\sA}(X_{d-2}, X_{d-1})\otimes\cdots\otimes \hom_{\sA}(X_0, X_1) \to D\sM(X_{d-1})[-\fn].
\end{equation}
Since $\sM$ is strictly unital and $\sA$ is directed, the map \eqref{AF.E.13} is only non-tautological if $X_0\prec X_1\prec\cdots\prec X_{d-1}$.  A similar condition holds also for any pre-module homomorphism $t\in \hom_{\sP_r}(\sM_1,\sM_0)$; hence, the dg category $\sP_r$ (and so $\sP_l^{\opp}$) has finite dimensional morphism spaces. Moreover, by \cite[Corollary 5.26]{S08}, $\sP_r$ (and so $\sP_l^{\opp}$) is a triangulated envelope of $\sA$. 

In the same vein, the duality functor provides a rotational operation on the space of finite directed $A_\infty$-categories. Suppose that $\Ob\sA$ is ordered such that $X_i\prec X_j$ if $i<j$; then define a new $A_\infty$-category $\sC\sA$ with 
\begin{align}\label{AF.E.25}
\Ob\sC\sA&:   X_2\prec \cdots\prec X_n\prec X_1,\\
\hom_{\sC\sA}(X_i, X_1)&\colonequals D\hom_{\sA}(X_1, X_n)[-\fn],\nonumber\\ \hom_{\sC\sA}(X_i, X_j)&\colonequals \hom_{\sA} (X_i, X_j), 2\leq i, j\leq n.\nonumber
\end{align}
The composition maps $\mu_{\sA}^d$ that involve the object $X_1$ are dualized as in \eqref{AF.E.13}, while others remain fixed. It is clear that $\sC^{n} \sA=\sA$ if $\sA$ has $n$ objects. The next three lemmas, which hold in general for any c-unital $A_\infty$-categories, allow us to speak of invariance of finite directed $A_\infty$-categories.

\begin{lemma}[{\cite[Corollary 1.14]{S08}}]\label{AF.L7.4} Let $\sF: \sA\to \sB$ be any quasi-isomorphism between finite directed $A_\infty$-categories, then $\sF$ has an inverse up to homotopy, i.e., there exists another quasi-isomorphism $\sG: \sB\to \sA$ such that $\sF\circ \sG$ is homotopic to $\Id_{\sB}$, and $\sG\circ \sF$ is homotopic to $\Id_{\sA}$. 
\end{lemma}

\begin{lemma}\label{AF.L7.5} Let $\sF_0, \sF_1: \sA\to \sB$ be quasi-isomorphisms between finite directed $A_\infty$-categories such that the induced functor $H(\sF_0)= H(\sF_1): H(\sA)\to H(\sB)$ are identical on the cohomological categories. Then $\sF_0$ is homotopic to $\sF_1$. 
\end{lemma}
 If $\sF_0, \sF_1$ induce the same map on $\Ob \sA$, then the difference $\sF_0-\sF_1\in \hom_{\fun(\sA, \sB)}(\sF_0, \sF_1)$ defines a cycle. $\sF_0$ and $\sF_1$ are called homotopic if $\sF_0-\sF_1$ is a boundary, which implies that $H(\sF_0)=H(\sF_1)$ as ordinary functors. Homotopy is an equivalence relation between $A_\infty$-functors \cite[Section (1h)]{S08}. With that being said, Lemma \ref{AF.L7.5} follows from the homological perturbation theory \cite[Section (1i)]{S08} and by repeating the proof of \cite[Corollary 1.14]{S08}.

\begin{lemma}\label{AF.L7.7}Let $\sF: \sA\to \sB$ be any quasi-equivalence between finite directed $A_\infty$-categories such that $\sF$ is injective on $\Ob\sA$. Then the cohomological functor $H(\sF)$ identifies $ H(\sA)$ with a full subcategory of $H(\sB)$. Let $G: H(\sB)\to H(\sA)$ be any equivalence of categories such that $G\circ H(\sF)=\Id_{H(\sA)}$, and $H(\sF)\circ G$ is natural isomorphic to $\Id_{H(\sB)}$. Then there is a quasi-equivalence $\sG: \sB\to \sA$ with $H(\sG)=G$. 
\end{lemma}
\begin{proof}[Proof of Lemma \ref{AF.L7.7}] Let $\tilde{\sB}$ be the full subcategory of $\sB$ formed by the image of $\sF$. By Lemma \ref{AF.L7.4}, $\sF$ has an inverse $\tilde{\sG}: \tilde{\sB}\to \sA$ up to homotopy. In particular, $H(\tilde{\sG})=G|_{H(\tilde{\sB})}$. In order to extend $\tilde{\sG}$ to the whole $A_\infty$-category $\sB$ while lifting the given cohomological functor $G$, one exploits the extension lemma \cite[Lemma 1.10]{S08} and repeats the argument in \cite[Lemma 2.8 \& Theorem 2.9]{S08}.
\end{proof}

Let $\sM_0, \sM_1$ be finite right $\sB$-modules. An $A_\infty$-module homomorphism $t:\sM_0\to \sM_1$ is called a quasi-isomorphism if $H(t): H(\sM_0(Y))\to H(\sM_1(Y))$ is an isomorphism for all objects $Y\in \Ob \sB$.  

\begin{lemma}[{\cite[Lemma 1.16]{S08}}] \label{AF.L7.8}If $t:\sM_0\to \sM_1$ is a quasi-isomorphism, then the left and right composition with $t$ induce quasi-isomorphisms
	\[
	\hom_{\sQ_r}(\sM_1, \sM')\to\hom_{\sQ_r}(\sM_0, \sM') \text{ and resp. } 	\hom_{\sQ_r}(\sM', \sM_0)\to\hom_{\sQ_r}(\sM', \sM_1)
	\]
	for any $\sM'\in \sQ_r=\rfmod(\sB)$. 
\end{lemma}

\begin{lemma}\label{AF.L.9} Any quasi-isomorphism $t: \sM_0\to \sM_1$ admit a quasi-inverse $s: \sM_1\to \sM_0$ such that $[\mu^2_{\sQ_r}(t,s)]=[e_{\sM_1}]$ in $H(\hom_{\sQ_r}(\sM_1, \sM_1))$ and $[\mu^2_{\sQ_r}(s,t)]=[e_{\sM_0}]$ in $H(\hom_{\sQ_r}(\sM_0, \sM_0))$. In particular, $H(s)=H(t)^{-1}: H(\sM_1(Y))\to H(\sM_0(Y))$ for all objects $Y\in \Ob \sB$. 
\end{lemma}

\begin{proof} Using Lemma \ref{AF.L7.8}, we find cycles $s_0,s_1\in \hom_{\sQ_r}(\sM_1, \sM_0)$ such that $\mu^2_{\sQ_r}(s_0,t)\cong e_{\sM_0}$ and $\mu^2_{\sQ_r}(t,s_1)\cong e_{\sM_1}$. Since $\sQ_r$ is a dg category with $\mu^3_{\sQ_r}\equiv 0, d\geq 3$, we have 
	\[
	s_1= \mu_{\sQ}^2( e_{\sM_0},s_1)\cong \mu_{\sQ}^2( \mu^2_{\sQ_r}(s_0, t),s_1)=\mu_{\sQ}^2(s_0,\mu^2_{\sQ_r}(t,s_1))\cong \mu_{\sQ}^2(s_0, e_{\sM_1})=s_0. \qedhere
	\]
\end{proof}

\subsection{Koszul Duality}\label{SecAF.KD}Let $\sA, \sB$ be finite directed $A_\infty$-categories, and let $\sN$ be an $(\sA, \sB)$-bimodules such that all morphism spaces $\sN(Y, X)$ are finite dimensional. Then the $A_\infty$-category $\sE_{\sN}$ defined by \eqref{AF.E.12} is a finite directed $A_\infty$-category if we insist that $Y\prec X$ for all $X\in \Ob \sA$ and $Y\in \Ob\sB$. Suppose that $\Ob \sB= \{Y_m,\cdots, Y_1\}$ is directed such that $Y_j\prec Y_k$ if $j>k$; then consider the right $\sB$-modules
\[
\sY_j^!(Y_j)=\BK  \text{ and }\sY_j^! (Y_k)=0 \text{ if }k\neq j,\ j=1,\cdots, m.
\]
A simple computation shows that $\{\sY_j^!\}_{j=1}^m$ form a finite directed $A_\infty$-subcategory of $\sQ_r=\rfmod(\sB)$, which we denote by $\sB^!$, with 
\[
\sY^!_1\prec \sY^!_2\prec\cdots \prec \sY^!_m.
\]
 Indeed, we have 
$
\hom_{\sQ_r}(\sY_j^!, \sY_k^!)=0 \text{ if } j>k$ and $\hom_{\sQ_r}(\sY_j^!, \sY_j^!)=\BK\cdot e_{\sY_j^!}.
$
\begin{definition}\label{AF.D.10} A finite directed $A_\infty$-category $\sA$ is called \textit{the Koszul dual category} of $\sB$, if there exists a finite $(\sA, \sB)$-bimodule $\Delta={}_{\sA}\Delta_{\sB}$ such that the induced functor
	\[
	r_{\Delta}: \sA\to \sQ_r=\rfmod(\sB)
	\] 
	identifies $\sA$ with the subcategory $\sB^!$ by a quasi-isomorphism. In particular, if $\Ob \sA$ is directed such that $X_1\prec X_2\prec \cdots \prec X_m$, then 
	\[
	\dim_{\BK}\Delta(Y_k, X_j)=\delta_{jk}.
	\]
	This bimodule is to specify the quasi-isomorphism $r_{\Delta}: \sA\to \sB^!$ and is viewed as part of the data that defines the Koszul duality. 
\end{definition}
Let $\bar{\Delta}={}_{\sB^!}\bar{\Delta}_{\sB}$ be the $(\sB^!, \sB)$-bimodule inducing the canonical embedding $\sB^!\embed \sQ_r$. Then any Koszul dual category $(\sA, {}_\sA\Delta_{\sB})$ of $\sB$ fits into a commutative diagram:
\[
\begin{tikzcd}
\sA \arrow[r,"r_\Delta"] \arrow[d, "r_{\Delta}"]& \sQ_r\arrow[d,equal]\\
\sB^!\arrow[r, "r_{\bar{\Delta}}"]& \sQ_r.
\end{tikzcd}
\]

In particular, the extended $A_\infty$-category $\sE_{\Delta}$ is quasi-isomorphic to $\sE_{\bar{\Delta}}$ by Lemma \ref{AF.L7.3}. 

\begin{lemma}\label{AF.L7.9} If $(\sA, {}_{\sA}\Delta_{\sB})$ is any Koszul dual category of $\sB$, then the dual $(\sB,\sA)$-bimodule $\sD\Delta={}_{\sB}(\sD\Delta)_{\sA}$ identifies $\sB$ as the Koszul dual category of $\sA$. In particular, the induced functor 
\begin{align*}
r_{\sD\Delta}: \sB&\to \sP_r=\rfmod(\sA)\\
Y_k&\mapsto (X_j\mapsto (\sD\Delta)(S_j, Y_k)\colonequals D(\Delta(Y_k, X_j))). 
\end{align*}
is cohomologically full and faithful. 
\end{lemma}
\begin{proof} It suffices to work out the special case that $(\sA, _{\sA}\Delta_{\sB})=(\sB^!, _{\sB^!}\bar{\Delta}_{\sB})$. The Yoneda embedding functor $r_{\sB}: \sB\to \sQ_r$ identifies $\sB$ with a full subcategory $\sB^\wedge$ of $\sQ_r$. Let $\sY^r_k$ be the Yoneda image of $Y_k\in \Ob \sB$, and consider the $(\sB^\wedge, \sB^!)$-bimodule $\Delta'$
	\[
\Delta'(\sY_k^r ,\sY_j^!)\colonequals \hom_{\sQ_r}(\sY_k^r, \sY_j^!). 
	\]
	By Lemma \ref{AF.L7.2}, the Yoneda functor $r_\sB: \sB\to \sB^\wedge$ extends to a quasi-isomorphism $\sE_{\bar{\Delta}}\to \sE_{\Delta'}$. Using Lemma \ref{AF.L7.3}, we may check instead that the functor $r_{\sD\Delta'}: \sB^\wedge \to \rfmod(\sB^!)$ is cohomologically full and faithful. 
	This functor can be viewed as a composition:
	\begin{equation}\label{AF.E.14}
	\begin{tikzcd}
\sB^\wedge\arrow[r,hook] &\sQ_r\arrow[r,"l_{\sQ_r}"] \arrow[rrd,dashed]& \lmod(\sQ_r)^{\opp}\arrow[r, "\iota^*"]& \lmod(\sB^!)^{\opp} &\\
& \sB^!\arrow[u,"\iota"] \arrow[rr, "l_{\sB^!}"]&  &	 \lfmod(\sB^!)^{\opp} \arrow[u,"\text{inclusion}"']\arrow[r,"\sD^*"]&\rfmod(\sB^!),
	\end{tikzcd}
	\end{equation}
	where $\iota: \sB^!\to \sQ_r$ denotes the inclusion functor. Since the morphism spaces of $\sQ_r$ are all finite dimensional, the image of $\iota^*\circ l_{\sQ_r}$ lies in the smaller category $\lfmod(\sB^!)^{\opp}$.  Since $\sQ_r$ is a triangulated envelop of $\sB^!$ and the left Yoneda embedding functor $l_{\sB^!}$ is cohomologically full and faithful, the same holds for the dashed arrow in \eqref{AF.E.14}.  This finishes the proof of Lemma \ref{AF.L7.9}. 
\end{proof}

The dashed arrow in \eqref{AF.E.14} can be understood more concretely as follows. Let $(\sA, {}_{\sA}\Delta_{\sB})$ be any Koszul dual category of $\sB$, then we have a pair of adjoint dg-equivalences:
\begin{align}
\sK_{\sA}^l:\sP_l^{\opp}&\to \sQ_r & \sK_{\sB}^r:\sQ_r&\to \sP_l^{\opp},\\
\sN={}_{\sA}\sN&\mapsto\hom_{\sP_l}(\sN, {}_{\sA}\Delta_{\sB}) & \sM=\sM_{\sB}&\mapsto \hom_{\sQ_r}(\sM, {}_{\sA}\Delta_{\sB}),\nonumber
\end{align}
with $\sP_l\colonequals\lfmod(\sA)$ and $\sQ_r\colonequals\rfmod(\sQ)$. The dual ($\sB, \sA$)-bimodule $\sD\Delta$ then induces adjoint dg-equivalences between $\sP_r\colonequals\rfmod(\sA)$ and $\sQ_l^{\opp}$ with $\sQ_l\colonequals\lfmod(\sB)$:
\begin{align}
\sK_{\sB}^l&=\hom_{\sQ_l}(-,{}_{\sB}(\sD\Delta)_{\sA}):\sQ_l^{\opp}\to \sP_r \\
 \sK_{\sA}^r&=\hom_{\sP_r}(-,{}_{\sB}(\sD\Delta)_{\sA}):\sP_r\to \sQ_l^{\opp}.\nonumber
\end{align}
These functors fit into a large diagram as follows:
\begin{equation}
\begin{tikzcd}
& \sA\arrow[d, "l_{\sA}"]\arrow[rrr,equal] \arrow[rd,"r_{\Delta}"]&[5em] &  &[5em] \sA\arrow[d,"r_{\sA}"] \arrow[ld,"l_{\sD\Delta}"']&\\
\sP_r \arrow[r,leftrightarrow,"\sD^*"] &\sP_l^{\opp} \arrow[r,rightharpoonup,"\sK_{\sA}^l"] &\arrow[l, rightharpoonup, shift left,"\sK_{\sB}^r"]
\sQ_r \arrow[r,leftrightarrow,"\sD^*"] &\sQ_l^{\opp}  \arrow[r,rightharpoondown,shift right,"\sK_{\sB}^l"']& 
\sP_r \arrow[r,leftrightarrow,"\sD^*"]\arrow[l, rightharpoondown, "\sK_{\sA}^r"'] &\sP_l \\
& & & \sB\arrow[u,"l_{\sB}"]\arrow[ur, "r_{\sD\Delta}"']& &,
\end{tikzcd}
\end{equation}
with canonical (natural) quasi-isomorphisms $r_\Delta\to \sK_{\sA}^l\circ l_{\sA}$ and $l_{\sD\Delta}\to \sK^r_{\sB}\circ r_{\sA}$ defined by the Yoneda embedding. Moreover, $\sD^*\circ r_{\Delta}=l_{\sD\Delta}$ by construction. 

\subsection{A Postnikov decomposition}\label{SecAF.5} This section is based on \cite[Section (5i)]{S08}. Let $\sB$ be any finite directed $A_\infty$-category with 
\[
\Ob \sB: Y_m\prec \cdots \prec Y_2\prec Y_1.
\]
Let $\sQ_{k,r}, 1\leq k\leq m$ denote the full dg-subcategory of $\sQ_r=\rfmod(\sB)$ consisting of finite right $\sB$-modules $\sN$ with $\sN(Y_j)$ acyclic for all $1\leq j\leq k$. Then we have canonical projection and inclusion functors:
\begin{equation}\label{AF.E.17}
\sQ_r\xrightarrow{\sF_k} \sQ_{k,r}\xrightarrow{\sI_k} \sQ_r,
\end{equation}
where $\sF_k$ is defined by
\[
\sF_k\sM(Y_j)=\left\{\begin{array}{ll}
\sM(Y_j) &\text{ if } k+1\leq j\leq m,\\
0 & \text{ otherwise}. 
\end{array}
\right.
\]
 Because $\sB$ is directed, the (counit) natural transformation $\sI_k\sF_k\to \Id_{\sQ_r}$, which assigns to each $\sM$ the inclusion homomorphism $\xi^k_{\sM}:\sI_k\sF_k\sM\to\sM$, gives a quasi-isomorphism 
 \begin{equation}\label{AF.E.22}
\hom_{\sQ_{k,r}}( \sN, \sF_k\sM)\to \hom_{\sQ_r}(\sI_k \sN, \sM)
 \end{equation}
 so the functor $\sF_k$ is right adjoint to $\sI_k$. When it is clear from the context, we shall not distinguish $\sN$ from its image $\sI_k \sN$. The rest of this section is devoted to the construction of a dg-functor $\sL_k$:
\begin{equation}\label{AF.E.18}
\sL_k:\sQ_r\to \sQ_{k,r}
\end{equation}
which is left-adjoint to $\sI_k$. For any $\sM\in\sQ_r$, we construct a closed morphism $\nu_{\sM}^{k}: \sM\to \sI_k\sL_k \sM$ (the unit of this adjunction) inducing a natural quasi-isomorphism for all $\sN\in \sQ_{k,r}$:
\begin{equation}\label{AF.E.19}
\hom_{\sQ_{k,r}}(\sL_k \sM, \sN)\to\hom_{\sQ_r}(\sM, \sI_k\sN).
\end{equation}
For all $k$, there is an evaluation $A_\infty$-homomorphism of degree $0$:
\[
t^k_{\sM}: \sM(Y_k)\otimes \sY_k^r\to \sM
\]
such that for every $d\geq 1$ and $m\geq j_d>j_{d-1}>\cdots >j_1>k$,
\[
(t^k_{\sM})^d=\mu_{\sM}^{d+1}:\sM(Y_k)\otimes \hom_{\sB}(Y_{j_1}, Y_k)\otimes \cdots \hom_{\sB}(Y_{j_d},Y_{j_{d-1}})\to\sM(Y_{j_d})[1-d]. 
\]

The construction of $\nu_{\sM}^{(k)}$ is done recursively. Define $\sL_1\sM$ to be the mapping cone of $t_{\sM}^1$, i.e., the abstract twist of $\sM$ along $Y_1$ in the sense of \cite[Section 5]{S08}. The horizontal map $\nu^1_{\sM}$ in the exact triangle \eqref{AF.E.20} below then gives a quasi-isomorphism 
\[
\hom_{\sQ_r}(\sL_1\sM,\sI_1\sN)\to \hom_{\sQ_r}(\sM,\sI_1\sN)
\]
for any $\sN\in \sQ_{1,r}$, since $\hom_{\sQ_r}(\sY^r_1, \sI_1\sN)$ is acyclic by the Yoneda embedding theorem. 
\begin{equation}\label{AF.E.20}
\begin{tikzcd}
\sM\arrow[rr,"\nu^1_{\sM}"]&&\sL_1\sM=\sC one(t_{\sM}^1)\arrow[ld,"{[1]}"]\\
& \sM(Y_1)\otimes \sY^r_1\arrow[lu,"t_{\sM}"] & 
\end{tikzcd}
\end{equation}
This construction is clearly functorial in $\sM\in \sQ_r$. For each $2\leq k\leq M$, define $\sL_k\sM$ recursively by the formula
\[
\sL_k\sM=\sC one(t^k_{\sL_{k-1}\sM}),
\]
and they fit into a ladder consisting of exact triangles on the left and commutative triangles on the right (see \cite[P.76]{S03}):


\begin{equation}\label{AF.E.21}
\begin{tikzcd}[column sep=3em]
\sL_m\sM\arrow[rd,"{[1]}"]& 0\arrow[l,equal]\arrow[d,"{[1]}"]\\
\sL_{m-1}\sM\arrow[rd,"{[1]}"]\arrow[u]& \sL_{m-1}\sM(Y_m)\otimes\sY_m^r\arrow[l,"t^m"']\arrow[d,"{[1]}"]\\
\cdots \arrow[u]\arrow[rd,"{[1]}"]& \cdots\arrow[d,"{[1]}"] \\
\sL_{2}\sM\arrow[rd,"{[1]}"]\arrow[u]& \sL_{2}\sM(Y_3)\otimes\sY_3^r\arrow[l,"t^3"']\arrow[d,"{[1]}"]\\
\sL_{1}\sM\arrow[rd,"{[1]}"]\arrow[u]& \sL_{1}\sM(Y_2)\otimes\sY_2^r\arrow[l,"t^2"']\arrow[d,"{[1]}"]\\
\sM=\sL_{0}\sM\arrow[u]& \sL_{0}\sM(Y_1)\otimes\sY_1^r\arrow[l,"t^1"']\\[-12pt]
{}\arrow[r,"(L)"',phantom]&{}
\end{tikzcd}
\qquad
\begin{tikzcd}[column sep=5em]
\sF_m\sN\arrow[d]\arrow[r,equal]&0 \\
\sF_{m-1}\sN\arrow[rd,"{[1]}",<-]\arrow[d]\arrow[r]&
\sN(Y_m)\otimes\sY_m^!\arrow[lu,"{[1]}"']\arrow[u,"{[1]}"']\\
\cdots \arrow[d]\arrow[rd,"{[1]}",<-]& \cdots\arrow[u,"{[1]}"'] \\
\sF_2\sN\arrow[rd,"{[1]}",<-]\arrow[d]& \sN(Y_3)\otimes\sY_3^!\arrow[l,<-]\arrow[u,"{[1]}"']\\
\sF_1\sN\arrow[rd,"{[1]}",<-]\arrow[d]& \sN(Y_2)\otimes\sY_2^!\arrow[l,<-]\arrow[u,"{[1]}"']\\
\sN=\sF_{0}\sN& \sN(Y_1)\otimes\sY_1^!\arrow[l,<-]\arrow[u,"{[1]}"']\\[-12pt]
{}\arrow[r, "(R)"',phantom]&{} 
\end{tikzcd}
\end{equation}
The morphism $\nu^k_{\sM}$ is then obtained as a suitable composition of left vertical arrows in \eqref{AF.E.21}$(L)$. This completes the construction of $\sL_k$. For any $\sN\in \sQ_r$, there is a similar ladder \eqref{AF.E.21}$(R)$ associated to the decreasing filtration 
\[
0=\sF_m\sN\to \sF_{m-1} \sN\to \cdots \to \sF_1 \sN=\sF_0\sN=\sN,
\]
where the direction of arrows has been reversed comparing to \eqref{AF.E.21}(L). In particular, each complex $\hom_{\sQ_r}(\sM, \sN)$ is filtered by 
\begin{equation}\label{AF.E.26}
0=\hom_{\sQ_r}(\sM, \sF_m\sN)\subset \cdots \subset \hom_{\sQ_r}(\sM, \sF_1\sN)\subset \hom_{\sQ_r}(\sM, \sF_0\sN)=\hom_{\sQ_r}(\sM, \sN).
\end{equation}
\begin{theorem}[{\cite[Proposition 5.17 \& Remark 5.25]{S03}}] \label{AF.T.12}This filtration induces a spectral sequence converging to $H(\hom_{\sQ_r}(\sM, \sN))$, whose starting term can be identified as
	\[
	E_1^{kj}=\big( H(\hom_{\sQ_r}(\sM, \sY_k^!))\otimes H(\sN(Y_k))\big)^{k+j}.
	\]
\end{theorem}

This spectral sequence can be understood also in terms of the ladder \eqref{AF.E.21}$(L)$ using the adjunction \eqref{AF.E.22} and \eqref{AF.E.19}. Instead of applying the covariant functor $H(\hom_{\sQ_r}(\sM,-))$ to \eqref{AF.E.21}$(R)$, one can apply the contravariant functor $H(\hom_{\sQ_r}(-,\sN))$ to \eqref{AF.E.21}$(L)$ to obtain an exact couple using \cite[P.263]{GM03}. These two points of views, as explained by Seidel in \cite[Proposition 5.17]{S03} and \cite[Remark 5.25]{S03} respectively, are unified by the diagram below:
\begin{equation}\label{AF.E.23}
\begin{tikzcd}[row sep=1em]
\cdots \arrow[d]&\cdots\arrow[d]&\cdots\arrow[d]\\
\hom_{\sQ_r}(\sL_2\sM,\sN)\arrow[d]&\hom_{\sQ_{2,r}}(\sL_2\sM, \sF_2\sN)\arrow[l]\arrow[r]\arrow[d]& \hom_{\sQ_r}(\sM, \sF_2\sN)\arrow[d]\\
\hom_{\sQ_r}(\sL_1\sM,\sN)\arrow[d]&\hom_{\sQ_{1,r}}(\sL_1\sM, \sF_1\sN)\arrow[l]\arrow[r]\arrow[d]& \hom_{\sQ_r}(\sM, \sF_1\sN)\arrow[d]\\
\hom_{\sQ_r}(\sL_0\sM,\sN)&\hom_{\sQ_{0,r}}(\sL_0\sM, \sF_0\sN)\arrow[l, equal]\arrow[r,equal]& \hom_{\sQ_r}(\sM, \sF_0\sN)
\end{tikzcd}
\end{equation}
The horizontal arrows in \eqref{AF.E.23} are quasi-isomorphisms. Unlike the right column, the left column of \eqref{AF.E.23} is not given by inclusions of chain complexes. Nevertheless, any chain map can be replaced by its mapping cylinder which is an inclusion, so the left (resp. the middle) column also defines a spectral sequence converging to $H(\hom_{\sQ_r}(\sM,\sN))$. The diagram \eqref{AF.E.23} then shows that these spectral sequences are all isomorphic to each other (no matter which column we pick).

By passing to the cohomological category of $\sQ_r$, the morphism $[\nu^k_{\sM}]:\sM\to \sL_k\sM$ is characterized by the property \eqref{AF.E.19}, and as such is uniquely defined up to isomorphisms in $H(\sQ_r)$. In Section \ref{SecGF} and Section \ref{SecSS}, we shall construct for any tame Landau-Ginzburg model a finite directed $A_\infty$-category $\sB$ and for any stable thimble or a compact exact graded Lagrangian submanifold a filtered $\sB$-module $\sM$:
\[
0=\sM^{(0)}\subset \sM^{(1)}\subset \sM^{(2)}\subset \cdots\subset \sM^{(m)}=\sM
\]
such that the projection morphism $\sM\to \sM/\sM^{(k)}$ is isomorphic to $[\nu^k_{\sM}]$ in $H(\sQ_r)$. Then the complex $\hom_{\sQ_r}(\sM, \sN)$ is filtered by $\hom_{\sQ_r}(\sM/\sM^{(k)},\sN),\ 0\leq k\leq m$. In some sense, the geometry is smarter and produces a better right $\sB$-module for us.

\subsection{Quotients and localizations}\label{SecAF.7} We review the basic theory of the quotients and localizations of $A_\infty$-categories due to Lyubashenko-Ovsienko \cite{LO06}, Lyubashenko-Manzyuk \cite{LM08} and Drinfeld \cite{D04}, generalizing much earlier work of Verdier \cite{V96} on quotients and localizations of triangulated categories. These techniques are mainly to invert certain morphisms in the $A_\infty$-categories and to verify the invariance of the Fukaya-Seidel categories. Since the quotient construction will be applied to only dg categories in this paper, the work of Drinfeld \cite[Section 3]{D04} would suffice for our purpose, but we shall mostly follow \cite[Section 2.6]{LO06} in this exposition. 

Let $\sB$ be any dg category, considered as an $A_\infty$-category with $\mu_{\sB}^d=0, d\geq 3$, and let $\sA\subset \sB$ be a full subcategory. The quotient dg category $\sB/\sA$ has the same set of objects as $\sB$ with morphism spaces replaced by 
\begin{align}\label{AF.E.15}
\hom_{\sB/\sA}(Y_0, Y_1) \colonequals\bigoplus_{X_1,\cdots, X_d\in \Ob\sA, d\geq 0} &\hom_{\sB}(X_d, Y_1)\otimes \hom_{\sB}(X_{d-1}, X_d)\nonumber\\
&\otimes \cdots \otimes \hom_{\sB}(X_1,X_2)\otimes\hom_{\sB}(Y_0,X_1)[d],
\end{align} 
for any $Y_0, Y_1\in \Ob \sB$. The $d=0$ component is, by convention, $\hom_{\sB}(Y_0, Y_1)$. The first order map $\mu_{\sB/\sA}^1$ is the usual bar differential:
\begin{align*}
\mu_{\sB/\sA}^1(b_{d+1}\otimes\cdots\otimes b_1)&=\sum_{n=1}^{d+1} b_{d+1}\otimes \cdots\otimes \mu_{\sB}^1(b_n)\otimes \cdots \otimes b_1\\
&\qquad+\sum_{n=1}^d b_{d+1}\otimes\cdots\otimes \mu_{\sB}^2(b_{n+1},b_n)\otimes \cdots \otimes b_1,
\end{align*}
and $\mu^2_{\sB/\sA}$ is defined by contracting the end points (note that $\mu^d_{\sB/\sA}=0$ for $d\geq 3$):
\[
\mu_{\sB/\sA}^2(a_{r+1}\otimes\cdots\otimes a_1, b_{r+1}\otimes\cdots\otimes b_1)=a_{r+1}\otimes\cdots\otimes a_2\otimes \mu_{\sB}^2( a_1,b_{r+1})\otimes b_r\otimes\cdots\otimes b_1.
\]

The quotient category $\sB/\sA$ comes with a canonical (quotient) dg functor $\pi:\sB\to \sB/\sA$ which sends $\hom_{\sB}(Y_0, Y_1)$ to the $d=0$ component of \eqref{AF.E.15}. The next lemma is proved by considering the length filtration of $\mu_{\sB/\sA}^1$ on \eqref{AF.E.15}. 
\begin{lemma}[{\cite[Lemma 3.13]{GPS20}}]\label{AF.L7.12} Suppose that $\hom_{\sB}(X, Y)$ is acyclic for all $X\in \sA$ $($i.e., $Y$ is right orthogonal to $\sA$$)$, then $\pi^1: \hom_{\sB}(Y_0, Y)\to \hom_{\sB/\sA}(Y_0, Y)$ is a quasi-isomorphism for all $Y_0\in \Ob \sB$. Similarly, if $\hom_{\sB}(Y, X)$ is acyclic for any $X\in \sA$  $($i.e.,$Y$ is left orthogonal to $\sA$$)$, then $\pi^1: \hom_{\sB}(Y, Y_1)\to \hom_{\sB/\sA}(Y, Y_1)$ is a quasi-isomorphism for all $Y_1\in \Ob \sB$.
\end{lemma}

Let $\sE$ be any finite directed $A_\infty$-category, and let $\sT$ be a collection of closed morphisms in $\sE$.  Then any $t: Y_0\to Y_1$ in $\sT$ can be also viewed as an $A_\infty$-homomorphism $t\in \hom_{\rfmod(\sE)}(\sY_0, \sY_1)$ using the Yoneda embedding functor. Define the localized $A_\infty$-category $\sE[\sT^{-1}]$ to be the image of the composition
\begin{equation}
\begin{tikzcd}
\sE\arrow[r,"r_{\sE}"]& \rfmod(\sE)\arrow[r,"\pi"]& \rfmod(\sE)/\sC one(\sT). 
\end{tikzcd}
\end{equation}
where $\sC one(\sT)$ denotes the full subcategory of $\rfmod(\sE)$ whose objects consist of all cones $\sC one(t)$ of $t\in \sT$. In particular, the localized $A_\infty$-category comes with a canonical quotient functor $\sE\to \sE[\sT^{-1}]$, denoted also by $\pi$. 

\begin{lemma}\label{AF.L7.13} Suppose that $Y_0\prec Y_1$ in $\Ob \sA$ and $t: Y_0\to Y_1$ is a closed morphism $\sT$. Then there exists another closed morphism $a\in \hom_{\sE[\sT^{-1}]}(Y_1, Y_0)$ such that when passing the cohomological category, $[\mu_{\sE[\sT^{-1}]}^2(\pi^1(t),a)]=[e_{Y_1}]$, and $[\mu_{\sE[\sT^{-1}]}^2(a,\pi^1(t))]=[e_{Y_0}]$.
\end{lemma}

 If one can verify that $\pi^1: \hom_{\sE}(Y_j, Y_j)\to \hom_{\sE[\sT^{-1}]}(Y_j, Y_j), j=0,1$ are quasi-isomorphisms using Lemma \ref{AF.L7.12} (and so $[e_{Y_j}]\neq 0, j=0,1$), then $[a]$ is the inverse of $[t]$ in $H(\sE[\sT^{-1}])$. 
\begin{proof} We identify $\sE$ with its Yoneda image in $\sQ=\rfmod(\sE)$. Let $\sC=\sC one(t)$ be the mapping cone of $t: \sY_0\to \sY_1$. The morphism $a$ is defined as 
	\[
a=\pi_0\otimes\iota_1\in \hom_{\sQ}(\sC, \sY_0)\otimes \hom_{\sQ}(\sY_1, \sC)[1]\subset \hom_{\sQ/\sC one(\sT)}(\sY_1, \sY_0),
	\] 
	where $\pi_0:\sC\to \sY_0$ and $\iota_1: \sY_1\to \sC$ are the projection and inclusion homomorphisms defined in \eqref{AF.E.16}. Then $\mu^1_{\sQ/\sC one(\sT)}(a)=0$ by Lemma \ref{AF.L7.6}. Moreover, 
	\[
\mu^2_{\sQ/\sC one(\sT)}(t, a)=\mu^2_{\sQ}(t, \pi_0)\otimes\iota_1 \in\hom_{\sQ}(\sC, \sY_1)\otimes \hom_{\sQ}(\sY_1, \sC)[1]\subset \hom_{\sQ/\sC one(\sT)}(\sY_1, \sY_1),
	\] 
	is cohomologous to $e_{Y_1}$. Indeed, by Lemma \ref{AF.L7.6} again,
	\[
	\mu^1_{\sQ/\sC one(\sT)}(\pi_1\otimes\iota_1)=\mu_{\sQ}^1(\pi_1)\otimes\iota_1+\mu^2_{\sQ}(\pi_1,\iota_1)=\mu^2_{\sQ}(t, \pi_0)\otimes \iota_1+e_{\sY_1}.
	\]
 Similarly, one verifies that $\mu^2_{\sQ/\sC one(\sT)}(a,t)=\pi_0\otimes\mu^2_{\sQ}(\iota_1, t)$ is cohomologous to $e_{\sY_0}$.
\end{proof}

\section{Meromorphic Quadratic Differentials with Double Poles}\label{SecQD}

To construct the Fukaya-Seidel category of $(M,W)$ without using Lagrangian boundary condition, we shall consider the Floer equation \eqref{E1.12} on Riemann surfaces with planar ends (Section \ref{SecQD.3})--- starting with a suitable metric on a pointed disk, we attach to each boundary component a lower half plane $\HH^-=\R_t\times \R^-_s$. The existence of such a metric on pointed disks becomes essential to our construction and is tied closely to the theory of measured foliations and meromorphic quadratic differentials on Riemann surfaces. This section is of expository nature: we review some important results on complex analysis of Riemann surfaces, especially the work of Gupta-Wolf \cite{GW17}, addressing this existence problem. The family version is then discussed in Section \ref{SecQD.8}.

\subsection{Measured foliations with centers} Let $\overline{S}$ be a closed oriented surface of genus $g\geq 0$. 

\begin{definition}\label{QD.D.1}A \textit{measured foliation with centers} $(\CF, \mu)$
	on $\overline{S}$ consists of a set of centers $\Sigma=\{x_1,\cdots x_n\}\subset \overline{S}$, a set of singular points $\Sigma'=\{y_1,\cdots, y_m\}$ disjoint from $\Sigma$, an open cover $\{U_j\}$ of $\overline{S}\setminus (\Sigma\sqcup \Sigma')$, and a collection of non-vanishing closed real-valued 1-form $\{\varphi_j\}$, one for each $U_j$, satisfying the following properties:
	\begin{enumerate}[label=(F\arabic*)]
		\item\label{F1} $\varphi_j$=$\pm\varphi_k$ on $U_j\cap U_k$;
		\item\label{F2}  for each $x_j\in \Sigma$, there exists a local chart $w_j: W_j\to \C$ and $a_j>0$ such that $\varphi_k=\im (a_jw_j^{-1}dw_j)$ on $W_j\cap U_k$;
		\item\label{F3} for each $y\in \Sigma'$, there exists a local chart $z: V\to \R^2$ and an integer $l\geq 1$ such that $\varphi_k=\im (z^{l/2}dz)$ on $V\cap U_k$ for some branch of $z^{l/2}$ on $V\cap U_k$. The number $l+2$ is called the order of the singularity $y$. \qedhere
	\end{enumerate}
\end{definition}

The kernels $\ker \varphi_j$ define a smooth line field of $\overline{S}$, which integrates to give a foliation $\CF$ on the complement $\overline{S}\setminus (\Sigma\sqcup \Sigma')$. The local model of $\CF$ near a center or a singular point are determined by \ref{F2} and \ref{F3}. Within the local chart $W_j$, the foliation $\CF$ consists of radial lines emanating from the center $x_j\in \Sigma$. Near a singular point of order $l+2$, $\CF$ is obtained by gluing $(l+2)$ rectangles $[-1,1]_t\times [0,b]_s$, foliated by the 1-form $ds$, according to the pattern in Figure \ref{Pic12}. 
\begin{figure}[H]
	\centering
	\begin{overpic}[scale=.10]{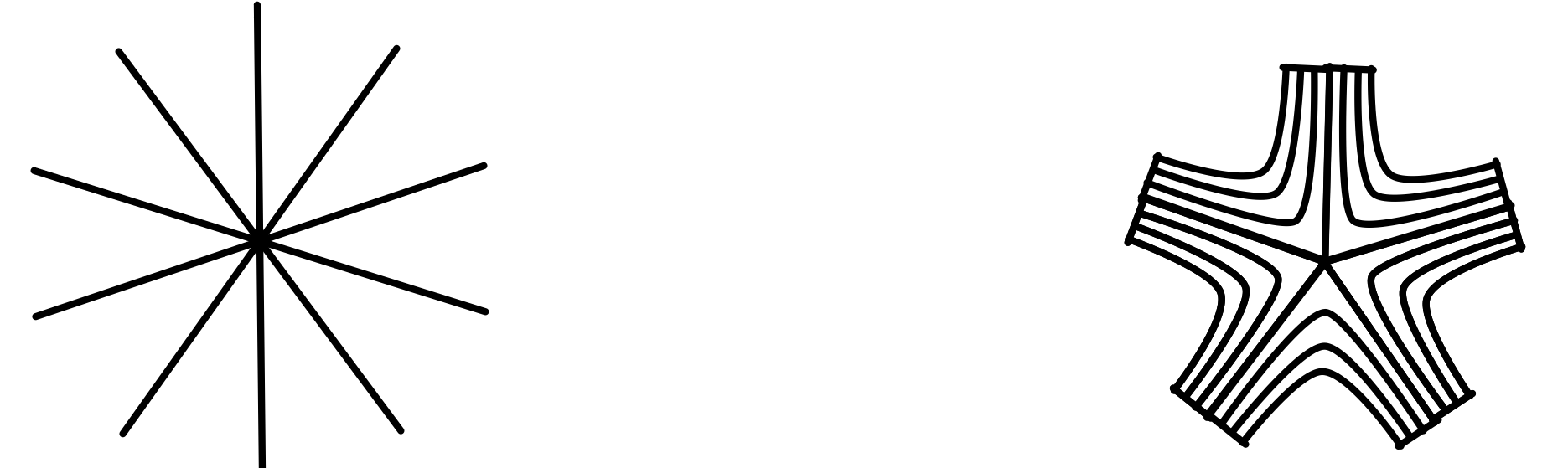}
	\end{overpic}	
	\caption{The local structure of $\CF$ near a center (left) and a singular point of order 5 (right).}
	\label{Pic12}
\end{figure}

The collection of 1-forms in \ref{F1} is patched to obtain a continuous map $|\varphi|: T(\overline{S}\setminus \Sigma\sqcup\Sigma')\to \R^+$, which is equal to $|\varphi_j|$ on $U_j$, and defines a transverse invariant measure $\mu$ on $\CF$: for any arc $A\subset \overline{S}\setminus (\Sigma\sqcup\Sigma')$, let $\mu(A)\colonequals\int_A |\varphi|$. If $A_0, A_1$ are arcs transverse to the foliation $\CF$, and if they are  isotopic via a family of such arcs with the end points lying on the same leaves, then $\mu(A_0)=\mu(A_1)$. For any free homotopy class $[\gamma]$ of closed curves in $\overline{S}\setminus \Sigma$, define $\mu([\gamma])=\inf_{\gamma\in [\gamma]}\mu(\gamma)$, where the infimum is taken over all curves $\gamma$ in the class $[\gamma]$. If $\gamma_j$ is any small loop linking the center $x_j\in \Sigma$ once, then the constant $a_j>0$ in \ref{F2} is equal to the length $\mu([\gamma_j])=\mu(\gamma_j)$ divided by $2\pi$  and is independent of the choice of the local charts.

A pair of measured foliations with centers are called \textit{equivalent}, if one is isotopic to the other after possibly some Whitehead moves (see Figure \ref{Pic13}). Denote by $\MF_2$ the space of equivalent classes of measured foliations with centers. If $(\CF, \mu)$ is equivalent to $(\CG, \nu)$, then $\mu([\gamma])=\nu([\gamma])$ for any closed $\gamma$. Conversely, the equivalent class of a measured foliation with centers is determined by its valuation on finitely many homotopy classes of curves. In fact, one can prove:

\begin{theorem}[\cite{FLP12}{\cite[Proposition 3.9]{ALPS16}}]\label{QD.T.1} Let $\overline{S}$ be any closed oriented surface of genus $g\geq 0$. Then the space $\MF_2$ of equivalent classes of measured foliations on $\overline{S}$ with $n$ centers is homeomorphic to $\R^{3(2g-2+n)}$ provided that $2g-2+n>0$.
\end{theorem}

\begin{figure}[H]
	\centering
	\begin{overpic}[scale=.12]{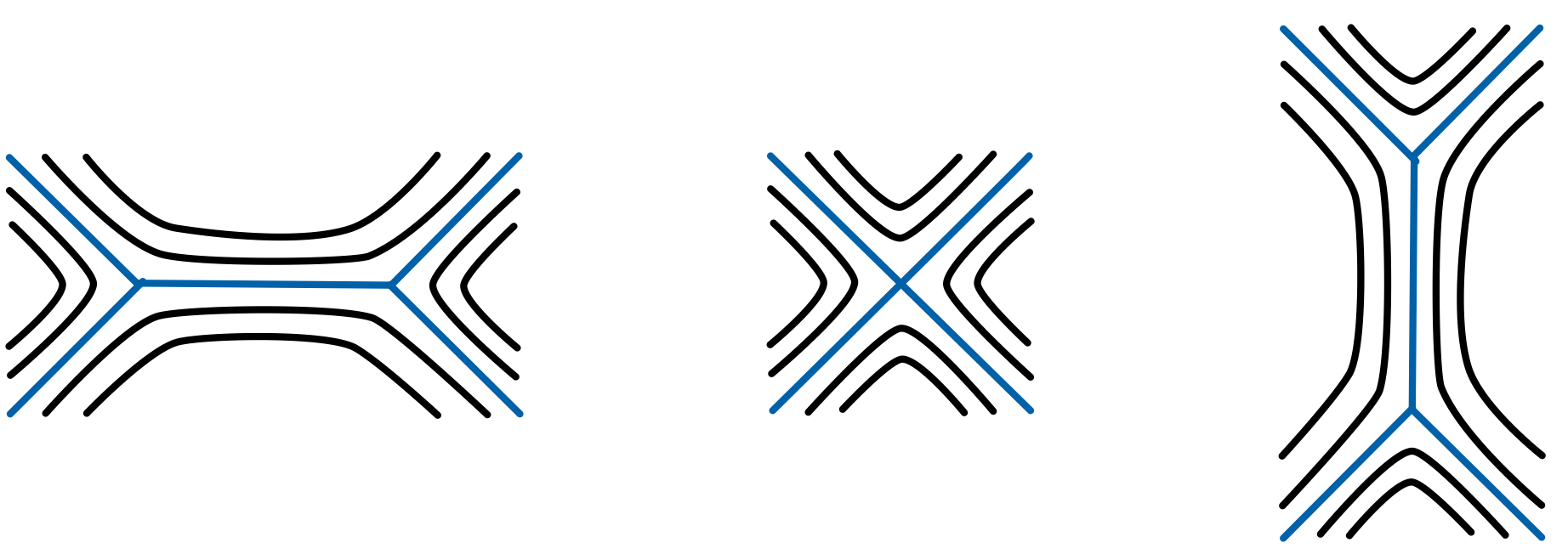}
	\end{overpic}	
	\caption{Whitehead moves.}
	\label{Pic13}
\end{figure}


When $\Sigma$ is empty, $(\CF, \mu)$ is a measured foliation on $\overline{S}$ in the sense of Thurston, and Theorem \ref{QD.T.1} is proved using a pair-of-pants decomposition of $\overline{S}$. By \cite[Proposition 6.7]{FLP12}, a measured foliation on a pair-of-pants is determined (up to isotopy and Whitehead moves) by the transverse measures of three boundary components. To obtain a measured foliation on $\overline{S}$ using the cut-and-pasting technique, one has to record one additional twisting number for each gluing curve, giving rise to $(3g-3)\cdot (1+1)$ parameters in total. This is known as the Thurston-Dehn parametrization; see \cite[Expos\'{e} 6]{FLP12}. The general case for measured foliations with centers is not very different: we remove a central disk in the sink neighborhood $W_j$ for each center $x_j\in\Sigma$ and obtain a measured foliation on a bordered surface with $n$ boundary components; then we repeat the previous argument. We refer to \cite[Section 3]{ALPS16} for a complete proof of Theorem \ref{QD.T.1}.

\subsection{Meromorphic quadratic differentials with double poles}\label{SecQD.2} From now on, let $\overline{S}$ be a compact Riemann surface of genus $g\geq 0$, and let $\Sigma=\{x_1,\cdots, x_n\}$ be $n$ distinct marked points on $\overline{S}$.
\begin{definition} \textit{A meromorphic quadratic differential} $\phi$ on $\overline{S}$ is a $(2,0)$-tensor that is locally of the form $h(w)dw\otimes dw$ for some meromorphic function $h(w)$. In this paper, we are concerned with meromorphic quadratic differentials with double poles at $\Sigma$. 
\end{definition}

The local expansion of $\phi$ near each $x_j\in \Sigma$ takes the form 
\[
(\frac{a_j^2}{w^2_j}+\frac{c_1}{w_j}+c_0+\cdots)dw_j\otimes dw_j,\ a_j, c_1, c_0\in \C.
\]
The leading coefficient $a_j^2$, called the complex residue of $\phi$ at $x_j$, is independent of the choice of the holomorphic coordinate $w_j$. In fact, one can normalize further and write 
\begin{equation}\label{QD.E.1}
\phi=\frac{a_j^2}{w_j^2}dw_j\otimes dw_j\cdot
\end{equation}
for a special choice of $w_j$. We always assume that $a_j^2>0$ for all $1\leq j\leq m$ and declare $a_j$ to be the positive square root. By the Riemann-Roch theorem, the space of meromorphic quadratic differentials with fixed complex residues has dimension $3g-3+n=\dim_\C H^0(\overline{S}, \K^{\otimes 2}(\Sigma))$, where $\K=\Lambda^{1,0}\overline{S}$ is the canonical bundle of $\overline{S}$. 

\begin{definition} The \textit{singular-flat metric} $g_\phi$ induced by a meromorphic quadratic differential $\phi$ is a singular metric defined by the formula $g_\phi(v,v)\colonequals |\phi(v,v)|, v\in T_y \overline{S}, y\in \overline{S}$. Away from the zero locus $\phi^{-1}(0)$ and $\Sigma$, one can locally write $\phi=dw\otimes dw$ for some $w$, then $g_\phi$ is the Euclidean metric on $\C_w$. The singularities of $g_\phi$ are at the zeros of $\phi$: a cone-angle of $(l+2)\pi$ is associated to any zeros of order $l\geq 1$. The metric $g_\phi$ is complete near any double pole $x_j\in \Sigma$: let $z_j=\ln(w_j)$ in the normal form \eqref{QD.E.1}, then $g_\phi$ is the product metric on $z_j\in (-\infty, b]\times S^1$ multiplied by $a_j^2$ (the circle factor has length $2\pi a_j$). 
\end{definition}

\begin{definition} Any meromorphic quadratic differential $\phi$ with double poles at $\Sigma$ induces a pair of measured foliations on $\overline{S}\setminus\Sigma$. The horizontal foliation $(\CF_h, \mu_h)$ is defined locally by the 1-form $\im(\pm \sqrt{\phi})$ away from $\phi^{-1}(0)$, which extends to a measured foliation with centers in the sense of Definition \ref{QD.D.1}. The vertical foliation $(\CF_v, \mu_v)$ is defined locally by $\re(\pm \sqrt{\phi})$; near each $x_j\in \Sigma$, $\CF_v$ consists of concentric circles linking $x_j$.
\end{definition}

The next theorem explains the relation between measured foliations and quadratic differentials. We are most interested in the case when $g=0$, $n\geq 3$, and when all centers are located on the equator of the Riemann sphere. 

\begin{theorem}[{\cite{HM79}\cite[Theorem 1.1]{GW17}}]\label{QD.T.6} Let $\overline{S}$ be any compact Riemann surface of genus $g\geq 0$, and let $\Sigma=\{x_1,\cdots, x_n\}$ be $n$ distinct marked points on $\overline{S}$ such that $\chi(\overline{S}\setminus \Sigma)=2-2g-n<0$. Then for any $(\CF,\mu)$ measured foliation on $\overline{S}$ with $n$ centers, there exists a unique meromorphic quadratic differential $\phi$ with double poles at $\Sigma$ and whose horizontal foliation is equivalent to $(\CF,\mu)$.  
\end{theorem}
\begin{proof}[Idea of Proof] When $\Sigma=\emptyset$, Theorem \ref{QD.T.6} was due to Hubbard-Masur \cite{HM79}. Shorter proofs were found later by Kerchhoff \cite{Ker80}, Gardiner \cite{Gar87} and Wolf \cite{Wol96}. 	The analytic proof by Wolf is generalized further to measured foliations with centers in \cite{GW17}. 
	
	Wolf's idea is to lift the measured foliation $(\CF,\mu)$ to the universal cover $\tilde{S}$ of $\overline{S}$, denoted by $(\tilde{\CF},\tilde{\mu})$. Then the leave space of $(\tilde{\CF},\tilde{\mu})$, which depends only on the equivalent class of $(\CF,\mu)$, is an $\R$-tree $\tilde{\CT}$. If $(\CF,\mu)$ is induced from a holomorphic quadratic differential, then the projection map 
	\[
	\tilde{S}\to \tilde{\CT}
	\]
	is a $\pi_1(\overline{S})$-equivariant harmonic map, and $\phi$ is equal to the Hopf differential. When $\Sigma=\emptyset$, the existence of this harmonic map follows from the variation principle using an energy minimizing sequence \cite{Wol96}. When $\Sigma\neq\emptyset$, the energy of such a harmonic map is infinite, so one has to establish uniform energy bounds for harmonic maps defined on a compact exhaustion of the surface. Readers are referred to \cite{GW17} for the complete argument. 
\end{proof}

\subsection{Pointed disks and their completions}\label{SecQD.3} For any $d\geq 2$, a $(d+1)$\textit{-pointed disk} $S$ is the closed unit disk $\BD\subset \C$ with finitely many boundary marked points $\Sigma=\Sigma^-\sqcup \Sigma^+$  removed and such that $|\Sigma^-|=1$ and $|\Sigma^+|=d$. The points in $\Sigma^-$ and $\Sigma^+$ are called incoming and respectively outgoing points of $S$ at infinity and are labeled anti-clockwise along $\partial \BD$ with $\Sigma^-=\{\zeta_0\}$ and $\Sigma^+=\{\zeta_1,\cdots,\zeta_d\}$. The connected components of $\partial S=\partial\BD\setminus\Sigma$ are denoted by $C_0,\cdots,C_d$ with each $C_k$ connecting $\zeta_k$ and $\zeta_{k+1}$, $0\leq k\leq d$.

\begin{figure}[H]
	\centering
	\begin{overpic}[scale=.08]{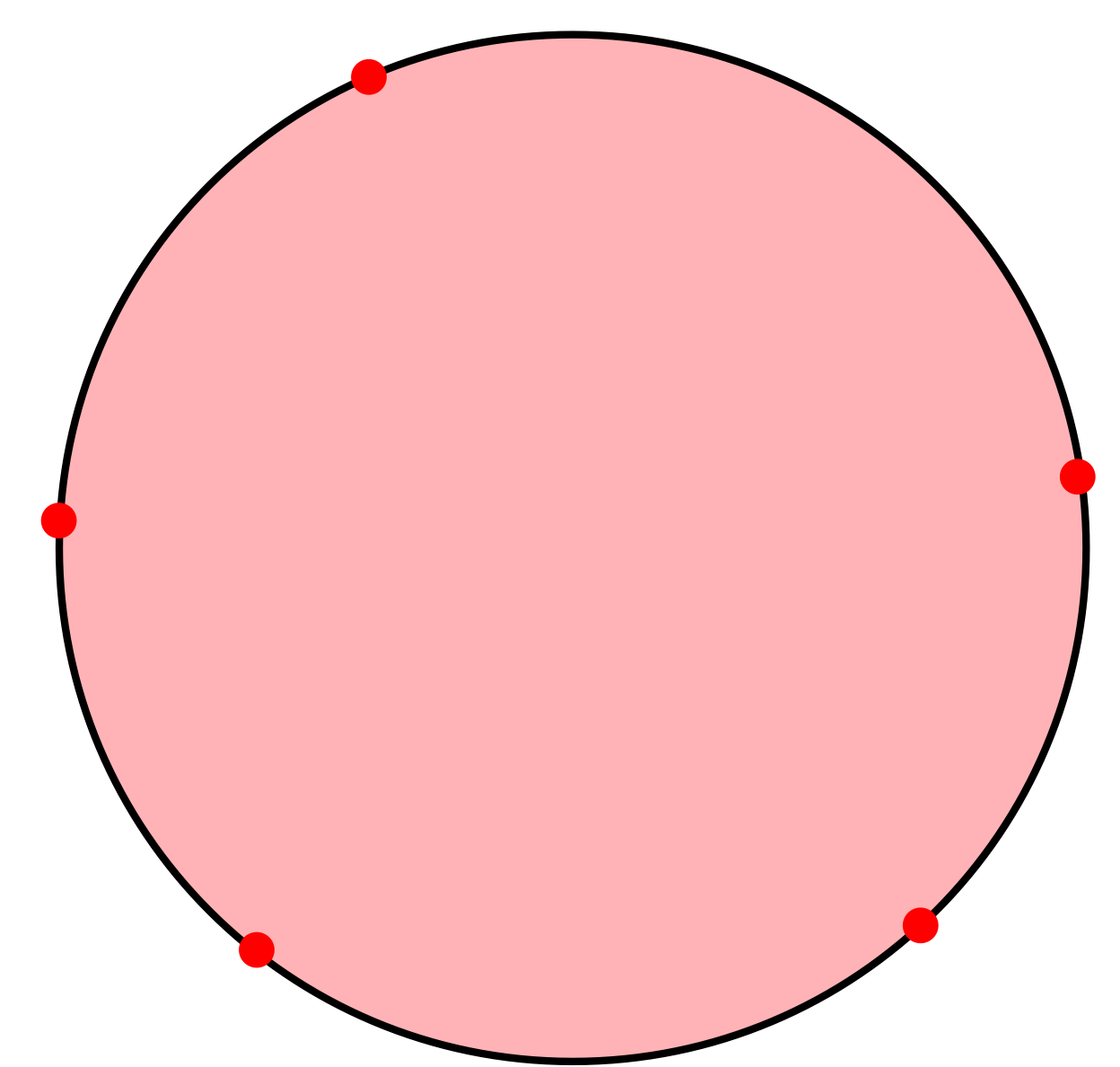}
		\put(-10,46){\small$\zeta_0$}
				\put(10,0){\small$\zeta_1$}
								\put(85,0){\small$\zeta_2$}
								\put(105,50){\small$\zeta_3$}
								\put(25,95){\small$\zeta_4$}
								\put(-2,20){\small $C_0$}
									\put(50,-10){\small $C_1$}
										\put(100,25){\small $C_2$}
											\put(80,90){\small $C_3$}
												\put(0,80){\small $C_4$}
	\end{overpic}	
	\caption{A 5-pointed disk.}
	\label{Pic14}
\end{figure}
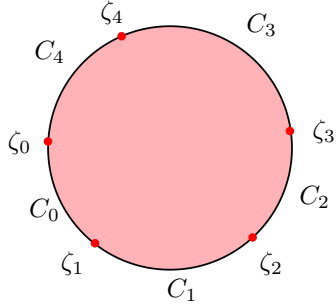

To construct the Fukaya-Seidel category of $(M,W)$, we attach planer ends to $S$, one for each $C_k$, to form a complete Riemannian surface without boundary. To this end, we use a suitable quadratic differential $\phi$ on $S$ to specify a metric that is flat near $\partial S$.  The existence of such $\phi$ is then deduced using Theorem \ref{QD.T.6}.

\begin{definition} A meromorphic quadratic differential $\phi$ on $\BD$ is called $S$-\textit{compatible}, if $\phi$ can be extended to a meromorphic quadratic differential $\Phi$ on $\CP^1=\C\cup\{\infty\}$ with double poles at $\Sigma$. Moreover, we require that the complex residue of $\Phi$ at each $\zeta_k\in \Sigma$ is $a_k^2>0$ for some $a_k>0$, and that $\Phi$ is totally real along $\partial S$. This means that $\Phi(v,v)>0$ for any $v\in T_w(\partial S), w\in \partial S$. In particular, each $C_k\subset \partial S$ is a closed leave of the horizontal foliation of $\Phi$ and is disjoint from the zero locus $\Phi^{-1}(0)$. By identifying $\BD\cong \HH^+$, one may recover $\Phi$ on $\HH^-$ from $\phi$ using the reflection principle, i.e., the formula $\Phi(\bar{z})=\bar{\phi}(z)dz\otimes dz, z\in \HH^+. $
\end{definition}

For any $S$-compatible quadratic differential $\phi$, the singular-flat metric induced by $\phi$ can be smoothed out in a small neighborhood of $\phi^{-1}(0)$ using a ``singular" conformal transformation to obtain a Riemannian metric on $S$, denoted by $g_\phi$. This metric is flat away from this neighborhood of $\phi^{-1}(0)$. Each component $C_k$ of $\partial S$ is totally geodesic, and a collar neighborhood of $C_k$ is isometric to $\R_t\times [0, \epsilon)_s, 0<\epsilon\ll 1$. We define the completion of $S$ by attaching lower half planes $\HH^-_k\colonequals \R_t\times \R_s^-$, one for each $C_k$:
\[
\hat{S}\colonequals S\ \bigcup\  \HH^-_0\cup\cdots\cup \HH^-_d,
\]
which inherits a complete Riemannian metric from $S$. $\HH^-_k$ is called \textit{the planar end} associated to $C_k$.
\begin{figure}[H]
	\centering
	\begin{overpic}[scale=.07]{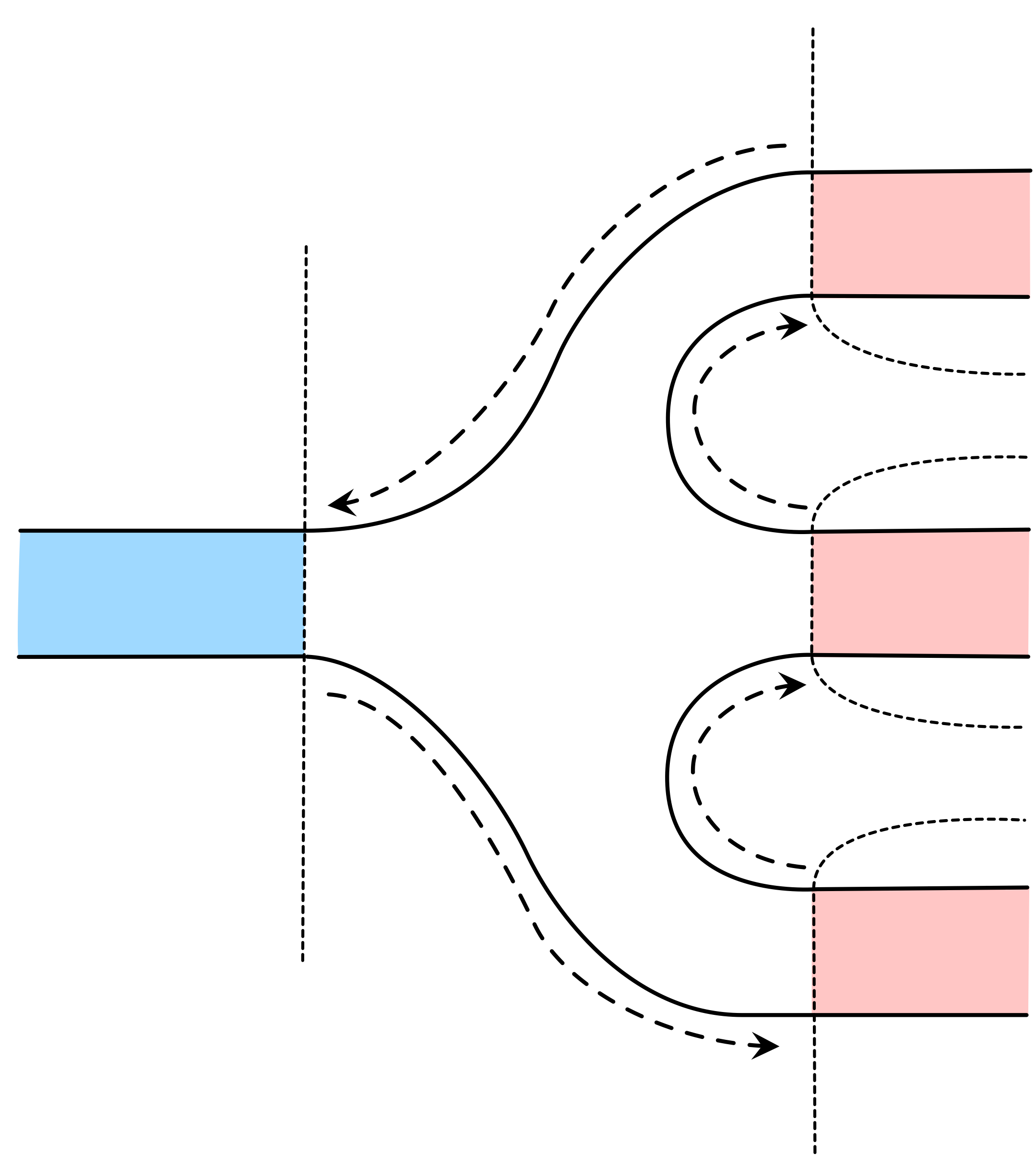}
		\put(-15,47){$\im\hat{\iota}_0$}
		\put(95,78){$\im\hat{\iota}_3$}
		\put(95,47){$\im\hat{\iota}_2$}
		\put(95,16){$\im\hat{\iota}_1$}	
		\put(30,0){$\HH^-_0$}			
		\put(30,90){$\HH^-_3$}	
		\put(65,63){$\HH^-_2$}	
		\put(65,30){$\HH^-_1$}	
		\put(35,20){\small$t_0^+$}
		\put(50,57){\small$t_1^+$}
		\put(51,36){\small$t_2^+$}
		\put(35,70){\small$t_3^+$}
	\end{overpic}	
	\caption{The completion of a 4-pointed disk.}
	\label{Pic15}
\end{figure}
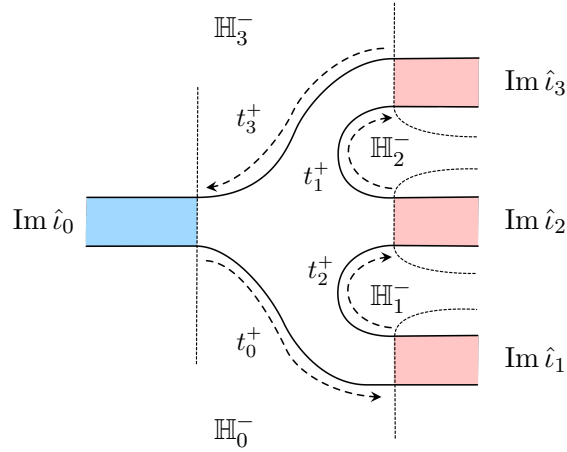

 A set of \textit{strip-like ends} $\{\iota_k\}_{k=0}^d$ of width $(a_0,\cdots, a_d)$ consists of proper holomorphic embeddings $\iota_k: \R^\pm_t \times [0, \pi a_k]_s\to S$, one for each $\zeta_k\in \Sigma^\pm$, such that 
\begin{itemize}
\item $\iota_k^{-1}(\partial S)=\R^\pm\times\{0,\pi a_k\}$,
\item $\lim_{t\to\pm\infty}\iota_k(t,\cdot)=\zeta_k$, and
\item  the images $V_k\colonequals\im\iota_k$ are pairwise disjoint. 
\end{itemize}

If in addition $\iota_k^* g_S$ is the product metric on $\R^\pm_t\times [0,\pi a_k]_s$, then we say that $\{\iota_k\}_{k=0}^d$ is \textit{adapted} to the quadratic differential $\phi$, in which case, the embedding $\iota_k$ is determined uniquely by $\phi$ up to a time translation, and $\iota_k^*\phi=dz\otimes dz$ on $\R_t^\pm \times [0,\pi a_k]_s$. Each $\iota_k$ extends to an isometric embedding $
\hat{\iota}_k=\R^\pm_t\times \R_s\to \hat{S}
$ called $\textit{the plane-like end}$ at $\zeta_k\in \Sigma^\pm$. Let $\sigma_k: \HH^-_k\to \hat{S}$ denote the inclusion map, which are normalized such that 
\[
\sigma_0^{-1}\circ \iota_0(0,0)=(0,0) \text{ and }\sigma_k^{-1}\circ \iota_k(0, \pi a_k) =(0,0),\ 1\leq k\leq d.
\]
The length of $g_\phi$ along $C_k$ is defined as the constant $t^+_k=t^+_k(\phi)\in \R^+$ such that (see Figure \ref{Pic15})
\begin{equation}\label{QD.E.18}
\sigma_d^{-1}\circ \iota_0(0, \pi a_0)=(t_d^+,0) \text{ and }\sigma_{k-1}^{-1}\circ \iota_k(0, 0) =(t_{k-1}^+,0),\ 1\leq k\leq d.
\end{equation}

 \subsection{Measured foliations on pointed disks} The existence of $S$-compatible quadratic differentials is best understood via measured foliations on pointed disks, which is tied closely to ribbon metric trees.

 \textit{A ribbon tree} $\CT$ is a tree with finitely many vertices and edges and with a cyclic ordering on the set of edges incident to each vertex. The boundary $\partial\CT$ of $\CT$, which consists of all exterior vertices of $\CT$ (i.e. vertices of valence $1$), then comes with a cyclic ordering. In this paper, we assume that a total order $\prec$ on $\partial \CT$ which lifts this cyclic ordering is fixed, and that the valence of each vertex of $\CT$ is either $1$ or $\geq 3$. Any ribbon tree $\CT$ admits an embedding $\lambda: (\CT, \partial \CT)\to (\BD,\partial \BD\setminus\{-1\})$ so that the cyclic ordering at each vertex is given by the positive orientation of $\BD$, and such that $\partial\CT$ is ordered positively along $\BD\setminus\{-1\}\cong [-\pi, \pi]$.  In the sequel, we shall always use such an embedding $\lambda$ to specify the ribbon structure on $\CT$. \textit{A metric ribbon tree} is a ribbon tree with a metric. 
 
\begin{example}\label{QD.EX.8} There are two types of metric ribbons trees $\CT^{d+1}$ and $\CT^{d_1,d_2}_R$ that we shall use frequently in this paper. In $\CT^{d+1}$, all exterior vertices are distance $\frac{\pi}{2}$ apart from the unique interior vertex, and $|\partial \CT^{d+1}|=d+1$. The second metric ribbon tree $\CT^{d_1,d_2}_R$ has two interior vertices $v_1, v_2$ connected by an edge of length $R$. $|\partial \CT^{d_1,d_2}_R|=d_1+d_2$. The first $d_1$ vertices (resp. the last $d_2$ vertices) in $\partial \CT$ are attached to $v_1$ (resp. $v_2$) with distance $\frac{\pi}{2}$. See Figure \ref{Pic16}.
\end{example}

 \begin{figure}[H]
 	\centering
 	\begin{overpic}[scale=.15]{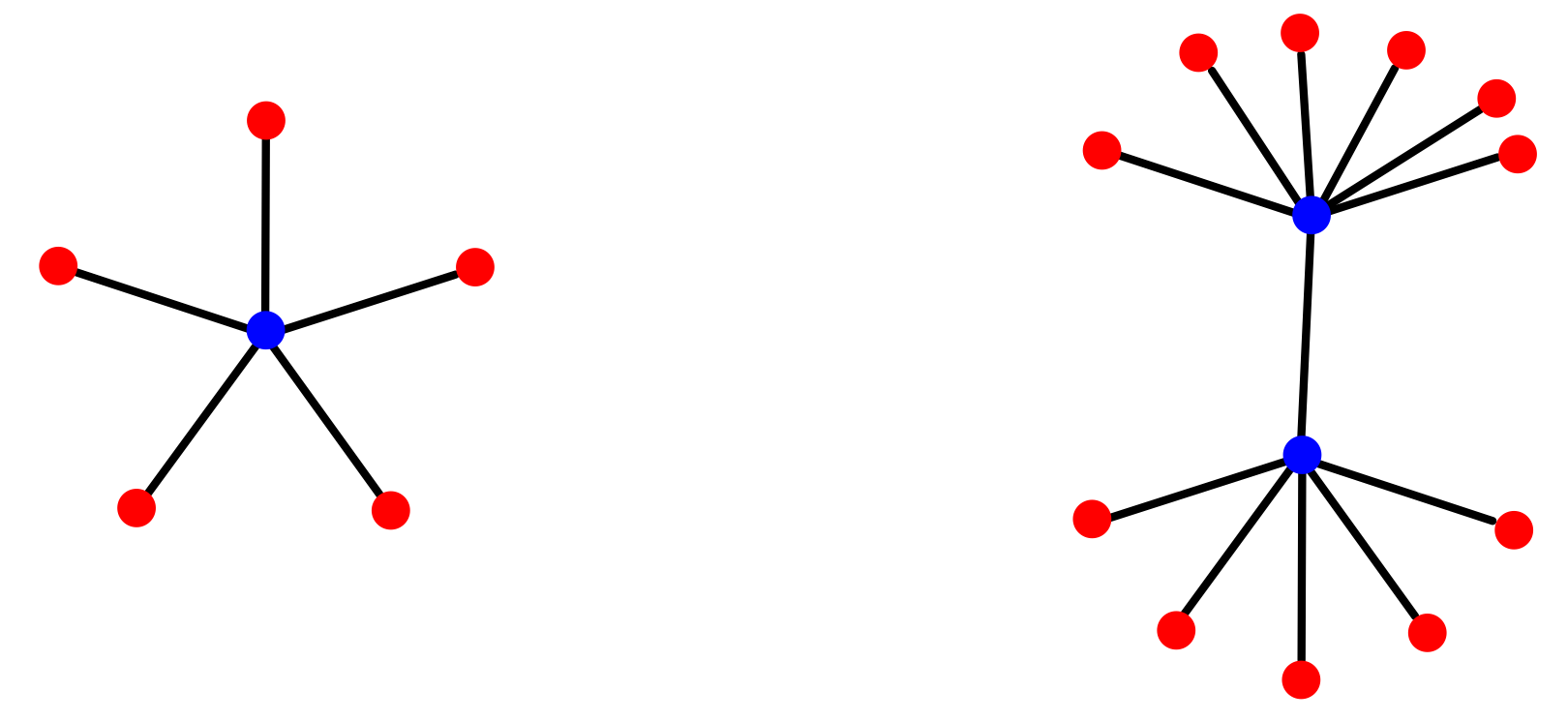}
\put(3,10){$\fo_0$}
\put(63,10){$\fo_0$}
\put(77,20){$v_1$}
\put(77,28){$v_2$}
\put(85,23){$R$}
 	\end{overpic}	
 	\caption{Metric ribbon trees $\CT^5$ and $\CT^{5,6}_R$.}
 	\label{Pic16}
 \end{figure}

\begin{example}\label{QD.EX.9} A more general class of examples will be used in Proposition \ref{FS.P.19}: the metric ribbon tree $\CT=\CT^{d_1,\cdots, d_n}_{R_1,\cdots, R_{n-1}}$ has $n$ interior vertices $v_1,\cdots, v_n$ with each $v_k$ lying on the unique geodesic connecting  $v_1$ and $v_n$ in $\CT$ and such that $d_{\CT}(v_1,v_k)=R_{k-1}, 2\leq k\leq n$ with
	\begin{equation}\label{QD.E.19}
	0\leq R_1\leq R_2\leq \cdots\leq R_{n-1}.
	\end{equation}
	
	Moreover, for every $1\leq k\leq n$, there are $d_k$ exterior vertices attached to $v_k$ with distance $\frac{\pi}{2}$. The total order on $\partial \CT$ is suggested as in Figure \ref{Pic19}. We allow some $v_k$'s to collide with each other, in which case some equalities are achieved in \eqref{QD.E.19}. For instance, if $R_k=0$ for all $k$, then $\CT^{d_1,\cdots, d_n}_{R_1,\cdots, R_{n-1}}=\CT^{d_1+\cdots+d_n}$.
\end{example}

 \begin{figure}[H]
	\centering
	\begin{overpic}[scale=.15]{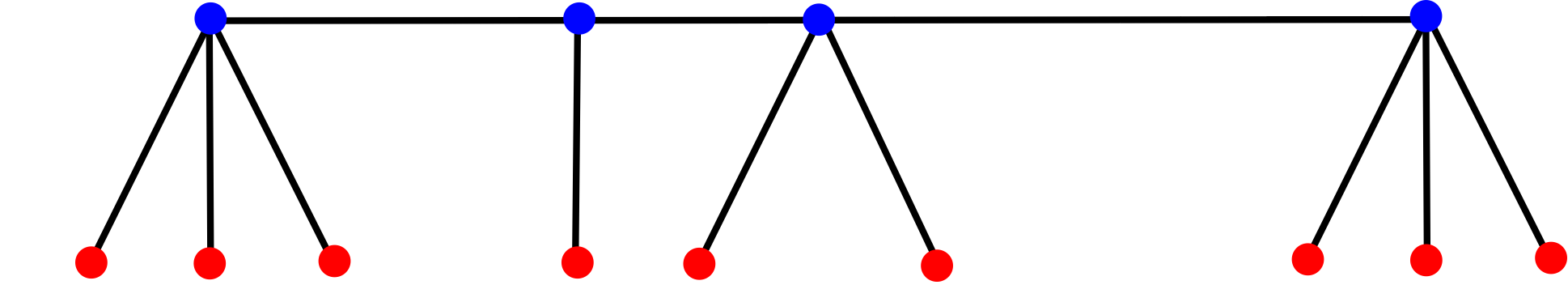}
		\put(0,0){$\fo_0$}
		\put(13,20){$v_1$}
		\put(36,20){$v_2$}
			\put(51,20){$v_3$}
				\put(90,20){$v_4$}
	\end{overpic}	
	\caption{The metric ribbon tree $\CT^{3,1,2,4}_{R_1,R_2,R_3}$ with $n=3$.}
	\label{Pic19}
\end{figure}
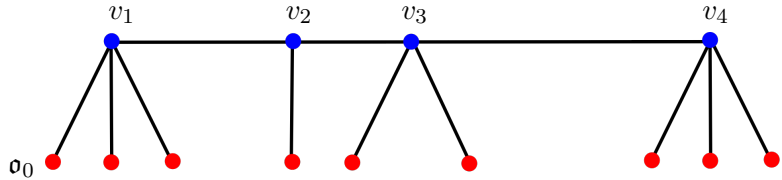
 
 \begin{definition}\label{QD.D.8} Let $\Gamma^{d+1}$ denote the space of metric ribbon trees $\CT$ with $(d+1)$ exterior vertices, and $\partial \CT$ is labeled as $(\fo_0,\cdots, \fo_d)$ with respect to the total ordering. To put a metric on $\Gamma^{d+1}$, we declare the embedding 
 	\begin{align*}
 	\Gamma^{d+1}&\to \R^{d(d+1)/2},\\
	\CT &\mapsto (d_{\CT}(\fo_j, \fo_k))_{0\leq j<k\leq d},
 	\end{align*}
to be isometric. In particular, $\CT$ and $\CT'$ are called $\epsilon$-close, if their images in $\R^{d(d+1)/2}$ have distance $<\epsilon$. 
 \end{definition}

 Let $S=\BD\setminus \Sigma$ be any pointed $(d+1)$-disk and $(\CF, \mu)$ a measured foliation on $\CP^1$ with centers at $\Sigma$. We always assume that 
 \begin{itemize}
\item $(\CF,\mu)$ is invariant under the anti-holomorphic involution fixing the equator $\partial \BD\subset \CP^1$, and that
\item $\partial S=\partial \BD\setminus \Sigma$ consists of only regular leaves. 
 \end{itemize}

In this case, the equivalent class of $(\CF, \mu)$ is uniquely determined by a metric Ribbon tree as follows. Let $\CT$ be the leaf-space of $(\CF,\mu)$ when restricted to $S$. By Poincar\'{e} Recurrence Theorem \cite[Theorem 5.2]{FLP12}, $\CF$ consists of only closed leaves. Since $S$ is contractible, this implies that $\CT$ is a ribbon tree. Each boundary component $C_k\subset \partial S$ is collapsed into an exterior vertex. The total order on $\partial \CT$ is defined such that  $f_S(C_j)\prec f_S(C_k) $ if  $j<k$. The push-forward of the measure $\mu$ then makes $\CT$ into a metric space (and so a metric Ribbon tree). The leaf-space of $(\CF,\mu)$ on $\CP^1$ is the metric ribbon graph $\overline{\CT}$ obtained by gluing two copies of $\CT$ along $\partial \CT$. Then the projection map
 \begin{equation}\label{QD.E.2}
 f_S: (S,\partial S)\to (\CT,\partial \CT)
 \end{equation}
 extends to a $\Z_2$-equivariant map on the double 
 \begin{equation}\label{QD.E.3}
 \overline{f}_S:\CP^1\to \overline{\CT}.
 \end{equation}

 \begin{figure}[H]
 	\centering
 	\begin{overpic}[scale=.15]{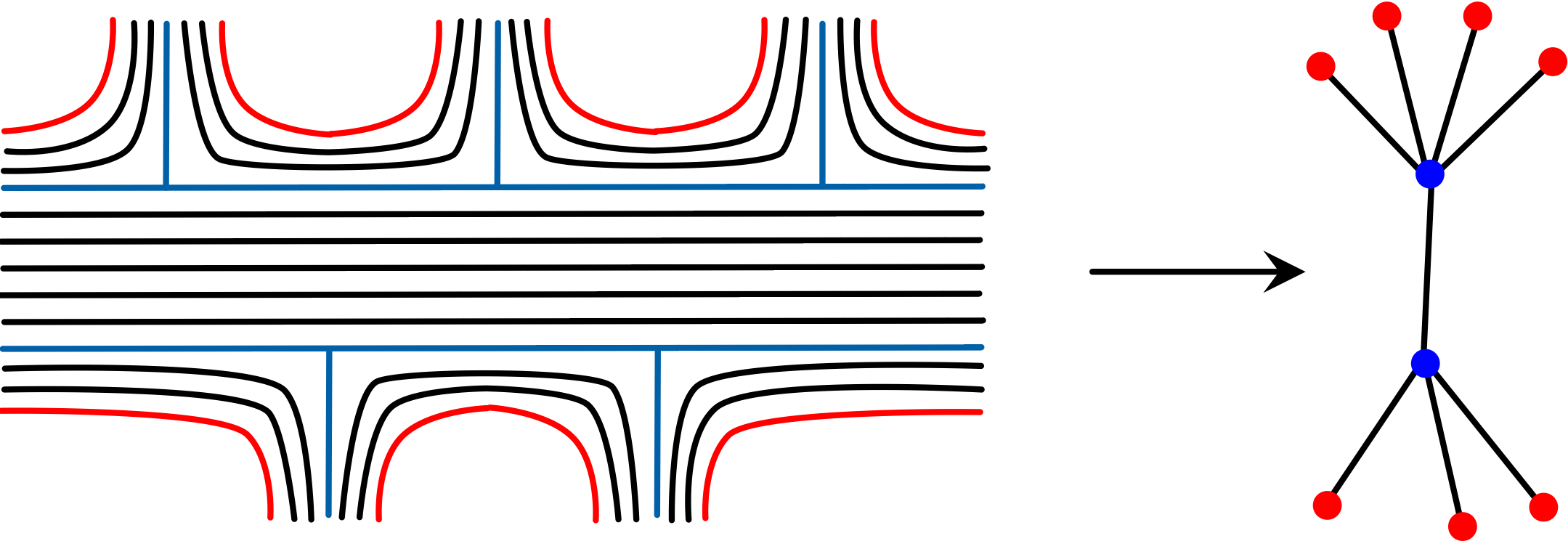}
 		\put(75,20){$f_S$}
 	\end{overpic}	
 	\caption{A measured foliation whose leaf-space is isometric to $\CT^{3,4}_R$. }
 	\label{Pic17}
 \end{figure}

The transverse measure of an arc connecting $C_j, C_k\subset \partial S$  is equal to the distance $d_{\CT}(f_S(C_j), f_S(C_k))$, so $\CT$ determines the equivalent class of $(\CF,\mu)$. Conversely, any metric Ribbon tree $\CT$ with $|\partial \CT|=d+1$ can be realized as the leaf-space of some $(\CF,\mu)$ on $S$. This follows from a simple induction argument on $|\partial \CT|=d+1$, starting with $d=2$. The next lemma follows from Theorem \ref{QD.T.6} by taking $\overline{S}=\CP^1$.

 \begin{lemma}\label{QD.L.8} For any metric Ribbon tree $\CT$ with $(d+1)$ exterior vertices and any $(d+1)$-pointed disk $S$, there exists a unique $S$-compatible quadratic differential $\phi$ such that the leaf-space of the horizontal foliation of $\phi$, which we denote by $\CT_\phi$, is isometric to $\CT$. 
\end{lemma}

\begin{remark} Following Wolf's idea this $S$-compatible quadratic differential is obtained from a harmonic projection map \eqref{QD.E.2}, or equivalently, a $\Z_2$-equivariant harmonic map \eqref{QD.E.3}, as the Hopf differential $\phi=-4(f_S)_w^2dw\otimes dw$. Since the topology of $S$ is rather simple, there is no need to pass to the universal cover, so Wolf's argument should be simplified. See \cite{Wol96}\cite{GW17} for more detail. 
\end{remark}

\begin{example}[{\cite[Example 5.1]{MP98}}] Identify $\BD\cong \HH^+$, and set $d=2$, $\zeta_0=0, \zeta_1=1$ and $\zeta_2=\infty$. Then 
	\[
	\phi(z)=b_0\big(\frac{dz}{z}\big)^2+b_1\big(\frac{dz}{1-z}\big)^2+b_2\big(\frac{dz}{z(1-z)}\big)^2
	\]
	with 
	\[
	b_0=\half (a_0^2+a_2^2-a_1^2),\
	b_1=\half (a_1^2+a_2^2-a_0^2),\
	b_2=\half (a_0^2+a_1^2-a_2^2),
	\]
	is the unique $S$-compatible quadratic differential with complex residues $a_k^2$ at $\zeta_k, 0\leq k\leq 2$. The zero locus $\phi^{-1}(0)$ is disjoint from $\partial \HH^+\cong \R$ if and only if $(a_0, a_1, a_2)$ is subject to the triangle inequalities (and so comes from a metric ribbon tree with three exterior vertices).
\end{example}

\subsection{Families of pointed disks} Our next step is to state the version of Lemma \ref{QD.L.8} when $S$ is allowed to vary in a family. Following \cite[Section 9]{S08}, let
\begin{equation}\label{QD.E.4}
\Sch^{d+1}\to \NR^{d+1}
\end{equation}
denote the universal family of $(d+1)$-pointed disk, and let $\{\iota_k^{d+1}\}_{k=0}^d$ be a universal choice of strip-like ends of width $(a_0,\cdots, a_d)$. Then the moduli space $\NR^{d+1}$ of $(d+1)$-pointed disk comes with a smooth vector bundle 
\begin{equation}\label{QD.E.5}
\V^{d+1}\to \NR^{d+1}
\end{equation}
whose fiber $\V^{d+1}_r$ at $r\in \NR^{d+1}$ is the space of $S_r$-compatible quadratic differentials on $\BD$ (where we allow the complex residues to be any real numbers). There are several ways to trivialize the bundle $\V^{d+1}$ locally, and we shall use spectral projections. For $r_0\in \NR^{d+1}$, fix a metric $g_{r_0}$ on $S_{r_0}$ such that $(\iota^{d+1}_k)^*g_{r_0}$ is the product metric on $\R^\pm_t\times [0,\pi a_k]_s, 0\leq k\leq d$, then pick a local trivialization of \eqref{QD.E.4} (as a family of smooth surfaces)  near $r_0$ that makes the ends $\{\iota_k^{d+1}\}_{k=0}^d$ constant in the sense of \cite[Section (9a)]{S08}. In this way, $g_{r_0}$ is viewed as a metric on nearby fibers $S_r$ (though not necessarily compatible with the complex structure). Now consider the truncated surface 
\[
S_r^{\trun}\colonequals S_r\setminus\coprod_{k=0}^d \iota_k^{d+1} (\inte(\R^\pm_t)\times [0,\pi a_k]_s),
\]
and define the vertical and horizontal boundary of $S_r^{\trun}$ as 
\begin{align*}
\partial^v S_r^{\trun}&=\coprod_{k=0}^d \iota_k^{d+1} (\{0\}\times [0,\pi a_k]_s),\\
\partial^h S_r^{\trun}&=S_r^{\trun}\cap \partial S_r. 
\end{align*}
Note that the standard $\bpartial$-operator on $\R^\pm_t\times [0,\pi a_k]_s$ takes the form $\pt+i\ps$. The operator $i\ps$ acting on the space
\[
L^2_j([0,\pi a_k]_s, \{0,\pi a_k\}; \C)\colonequals \{h\in L^2_k([0, \pi a_k]_s;\C): h(0), h(\pi a_k)\in\R\},\ j\geq 1,
\]
is $L^2$-self-adjoint, with 1-dimensional kernel. Let 
\begin{align*}
\Pi_k^\pm: L^2_j([0,\pi a_k]_s, \{0,\pi a_k\}; \C)&\to H^\pm_{k,j}, \\
\Pi^0_k: L^2_j([0,\pi a_k]_s, \{0,\pi a_k\}; \C)&\to \R,\ 0\leq k\leq d
\end{align*}
denote the spectral projections onto the positive (resp. negative) subspace and the kernel. Then the fiber $\V^{d+1}_r$ is identified with the kernel of 
\begin{equation}\label{QD.E.6}
\BD_r: \bpartial\oplus (\Pi^-\circ \res): L^2_j(S_r^{\trun}, \partial^h S_r^{\trun}; E_r)\to L^2_{j-1}(S_r^{\trun};\Lambda^{0,1}S_r^{\trun}\otimes E_r)\oplus  H^{d+1,-}_{j-\half},\ j\geq 1
\end{equation}
with $E_r\colonequals (\Lambda^{1,0}S_r^{\trun})^{\otimes 2}$. We explain the meaning of \eqref{QD.E.6} in more detail. The operator 
\[
\res: L^2_j\to L^2_{j-1/2}(\partial^v S_r^{\trun}, \partial^v S_r^{\trun}\cap \partial^hS_r^{\trun};\C (dz)^{\otimes 2})
\]
is the boundary restriction map, and $\Pi^-\colonequals(\Pi^+_0, \Pi^-_1,\cdots, \Pi^-_d)$ is the spectral projection onto the subspace  
\begin{equation}\label{QD.E.15}
H^{d+1,-}_{j-\half}\colonequals H_{0, j-\half}^+\oplus H_{1, j-\half}^-\oplus \cdots\oplus H_{d, j-\half}^-.
\end{equation}
Moreover, the domain $L^2_j(S_r^{\trun}, \partial^h S_r^{\trun}; E_r)$ of \eqref{QD.E.6} consists of $L^2_j$-sections which are totally real along the horizontal boundary $\partial^h S_r^{\trun}$.  

\medskip

As $r$ varies in a neighborhood of $r_0$, \eqref{QD.E.6} is a smooth family of surjective Fredholm operators over a fixed domain. Thus $\V^{d+1}_r=\ker \BD_r$ can be trivialized locally, and we compute $\rank_\R \V^{d+1}=\Ind_\R \bpartial=2d-1$. A smooth section $\phi$ of $\V^{d+1}$ is tested locally by the following property: for any compactly supported smooth section $\psi\in C^\infty_c(S^{\trun}_{r_0}; (T S^{\trun}_{r_0})^{\otimes 2})$, the pairing
\begin{equation}\label{QD.E.12}
r\mapsto \langle \phi_r, \psi\rangle
\end{equation}
is a smooth function near $r_0$ ($S^{\trun}_{r}$ and $S^{\trun}_{r_0}$ are identified as smooth surfaces using a local trivialization of \eqref{QD.E.4}). By the unique continuation property, one may further require that the support of $\psi$ to lie in any fixed small ball of $S^{\trun}_{r_0}$.

$\V^{d+1}$ comes with a surjective bundle map
\[
\Res^{d+1}=(\Res_k^{d+1})_{k=0}^d: \V^{d+1}\to \R^{d+1}
\]
which associates a compatible quadratic differential with its complex residue at $\zeta_k\in \Sigma$. Equivalently, $\Res_k^{d+1}=\Pi_k^0\circ \res$. Then the family version of Lemma \ref{QD.L.8} states the following:

\begin{lemma}\label{QD.L.10} For any $d\geq 2$ and any metric ribbon tree $\CT$ with $|\partial \CT|=d+1$, there exists a continuous section $\phi_\CT$ of $\V^{d+1}\to \NR^{d+1}$ such that $\phi_{\CT,r}$ is the unique $S_r$-compatible quadratic differential such that the leaf space of the horizontal foliation is isometric to $\CT$, i.e., $\CT_{\phi_{\CT,r}}=\CT$. By construction, $\Res^{d+1}(\phi_{\CT,r})=(a_0^2,\cdots, a_d^2)$ for all $r\in \NR^{d+1}$, where $a_k=d_{\CT}(\fo_{k-1}, \fo_k)/\pi, 0\leq k\leq d$ $($by convention $\fo_{d+1}\colonequals \fo_0)$. 
\end{lemma}

\begin{remark}\label{QD.R.15} The continuity of $\phi_\CT$ follows from a basic property of harmonic maps defined on varying domains and is not stated explicitly in \cite{GW17}. For holomorphic quadratic differentials on closed Riemann surfaces, this was proved in \cite[Section 1.2]{HM79}. The section $\phi_\CT$ might not be smooth when $\phi_{\CT,r}$ has a zero of order $\geq 2$. This can be avoided by requiring each interior vertex of $\CT$ to have valence $=3$ (which is the generic case), so the horizontal foliation of $\phi_{\CT,r}$ must have $d-1$ distinct singular leaves. In practice, we  shall work with a smooth approximation of $\phi_\CT$, so the metric ribbon trees in Example \ref{QD.EX.8} would suffice for our work.
\end{remark}

\subsection{Leafed trees} A $d$-leafed tree is a properly embedded planar tree $T\subset \R^2$ with $d+1$ semi-infinite edges: one root and $d$ leaves. Two $d$-leafed trees are considered same if they are isotopic in $\R^2$. A flag $f=(v,e)$ in a graph is a pair consisting of a vertex $v$ and an adjacent edge $e$. A finite edge is also called interior (otherwise exterior). Let $\Ve(T)$ denote the set of vertices, $\Ed^{\inte}(T)$ the set of interior edges, and $\Fl(T)$ the set of flags. 

Define a partial order on the space of $d$-leafed trees as follows: $T'<T$ if $T$ is obtained from $T'$ by collapsing some interior edges into vertices, in which case we identify $\Ed^{\inte}(T)$ with a subset of $\Ed^{\inte}(T')$. The final object of this partial order is the $d$-leafed tree $T_*$ with no interior edges and with a single vertex.

A $d$-leafed tree $T$ is oriented by flowing upwards from the root to the leaves; see Figure \ref{Pic18}. A flag $f=(v,e)$ is called positive if $e$ is oriented away from $v$, and negative otherwise. For any interior edge $e$, let $f^\pm(e)=(v^\pm(e), e)$ be the positive (resp. the negative) flag attached to $e$, where $v^+(e)$ (resp. $v^-(e)$) denotes the initial point (resp. the end point) of $e$. For any vertex $v$, we label the flags attached to $v$ anticlockwise as $f_0(v),\cdots, f_{|v|-1}(v)$ starting with the unique negative flag $f_0(v)$, and $|v|$ is the valence of $v$. A $d$-leafed tree $T$ is called stable if $|v|\geq 3$ for all $v\in \Ve(T)$.

\subsection{Deligne-Mumford-Stasheff compactifications} As a set, the Deligne-Mumford-Stasheff compactification of $\NR^{d+1}$ is 
\[
\overline{\NR}^{d+1}=\coprod_{T} \NR^T
\]
where the union is over all stable $d$-leafed trees. Define  $\NR_v\colonequals \NR^{|v|}$ and $\Sch_v\colonequals \Sch^{|v|}$, $v\in \Ve(T)$. Then the $T$-stratum $\NR^T\colonequals \prod_{v\in \Ve(T)} \NR_v$ comes with a fiber bundle 
\[
\Sch^T\to \NR^T
\]
which is the product of $\Sch_v\to \NR_v, v\in \Ve(T)$. The gluing construction from \cite[Section (9f)]{S08} then makes $\overline{\NR}^{d+1}$ into a smooth manifold with corners, with $\NR^{T_*}=\NR^{d+1}$ as the top dimensional stratum. 

\medskip

Our next goal is to extend $\V^{d+1}$ into a smooth vector bundle over $\overline{\NR}^{d+1}$. For any $d$-leafed tree $T$ and $v\in \Ve(T)$, let $\V_v=\V^{|v|}\to \NR_v$ be the bundle of compatible quadratic differentials over $\NR_v$. For any flag $f=(v,e)$, let $\Res_f: \V_v\to \R$ denote the map taking the complex residue at the marked point at $f$. Define $\V^T$ to be the kernel of the surjective bundle map over $\NR^T$:
\begin{align}\label{QD.E.7}
\Res^{T}:\bigoplus_{v\in \Ve(T)} \V_v & \to \R^{\Ed^{\inte}(T)},\\
(\phi_v)_{v\in \Ve(T)} &\mapsto  \big(\Res_{f^+(e)}(\phi_{v^+(e)})-\Res_{f^-(e)}(\phi_{v^-(e)})\big)_{e\in \Ed^{\inte}(T)}.\nonumber
\end{align}

\begin{lemma}\label{QD.L.15} There exists a smooth vector bundle $\overline{\V}^{d+1}\to \overline{\NR}^{d+1}$ which restricts to the bundle $\V^T\to \NR^T$ over each $T$-stratum of $\overline{\NR}^{d+1}$.  
\end{lemma}
\begin{proof}[Proof of Lemma \ref{QD.L.15}] Recall from \cite[Section (9f)]{S08} that any universal choice of strip-like ends $\{\iota_f\}_{f\in \Fl(T)}$ for the family $\Sch^T\to\NR^T$ defines a smooth embedding
	\begin{equation}\label{QD.E.9}
	\overline{\gamma}^T: (-1,0]^{\Ed^{\inte}(T)}\times \NR^{T}\to \overline{\NR}^{d+1}. 
	\end{equation}
	which identifies $\{0\}\times \NR^T$ with the $T$-stratum $\NR^T\subset \NR^{d+1}$ and defines the smooth structure of $\overline{\NR}^{d+1}$ near $\NR^T$. For simplicity, we assume the width of $\iota_f$ is 1 for all $f=(v,e)$, so every $\iota_f$ is a proper embedding
	\[
	\iota_f:\R^\pm_t\times [0,\pi]\times \NR_v\to \Sch_v.
	\]
	fibered over $\NR_v$, which restricts to a strip-like end on each fiber. We must construct local trivializations of $\overline{\V}^{d+1}$ near $\NR^T\subset \overline{\NR}^{d+1}$ for all $d$-leafed trees $T$. To ease our notation, we focus on the case for $\codim$-1 strata, i.e., when $|\Ed^{\inte}(T)|=1$ and $|\Ve(T)|=2$, and leave the general case to interested readers. Let $e$ be the unique interior edge of $T$ and $v^\pm\colonequals v^\pm(e)$. 
	For each $r_0=(r_{v^+}, r_{v^-})\in \NR^T$, we construct a bundle isomorphism 
	\[
\V^T\to (\overline{\gamma}^T)^*\overline{\V}^{d+1}
	\]
	defined in a neighborhood of $(0,r_0)\in (-1,0]\times \NR^{T}$, which trivializes $\overline{\V}^{d+1}$ locally. Note first that the truncated surface $S^{\trun}_{r_*}$ associated to $r_*=\overline{\gamma}^{T}(\rho_e, r_{v^\pm})\in \NR^{d+1}, \rho_e\in (-1,0)$ is obtained by gluing $S_{v^\pm,r_{v^\pm}}^{\trun}$ with a finite strip along the vertical boundary:
		\begin{align}\label{QD.E.11}
	S_{r_*}^{\trun}&\colonequals S_{v^+,r_{v^+}}^{\trun}\coprod \Stp_e\coprod S_{v^-,r_{v^-}}^{\trun}/\sim,\\
	\Stp_e&\colonequals [0, -\ln (-\rho_e)]_t\times [0,\pi]_s\nonumber
	\end{align}
		where $\iota_{f^+(e)}(0,s)$ is identified with $(0, s)\in \Stp_e$, and $\iota_{f^-(e)}(0,s)$ with $(-\ln (-\rho_e), s)\in \Stp_e, s\in [0,\pi]$. The pointed disk $S_{r_*}$ is then obtained by attaching semi-infinite strips to $\partial^v S_{r_*}^{\trun}$. The fiber $\V_{r_*}^{d+1}$ can be understood as the kernel of a surjective Fredholm operator \eqref{QD.E.6} with spectral projections, but we shall exploit the decomposition \eqref{QD.E.11} this time. Recall that the spectral projections are defined using the $L^2$-adjoint operator $i\partial_s$:
		\[
		\Pi^\pm\ \text{(resp. $\Pi^0$)}: V_{j-\half }\colonequals L^2_{j-\half}([0, \pi],\{0,\pi\};\C)\to H^\pm_{j-\half} \text{(resp. $\R$)}.
		\] 
		
		Then the fiber $\V^{d+1}_{r_*}$ is also identified with the kernel of 
		\begin{align*}
		\BD'_{r_*}: L^2_j (S_{v^+, r_{v^+}}^{\trun}, \partial^h S_{v^+, r_{v^+}}^{\trun}; E_+)&\oplus L^2_j (S_{v^-, r_{v^-}}^{\trun}, \partial^h S_{v^-, r_{v^-}}^{\trun}; E_-)\\
		&\to L^2_{j-1} (S_{v^+, r_{v^+}}^{\trun}; F_+)\oplus L^2_{j-1} (S_{v^+, r_{v^+}}^{\trun}; F_-)\oplus H^{T,-}_{j-\half}\\
		(\phi_+,\phi_-)&\mapsto \big(\bpartial \phi_+, \bpartial \phi_-, \Pi^-_{\rho_e}(\res(\phi_+), \res(\phi_-))\big). 
		\end{align*}
		where $E_\pm=(\Lambda^{1,0}S_{v^\pm, r_{v^\pm}}^{\trun})^{\otimes 2}$, $F_\pm=E_\pm\otimes\Lambda^{0,1}S_{v^\pm, r_{v^\pm}}^{\trun}$,
		\[
		H^{T,-}_{j-\half}\colonequals H^{|v_+|,-}_{j-\half}\oplus H^{|v_-|,-}_{j-\half}\oplus \R
		\]
		and $H^{|v_\pm|,-}_{j-\half}$ is the negative spectral subspace \eqref{QD.E.15} associated to $v^\pm$. $\Pi^-_{\rho_e}$ is the spectral projection map as in \eqref{QD.E.6} except at the flags $f_\pm(e)$. More precisely, the entry of $\Pi^-_{\rho_e}$ associated to $f_\pm(e)$ is defined by 
		\begin{align}\label{QD.E.10}
		V_{j-\half }\oplus  V_{j-\half }&\to  V_{j-\half }=H^-_{ j-\half}\oplus H^+_{j-\half}\oplus \R\nonumber\\
		\begin{pmatrix}
		\sigma^+\\
		\sigma^-
		\end{pmatrix}&\mapsto
		\begin{pmatrix}
		\Pi^-&-e^{l_e(i\partial_s)}\Pi^-\\ -e^{l_e(-i\partial_s)}\Pi^+& \Pi^+\\ \Pi^0 & -\Pi^0
		\end{pmatrix} 
		\begin{pmatrix}
		\sigma^+\\
		\sigma^-
		\end{pmatrix}
		\end{align}
		where $l_e=-\ln(-\rho_e)$ is the length of $\Stp_e$. 
		Since the spectrum of $(i\ps)$ on $V_{j-\half}$ is $\Z$, one verifies that the $(1,2)$- and $(2,1)$-entry of \eqref{QD.E.10} are smooth in the variable $\rho_e\in (-1,0]$ and vanish when $\rho_e=0$. In the latter case, $r_*=r_0=(r_{v^+}, r_{v^-})$ corresponds to the unglued surface, and $\BD_{r_0}'$ is the operator $\BD_{r_{v^+}}\oplus \BD_{r_{v^-}}$ in \eqref{QD.E.6} combined with the third row of \eqref{QD.E.10}. Recall that the bundle map $	\Res^T: \V_{v^+}\oplus \V_{v^-}\to \R$ is defined as 
		\begin{align*}
(\phi_{v^+}, \phi_{v^-})\mapsto\Pi^0\circ \res (\phi_{v^+})-\Pi^0\circ \res (\phi_{v^-})=\Res_{f^+(e)}(\phi_{v^+})-\Res_{f^-(e)}(\phi_{v^-}).
		\end{align*}
		Thus $\ker \BD_{r_0}'$ and $\V^T_{r_0}$ are equal by construction. Combined with a trivialization of $\V^T$ near $r_0$, this trivializes $\overline{\V}^{d+1}$ in a neighborhood of $\overline{\gamma}^T(0, r_0)$.

Finally, we have to verify that the smooth structure on $\overline{\V}^{d+1}$ would agree for different trivializations. Similar to \eqref{QD.E.12}, a smooth section $\phi$ of $\overline{\V}^{d+1}$ can be tested locally by the following property: for any compactly supported sections $\psi_\pm\in C^\infty_c(S^{\trun}_{v^\pm, r_{v^\pm}}, (TS^{\trun}_{v^\pm, r_{v^\pm}})^{\otimes 2})$, the pairing
\begin{equation}\label{QD.E.13}
(\rho_e, r)\in (-1,0]\times \NR^T\mapsto \langle \phi_{\overline{\gamma}^T(\rho_e, r)}, (\psi_+,\psi_-)\rangle.
\end{equation}
is smooth near $(0,r_0)$. By the unique continuation property, we may further require the support of  $\psi_\pm$ to lie in any fixed small ball in $S^{\trun}_{v^\pm, r_{v^\pm}}$. This property is independent of the choice of strip-like ends $\{\iota_f\}_{f\in \Fl(T)}$ (which is used for the construction of $\overline{\gamma}^T$). In the same vein, one verifies that the trivializations are consistent also near strata of $\codim\geq 2$. This completes the proof of Lemma \ref{QD.L.15}.  
\end{proof}

\subsection{Consistent families of quadratic differentials}\label{SecQD.8} The complement $\R^2\setminus T$ of a $d$-leafed tree consists of $d+1$ connected regions, which we identify as the boundary of a metric ribbon tree $\CT$, labeled anticlockwise by $\fo_0,\cdots, \fo_d$. By convention, the root of $T$ is adjacent to $\fo_0$ and $\fo_d$. If the flag $f$ adjacent to $\fo_j, \fo_k$ in $\R^2$, define the width $\w(f)\colonequals d_{\CT}(\fo_j, \fo_k)/\pi$, where $d_{\CT}$ is the distance function in $\CT$.
For any subset $A\subset \partial \CT$ and  any $v\in \Ve(T)$, denote by $\CT_A$ the unique metric ribbon subtree of $\CT$ with $\partial \CT_A=A$, and by $\CT_v$ the subtree of $\CT$ whose boundary consists of all connected regions adjacent to $v$. 

Following Lemma \ref{QD.L.8}, for any $(d+1)$-pointed disk $S$, there is a bijection $\partial S\to \partial \CT$ as ordered sets, sending the $j$-th component $C_j$ to $\fo_j$. This map is locally constant as $S$ varies in a family. Hence, one can think of $\partial\CT$ as a set of labels for the universal family $\Sch^{d+1}\to \NR^{d+1}$. In the same vein, the family $\Sch_v\to \NR_v$ is labeled by $\partial\CT_v$.  

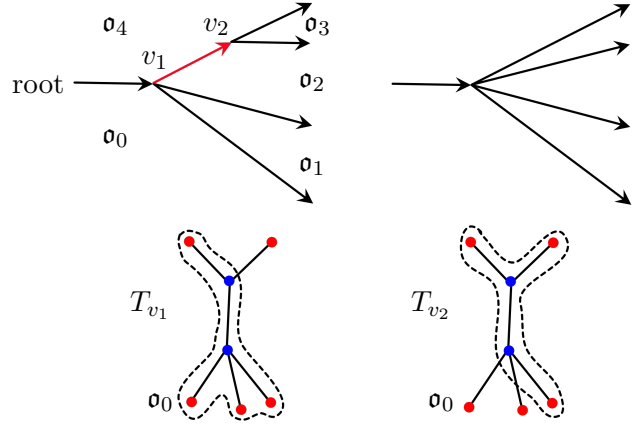
\begin{figure}[H]
	\centering
	\begin{overpic}[scale=.10]{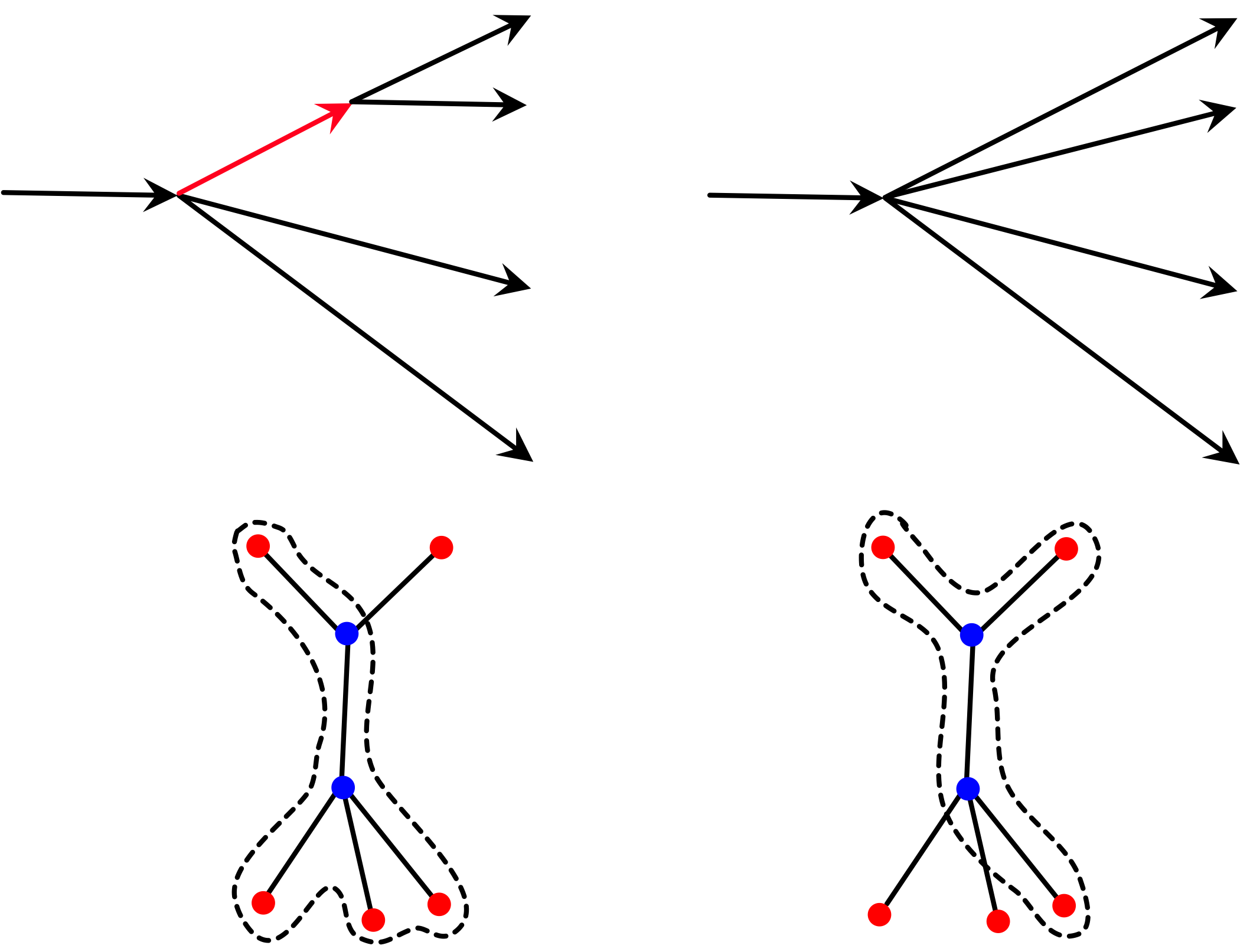}
		\put(-11, 59){root}
		\put(5,50){$\fo_0$}
		\put(40,45){$\fo_1$}
			\put(40,60){$\fo_2$}
				\put(41,70){$\fo_3$}
				\put(5,70){$\fo_4$}
				\put(12,64){$v_1$}
					\put(13,3){$\fo_0$}
						\put(63,3){$\fo_0$}
								\put(23,70){$v_2$}
								\put(10,20){$T_{v_1}$}
								\put(60,20){$T_{v_2}$}
	\end{overpic}	
	\caption{Collapsing an interior edge of a $5$-leafed tree $T$ (top). The metric ribbon subtrees associated to $v_1, v_2\in \Ve(T)$ (bottom).}
	\label{Pic18}
\end{figure}

\begin{definition}\label{QD.D.17} A smooth section $\phi_T=(\phi_{T,v})_{v\in \Ve(T)}$ of $\V^T\to \NR^T$ is called $\epsilon$-close to $\CT$, if the metric ribbon tree induced by $\phi_{T, v,r}$ is $\epsilon$-close to $\CT_v$ (in the sense of Definition \ref{QD.D.8}) for all $r\in \NR_v$, and additionally,
	\[
	\Res_f(\phi_{T,v,r})=\w(f)^2.
	\]
	for all flags $f$ attached to $v$. In this case, one can choose a set of $\phi_T$-adapted strip-like ends $\iota_T=(\iota_f)_{f\in \Fl(T)}$ for the family $\Sch^T\to \NR^T$, that is, a set of proper embeddings,
	\[
	\iota_f:  \R_t^\pm\times [0,\pi \w(f)]_s\times \NR_v\to \Sch_v,
	\]
	fibered over $\NR_v$, one for each flag $f$ attached to $v$, and which is $\phi_{T,v,r}$-adapted on each fiber $S_{v,r}$ (cf. Section \ref{SecQD.3}). The width of each $\iota_f$ is $\w(f)$. In particular, $\iota_f^*(\phi_{T,v,r})$ is the standard $(2,0)$-tensor $dz\otimes dz, z=t+is$ on the semi-finite strip $\R_t^\pm\times [0,\pi \w(f)]_s$. 
\end{definition}

For any smooth section $\phi_T=(\phi_{T,v})_{v\in \Ve(T)}$ of $\V^T\to \NR^T$ and any pair $T'>T$, consider the gluing map 
\[
\gamma^{T,T'}: (-1,0)^{\Ed^{\inte}(T)\setminus\Ed^{\inte}(T')}\times \NR^T\to \NR^{T'}, 
\]
defined using a set of $\phi_T$-adapted strip-like ends. We briefly recall the construction below. Similar to \eqref{QD.E.11}, for any $(r_v)\in \NR^T$ and $(\rho_e)\in (-1,0)^{ \Ed^{\inte}(T)\setminus\Ed^{\inte}(T')}$, the truncated surface associated to $r_*=\gamma^{T,T'}((\rho_e),(r_v))\in \NR^{T'}$ is obtained by gluing $S_{v,r_v}^{\trun}, v\in \Ve(T)$ with some strips, one for each interior edge in $T$ but not in $T'$,  along the vertical boundary:
\begin{align}\label{QD.E.14}
S_{r_*}^{\trun}&\colonequals\coprod_{v\in \Ve(T)}S_{v,r_v}^{\trun}\coprod_{e\in\Ed^{\inte}(T)\setminus\Ed^{\inte}(T')} \Stp_e/\sim,\\
\Stp_e&\colonequals [0, -\w(e)\ln (-\rho_e)]_t\times [0,\pi\w(e)]_s\nonumber
\end{align}
	where $\iota_{f^+(e)}(0,s)$ is identified with $(0, s)\in \Stp_e$, and $\iota_{f^-(e)}(0,s)$ with $(-\ln (-\rho_e), s)\in \Stp_e, s\in [0,\pi\w(e)]_s$. This time the width of each $\Stp_e$ is dictated by the metric ribbon tree $\CT$. Finally, we attach semi-infinite strips, one for each component of $\partial^v S^{\trun}_{r_*}$, to obtain the union of pointed disks that corresponds to $r_*$. 
	
	The upshot is that in the meantime one may glue quadratic differentials $\phi_T$ to obtain a section of $\V^{T'}$ on the image of $\gamma^{T, T'}$: the quadratic differential on the fiber of $r_*=\overline{\gamma}^{T,T_*}((\rho_e), (r_v))$ is defined by patching $\phi_{T, v, r_v}$ on $S^{\trun}_{v, r_v}$ with the standard $(2,0)$-tensor $dz\otimes dz$ on all $\Stp_e$ and on semi-infinite strips under the decomposition \eqref{QD.E.14}. This quadratic differential is holomorphic, since the strip-like ends are adapted to $\phi_T$ (in other words, we are gluing singular flat surfaces; see Figure \ref{Pic26} in the next section for an illustration). A smooth $\phi=(\phi_T)$ of $\overline{\V}^{d+1}\to \overline{\NR}^{d+1}$ is called \textit{consistent} if for any $T<T'$, $\phi^{T'}$ is the section obtained by gluing $\phi_T$ on the image of $\gamma^{T,T'}$. 
	
In order to construct the Fukaya-Seidel categories, there is one final requirement for $\phi=(\phi_T)$: families with the same set of labels are considered the same, and so they come with the same family of quadratic differentials and strip-like ends. This means that to any subset $A\subset \partial \CT$ is associated a smooth section $\psi_A$ of $\V^{|A|}\to \NR^{|A|}$, and for any $T$ and $v\in \Ve(T)$, we require that $\phi_{T,v}=\psi_A$ if $\partial \CT_v=A$, i.e., if the family $\Sch_v\to \NR_v$ is labeled by $A$. We insist a similar condition for the set of strip-like ends. More concretely, one can prove:
	
	\medskip

\begin{lemma}\label{QD.L.14} For $\epsilon>0$ and any metric ribbon tree $\CT$ with $(d+1)$ exterior vertices, one can construct inductively for any subset $A\subset \partial \CT$ with $|A|\geq 3$, 
	\begin{itemize}
	\item a smooth section $\psi_A$ of $\V^{|A|}\to \NR^{|A|}$ that is $\epsilon$-close to $\CT_A$ (in the sense of Definition of \ref{QD.D.17});
	\item  a set of $\psi_A$-adapted strip-like ends $(\iota_{A,k})_{k=0}^{|A|-1}$ for the family $\Sch^{|A|}\to \NR^{|A|}$;
	\end{itemize}
which defines for each $d$-leafed tree $T$, 
	\begin{itemize}
\item a smooth section $\phi_T=(\phi_{T,v})_{v\in\Ve(T)}$ of $\V^T\to \NR^T$;
\item a set of $\phi_T$-adapted strip-like ends $\iota_T=(\iota_f)_{f\in \Fl(T)}$ for the family $\Sch^T\to \NR^T$;
	\end{itemize}
and which satisfies the following properties:
	\begin{enumerate}[label=$(\roman*)$]
		\item $\phi_{T,v}=\psi_A$ and $\iota_{ f_k(v)}=\iota_{A,k}, 0\leq k\leq |v|-1,$ if $A=\partial\CT_v$;
		\item$\phi=(\phi_T)$ is a  smooth section of $\overline{\V}^{d+1}\to \overline{\NR}^{d+1}$;
		\item \label{QD.L.14i}  for any pair $T<T'$, $\phi_{T'}$ is obtained by gluing $\phi_T$ on the image of $\gamma^{T,T'}$;
\item\label{QD.L.14ii} $\{(\iota_f)_{f\in \Fl(T)}\}_T$ is a consistent choice of strip-like ends in the sense of \cite[Lemma 9.3]{S08}.
	\end{enumerate}
In this case, we say that $\phi=(\phi_T)$ is a $(\epsilon, \CT)$-consistent section of $\overline{\V}^{d+1}\to \overline{\NR}^{d+1}$, and that$(\iota_T)$ is a set of strip-like ends adapted to $(\phi_T)$. 
\end{lemma}
\begin{remark} By the proof of \cite[Lemma 9.3]{S08}, the last property \ref{QD.L.14ii} implies that the following diagram is commutative for any $T<T'<T''$:
	\begin{equation}\label{QD.E.8}
	\begin{tikzcd}
(-1,0)^{\Ed^{\inte}(T')\setminus\Ed^{\inte}(T'')}\times	(-1,0)^{\Ed^{\inte}(T)\setminus\Ed^{\inte}(T)}\times  \NR^{T} \arrow[r,"\gamma^{T,T''}"]\arrow[d,"\Id\times \gamma^{T,T'}"]&\NR^{T''}. \\
	(-1,0)^{\Ed^{\inte}(T')\setminus \Ed^{\inte}(T'')}\times \NR^{T'} \arrow[ru,"\gamma^{T',T''}"'] & 
	\end{tikzcd}\qedhere
	\end{equation}
\end{remark}

\begin{proof}[Proof of Lemma \ref{QD.L.14}] The argument is similar to that of \cite[Lemma 9.3]{S08}. For any $3\leq n\leq d+1$, let $0<\epsilon_n\ll \epsilon$ be a small constant to be fixed in the proof. For any $A \subset \partial\CT$, we shall construct  a $(\epsilon_{|A|}, \CT_A)$-consistent family of quadratic differentials and a set of strip-like ends adapted to this family over $\overline{\NR}^{|A|}$ using induction. When $|A|=3$, $\NR^3$ is a single point, and a compatible quadratic differential is determined uniquely by its complex residues; see Example \ref{QD.EX.8}; so the statement holds trivially, and we can take $\epsilon_3$ as small as we wish.

The induction then proceeds on $n=|A|$, and without loss of generality, one may focus on the last step when $n=d+1$, as this explains the general case as well. Hence, we may assume that $\phi_T=(\phi_{T,v})$ and $(\iota_f)_{f\in \Fl(T)}$ have been constructed for all $T<T_*$ which satisfy the consistence conditions.

 The section $\phi_{T_*}=\psi_{\partial \CT}$ on the top stratum  $\NR^{T_*}= \NR^{d+1}\subset \overline{\NR}^{d+1}$ is determined by the third property \ref{QD.L.14i} on the image of $\gamma^{T,T_*}$, which is smooth by the test \eqref{QD.E.13}. By induction hypothesis, each $\phi_{T,v}$ is $\epsilon_{|v|}$-close to the subtree $\CT_v$. This implies that $\phi_{T_*}$ is at least $\epsilon_T^*$-close to $\CT$ on the image of $\gamma^{T,T_*}$ with
 
 \begin{equation}\label{QD.E.16}
\epsilon_{T}^*\colonequals 10\sum_{v\in \Ve(T)} \epsilon_{|v|}
 \end{equation} 
(this follows again from an inductive argument on $|\Ve(T)|$). The commutativity of \eqref{QD.E.8} (with $T''=T_*$) then implies that the section $\phi_{T_*}$ would agree on the overlapped region.  Set 
\begin{equation}\label{QD.E.17}
\epsilon_{d+1}\colonequals \max_{T<T_*}\epsilon_{T}^*.
\end{equation}

By Lemma \ref{QD.L.10}, $\phi_{T_*}$ can be extended to a smooth section over all $\NR^{d+1}$ which is still $\epsilon_{d+1}$-close to $\CT$. In a similar vein, the gluing construction determines the strip-like ends $(\iota_f)_{f\in \Fl(T_*)}$ near the boundary of  $\overline{\NR}^{d+1}$; cf. \cite[Lemma 9.3]{S08}.  An extension of $\phi_{T_*}$-adapted strip-like ends over $\NR^{d+1}$ clearly exists. 

Finally, the constant $\epsilon_n$ is determined recursively by \eqref{QD.E.16} and \eqref{QD.E.17}, which depends only on the possible degenerations of a $(n+1)$-leafed tree (which is dependent of the metric ribbon subtree $\CT_A$). To make the section $\phi_T$ $\epsilon$-close to the metric ribbon tree $\CT$ for all $T$, it suffices to choose $\epsilon_3\ll 1$ sufficiently small to start. This completes the proof of Lemma \ref{QD.L.14}.
\end{proof}

\section{Fukaya-Seidel Categories}\label{SecFS}

\subsection{Introduction}\label{SecFS.1} Let $(q_j)_{j=0}^n$ be a collection of critical points of $W: M\to \C$, not necessarily distinct. A set $\Th=(\Lambda_{q_0,\theta_0}, \Lambda_{q_1,\theta_1},\cdots, \Lambda_{q_n,\theta_n})$ of thimbles is called \textit{admissible} if 
	\begin{itemize}
	\item $\theta_n<\cdots<\theta_1<\theta_0<\theta_n+2\pi$;
	\item for all $0\leq j\leq n$, the projection $l_{q_j,\theta_j}=W(\Lambda_{q_j,\theta_j})$ does not contain any critical values of $W$ other than its end point.
\end{itemize}

 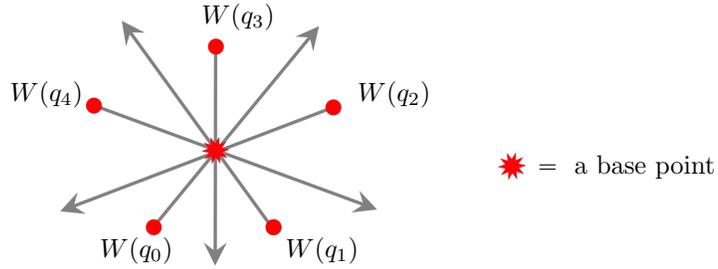
\begin{figure}[H]
	\centering
	\begin{overpic}[scale=.15]{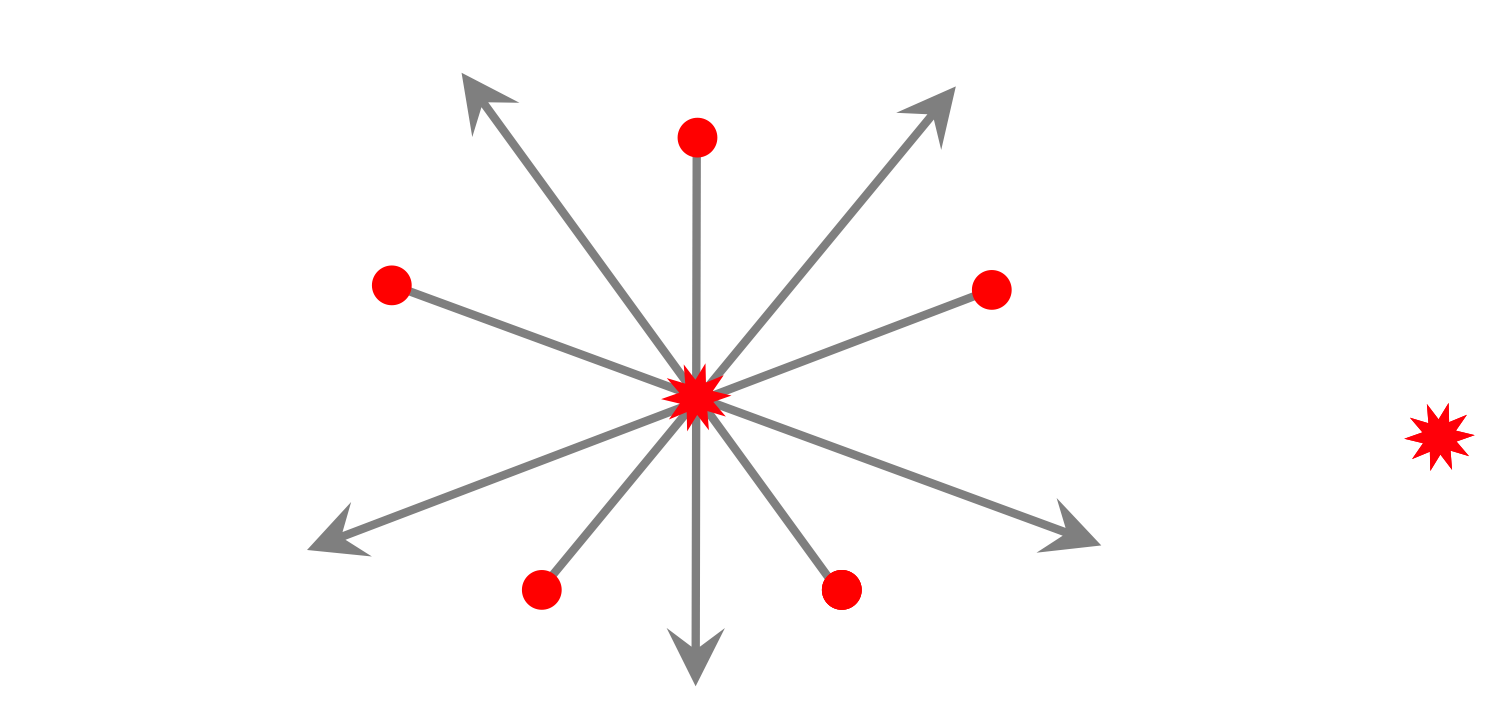}
		\put(27,3){\small$W(q_0)$}
		\put(58,3){\small$W(q_1)$}
		\put(70,29){\small$W(q_2)$}
		\put(44,42){\small$W(q_3)$}
		\put(12,29){\small$W(q_4)$}
		\put(100,17){\small $= \text{ a base point}$}
	\end{overpic}	
	\caption{An admissible set of thimbles.}
	\label{Pic20}
\end{figure}

The set $\Th$ is ordered such that  $\Lambda_{q_j,\theta_j}\prec\Lambda_{q_k,\theta_k}$ if $\theta_k<\theta_j$. For convenience, write $\Lambda_j=\Lambda_{q_j, \theta_j}$. Choose admissible Floer data $\fa_{jk}$, one for each pair $(\Lambda_j, \Lambda_k),\ j<k$. In this section, we shall construct a directed $A_\infty$-category $\sE=\sE(\Th)$ with $\Ob \sE=\Th$ and whose morphism spaces are defined by the formula
\begin{equation}\label{FSE.22}
\hom_{\sE}(\Lambda_j, \Lambda_k)=\left\{\begin{array}{cl}
0 & \text{ if }j>k,\\
\BF\cdot e_{\Lambda_j} &\text{ if }j=k,\\
\Ch_\natural^*(\Lambda_j, \Lambda_k;\fa_{jk}) &\text{ if }j<k,
\end{array}
\right.
\end{equation}
where the Floer complex $\Ch_\natural^*(\Lambda_j, \Lambda_k;\fa_{jk})$ is  defined as in Section \ref{Sec1}. We shall verify that $\sE$ is well-defined up to canonical quasi-isomorphisms, meaning that the $A_\infty$-categories $\sE$ defined using different choices of auxiliary data in the construction or different Floer data $\fa_{jk}$ are all quasi-isomorphic, and the quasi-isomorphism is unique up to homotopy. The most relevant examples for our applications are the following.

\begin{example}\label{EX3.3} Define the set of \textit{critical angles} of $(M, W)$ as 
	\[
	\SC(M,W)\colonequals \{ \frac{W(q_0)-W(q_1)}{|W(q_0)-W(q_1)|}: q_0, q_1\in \Crit(W), W(q_0)\neq W(q_1) \}\subset S^1.
	\]
Any $\theta_*\in \R$ is called \textit{admissible} if $e^{i\theta_*}\not\in \SC(M,W)$. Denote by $\theta_*^{\crit}$ the smallest angle greater than $\theta_*$ such that $e^{i\theta_*^{\crit}}$ is critical. Thus any angles in $[\theta_*,\theta_*^{\crit})$ are still admissible, and the critical values $\im (e^{-i\theta_*}W(q)), q\in \Crit(W)$ are mutually distinct. Label the set $\Crit(W)=\{y_j\}_{j=1}^m$ such that
	\begin{equation}\label{FSE.1}
	\im(e^{-i\theta_*}W(y_1))<\im(e^{-i\theta_*}W(y_2))<\cdots<\im(e^{-i\theta_*}W(y_m)),
	\end{equation}
	and choose angles $\theta_j$, $1\leq j\leq m$ such that
	\begin{equation}\label{FSE.19}
		\theta_m<\cdots<\theta_2<\theta_1\subset (\theta_*, \theta_*^{\crit}). 
	\end{equation}
	Then the set $\Th_{\theta_*}\colonequals (\Lambda_{y_1,\theta_1},\cdots, \Lambda_{y_m,\theta_m})$ is admissible. In this case, $\sE(\Th_{\theta_*})$ is called \textit{the Fukaya-Seidel category} of the Landau-Ginzburg model $(M,W)$ in the direction of $\theta_*$. By Proposition \ref{FS.P.18} below, $\sE(\Th_{\theta_*})$ is independent of the choice of $(\theta_1,\cdots, \theta_m)$ satisfying \eqref{FSE.19} up to canonical quasi-isomorphisms.

	 It is sometimes convenient to choose a base point $w_{\theta_*}\in \C$ far away in the direction of $e^{i\theta_*}$ and ask all rays $l_{y_j,\theta_j}, 1\leq j\leq m$ to pass through $w_{\theta_*}$; see Figure \ref{Pic7} below. In particular, $\sE(\Th_{\theta_*})$ will be independent of the choice of $w_{\theta_*}$. 
\end{example}

\begin{example}\label{FS.EX.2} For convenience, we shall always assume in this paper that $\theta_*=0$ is admissible, and so is $\theta_*=\pi$. Let $\theta_\star\colonequals 0^{\crit}$. A thimble $\Lambda_{q,\theta},\ q\in \Crit(W)$ is called \textit{stable} if $\theta\in (0, \theta_\star)$ and respectively \textit{unstable} if $\theta\in (\pi, \pi+\theta_\star)$. Choose angles such that
	\begin{equation}\label{FSE.24}
0<	\theta_m<\cdots<\theta_1<\theta_0<\theta_\star<\pi<\eta_0<\eta_1<\cdots <\eta_m<\pi+\theta_\star,
	\end{equation}
	and define 
	\begin{align*}
	\Th_0&=(S_1,S_2,\cdots, S_m),& S_j&\colonequals \Lambda_{x_j,\theta_j},\\
	\Th_\pi&=(U_m,U_{m-1},\cdots, U_1),& U_j&\colonequals \Lambda_{x_j, \eta_j},\ 1\leq j\leq m.
	\end{align*}
	
	We insist that  $\Crit(W)=\{x_j\}_{j=1}^m$ is ordered as in the set $\Th_0$, so  
		\begin{equation}\label{FSE.2}
	H(x_1)<H(x_1)<\cdots<H(x_m).
	\end{equation}
	Then $\sA\colonequals \sE(\Th_0)$ is called the Fukaya-Seidel category of $(M,W)$, and $\sB\colonequals \sE(\Th_\pi)$ the Koszul-dual category of $\sA$. 
	 
	The union $\Th_\pi\sqcup \Th_0$ is also admissible, which ordered by
\begin{equation}\label{FSE.23}
U_m\prec U_{m-1}\prec\cdots U_1\prec S_1\prec S_2\prec\cdots \prec S_m. 
\end{equation}

 Following the convention in Section \ref{SecAF.2}, the directed $A_\infty$ category  $\sE(\Th_{\pi}\sqcup\Th_0)$ defines a diagonal $(\sA,\sB)$-bimodule ${}_{\sA}\Delta_{\sB}$ which assigns to each pair $(U_j, S_k)$ the Floer complex $\Ch^*_\natural (U_j, S_k; \fa_{jk}^*)$. By Lemma \ref{lemma:Vanishing} and Lemma \ref{lemma:CanonicalGenerator}, for some good choices of Floer data $\fa_{jk}^*$, 
 \begin{equation}\label{FSE.30}
 \Delta(U_j, S_k)=\Ch^*_\natural (U_j, S_k; \fa_{jk}^*)=\left\{\begin{array}{ll}
0 & \text{ if }j\neq k,\\
\BK & \text{ if }j=k. 
\end{array}
\right.
 \end{equation}
 
Thus the induced functor $r_\Delta: \sA\to \sQ_r=\rfmod(\sB)$ identifies $\Ob\sA$ with $\Ob\sB^!$. The Koszul duality between $\sA$ and $\sB$ is verified by the following theorem, whose proof is accomplished in Section \ref{SecPT}.
\end{example}
 \begin{theorem}\label{FS.T.3} Let $\sS_k\in \Ob \sQ_r, 1\leq k\leq m$ denote the image of $S_k$ under $r_{\Delta}$. Then for all $1\leq k_1, k_2\leq m$, the chain map
 	\begin{equation}\label{FSE.29}
	(r_\Delta)^1: \hom_{\sA}(S_{k_1}, S_{k_2})\to \hom_{\sQ_r}(\sS_{k_1},\sS_{k_2})
 	\end{equation}
 	is a quasi-isomorphism. Thus the $A_\infty$-functor $r_\Delta: \sA\to \sQ_r$ is a quasi-isomorphism onto its image, and $(\sA,\sB)$ form a Koszul duality pair in the sense of Definition \ref{AF.D.10}. 
\end{theorem}

\begin{remark}\label{FS.R.4} By Proposition \ref{FS.P.18}, $\sE(\Th_{\pi}\sqcup\Th_0)$ is well-defined up to canonical quasi-isomorphism. Then by Lemma \ref{AF.L7.3}, it suffices to verify \eqref{FSE.29} for some convenient choices of Floer data. If $\fa^*_{jk}$ are chosen so that \eqref{FSE.30} holds, then \eqref{FSE.29} is easily verified for all $k_1\geq k_2$, because $\sA$ and $\sB^!$ are both directed. Thus the challenging part of Theorem \ref{FS.T.3} is for $k_1<k_2$.
\end{remark}

\begin{remark}\label{FS.R.5} The construction of $\Delta$ can be generalized slightly to the following scenario: let $\Th_{st}$ (resp. $\Th_{un}$) be any admissible set of stable (resp. unstable) thimbles, and set $\widetilde{\sA}=\sE(\Th_{st})$ and $\widetilde{\sB}=\sE(\Th_{un})$. Then $\sE(\Th_{un}\sqcup \Th_{st})$ defines a $(\widetilde{\sA},\widetilde{\sB})$-bimodule $\widetilde{\Delta}$. This point of view will become useful in the proof of Theorem \ref{FS.T.3}.
\end{remark}

\begin{figure}[H]
	\centering
	\begin{overpic}[scale=.15]{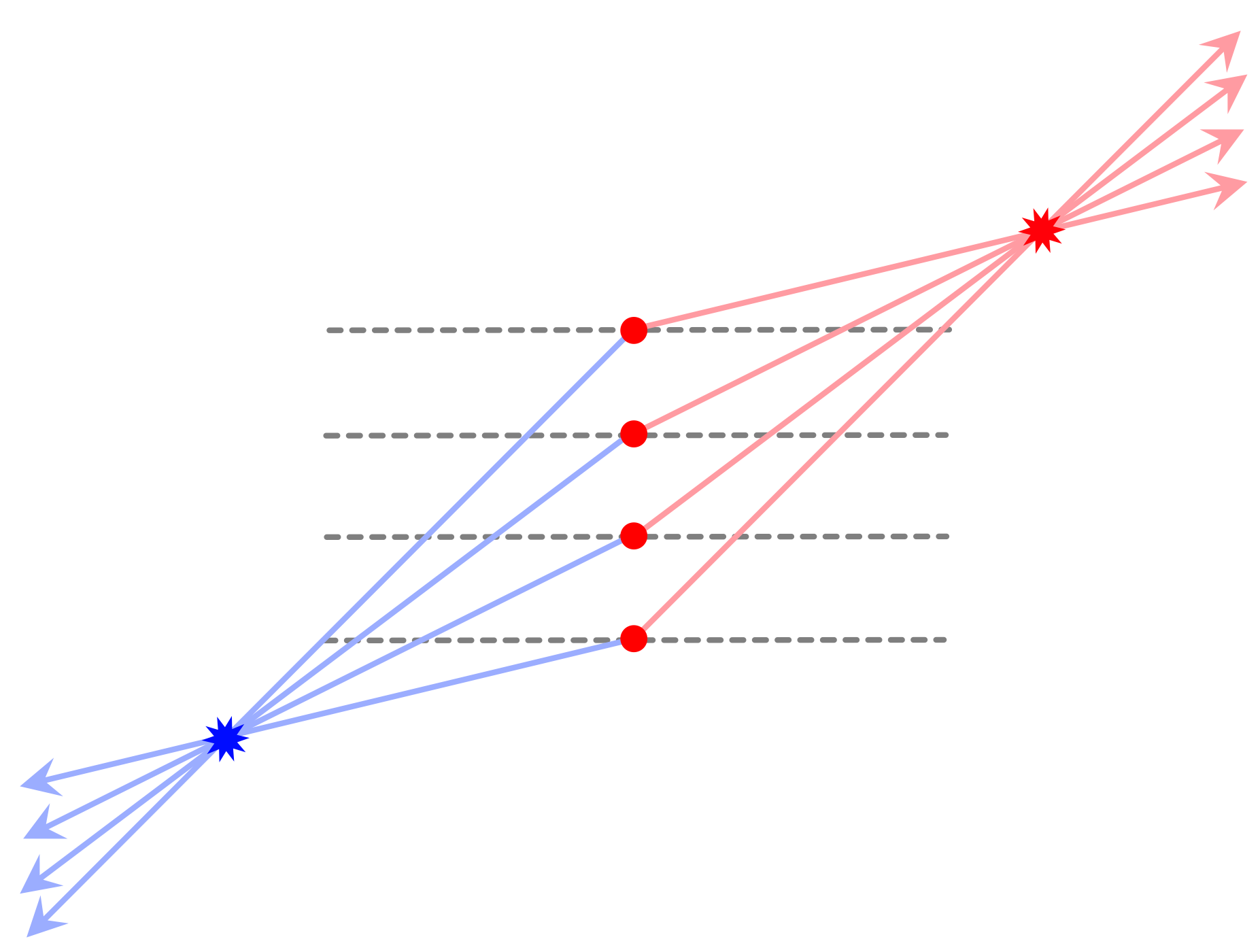}
		\put(46,52){\small $W(x_4)$}		
		\put(46,43){\small $W(x_3)$}		
		\put(46,28){\small $W(x_2)$}		
		\put(46,20){\small $W(x_1)$}
				\put(80,48){\small $H(x_4)$}		
		\put(80,40){\small $H(x_3)$}		
		\put(80,32){\small $H(x_2)$}		
		\put(80,24){\small $H(x_1)$}	
		\put(85,55){\small $w_0$}		
				\put(18,13){\small $w_\pi$}		
				\put(100,72){\small $S_1$}
								\put(100,68){\small $S_2$}	
												\put(100,64){\small $S_3$}	
																\put(100,60){\small $S_4$}	
																	\put(-3,0){\small $U_4$}
																\put(-3,5){\small $U_3$}	
																\put(-3,10){\small $U_2$}	
																\put(-3,15){\small $U_1$}	
	\end{overpic}	
	\caption{The diagonal bimodule ${}_{\sA}\Delta_{\sB}$.}
	\label{Pic7}
\end{figure}

\begin{remark}\label{FS.R.6}There is a rotational operation $\sR$ on the space of admissible sets of thimbles which moves the first thimble to the last one. This means that for any $\Th=(\Lambda_{q_0,\theta_0},\cdots,\Lambda_{q_n,\theta_n})$, $\sR\Th$ is defined by 
$
(\Lambda_{q_1,\theta_1},\cdots, \Lambda_{q_n,\theta_n},\Lambda_{q_0,\theta_0-2\pi}),
$
so the set of angles becomes $\theta_0-2\pi<\theta_n<\cdots<\theta_1$. By Lemma \ref{L2.11}, the canonical isomorphism 
\[
\HFF_\natural^*( \Lambda_{q_j,\theta_j},\Lambda_{q_0,\theta_0-2\pi})\cong D\HFF_\natural^*(\Lambda_{q_0,\theta_0}, \Lambda_{q_j,\theta_j})
\]
is grading-preserving. Thus $\sE(\sR\Th)=\sC\sE(\Th)$ if one ignores the degree shifting in the definition of  $\sC$ in \eqref{AF.E.25}. To some extent, only the cyclic ordering on $\Th$ is essential to this construction. 
\end{remark}

\subsection{Cobordism data and moduli spaces for families}\label{SecFS.2} The rest of Section \ref{SecFS} is devoted to the construction of $\sE(\Th)$. From now on fix an admissible set of thimbles
\[
\Th=(\Lambda_0, \cdots, \Lambda_n),\ \Lambda_j=\Lambda_{q_j,\theta_j},\ 0\leq j\leq n,\  n\geq 2,
\]
along with a metric ribbon tree $\CT$ with $(n+1)$ exterior vertices. We shall think of $\Th$ as a set of labels for the universal family $\Sch^{n+1}\to \NR^{n+1}$ by identifying 
\begin{equation}\label{FSE.7}
\Th\cong \partial \CT=\{\fo_0,\cdots,\fo_n\},\ \Lambda_j\mapsto \fo_j
\end{equation}
as ordered sets. Let $
\w_{jk}\colonequals d_\CT(\fo_j, \fo_k),\ 0\leq  j<k\leq n$ denote the distance of $\fo_j,\fo_k$ in $\CT$, and choose any admissible Floer data 
$$\fa_{jk}=(\w_{jk}, \alpha_{jk}(s), \beta_{jk},\epsilon_{jk}, \delta H_{jk}),$$
one for each pair $(\Lambda_j,\Lambda_k),\ j<k$. Thus for some $0<\delta\ll 1$ we must have
\begin{align*}
\alpha_{jk}(s)&=\left\{
\begin{array}{ll}
\theta_k &\text{ if } s\geq  \w_{jk}-\delta,\\
\theta_j-\pi &\text{ if }s\leq \delta,
\end{array}
\right.&
&\re(e^{i(\beta_{jk}-\alpha_{jk}(s)})>\epsilon_{jk}\ \forall s\in \R_s, 
\end{align*}
and the perturbation 1-form $\delta H_{jk}$ is supported on $[0,\w_{jk}]_s\subset \R_s$. For simplicity, write $\fa_0=(\w_0,\alpha_0(s), \beta_0,\epsilon_0,\delta H_0)$ for $\fa_{0n}$ and $\fa_k=(\w_0,\alpha_k(s), \beta_k,\epsilon_k,\delta H_k)$ for $\fa_{(k-1)k}, 1\leq k\leq n$. Let $\gamma_k:\R_s\to \C$ be a characteristic curve of $\alpha_k(s)$. To construct the $A_\infty$-category $\sE=\sE(\Th)$, it suffices to take $\CT$ to be the metric ribbon tree $\CT^{n+1}$ in Example \ref{QD.EX.8}, in which case $\w_{jk}=\pi$ for all $j<k$. But for future reference we shall develop the theory for a general $\CT$ as well.

\medskip

Let $S$ be any $(n+1)$-pointed disk labeled by $\Th$ and equipped with an $S$-compatible quadratic differential $\phi$. The completion of $(S,\phi)$
\[
\hat{S}=S\ \bigcup\ \HH^-_0\cup\cdots \cup \HH^-_n,
\]
carries a complete Riemannian metric. Fix a set of strip-like ends $(\iota_k)_{k=0}^n$ adapted to $\phi$. Assume that $\phi$ is $\epsilon$-close to $\CT$ for some $\epsilon>0$ and the complex residues of $\phi$ are fixed by $\CT$; see Section \ref{SecQD.2} for relevant notations. 

The Floer equation \eqref{E1.12} on this completion is specified by a phase pair $(\Xi, \delta\kappa)$, whose construction will follow closely the case for continuation maps; see Section \ref{Subsec:Continuation}. Note that the inclusion map $\sigma_k: \HH^+_k\to \hat{S}$ extends to an isometric embedding $\R_t\times (-\infty, \delta]\to \hat{S}$ for some $0<\delta\ll 1$. For any $K>0$, we consider smooth embeddings
\[
\Xi^\dagger: S\to \C
\]
that satisfy the following properties:
\begin{itemize}
\item for some constant $c_{\Xi^\dagger}^k\in \C, 1\leq k\leq n$, we have 
\begin{align}\label{FSE.3}
\Xi^\dagger\circ\iota_0(t,s)&=-i\epsilon_0e^{i\beta_0}\cdot t+\gamma_0(s)
	&&\text{ on }\R^-_t\times [0, \w_0]_s,\\
	\Xi^\dagger\circ\iota_k(t,s)&=-i\epsilon_ke^{i\beta_k}\cdot t+\gamma_k(s)+c^k_{\Xi^\dagger}
		&&\text{ on }[K,+\infty)_t\times [0, \w_k]_s;\nonumber
\end{align}
\item for some $0<\delta\ll 1$, 
\begin{equation}\label{FSE.4}
\Xi^\dagger\circ \sigma_k(t,s)=g_k(t)+e^{i\theta_k}\cdot s \text{ on } \R_t\times [0,\delta]_s,
\end{equation}
where $g_k(t)\colonequals \Xi^\dagger\circ \sigma_k(t,0)$;
\item for any $t\in \R_t$ and $0\leq k\leq n$, 
\begin{equation}\label{FSE.5}
-\im (e^{-i\theta_k}\pt g_k(t))-\half |\pt g_k(t))|^2>0.
\end{equation}
\end{itemize}

Denote by $\Emb_K^S$ the space of such embeddings and by
\[
S^{\trun}_K\colonequals S\big\backslash \bigg( \iota_0\big((-\infty, 0)_t\times [0,\w_0]_s\big)\coprod_{k=1}^n \iota_k\big((K,+\infty)_t\times [0,\w_k]_s\big)\bigg)
\]
the truncation of $S$.  Because any embedding $\Xi^\dagger$ is determined by its restriction on $S^{\trun}_K$, $\Emb_K^S$ is identified with a subspace of $C^\infty(S^{\trun}_K;\C)$ and equipped with the smooth topology. For any $K<K'$, the inclusion map $\Emb_K^S\to \Emb_{K'}^S$ is continuous.
\begin{lemma}\label{FSL.4} The direct limit $\Emb^S\colonequals \varinjlim \Emb^S_K$ is weakly contractible. 
\end{lemma}
\begin{proof}[Sketch of Proof]  The proof follows the same line of arguments as in Lemma \ref{L2.5}, and we only give a sketch. Let  $\Emb_{K}^{S,*}$ denote the space of  embeddings $\Xi^\dagger: S\to \C$ satisfying \eqref{FSE.3} and a weaker version of \eqref{FSE.5}:
\begin{equation}\label{FSE.20}
-\im (e^{-i\theta_k}\pt g_k(t))>0,\ \forall t\in \R_t,\ 0\leq k\leq n. 
\end{equation}
	
	The second space $\Emb_{K}^{S,*}$ is easier to work with for two reasons: 1) it depends only on the underlying smooth surface of $S$ (not on the metric); 2) it is invariant under self-diffeomorphisms of $S$ fixing the complement $S\setminus S^{\trun}_K$. For any continuous map $\varphi: S^n\to \Emb^S_K$, the null-homotopy of $\varphi$ in $\Emb^S$ is constructed in three steps, using $\Emb_{K}^{S,*}$ as an intermediate space: $i$) by taking $c^k_{\Xi^\dagger}=0$, $ii$) by pushing $\Xi^\dagger$ ``outwards" along the boundary, and $iii$) by applying a family of self-diffeomorphisms of $S$. At each step, one may take the freedom to increase $K$; cf. the proof of Lemma \ref{L2.5}. We leave the details to interested readers. 
\end{proof}
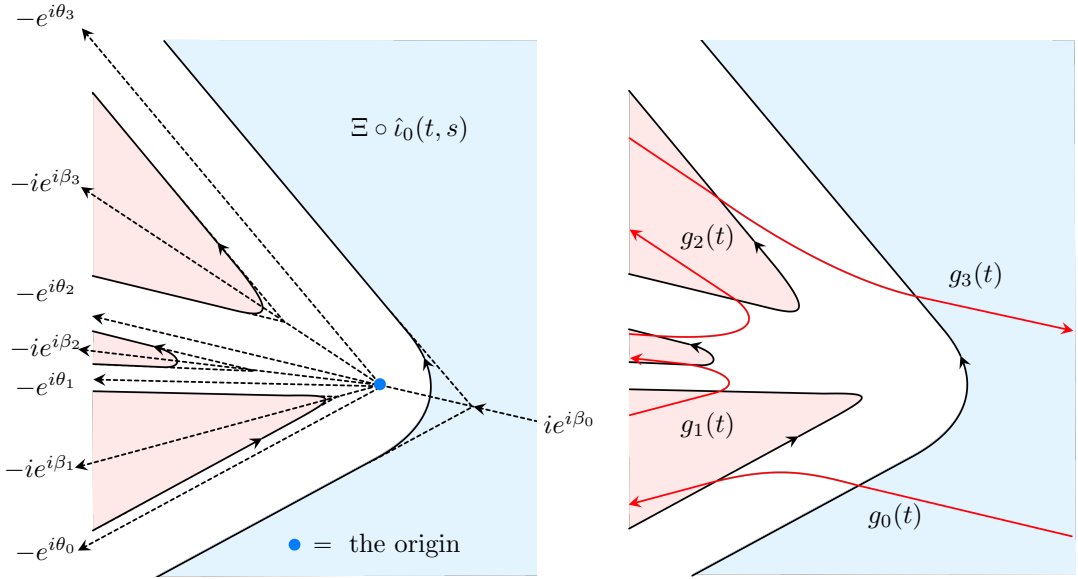
\begin{figure}[H]
	\centering
	\begin{overpic}[scale=.08]{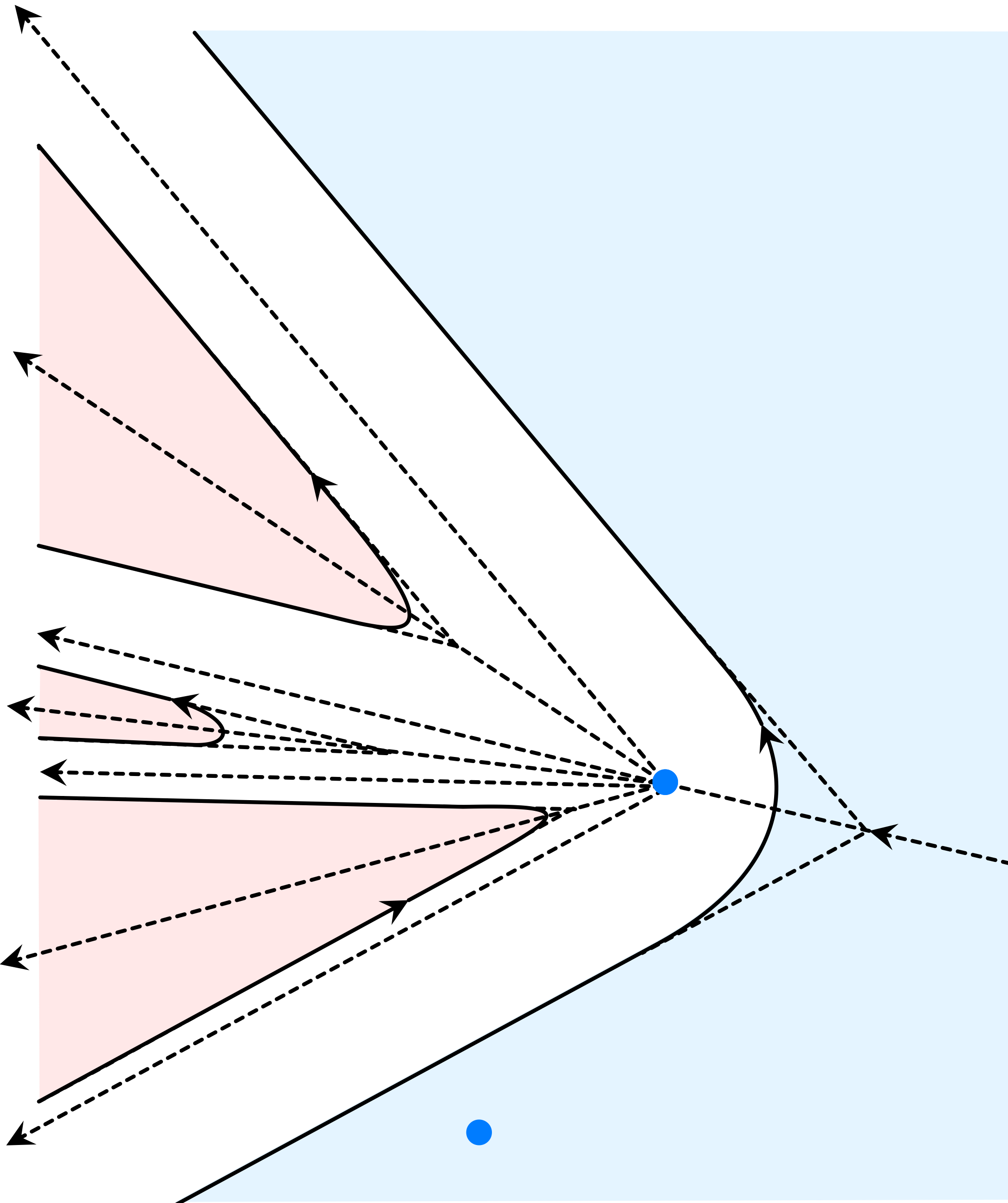}
		\put(-11,100){\small$-e^{i\theta_3}$}
		\put(-12,70){\small$-ie^{i\beta_3}$}
		\put(-12,40){\small$-ie^{i\beta_2}$}
		\put(-13,18){\small$-ie^{i\beta_1}$}
		\put(-11,50){\small$-e^{i\theta_2}$}
		\put(-11,33){\small$-e^{i\theta_1}$}
		\put(-11,3){\small$-e^{i\theta_0}$}
		\put(85,26){\small$ie^{i\beta_0}$}
		\put(43,4.5){\small$=\text{ the origin}$}
		\put(50,80){\small $\Xi\circ \hat{\iota}_0(t,s)$}
	\end{overpic}	
	\hspace{1cm}
	\begin{overpic}[scale=.08]{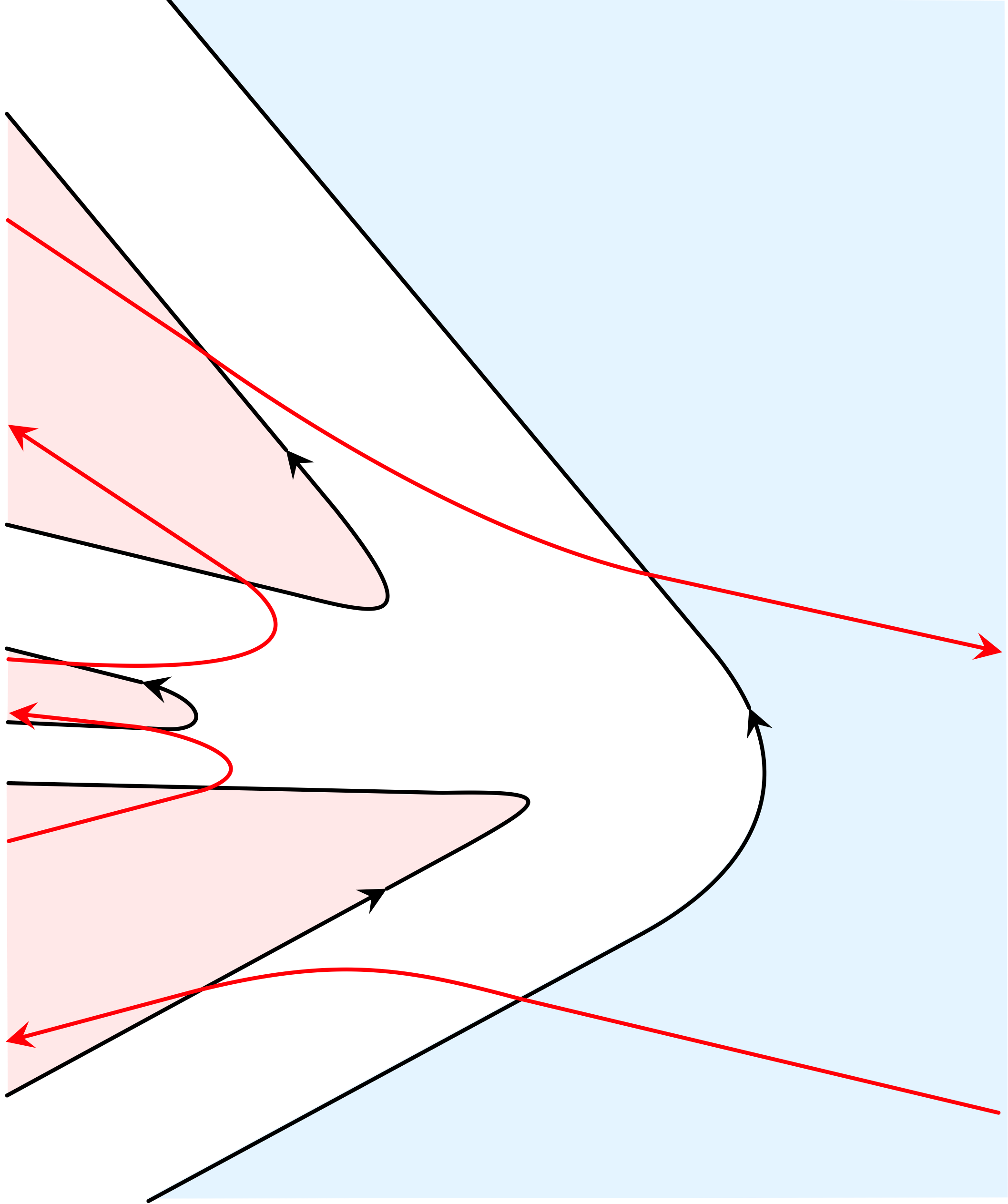}
		\put(45,10){\small $g_0(t)$}
		\put(60,55){\small $g_3(t)$}
		\put(10,62){\small $g_2(t)$}
		\put(10,27){\small $g_1(t)$}
	\end{overpic}
	\caption{The phase function $\Xi$ on $\hat{S}$. Shaded regions denote the image of $\Xi\circ \hat{\iota}_k$ when either $t\leq 0$ or $t\geq K$. Red curves represent the images of $g_k(t)$. This figure should be compared with Figure \ref{Pic15}.}
	\label{Pic9}
\end{figure}

By \eqref{FSE.4}, any $\Xi^\dagger\in \Emb_{K}$ extends to an orientation-reversing diffeomorphism $\Xi: \hat{S}\to \C^2$ by setting
\begin{equation}\label{FSE.13}
\Xi\circ \sigma_k(t,s)=g_k(t)+e^{i\theta_k}\cdot s \text{ on } \R_t\times (-\infty, \delta], 0\leq k\leq n
\end{equation}
and $\Xi=\Xi^\dagger$ on $S$. Away from $S^{\trun}_K$, the smooth 1-form $\delta\kappa\in \Omega^1(\hat{S};\C)$ is defined by the formulae
\begin{align}\label{FSE.14}
\sigma_k^*(\delta\kappa)&=\im (\overline{\pt g_k(t)}\cdot W), 0\leq k\leq n, && \text{ on }\HH^-_k,\\
\hat{\iota}_0^*(\delta\kappa)&=\im (i\epsilon_0 e^{-i\beta_0}dt \cdot W), &&\text{ on } \R^-_t\times \R_s,\nonumber
\\
	\hat{\iota}_k^*(\delta\kappa)&=\im (i\epsilon_k e^{-i\beta_k}dt \cdot W), 1\leq k\leq n, &&\text{ on } [K,+\infty)_s\times \R_s.\nonumber
\end{align}

\begin{remark}\label{FSR.5} By \eqref{FSE.3} and \eqref{FSE.14}, the phase pair $(\Xi, \delta\kappa)$ agrees with the canonical one associated to $\fa_0$ on the incoming plane-like end $\hat{\iota}_0$ and with the one associated to $\fa_k$ on $\hat{\iota}_k([K,+\infty]_t\times \R_s)$ up to a translation by the constant $c_{\Xi^\dagger}^k\in \C$. In particular, the phase function $\Xi$ takes the form (see Figure \ref{Pic9} and compare \eqref{FSE.3})
	\begin{align*}
	\Xi\circ \hat{\iota}_0(t,s)&=-i\epsilon_0e^{i\beta_0} \cdot t+\gamma_0(s), &&\text{ on }\R_t^-\times \R_s,\\
	\Xi\circ\iota_k(t,s)&=-i\epsilon_ke^{i\beta_k}\cdot t+\gamma_k(s)+c^k_{\Xi^\dagger},\ 1\leq k\leq n,
	&&\text{ on }[K,+\infty)_t\times \R_s. \qedhere
	\end{align*}
\end{remark}

\begin{definition}\label{FSD.5} Given an admissible Floer datum $\fa_0$ for the pair $(\Lambda_0,\Lambda_n)$ and $\fa_k$ for each $(\Lambda_{k-1}, \Lambda_k), 1\leq k\leq n$, a cobordism datum on $(S,\phi)$ is a quadruple $\fb_\phi=(K, \Xi, \delta \kappa, \delta H)$ where 
	\begin{itemize}
	\item $K>0$ and the phase pair $(\Xi,\delta\kappa)$ is constructed using an embedding $\Xi^\dagger\in \Emb_K^S$ such that for some $\epsilon_S>0$,
	\begin{equation}\label{FSE.6}
	-\det D\Xi-|\delta \kappa^{0,1}|^2>\epsilon_S \text{ on } \hat{S};
	\end{equation}
	\item the perturbation 1-form $\delta H\in \Omega^1(\hat{S}, \SH)$ is supported on $S$ and for any $0\leq k\leq n$, $\iota_k^*\delta H=\delta H_k$ on the strip-like end.
	\end{itemize}
 By  Remark \ref{FSR.5} and \eqref{FSE.5} , the pointwise estimate \eqref{FSE.6} holds for some $\epsilon_S>0$ away from the truncated surface $S^{\trun}_K$, and \eqref{FSE.6} will hold on $\hat{S}$ for a possibly smaller $\epsilon_S$ if we extend $\delta\kappa$ by the zero form over $S^{\trun}_K$. Thus the space of cobordism data is nonempty. 
\end{definition}
\begin{remark} One should think of $\fb_\phi$ as associated to the singular flat surface $(S,\phi)$ not just to the smooth surface $S$. A  cobordism datum encodes the information to define the Floer equation as well as the way to write the energy estimate.
\end{remark}

The Floer equation \eqref{E1.12} associated to any cobordism datum $\fb_\phi$ recovers the $\alpha_k$-instanton equation on each plane-like end $\hat{\iota}_k$ and takes the standard form 
\[
\pt P+J(\ps P+\nabla \re(e^{-i(\theta_k-\pi)}W)\circ P)=0,\ 0\leq k\leq n
\]
on each planar end $\HH^-_k$. Now choose an $\alpha_k$-soliton, one for each marked point $\zeta_k$ of $S$,
\begin{align*}
p_0(s)&\in \FC(\Lambda_0,\Lambda_n;\fa_0),\\
p_k(s)&\in \FC(\Lambda_{k-1},\Lambda_k;\fa_k),\ 1\leq k\leq n,
\end{align*}

The boundary condition for a solution of \eqref{E1.12} is specified by a model map $P_{\model}: \hat{S}\to M$ such that $P_{\model}\circ\iota_k(t,s)=p_k(s)$ for all $t\in \R^\pm_t$ (i.e., it is time-independent) and such that for all $0\leq k\leq n$,
\[
P_{\model}\circ \sigma_k(t ,s)\to q_k \text{ as }s\to -\infty.
\]
This decay is understood as in \eqref{E1.7} and is required to be exponential and uniform in $t\in \R_t$. The moduli space 
\begin{equation}\label{FSE.31}
\M_S(\{p_k\}_{k=0}^n;\fb_\phi)
\end{equation}
 then consists of all solutions of \eqref{E1.12} on $\hat{S}$ which have finite $L^2_k$-distance to $P_{\model}$ for all $k\geq 2$.  The excision argument in the proof of Proposition \ref{LG.P.2} reduces the index computation of $\M_S$ to the case of Lagrangian boundary condition. Then \cite[Proposition 11.13]{S08} implies that 
 \begin{equation}\label{FSE.33}
\text{the expected dimension of } \M_S(\{p_k\}_{k=0}^n;\fb_\phi)=\gr(p_0)-\sum_{k=1}^n\gr(p_k). 
 \end{equation}

To define the $A_\infty$-category $\sE(\Th)$, we have to consider the family version of \eqref{FSE.31} when $S$ is allowed to vary in the universal family $\Sch^{n+1}\to\NR^{n+1}$. From now on let $\phi$ be a smooth section of $\V^{n+1}\to \NR^{n+1}$ which is $\epsilon$-close to $\CT$; then $\phi$ determines a smooth fiber bundle 
\[
\hat{\Sch}^{n+1}\to \NR^{n+1},
\] 
whose fiber at any $r\in \NR^{n+1}$ is the completion $\hat{S}_r$ of $(S_r, \phi_r)$.

\begin{definition}\label{FSD.6} Given any admissible Floer datum $\fa_0$ for the pair $(\Lambda_0, \Lambda_n)$ and $\fa_k$ for each  $(\Lambda_{k-1}, \Lambda_k), 1\leq k\leq n$, a cobordism datum for the family $\Sch^{n+1}\to \NR^{n+1}$ labeled by $\Th$ and equipped with the section $\phi$ is a quadruple $\fb_\Th=(K_{\Th}, \Xi_{\Th}, \delta\kappa_{\Th}, \delta H_{\Th})$ consisting of 
	\begin{itemize}
\item a constant $K_\Th>0$ and a smooth function $\Xi_{\Th}\in C^\infty(\hat{\Sch}^{n+1};\C)$;
\item relative 1-forms $\delta\kappa_{\Th}\in \Omega^1_{\hat{\Sch}^{n+1}/\NR^{n+1}}(\hat{\Sch}^{n+1};\C)$ and $\delta H_{\Th}\in \Omega^1_{\hat{\Sch}^{n+1}/\NR^{n+1}}(\hat{\Sch}^{n+1};\SH)$
	\end{itemize}
such that $\fb_\Th$ restricts to a cobordism data $\fb_{\phi_r}=(K_\Th, \Xi_r, \delta\kappa_r, \delta H_r )$ on each fiber $(S_r,\phi_r),\ r\in \NR^{n+1}$, and \eqref{FSE.6} holds for a uniform constant $\epsilon_\Th>0$ independent of $r$. 
	\end{definition}

The cobordism datum $\fb_\Th$ is called \textit{admissible} if for every $(n+1)$-tuple $\{p_k\}_{k=0}^n$, the parametrized moduli space 
\begin{equation}\label{FSE.8}
\M_{\NR^{n+1}}(\{p_k\}_{k=0}^n;\fb_{\Th})\colonequals \coprod_{r\in \NR^{n+1}} \M_{S_r}(\{p_k\}_{k=0}^n; \fb_{\phi_r}),
\end{equation}
is cut out transversely. This means that the extended linearized operator is surjective at any point of \eqref{FSE.8}; see for instance \cite[Section (9h)]{S03}.

\subsection{Choosing cobordism data consistently} From now on fix an $(\epsilon, \CT)$-consistent section $\phi$  for the bundle $\overline{\V}^{n+1}\to \overline{\NR}^{n+1}$ and a set of strip-like ends $(\iota_T)$  adapted to $\phi=(\phi_T)$ in the sense of Lemma \ref{QD.L.14}. 

The construction of Section \ref{SecFS.2} can be carried out also for any subset $A\subset \Th$ with $|A|\geq 3$. Denote by $\CT_A\subset \CT$ the unique metric ribbon subtree with $\partial \CT_A\cong A$ under the map \eqref{FSE.7}. A component of $\phi=(\phi_T)$ and $(\iota_T)$ then defines a smooth section $\psi_A$ of 
\[
\V^{|A|}\to \NR^{|A|},
\]
which is $\epsilon$-close to $\CT_A$, together with a set of strip-like ends $(\iota_{A,k})_{k=0}^{|A|-1}$ adapted to $\psi_A$. Choose any admissible cobordism datum $\fb_A=(K_A, \Xi_A,\delta\kappa_A,\delta H_A)$ for the section $\psi_A$ and the family $\Sch^{|A|}\to \NR^{|A|}$ labeled by $A$, then one can form a parametrized moduli space similar to \eqref{FSE.8}. 

\medskip

For any $n$-leafed tree $T$ and $v\in \Ve(T)$, write $\fb_v\colonequals \fb_{\partial\CT_v}$ for the cobordism datum on the family $\Sch_v\to \NR_v$, which is labeled by $\partial \CT_v\subset \Th$, and $K_v\colonequals K_{\partial\CT_v}>0$ for the first component of $\fb_v$. For any $T<T'$, the gluing construction in Section \ref{SecQD.8} applied to $(\fb_v)_{v\in \Ve(T)}$ defines a cobordism datum on the image 
\[
\gamma^{T,T'}\bigg(\prod_{e\in \Ed^{\inte}(T)\setminus \Ed^{\inte}(T')}(-e^{-K_{v^+(e)}}, 0)\times \NR^T\bigg)\subset \NR^{T'}.
\]

We shall only sketch this construction for the codimension-1 stratum and when $T'=T_*$; the general case follows from a similar scheme. Let $e\in \Ed^{\inte}(T)$ denote the unique interior edge of $T$. Suppose that $e$ is adjacent to $\fo_j, \fo_k$ in $\R^2$. Write $v^\pm=v^\pm(e)$, $f^\pm=f^\pm(e)=(v^\pm, e)$ and $K_\pm=K_{v^\pm}$. For any $r_\pm\in \NR_{v^\pm}$ and $\rho_e\in (-e^{-K_+},0)$, there is a composition map 
\begin{equation}\label{FSE.9}
\Emb_{K_+}^{S_{r_+}}\times \Emb_{K_-}^{S_{r_-}}\to \Emb_{\max\{K_+, K_-\}}^{S_{r_*}}.
\end{equation}
with $r_*=\gamma^{T,T'}(\rho_e, (r_\pm))$. Recall that $S_{r_*}$ is obtained by gluing 
\[
S_{r_+}\setminus \iota_{f^+}\big((l_e,+\infty)_t\times [0,\w_{jk}]_s\big)\text{ with } S_{r_-}\setminus \iota_{f^-}\big((-\infty, 0)_t\times [0,\w_{jk}]_s\big), \ l_e\colonequals-\ln(-\rho_e), 
\]
where $\iota_{f^+}(l_e, s)$ is identified with $\iota_{f^-}(0,s), s\in [0,\w_{jk}]_s$. Then for any $\Xi^\dagger_\pm\in \Emb_{K_\pm}^{S_{r_\pm}}$, the image of $(\Xi^\dagger_+,\Xi^\dagger_-)$ under \eqref{FSE.9} is defined by the formula
\begin{equation}\label{FSE.21}
\Xi^\dagger (z)=\left\{\begin{array}{ll}
\Xi^\dagger_+(z) & z\in S_{r_+}\setminus \iota_{f^+}\big((l_e,+\infty)_t\times [0,\w_{jk}]_s\big), \\
\Xi^\dagger_-(z)+c_{\Xi^{\dagger}_+}^{f^+}+(-i\epsilon_{jk}e^{i\beta_{jk}})\cdot l_e &z\in S_{r_-}\setminus \iota_{f^-}\big((-\infty, 0)_t\times [0,\w_{jk}]_s\big).
\end{array}
\right.
\end{equation}

\begin{remark} Roughly speaking, \eqref{FSE.21} is saying that the phase function is normalized such that on the unique incoming end it agrees with the canonical phase function associated to $\fa_0$ (see Example \ref{EX2.2}), and $\Xi^\dagger_-$ is translated to match the function $\Xi^\dagger_+$ on an outgoing end of $S_{r_+}$.

	The composition map \eqref{FSE.9} is defined only when $l_e\geq  K_+$, or equivalently, $\rho_e\geq -e^{-K_+}$, and it allows us to obtain a phase function on $\hat{S}_{r_*}$ from the ones on  $\hat{S}_{r_\pm}$. One should compare \eqref{FSE.21} with the composition map \eqref{E2.22} in the case of continuation maps.

	 It is slightly easier to think of this gluing construction in terms of $\kappa_{\Xi}$, the induced 1-form of $\Xi$ defined by \eqref{E1.17}. The primary reason to work with $\Xi$ is that one has to verify that the space of the cobordism data is weakly contractible, which requires that the connection $d+\kappa_{\Xi}$ be negatively curved all the time. Using phase functions proves extremely useful in this regard --- it allows us to prove Lemma \ref{FSL.4} using the contractibility \cite{Smale59} of $\Diff(D^2,\partial D^2)$. 
\end{remark}

The correction 1-form $\delta\kappa$ (and resp. $\delta H$) can be glued in a straightforward manner, as it is already set to equal on 
\begin{align*}
&\iota_{f^+}([K_+, l_e]_t\times [0,\w_{jk}]_s)\subset S_{r_+} \text{ and }\\
&\iota_{f^-}([K_+-l_e, 0]_t\times [0,\w_{jk}]_s)\subset S_{r_-}.
\end{align*}

For any $A\subset \Th$ with $|A|\geq 3$, choose an admissible cobordism data $\fb_A$. Then this family $(\fb_A)_{A\subset \Th}$ is called \textit{consistent} if for all $A$, $\fb_A$ agrees with the cobordism datum obtained by this gluing construction in a neighborhood of the boundary of $\overline{\NR}^{|A|}$. For instance, when $A=\Th$, this means that for any $d$-leafed tree, $\fb_\Th$ is obtained by gluing $(\fb_v)_{v\in \Ve(T)}$ on the image 
\[
\gamma^{T,T_*}\bigg((-\delta,0)^{\Ed^{\inte(T)}}\times \NR^T\bigg)\subset \NR^{n+1}
\]
for some $0<\delta\ll 1$. 

\begin{lemma}\label{FSL.8} A consistent family $(\fb_A)_{A\subset \Th}$ of cobordism data exists. 
\end{lemma}
\begin{proof}[Proof of Lemma \ref{FSL.8}] The proof follows the same line of arguments as in  Lemma \ref{QD.L.14}. The induction starts with the trivial case when $|A|=3$, where no consistency conditions are required. For the induction step, we focus on the last step when $A=\Th$, in which case the cobordism data obtained by the gluing construction near different $T$-strata of $\overline{\NR}^{n+1}$ all match up and therefore determines $\fb_\Th$ in a neighborhood of the boundary of $\overline{\NR}^{n+1}$. The remaining question is to extend $\fb_\Th$ over the whole family $\Sch^{n+1}\to \NR^{n+1}$; in particular, one has to find a phase function 
	\[
	\Xi_\Th: \hat{\Sch}^{n+1}\to \C
	\]
	extending a given one defined near the boundary of $\overline{\NR}^{n+1}$. By construction, each $\Xi_r: \hat{S}_r\to \C$ is determined by an embedding $\Xi_r^\dagger:S_r\to \C$ satisfying the conditions \eqref{FSE.3}\eqref{FSE.4}\eqref{FSE.5}.

	Topologically, $(\overline{\NR}^{n+1},\partial \overline{\NR}^{n+1})$ is homeomorphic to $(D^{n-2}, S^{n-3})$, and $\NR^{n+1}$ is identified with the interior of $D^{n-2}$. We are in the same situation as in Lemma \ref{FSL.4}, except that the conditions \eqref{FSE.4} and \eqref{FSE.5} are stated for a family of $(n+1)$-pointed disks instead of a single one. However, the space $\Emb_{K}^{S,*}$ introduced in the proof of Lemma \ref{FSL.4} depends only on the underlying smooth surface of $S$. Following the proof of Lemma \ref{L2.5}, one may first construct an extension in $\Emb_{K_\Th}^{S,*}$ for some $K_\Th\gg1$ and then apply a family of self-diffeomorphisms of the domain to achieve \eqref{FSE.4} and \eqref{FSE.5} for each fiber $S_r$. 
	
	The extension of $\delta \kappa_\Th$ and $\delta H_\Th$ is much more straightforward. Finally, $\fb_\Th$ is admissible if the extension of $\delta H_\Th$ is chosen generically. 
\end{proof}

\begin{remark} As we shall never consider the Floer cohomology of a thimble with itself, our consistency condition is slightly more rigid (and more convenient) than the one in \cite[Lemma 9.5]{S08}.
\end{remark}

\subsection{The energy estimate and compactness} Given a consistent family of cobordism data $(\fb_A)_{A\subset \Th}$ associated to $\Th$,  we have to verify that the parametrized moduli space \eqref{FSE.8} is compact when its expected dimension is zero. This is guaranteed by the energy estimate from Lemma \ref{L1.12} combined with the argument in Section \ref{Sec2}. For future reference, we record this estimate as follows: 

\begin{lemma}[The Energy Estimate III]\label{L3.5} For any $(n+1)$-pointed disk $(S,\phi)$ equipped with a compatible differential and any cobordism datum $\fb_\phi=(K, \Xi, \delta\kappa, \delta H)$, consider the truncated surface 
	\[
	\hat{S}_{K}\colonequals \hat{S}\setminus \bigg(\hat{\iota}_0(\{t<-K\})\coprod_{k=1}^n \hat{\iota}_k(\{t>K\})\bigg).
	\]
	For any solution $P\in \M_S(\{p_k\}_{k=0}^n; \fb_S)$, let 
	\[
	\tilde{p}_0(s)=P\circ\hat{\iota}_0(-K,s) \text{ and }\tilde{p}_k(s)=P\circ \hat{\iota}_k(K,s), s\in \R_s, 1\leq k\leq d,
	\]
denote the restrictions of $P$ on certain time slices. Then	the energy of $P$ on the truncation $\hat{S}_{K}$ can be bounded as follows
	\begin{align}\label{FSE.16}
	\CA_{W,\fa_0}(p_0)&-\sum_{k=1}^n \CA_{W,\fa_k}(p_k)+\sum_{k=0}^n \im(W(q_k)\cdot \overline{g_k(t_k^++K)-g_k(-K)})\\
	&\geq  (\CA_{W,\fa_0}(p_0)-	\CA_{W,\fa_0}(\tilde{p}_0))+\sum_{k=1}^n (\CA_{W,\fa_k}(\tilde{p}_k)-\CA_{W,\fa_k}(p_k))+\epsilon_S^*\int_{\hat{S}_{K}} u,\nonumber
	\end{align}
	where $\epsilon_S^*$ is independent of $P$ and $t_k^+=t^+_k(\phi), 0\leq k\leq n$ is the length of the $k$-th boundary component $C_k$ as defined in \eqref{QD.E.18}. Each term on the right hand side of \eqref{FSE.16} is non-negative and therefore is bounded above by the left hand side of \eqref{FSE.16}.
\end{lemma}

\begin{proof} Apply Lemma \ref{L1.12} to the surface $\hat{S}_{K}$ to derive that 
		\begin{align}\label{FSE.17}
	\CA_{W,\fa_0}(\tilde{p}_0)-\sum_{k=1}^n \CA_{W,\fa_k}(\tilde{p}_k)+\sum_{k=0}^n \im(W(q_k)\cdot \overline{g_k(t_k^++K)-g_k(-K)})\geq\epsilon_S^*\int_{\hat{S}_{K}} u,
	\end{align}
	which is \eqref{FSE.16}.  On the plane-like end $\hat{\iota}_0((-\infty, -K]\times \R_s)$, the solution $P$ is a downward gradient flowline of $\CA_{W,\fa_0}$, and so 
	\[
	\CA_{W,\fa_0}(p_0)\geq \CA_{W,\fa_0}(\tilde{p}_0).
	\]
	Similarly, $\CA_{W,\fa_k}(\tilde{p}_k)\geq	\CA_{W,\fa_k}(p_k), 1\leq k\leq d$.
\end{proof}

By \eqref{FSE.5}, $|\pt g_k(t)|<2$. Lemma \ref{L3.5} implies that the total drop of the action functional is bounded below by 
\begin{equation}\label{FSE.27}
\CA_{W,\fa_0}(p_0)-\sum_{k=1}^n \CA_{W,\fa_k}(p_k)\geq -2\sum_{k=0}^n (2K+t^+_k)|W(q_k)|,
\end{equation}
if the moduli space  $\M_S(\{p_k\}_{k=0}^n; \fb_S)$ is non-empty. The next lemma says that this lower bound can be made uniform for the parametrized moduli spaces.

\begin{lemma}\label{FS.L.12} Given any consistent family of cobordism data $(\fb_A)_{A\subset\Th}$, there exists a constant $C_{\Th}>0$ depending on the phase functions such that if $\M_{\NR^{n+1}}(\{p_k\}_{k=0}^n; \fb_\Th)$ is non-empty, then 
	\begin{equation}\label{FSE.32}
\CA_{W,\fa_0}(p_0)-\sum_{k=1}^n \CA_{W,\fa_k}(p_k)\geq -C_{\Th}.
	\end{equation}
\end{lemma}
\begin{proof} It suffices to show that for any $r\in \NR^{n+1}$, if the moduli space $\M_{S_r}(\{p_k\}_{k=0}^n; \fb_{\phi_r})$ is non-empty, then \eqref{FSE.32} holds for some uniform constant of $C_{\Th}>0$. If $r$ lies in a fixed compact subset of $\NR^{n+1}$, then $\eqref{FSE.32}$ follows from \eqref{FSE.27}, since $K=K_{\Th}$ is fixed and $t^+_k(S_r,\phi_r), 0\leq k\leq n$ are bounded uniformly. Fix some $0<\delta\ll 1$. In general, for any $A\subset \Th$ with $|A|\geq 3$, one can choose inductively an open subset $U_A\subset \NR^{|A|}$ with compact closure to obtain an open cover $(V_T)$ of $\NR^{n+1}$ as follows: let 
	\[
	V_T\colonequals  \gamma^{T,T^*}\bigg((-\delta, 0)^{\Ed^{\inte(T)}}\times U_T\bigg)\subset  \NR^{n+1},\ U_T\colonequals \prod_{v\in \Ve(T)}U_{\partial \CT_v}\subset \NR^{T},
	\]
	for all $T<T^*$ and $V_{T^*}\colonequals U_{\Th}$. If $r\in V^T$ for some $T<T^*$, then we write $r=\gamma^{T,T^*}((\rho_e)_{e\in \Ed^{\inte(T)}},(r_v)_{v\in \Ve(T)})$, and \eqref{FSE.32} is proved using the estimate \eqref{FSE.17} on each $S_{r_v}, v\in \Ve(T)$ and by noting that over each gluing region the solution $P$ is a downward gradient flowline for an action functional, so when we add up \eqref{FSE.17} for all $S_{r_v}, v\in \Ve(T)$, each interior edge of $T$ contributes a negative term on the left hand side, which allows us to prove \eqref{FSE.32} for this $r$. The figure below illustrates the case when $n=3$. 
\end{proof}

 \begin{figure}[H]
	\centering
	\begin{overpic}[scale=.15]{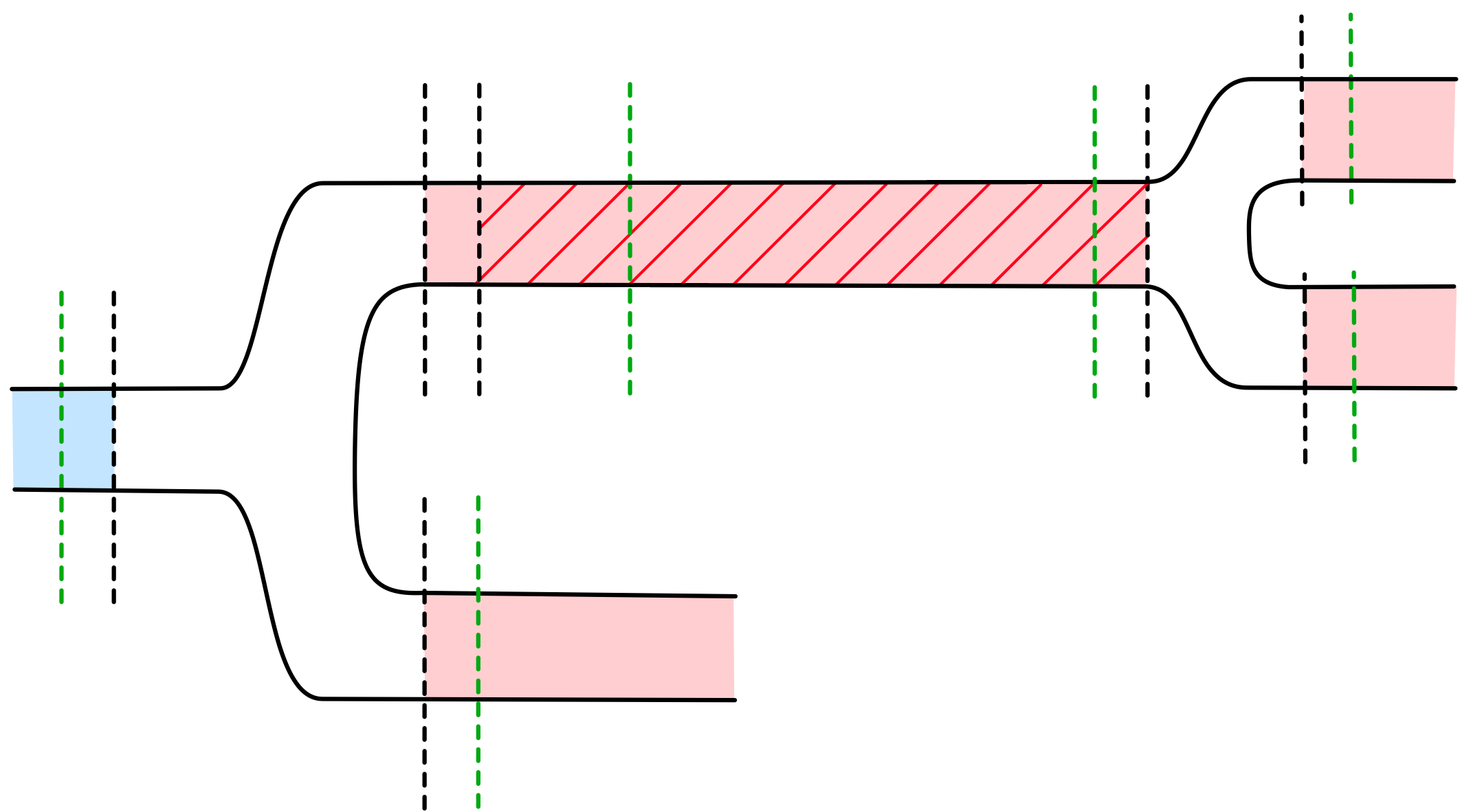}
		\put(-10,10){\small $\tilde{p}_0(s)=\hat{\iota}_0(-K,s)$}
				\put(35,3){\small $\tilde{p}_1(s)=\hat{\iota}_1(K,s)$}
				\put(95,25){\small $\tilde{p}_2(s)=\hat{\iota}_2(K,s)$}
								\put(95,52){\small $\tilde{p}_3(s)=\hat{\iota}_3(K,s)$}
								\put(35,52){\small $\tilde{p}_e^+(s)=\hat{\iota}_{f^+(e)}(-\ln(\delta), s)$}
												\put(55,25){\small $\tilde{p}_e^-(s)=\hat{\iota}_{f^-(e)}(-K,s)$}
	\end{overpic}	
	\caption{The gluing region is shaded by red. The value of the action functional at $\tilde{p}^+_e$ is always greater than that at $\tilde{p}^-_e$.}
	\label{Pic26}
\end{figure}
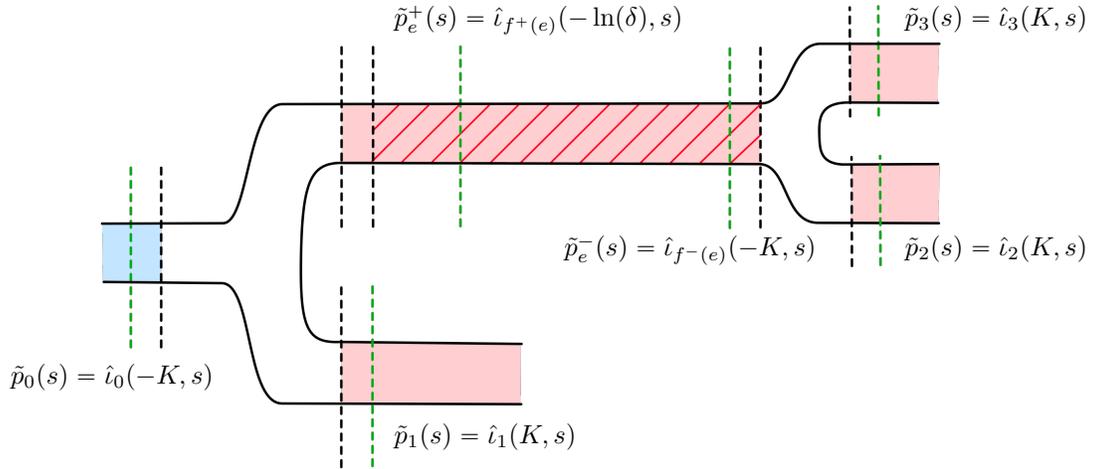

\begin{remark} At the first sight, the estimate \eqref{FSE.32} does not seems very interesting as the $(n+1)$-tuple $\{p_k\}_{k=0}^n$ has only finitely many possibilities on the left hand side. The upshot is that this constant $C_{\Th}$ does not even depend on $(M,W)$ but only on the phase functions on $\Sch^{n+1}\to \NR^{n+1}$. The proof of Lemma \ref{FS.L.12} will be used again when we construct the geometric filtration on the bimodule ${}_{\sA}\Delta_{\sB}$ in Section \ref{SecGF.4}; cf. Lemma \ref{GF.L.17}. 
\end{remark}

\subsection{Construction} Given an $(\epsilon, \CT)$-consistent section $\phi=(\phi_T)$ of the bundle $\overline{\V}^{n+1}\to \overline{\NR}^{n+1}$, a set of strip-like ends $(\iota_T)$ adapted to $\phi$ and a consistent family $(\fb_A)$ of cobordism data, the directed $A_\infty$-category $\sE=\sE(\Th; \CT, \phi, (\fb_A))$ is constructed as follows. The morphism spaces are defined by \eqref{FSE.22}. For any subset $A=(\Lambda_{j_0},\cdots, \Lambda_{j_d})\subset \Th$, $d\geq 2$,  the $A_\infty$-operation 
\begin{equation}\label{FSE.10}
\mu_{\sE}^d: \hom_{\sE}(\Lambda_{j_{d-1}}, \Lambda_{j_d})\otimes \cdots \otimes \hom_{\sE}(\Lambda_{j_0}, \Lambda_{j_1})\to \hom_{\sE} (\Lambda_{j_0}, \Lambda_{j_d})[2-d]
\end{equation}
is defined by counting the $0$-dimensional parametrized moduli space over the family $\hat{S}^{|A|}\to \NR^{|A|}$, with the Floer equation \eqref{E1.12} determined by the cobordism datum $\fb_A$. More concretely, when $A=\Th$, \eqref{FSE.10} is defined by the formula (see \eqref{FSE.8})
\begin{equation}\label{FSE.15}
\bigotimes_{k=1}^n p_k\mapsto \sum_{p_0: \dim \M=0} \#\M_{\NR^{n+1}}(\{p_k\}_{k=0}^n; \fb_\Th)\cdot p_0.
\end{equation}
The $A_\infty$-associativity relations \eqref{AF.E.1} are then verified by the standard argument; cf. \cite[Section 9]{S08}.
 
\subsection{The cohomological category}\label{SecFS.6} Our next step is to understand the invariance of $\sE=\sE(\Th; \CT, \phi, (\fb_A))$. When $n=2$, $\NR^3$ is a single point, and the quadratic differential $\phi$ is uniquely determined by the metric ribbon tree $\CT$ by Example \ref{QD.EX.8} . In this case, the only information carried by $\sE$ is the product map on the cohomological category:
\[
[\mu^2_{\sE}]: \HFF_\natural^*(\Lambda_1,\Lambda_2;\fa_{12})\otimes \HFF_\natural^*(\Lambda_0,\Lambda_1;\fa_{01})\to \HFF_\natural^*(\Lambda_0, \Lambda_2;\fa_{02}).
\]
The continuation method in Section \ref{Subsec:Continuation} can be generalized further, using Lemma \ref{FSL.4} instead of Lemma \ref{L2.5}, to show that $[\mu^2_{\sE}]$ is independent of the cobordism data $\fb_\Th$ and is natural with respect to the continuation map \eqref{ContinuationMaps}. Note also that the completion $\hat{S}$ of a (2+1)-pointed disk $S$ is independent of the metric ribbon tree $\CT$ (though the metric may change on a compact subset); the different choices of $\CT$ only affect the inclusion map
\[
\sigma_k: \HH^-_k\to \hat{S},\ 0\leq k\leq 2
\]
by a translation in the $s$-coordinate. In this case, $\CT$ is only used to determined the region of $\hat{S}$ on which the phase pair $(\Xi, \delta\kappa)$ takes the standard form \eqref{FSE.13}\eqref{FSE.14}. We conclude that the product map
\begin{equation}\label{FSE.12}
m_*=[\mu^2_{\sE}]: \HFF_\natural^*(\Lambda_1,\Lambda_2)\otimes \HFF_\natural^*(\Lambda_0,\Lambda_1)\to \HFF_\natural^*(\Lambda_0, \Lambda_2),
\end{equation}
is canonically defined for any admissible triple $(\Lambda_0, \Lambda_1,\Lambda_2)$. This shows that for any admissible set $\Th$ of thimbles the cohomological category of $\sE(\Th; \CT, \phi, (\fb_A)_{A\subset \Th})$ is independent of $\CT,\ \phi$ and $(\fb_A)$ and that \eqref{FSE.12} is associative. 

\subsection{Quasi-units}\label{SecFS.7} For any $q\in \Crit(W)$ and $\theta_1<\theta_0<\theta_1+2\pi$, let $e_q$ denote the canonical generator of $\HFF_\natural^*(\Lambda_{q,\theta_0},\Lambda_{q,\theta_1})$ obtained in Lemma \ref{lemma:CanonicalGenerator}. One major failure of our construction so far lies in the fact that the strict units of $\sE(\Th; \CT, \phi, (\fb_A))$ are defined formally, as we were unable to define the Floer cohomology of a thimble with itself. However, these generators $e_q$, which are called \textit{quasi-units}, will partially remedy this failure, as they behave pretty much like isomorphisms in the cohomological category of $\sE$, except that they do not admit an inverse. This will be fixed in Section \ref{SecFS.9} using categorical localization to establish the invariance of $\sE=\sE(\Th)$.

\begin{proposition}\label{prop:Quasi-Units} Given any admissible set $(\Lambda_{q_0,\theta_0},\Lambda_{q_1,\theta_1},\Lambda_{q_2,\theta_2})$ with $q_1=q_2=q\in \Crit(W)$, suppose that the ray $l_{q_0,\theta_0}$ interests with $l_{q,\theta_j}$ at $y_j\in \C, j=1,2$ and that the triangle formed by $(y_0, y_1, W(q))$ contains no critical values of $W$ in the interior, then the product map
		\[
	m_*(e_{q}, \cdot): \HFF_\natural^*(\Lambda_{q_0,\theta_0},\Lambda_{q,\theta_1})\to \HFF_\natural^*(\Lambda_{q_0,\theta_0},\Lambda_{q,\theta_2}). 
	\]
	is an isomorphism. Similarly, if $q_0=q_1=q$, $l_{q_2, \theta_2}$ intersects with $l_{q, \theta_j}$ at $y_j'\in \C, j=0,1$ and the triangle formed by $(y_0', y_1', W(q))$ contains no critical values of $W$ in the interior, then the map 
	\[
m_*(\cdot, e_q): \HFF_\natural^*(\Lambda_{q,\theta_1},\Lambda_{q_2,\theta_2})\to \HFF_\natural^*(\Lambda_{q,\theta_0},\Lambda_{q_2,\theta_2}). 
\]
	is an isomorphism. 
\end{proposition} 

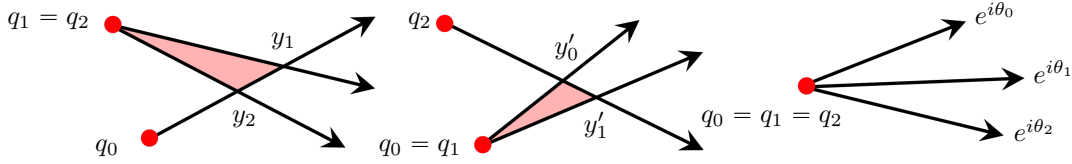
\begin{figure}[H]
	\centering
	\begin{overpic}[scale=.15]{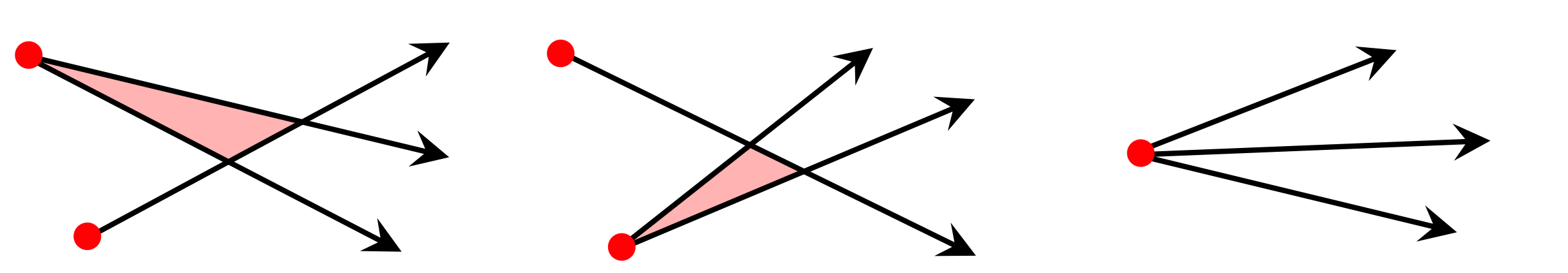}
		\put(-9,14){\small $q_1=q_2$}		
		\put(0,1){\small$q_0$}
		\put(32,14){\small$q_2$}	
		\put(29,1){\small $q_0=q_1$}	
		\put(62,4){\small$q_0=q_1=q_2$}		
		\put(90,14){\small$e^{i\theta_0}$}
		\put(96,8){\small$e^{i\theta_1}$}
		\put(94,2){\small$e^{i\theta_2}$}
		\put(18,12){\small $y_1$}
		\put(14,4){\small$y_2$}
		\put(50,3){\small $y_1'$}
		\put(47,11){\small $y_0'$}
	\end{overpic}	
	\caption{Quasi-units.}
	\label{Pic10}
\end{figure}

\begin{remark}The heuristic behind Proposition \ref{prop:Quasi-Units} is straightforward: as we vary the angle $\theta_1$ continuously to deform the thimble $\Lambda_{q,\theta_1}$ into $\Lambda_{q,\theta_2}$ (in the first case), the Floer cohomology is unaffected if the triangle of interest contains no other critical values of $W$; otherwise, the angle $\theta_1$ may come across a value for which the condition \eqref{E1.5} is violated and the Floer cohomology is undefined. 
\end{remark}

Proposition \ref{prop:Quasi-Units} will be proved in Section \ref{SecVT.6} using a vertical gluing theorem Here we content ourselves with a special case.

\begin{lemma}\label{L3.7} For any admissible set $\Th=(\Lambda_{q_0,\theta_0},\Lambda_{q_1,\theta_1},\Lambda_{q_2,\theta_2})$ of thimbles with $q_k=q\in \Crit(W),\ 0\leq k\leq 2$, we have $m_*(e_q, e_q)=e_q$. 
\end{lemma}
\begin{proof}[Proof of Lemma] Let $S$ be a (2+1)-pointed disk inducing $m_*$ and $\CT$ a metric ribbon tree with $|\partial \CT|=3$. Choose a collection of admissible Floer data $\fa_{jk}$, one for each pair $(\Lambda_j, \Lambda_k), 0\leq j<k\leq 2$. We take the perturbation 1-forms $\delta H$ to be zero in $\fa_{jk}$ and also in the cobordism datum $\fb_\Th$. Recall from Lemma \ref{lemma:CanonicalGenerator} that in this case any $\alpha_{jk}$-soliton is a constant map at $q$ with $\gr=0$. Let  $P\in \M_S(p_0, p_1,p_2; \fb)$ be any solution that contribute to the chain-level map \eqref{FSE.15} that defines $m_*$; then each $p_k, 0\leq k\leq 2$ is constant at $q$. The energy estimate from Lemma \ref{L3.5} implies that $P$ is also constant at $q$. The easiest to see this is to set $P\equiv q$ to start, then the equality in \eqref{FSE.16} must hold (one has to go back to the proof of Lemma \ref{L1.12} and check that the equality indeed holds at each step). This forces the left hand side of \eqref{FSE.16} to vanish.  Alternatively, one may simply repeat the proof of Lemma \ref{lemma:CanonicalGenerator}.
	
	Finally, we have to show that the linearized operator $\D$ at the constant map $P\equiv q$ is invertible, so the moduli space $\M_S$ is regular. As in the proof of Lemma \ref{lemma:CanonicalGenerator}, we reduce this problem to the case of Example \ref{EX1.2} and the same energy estimate then shows that $\D$ is injective. By \eqref{FSE.33}, the Fredholm index of $\D$ is $\gr(e_q)-\gr(e_q)-\gr(e_q)=0$, so $\D$ is invertible.
\end{proof}

\subsection{A remark on the pair-of-pants coproduct} We add a short digression to discuss a coproduct structure on $H(\sE)$. In Section \ref{SecFS.6}, the operations $\mu_{\sE}^d,\ d\geq 2$ are defined using phase functions $\Xi_r: \hat{S}_r\to \C$, whose mapping degree is equal to $-1$, but one may also consider phase functions of degree $k\in [-d,-1]$, which is modeled on a map
\[
\hat{S}_r\xrightarrow{\text{degree } (-1)} \C\xrightarrow{z\mapsto z^k} \C. 
\] 
Although the pointwise estimate \eqref{PointwiseEstimate} will be violated at the zero locus of $\det D\Xi$, we still have $-\det D\Xi-|\delta\kappa^{0,1}|^2\geq 0$ if  $\delta\kappa\equiv 0$ in a neighborhood of this locus; the proof of the compactness theorem then carries over with only minor differences. In particular, for any admissible $(\Lambda_0,\Lambda_1, \Lambda_2)$, one may construct a coproduct map:
\[
\blacktriangle: \HFF_\natural^*(\Lambda_0, \Lambda_2)\to \HFF_\natural^*(\Lambda_1, \Lambda_2)\otimes\HFF_\natural^*(\Lambda_0, \Lambda_1). 
\]
using a phase function of degree $(-2)$ on the completion of an $(1+2)$-pointed disk $S^\ddagger$ with $|\Sigma^-|=2$ and $|\Sigma^+|$=1; see Figure \ref{Pic11}. 

It is not our goal to incorporate this coproduct map in the present work, but we shall point out a notable difference of $\blacktriangle$ with $m_*$: Lemma \ref{L3.7} is not true for $\blacktriangle$. If $q_k=q,\ 0\leq k\leq 2$ and $\delta H\equiv 0$, the Floer equation \eqref{E1.12} can only have a constant solution $P\equiv q$ on $\hat{S}^\ddagger$, but this solution is not regular and the Fredholm index of the linearized operator $\D^\ddagger$ is $-\dim_\C M$. Indeed, using the same excision principle as in Proposition \ref{LG.P.2} and \cite[Proposition 11.13]{S08}, we compute that
\[
\Ind\D^\ddagger=n(1-|\Sigma^-|)+\gr(e_q)+\gr(e_q)-\gr(e_q)=-n.
\]
Thus the constant solution on $\hat{S}^\ddagger$ is not regular and will disappear under a generic perturbation. This explains why $m_*$ is the right operation to define the Fukaya-Seidel category of $(M,W)$. 
\begin{figure}[H]
	\centering
	\begin{overpic}[scale=.15]{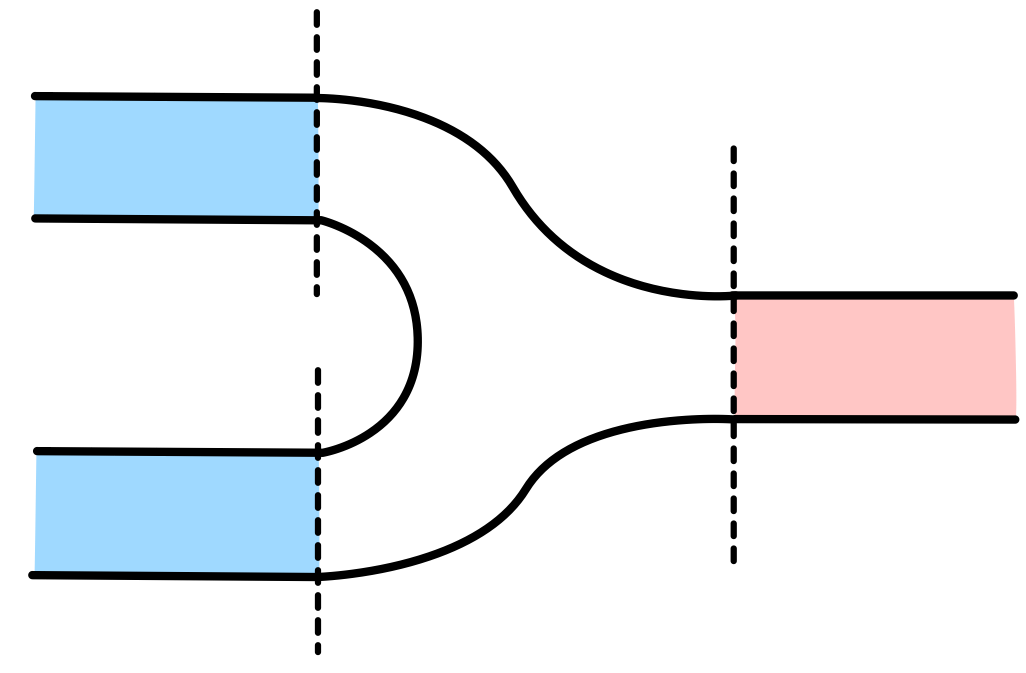}
		\put(50,29){$\hat{S}^\ddagger$}
	\end{overpic}	
	\caption{The coproduct map.}
	\label{Pic11}
\end{figure}

\subsection{Invariance}\label{SecFS.9} To prove that  $\sE(\Th; \CT, \phi, (\fb_A))$ is independent of the metric ribbon tree $\CT$ turns out to be a subtle problem; to simplify our exposition, we fix the choice of $\CT$ to define the $A_\infty$-category $\sE(\Th)$:
\begin{proposition}\label{FS.P.15} For any admissible set $\Th$ of thimbles with $n+1=|\Th|$, take $\CT$ to be the metric ribbon tree $\CT^{n+1}$ in Example \ref{QD.EX.8}. Then the directed $A_\infty$-category
	\[
	\sE(\Th)\colonequals \sE(\Th; \CT^{n+1}, \phi, (\fb_A))
	\]
is independent of the choice of $\phi$ and $(\fb_A)$ up to canonical quasi-isomorphisms.
\end{proposition}

The proof below is inspired by \cite{GPS20}\cite[Appendix A]{AS19}. 
\begin{proof} Let $\Gamma$ denote the space of admissible sets $\Th'=(\Lambda'_0,\cdots,\Lambda'_n), \Lambda_j'=\Lambda_{q_j, \theta_j'}, 0\leq j\leq n$. Define a partial order $\sto$ on $\Gamma$ as follows: $\Th'\sto \Th''$ if 
		\[
	\theta_n''<\theta_n'<\cdots<\theta_1''<\theta_1'<\theta_0''<\theta_0',
	\]
	and if the interval $[\theta_j'',\theta_j']$ consists of admissible angles for all $j$ (in the sense of Example \ref{EX3.3}). This means that the ray $l_{q_j, \theta_j'}$ can be continuously deformed into $l_{q_j, \theta_j''}$ without touching any critical values of $W$; see Proposition \ref{prop:Quasi-Units}.
	
	\medskip
	
	Given $\Th\sto \Th'$, choose quadratic differentials $\phi'$ and cobordism data $(\fb_A')_{A\subset \Th'}$ to define the $A_\infty$-category
	 \[
	 \sE'=\sE'(\Th'; \CT^{n+1}, \phi', (\fb_A')).
	 \]
	 
	 The idea is to construct a quasi-isomorphism between $\sE=\sE(\Th; \CT^{n+1},\phi, (\fb_A))$ and $\sE'$ which sends $\Lambda_j$ to $\Lambda_j'$. Note that the union $\widetilde{\Th}\colonequals \Th\sqcup \Th'$ is ordered such that

\[
 \Lambda_0\prec \Lambda_0'\prec \Lambda_1\prec \Lambda_1'\prec\cdots\prec\Lambda_n\prec\Lambda_n'.
\]

Due to the inductive nature of Lemma \ref{QD.L.14} and Lemma \ref{FSL.8}, for some $\tilde{\epsilon}>\epsilon$, one can find an $(\tilde{\epsilon}, \CT^{2n+2})$-consistent section $\tilde{\phi}$ of the bundle $\overline{\V}^{2n+2}\to \overline{\NR}^{2n+2}$, a set of strip-like ends adapted to $\tilde{\phi}$, and consistent cobordism data $(\tilde{\fb}_A)_{A\subset \widetilde{\Th}}$, extending the given ones on the subsets $\Th, \Th'\subset\widetilde{\Th}$. Then the $A_\infty$-category
\[
\tilde{\sE}\colonequals \sE(\widetilde{\Th}; \CT^{2n+2}, \tilde{\phi}, (\tilde{\fb}_{A})),
\]
contains $\sE'$ and $\sE$ as $A_\infty$-subcategories. For every $0\leq j\leq n$, set the Hamiltonian perturbation $\delta H$ in the Floer datum of $(\Lambda_j, \Lambda_j')$ to be zero, so by Lemma \ref{lemma:CanonicalGenerator},
$\hom_{\sE'}(\Lambda_j, \Lambda_j')=\BF\cdot e_j$
is generated by the constant soliton $e_j$. Denote by $\tilde{\sE}[\sT^{-1}]$ the localization of $\tilde{\sE}$ at  $\sT\colonequals
\{e_j: 0\leq j\leq n \}$ (see Section \ref{SecAF.7}) and by $\pi: \tilde{\sE}\to \tilde{\sE}[\sT^{-1}]$ the quotient functor. Then we obtain a diagram:
\begin{equation}
\begin{tikzcd}
&& \tilde{\sE}[\sT^{-1}]\arrow[ddl, dashed, bend left, "\sG'"]\\
\sE \arrow[r,"\eta"]\arrow[rd,dashed]& \tilde{\sE}\arrow[ru,"\pi"]&  \\
& \sE'.\arrow[u,"\eta'"']&
\end{tikzcd}
\end{equation}
where $\eta, \eta'$ denote the inclusion functors. We shall not distinguish an object $Y\in \Ob \tilde{\sE}$ with its Yoneda image in the dg-category $\sQ_r\colonequals \rfmod(\tilde{\sE})$ of finite right $\tilde{\sE}$-modules. To compute the cohomological category of $\tilde{\sE}[\sT^{-1}]$, we exploit Lemma \ref{AF.L7.12}. By Lemma \ref{AF.L7.6} and \ref{AF.L7.8}, for any $Y\in \Ob\tilde{\sE}$, the triangle in the diagram below is exact:
\[
\begin{tikzcd}
H(\hom_{\tilde{\sE}}(\Lambda_j',Y)) \arrow[r,"\sim"]\arrow[d,"{[\mu^2_{\sE}](e_j,\cdot )}"]& 
H(\hom_{\sQ_r}(\Lambda_j', Y)) \arrow[d]\arrow[r]& H(\hom_{\sQ_r}(\sC one(e_j), Y))\arrow[ld]\\
H(\hom_{\tilde{\sE}}(\Lambda_j,Y))\arrow[r,"\sim"]& 
H(\hom_{\sQ_r}(\Lambda_j, Y)). &
\end{tikzcd}
\]
 The left horizontal arrows (both top and bottom) are given by the Yoneda embedding functor. If $Y\neq \Lambda_j$, then the left vertical map is an isomorphism by Proposition \ref{prop:Quasi-Units}, so the third group in the exact triangle is trivial
\[
H(\hom_{\sQ_r}(\sC one(e_j), Y))=0. 
\]

Hence, any $\Lambda_k', 0\leq k\leq n$ is right orthogonal to $\sC one(\sT)$. A similar argument shows that 
\begin{align*}
H(\hom_{\sQ_r}( Y,\sC one(e_j)))&=H(\sC one(e_j)(Y))=0.
\end{align*}
for all $Y\neq \Lambda_j'$, which implies that any $\Lambda_k$ is left orthogonal to $\sC one(\sT)$. By Lemma \ref{AF.L7.12}, we conclude that 
\[
H(\hom_{\tilde{\sE}}(X, Y))=H(\hom_{\tilde{\sE}[\sT^{-1}]}(X, Y)) 
\]
for all $X,Y\in \Ob \tilde{\sE}$ unless $(X,Y)=(\Lambda_j',\Lambda_j), 0\leq j\leq n$. In the latter case, the group $H(\hom_{\tilde{\sE}[\sT^{-1}]}(\Lambda_j', \Lambda_j))$ is computed by Lemma \ref{AF.L7.13} and contains an inverse of $[e_j]$. We conclude that 
\[
 \pi\circ\eta: \sE\to \tilde{\sE}[\sT^{-1}] \text{ and }
  \pi\circ\eta': \sE'\to \tilde{\sE}[\sT^{-1}] 
\]
are quasi-equivalences. A left inverse of $H(\pi\circ \eta')$ 
\[
G: H(\tilde{\sE}[\sT^{-1}])\to H(\sE'),
\]
is given by sending $\Lambda_j, \Lambda_j'\mapsto \Lambda_j'$, which is then lifted to a quasi-equivalence $\sG: \tilde{\sE}[\sT^{-1}]\to \sE$ by Lemma \ref{AF.L7.7}. Then the composition 
\begin{equation}\label{key}
l_{\Th, \Th'}\colonequals \sG'\circ \pi\circ \eta:\sE\to \sE',\ \Lambda_j\mapsto \Lambda_j',\ 0\leq j\leq n
\end{equation}
is the desired quasi-isomorphism. By Lemma \ref{AF.L7.5}, this $A_\infty$-functor is canonical up to homotopy as the cohomological functor is independent of the choice of $\tilde{\phi}$ and $  (\tilde{\fb}_A)_{A\subset \widetilde{\Th}}$ by Lemma \ref{L3.7}. Indeed, $H(l_{\Th, \Th'})$ is defined as a zigzag:
\[
\begin{tikzcd}
& \HFF_\natural^*(\Lambda_j',\Lambda_k)\arrow[rd,"{m_*(e_k,\cdot)}"]\arrow[ld,"{m_*(\cdot,e_j)}"',shift right] &\\
\HFF_\natural^*(\Lambda_j, \Lambda_k) \arrow[rr,"H(l_{\Th, \Th'})"] \arrow[ru,shift right,dashed]&&
\HFF_\natural^*(\Lambda_j',\Lambda_k' ).
\end{tikzcd}
\]
for all $j<k$. In particular, if $\Th''\in \Gamma$ is admissible such that $\Th\sto\Th'\sto \Th''$, then $l_{\Th', \Th''}\circ l_{\Th,\Th'}$ is homotopic to $l_{\Th, \Th''}$. This shows that the $\sE$ is independent of the choice of $\phi$ and $(\fb_A)_{A\subset \Th}$: difference choices will yield directed $A_\infty$-categories quasi-isomorphic to the same $\sE'$. 
\end{proof}

The next corollary follows immediately from the proof of Proposition \ref{FS.P.15}.
\begin{corollary} The directed $A_\infty$-category $\sE(\Th)$ is independent of small variations of the angles $\theta_j, 0\leq j\leq n$ up to canonical quasi-isomorphisms.
\end{corollary}

\begin{remark} If one uses a general metric ribbon tree $\CT$ for  $\Th$ and $\CT'$ for $\Th'$, it is not guaranteed that there is a metric ribbon tree $\widetilde{\CT}$ with $\partial\widetilde{\CT}\cong\widetilde{\Th}$ and which contains $\CT, \CT'$ as subtrees. This is the main reason to take  $\CT=\CT^{n+1}$. 
\end{remark}

\subsection{Invariance of Fukaya-Seidel categories}  Given any admissible angle $\theta_*$, let $\Gamma_{\theta_*}$ denote the space of admissible sets $\Th=(\Lambda_1,\cdots, \Lambda_m),\ \Lambda_j=\Lambda_{y_j, \theta_j}$ such that 
\begin{equation}\label{FSE.18}
\theta_*<\theta_m<\cdots< \theta_1<\theta_*^{\crit},
\end{equation}
and $\Crit(W)=\{y_j\}_{j=1}^m$ is ordered as in \eqref{FSE.1}. Proposition \ref{FS.P.15} implies that the $A_\infty$-category $\sA=\sE(\Th),\ \Th\in \Gamma_{\theta_*}$ is independent of the choice of the quadratic differentials and the cobordism data. For this special case, there is a slightly different route to verify this invariance. Define a partial order on $\Gamma_{\theta_*}$ as follows: 
\[
\Th=(\Lambda_j)_{j=1}^m\sto \Th'=(\Lambda_j')_{j=1}^m,\ \Lambda_j=\Lambda_{y_j, \theta_j},\ \Lambda_j'=\Lambda_{y_j, \theta_j'},\ 1\leq j\leq m
\]
if  
\[
\theta_*<\theta_m'<\cdots< \theta_1'<\theta_m<\cdots< \theta_1<\theta_*^{\crit}.
\]
\begin{proposition}\label{FS.P.18} For any pair $\Th\sto \Th'$, there is a quasi-isomorphism 
	\[
	l_{\Th,\Th'}: \sE(\Th)\to \sE(\Th')
	\]
	such that $l_{\Th',\Th''}\circ l_{\Th, \Th'}$ is homotopic to $l_{\Th,\Th''}$ for any triple $\Th\sto \Th'\sto \Th''$.
\end{proposition}
\begin{proof} The proof follows the same line of arguments as in Proposition \ref{FS.P.15}. The only difference arises from the fact that $\Th\sqcup \Th'$ is ordered such that 
	\[
	\Lambda_1\prec \Lambda_2\prec\cdots \prec\Lambda_m\prec \Lambda_1'\prec\Lambda_2'\prec\cdots\prec\Lambda_m'.
	\]
	To verify that $\Lambda_k$ is left orthogonal to $\sC one(e_j)$ when $j<k$ (the $j\geq k$ case is identical to Proposition \ref{FS.P.15}) , one has to use the fact that 
	\[
\HFF_\natural^*(\Lambda_k, \Lambda_j')=0
	\]
	for all $j<k$. This follows from Lemma \ref{lemma:Vanishing} and the fact that $l_{y_k, \theta_k}$ does not intersect with $l_{y_j, \theta_j'}$, since $\im(e^{-\theta_*}W(y_j))<\im(e^{-\theta_*}W(y_k))$ and $\theta_*<\theta_j'<\theta_k<\theta_*+\epsilon$; see Figure \ref{Pic22} below. The same property is used to verify  that $\Lambda_j'$ is left orthogonal to $\sC one(e_k)$ when $j<k$.
\end{proof}

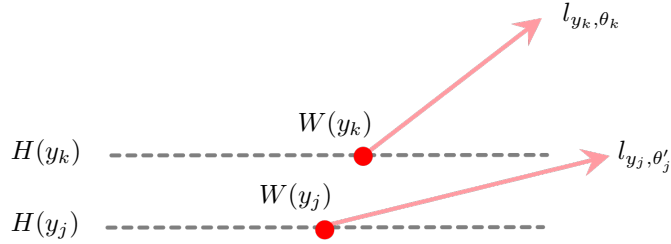
\begin{figure}[H]
	\centering
	\begin{overpic}[scale=.18]{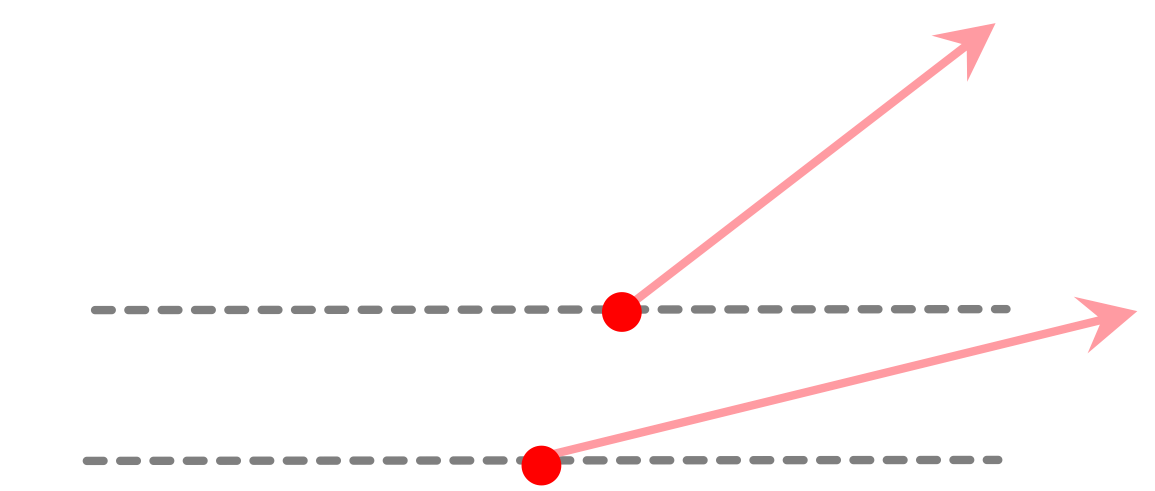}
		\put(-10, 15){\small $H(y_k)$}
		\put(-10,2){\small $H(y_j)$}
		\put(42,20){\small $W(y_k)$}
		\put(35,7){\small $W(y_j)$}
		\put(90,40){\small $l_{y_k,\theta_k}$}
		\put(100, 15){\small $l_{y_j,\theta_j'}$}
	\end{overpic}	
	\caption{$l_{y_k, \theta_k}$ does not interest with $l_{y_j, \theta_j'}$ ($\theta_*=0$).}
	\label{Pic22}
\end{figure}

\begin{remark} The partial order on $\Gamma$ is inspired by the construction of wrapped Fukaya category $\W(M,W)$ in \cite[Section 3.5]{GPS20}, and one should think of $\sA=\sE(\Th)$ as a subcategory of $\W(M, W)$ if $(M,W)$ is a Liouville Landau-Ginzburg model in the sense of \cite[Example 2.20]{GPS20}.
\end{remark}
\subsection{Invariance of the diagonal bimodule $\Delta$}\label{SecFS.11} To construct the diagonal bimodule $\Delta$ in Example \ref{FS.EX.2}, one can be more flexible about the choice of the metric ribbon tree $\CT$. This extra flexibility is essential to the proof of Theorem \ref{FS.T.3}, as we shall see shortly in Section \ref{SecPT}.  Recall that $\theta_*=0$ is assumed to be admissible in the sense of Example \ref{EX3.3}, and let $\Gamma_\Delta$ denote the space of triples $\Th_\Delta=(R, \Th_\pi, \Th_0)$ with $R\geq 0$, $\Th_\pi\in \Gamma_{\pi}$ and $\Th_0\in \Gamma_0$. For any $\Th_\Delta\in \Gamma_\Delta$, one can define a directed $A_\infty$-category
\[
\sE(\Th_\Delta)\colonequals \sE(\Th_\pi\sqcup \Th_0; \CT^{m, m}_R, \phi,(\fb_A))
\]
which may a priori depends on the choice of $R\geq 0$, $\phi$ and the cobordism data $(\fb_A)$. $\CT^{m,m}_R$ is the metric ribbon tree defined in Example \ref{QD.EX.8}, which has two interior vertices and $2m$ exterior ones. Any $S_j\in \Th_0$ and $U_k\in \Th_\pi$ have distance $=R+\pi$ in $\CT^{m,m}_R$. The subtree of $\CT^{m,m}_R$ associated to the subset $\Th_0$ (resp. $\Th_\pi$) is $\CT^{m+1}$, so $\sE(\Th_\Delta)$ contains $\sA\colonequals \sE(\Th_0)$ and $\sB\colonequals \sE(\Th_\pi)$ as $A_\infty$-subcategories. When $R=0$, $\CT^{m,m}_0=\CT^{2m}$, and $\sE(\Th_\Delta)$ agrees with the previous construction.

\begin{proposition}\label{FS.P.19} The $A_\infty$-category $\sE(\Th_\Delta)$ is independent of the choice of $\Th_\Delta\in \Gamma_\Delta$, $\phi$ and $(\fb_{A})$ up to canonical quasi-isomorphisms.
\end{proposition}
\begin{proof} Define a partial order on $\Gamma_\Delta$ such that 
	\[
	\Th_\Delta=(R, \Th_\pi, \Th_0)\sto \Th_\Delta'(R', \Th_\pi', \Th_0')
	\]
	if $R\leq R'$, $\Th_\pi\sto \Th'_\pi$ in $\Gamma_{\pi}$ and $\Th_0\sto \Th_0'$ in $\Gamma_0$. Following the proof of Proposition \ref{FS.P.15}, one has to construct a directed $A_\infty$-category $\tilde{\sE}$ with 
	\[
	\Ob\tilde{\sE}= \Th_\pi\sqcup \Th_\pi'\sqcup \Th_0\sqcup \Th_0',
	\]
	and which contains $\sE(\Th_\Delta)$ and $\sE(\Th_\Delta')$ as $A_\infty$-subcategories. To this end, we use the metric ribbon tree from Example \ref{QD.EX.9}: take  $\CT=\CT^{d_1,d_2,d_3}_{R_1, R_2}$ with
	\[
	R_1=R,\ R_2=R',\ d_1=2m \text{ and }d_2=d_3=m.
	\]
	This means that $\CT$ has three interior vertices $v_1,v_2, v_3$ with $v_2$ lying on the interval $[v_1, v_3]$ and with 
	\[
	d_{\CT}(v_1, v_2)=R \text{ and }d_{\CT}(v_1, v_3)=R'.
	\]
	 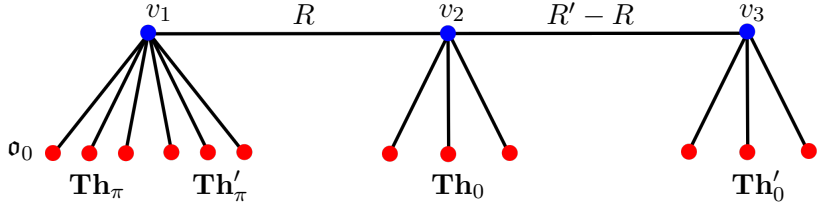
\begin{figure}[H]
		\centering
		\vspace{1em}
		\begin{overpic}[scale=.15]{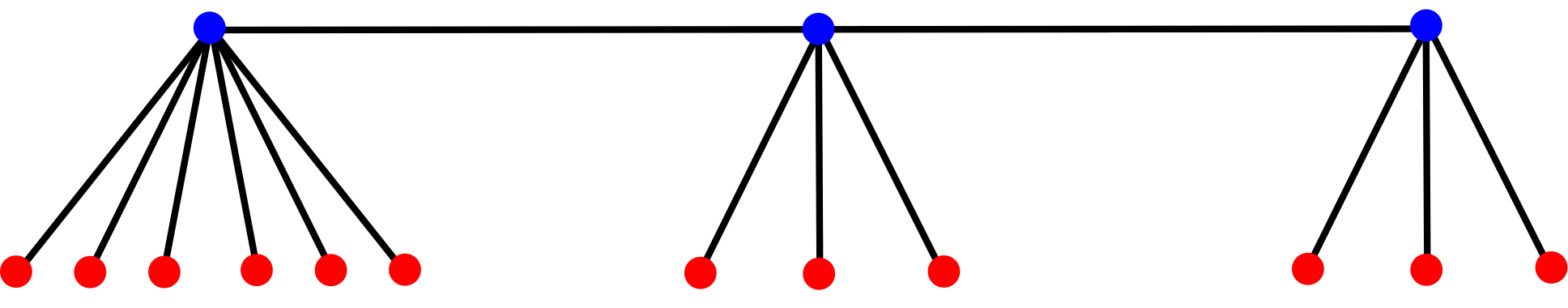}
			\put(-5,2){$\fo_0$}
			\put(13,20){$v_1$}
			\put(51,20){$v_2$}
			\put(90,20){$v_3$}
			\put(32, 19){$R$}
			\put(65,19){$R'-R$}
			\put(3, -3){$\Th_\pi$}
			\put(19,-3){$\Th_\pi'$}
				\put(50,-3){$\Th_0$}
							\put(89,-3){$\Th_0'$}
		\end{overpic}	
\vspace{1em}
		\caption{The metric ribbon tree $\CT^{6,3,3}_{R, R'}$ with $m=2$.}
		\label{Pic21}
	\end{figure}
	
The first $2m$ vertices in $\partial \CT$ are adjacent to $v_1$ and labeled by $\Th_\pi \sqcup \Th_\pi'$. The next $m$ vertices are adjacent to $v_2$ and labeled by $\Th_0$, while the rest are adjacent to $v_3$ and labeled by $\Th_0'$. The rest of the proof is identical to that of Proposition \ref{FS.P.18}: using $\tilde{\sE}$ one constructs a quasi-isomorphism
	\[
	l_{\Th_\Delta,\Th'_\Delta}:\sE(\Th_\Delta)\to \sE(\Th'_\Delta)
	\]
	and verifies that $	l_{\Th'_\Delta,\Th''_\Delta}\circ 	l_{\Th_\Delta,\Th'_\Delta}$ is homotopic to $	l_{\Th_\Delta,\Th''_\Delta}$ for any triple $\Th_\Delta\sto \Th_\Delta'\sto \Th_\Delta''$. Details are left to the readers. 
\end{proof}
\begin{remark} There are multiple ways to choose the metric ribbon tree $\CT$ in the proof of Proposition \ref{FS.P.19} to define $\tilde{\sE}$. The advantage of using Example \ref{QD.EX.9} is in that for any chain in $\Gamma_\Delta$:
	\[
	\Th^{(1)}_\Delta\sto \Th^{(2)}_\Delta\sto \cdots \sto \Th^{(n)}_{\Delta}
	\]
	with $\Th^{(k)}_\Delta=(R_k, \Th_\pi^{(k)}, \Th_0^{(k)})$, $1\leq k\leq n$, one can use the metric ribbon tree $\CT_{R_1,\cdots, R_k}^{d_1,\cdots,d_{n+1} }$ with 
	\[
	d_1=n\cdot m \text{ and }d_k=m=|\Crit(W)|,\ 2\leq k\leq n+1,
	\]
	to construct a directed $A_\infty$-category that contains each $\sE(\Th_\Delta^{(k)})$ as an $A_\infty$-subcategory, although this property is not used in the proof of Proposition \ref{FS.P.19}.
\end{remark}

\subsection{A generalization}\label{SecFS.12} There is one more preliminary step to go before we turn to the proof of Theorem \ref{FS.T.3}. Let $\Th$ be any admissible set of thimbles and $\CT$ a metric ribbon tree with $n+1\colonequals |\partial \CT|=|\Th|$. The combinatorial structure of $\overline{\NR}^{n+1}$ as a smooth manifold with corners, as described by $n$-leafed trees, is critical to essentially any construction of Fukaya categories. On the other hand, the completions of $(n+1)$-pointed disks, which are used to define the moduli space $\M_{\NR^{n+1}}$, are specified by a consistent section of $\overline{\V}^{n+1}\to \overline{\NR}^{n+1}$, but this requirement can be relaxed a little bit: in general $\sE(\Th)$ can be defined using a smooth map $\phi^\circ:\overline{\NR}^{n+1}\to \overline{\V}^{n+1}$ which is not necessarily a section. 

In other words, we would like to distinguish \textit{the parameter space} of $\M_{\NR^{n+1}}$ from \textit{the actual family of pointed disks} (and their completions) on which the Floer equation \eqref{E1.12} is defined. We explain this in more detail. From now on we always write $\overline{\NR}^{n+1}_{\base}$ for this parameter space and $\phi^\circ$ for a smooth map $\overline{\NR}^{n+1}_{\base}\to \overline{\V}^{n+1}$. For any subset $A\subset\Th$, the family $\Sch^{|A|}_{\base}\to \NR^{|A|}_{\base}$ comes with a set of strip-like ends $\{\iota_{A,k}^{\base}\}_{k=0}^{|A|-1}$, not necessarily adapted to any quadratic differentials, and which defines for any $n$-leafed tree $T$ a set of strip-like ends $\iota_T^{\base}=(\iota_f^{\base})_{f\in \Fl(T)}$ on the family $\Sch^T_{\base}\to \NR^T_{\base}$ and for any pair $T<T'$ a gluing map
\[
\gamma^{T,T'}_{\base}: (-1,0)^{\Ed^{\inte}(T)\setminus\Ed^{\inte}(T')}\times \NR^T_{\base}\to \NR^{T'}_{\base}.
\]

The set of strip-like ends $\{\iota_T^{\base}\}$ is assumed to be consistent in the sense of \cite[Lemma 9.3]{S08}. In particular, the diagram $\eqref{QD.E.8}$ holds for any triple $T<T'<T''$. The subscript (resp. superscript) $\base$ is to indicate that these data are associated to the base manifold $\overline{\NR}^{n+1}_{\base}$, which are chosen once and for all. Now we spell out the conditions for the smooth map $\phi^\circ:\overline{\NR}^{n+1}_{\base}\to \overline{\V}^{n+1}$. 
\begin{itemize}
\item for any subset $A\subset \Th$ with $|A|\geq 3$, choose a smooth function $g_A^\circ: \NR^{|A|}_{\base}\to\NR^{|A|}$ and a smooth section $\psi^\circ_A$ of $(g_A^\circ)^* \V^{|A|}\to \NR^{|A|}_{\base}$ that is $\epsilon$-close to $\CT_A$, i.e., $\psi^\circ_A: \NR^{|A|}_{\base}\to \V^{|A|}$ is a smooth map covering $g_A$:
\[
\begin{tikzcd}
& \V^{|A|}\arrow[d] & (g_A^\circ)^*\Sch^{|A|}\arrow[r]\arrow[d]& \Sch^{|A|} \arrow[d]\\
\NR^{|A|}_{\base}\arrow[r,"g_A^\circ"]\arrow[ru,"\psi^\circ_A"]& \NR^{|A|}, & \NR^{|A|}_{\base}\arrow[r, "g_A^\circ"]&\NR^{|A|}. 
\end{tikzcd}
\]
\item choose a set of $\psi_A^\circ$-adapted strip-like ends $\{\iota^\circ_{A,k}\}_{k=0}^{|A|-1}$ for the family $g_A^*\Sch^{|A|}\to \NR^{|A|}_{\base}$, that is, a set of proper embeddings,
\[
\iota^\circ_{A,k}: \R^\pm_t\times [0,\w_{A,k}^\circ]_s\times \NR^{|A|}_{\base}\to (g_A^\circ)^*\Sch^{|A|},
\]
fibered over $\NR^{|A|}_{\base}$, one for each $k$, which is $\psi_{A,r}^\circ$-adapted on the fiber $S_{g_A^\circ(r)}$ for all $r\in \NR^{|A|}_{\base}$. The width $\w_{A,k}^\circ$ of $\iota_{A,k}^\circ$ is determined by the metric ribbon tree.
\end{itemize}

These data define for each $n$-leafed tree $T$, 
\begin{itemize}
	\item a smooth map $g_T^\circ=(g_{T,v}^\circ): \NR^T_{\base}\to \NR^T$;
	\item a smooth section $\phi_T^\circ=(\phi_{T,v}^\circ)_{v\in\Ve(T)}$ of $(g_T^\circ)^*\V^T\to \NR^T_{\base}$;
	\item a set of $\phi_T^\circ$-adapted strip-like ends $\iota_T^\circ=(\iota_f^\circ)_{f\in \Fl(T)}$ for the family $g_T^*\Sch^T\to \NR^T_{\base}$;
\end{itemize}
which satisfy the following properties:
\begin{enumerate}[label=$(\roman*)$]
	\item $g_{T,v}^\circ=g_A^\circ$, $\phi_{T,v}^\circ=\psi_A^\circ$ and $\iota_{f_k(v)}^\circ=\iota_{A,k}^\circ,\ 0\leq k\leq |v|-1,$ if $A=\partial\CT_v$;
	\item\label{Cii} $g^\circ=(g_T^\circ):(\overline{\NR}^{n+1}_{\base},\partial \overline{\NR}^{n+1}_{\base})\to (\overline{\NR}^{n+1},\partial \overline{\NR}^{n+1})$ is a smooth map of degree $1$;
	\item$\phi^\circ=(\phi_T^\circ)$ is a  smooth section of $(g^\circ)^*\overline{\V}^{d+1}\to \overline{\NR}^{d+1}_{\base}$;
	\item\label{Cv} for any pair $T<T'$, $\phi_{T'}^\circ$ is obtained by gluing $\phi_T^\circ$ using the strip-like ends $(\iota^\circ_f)_{f\in \Fl(T)}$ on the image of $\gamma^{T,T'}_{\base}$.
	\end{enumerate}

The last property \ref{Cv} means that for any $T<T'$ there is a commutative diagram:
\begin{equation}\label{FSE.28}
\begin{tikzcd}
(-1,0)^{\Ed^{\inte}(T)\setminus\Ed^{\inte}(T')}\times \NR^T_{\base}\arrow[d,"\gamma^{T,T'}_{\base}"]\arrow[r,"\Id\times \phi_{T}^\circ"] & (-1,0)^{\Ed^{\inte}(T)\setminus\Ed^{\inte}(T')}\times \V^{T}\arrow[d,dashed,"\text{ glue }"]&\\
\NR^{T'}_{\base}\arrow[r,"\phi_{T'}^\circ"]\arrow[rd,"g_{T'}^\circ"'] & \V^{T'}\arrow[d, "\text{projection}"]\\
 & \NR^{T'}.
\end{tikzcd}
\end{equation}

The dashed arrow is not canonically defined. On the image of $\Id\times \phi^\circ_T$, it is the gluing map defined using the strip-like ends $(\iota^\circ_f)_{f\in \Fl(T)}$ and whose construction is outlined in Section \ref{SecQD.8}. The smooth map $g_{T'}^\circ$ is determined by $\phi_{T}^\circ$ and the property \ref{Cv} on $\im \gamma^{T,T'}_{\base}$.

\begin{definition}\label{FS.D.23} We say that $\phi^\circ=(\phi^\circ_T)$ is an $(\epsilon, \CT)$-consistent family of quadratic differentials parametrized by $g^\circ=(g_T^\circ):\overline{\NR}_{\base}^{n+1}\to \overline{\NR}^{n+1}$ if $(\phi^\circ, g^\circ)$ satisfies the conditions above for some choice of strip-like ends $(\iota^\circ_T)$.
\end{definition}
\begin{remark} When $|A|=3$, $g_{A}^\circ: \NR_{\base}^{|A|}\to \NR^{|A|}$ is a bijection between points. In general, $g_{A}^\circ$ can be extended smoothly into a map $\overline{\NR}_{\base}^{|A|}\to \overline{\NR}^{|A|}$ using \eqref{FSE.28}, and one can verify inductively that this mapping degree is $1$; see Figure \ref{Pic23}. Thus the second property \ref{Cii} is somewhat redundant. 
\end{remark}
\begin{example} If one simply takes $g^\circ:\overline{\NR}_{\base}^{n+1}\to \overline{\NR}^{n+1}$ to be the identity map in Definition \ref{FS.D.23}, and if the strip-like ends $\{\iota^{\base}_T\}$ are already adapted to $\phi^\circ$, then Definition \ref{FS.D.23} recovers with the special case in Lemma \ref{QD.L.14}. For another possible reduction, one starts with an $(\epsilon,\CT)$-consistent section $\phi$ of  $\overline{\V}^{n+1}\to \overline{\NR}^{n+1}$; then a set of $\phi$-adapted strip-like ends $(\iota_T)$ defines the gluing map
\[
\gamma^{T,T'}:(-1,0)^{\Ed^{\inte}(T)\setminus\Ed^{\inte}(T')}\times \NR^{T}\to \NR^{T'}.
\]
 Take $g^\circ:\overline{\NR}_{\base}^{n+1}\to \overline{\NR}^{n+1}$ to be any smooth diffeomorphism between manifolds with corners that intertwines $\gamma^{T,T'}_{\base}$ and $\gamma^{T,T'}$ for any $T<T'$, and define $\phi^\circ=(g^\circ)^*\phi$. The upshot is that $g^\circ$ is not required to be a diffeomorphism (or even injective) in Definition \ref{FS.D.23}.
\end{example}

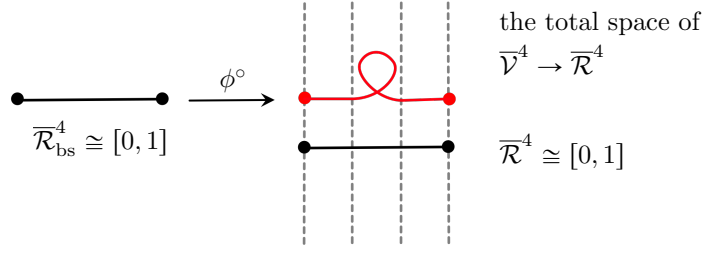
\begin{figure}[H]
	\centering
	\begin{overpic}[scale=.12]{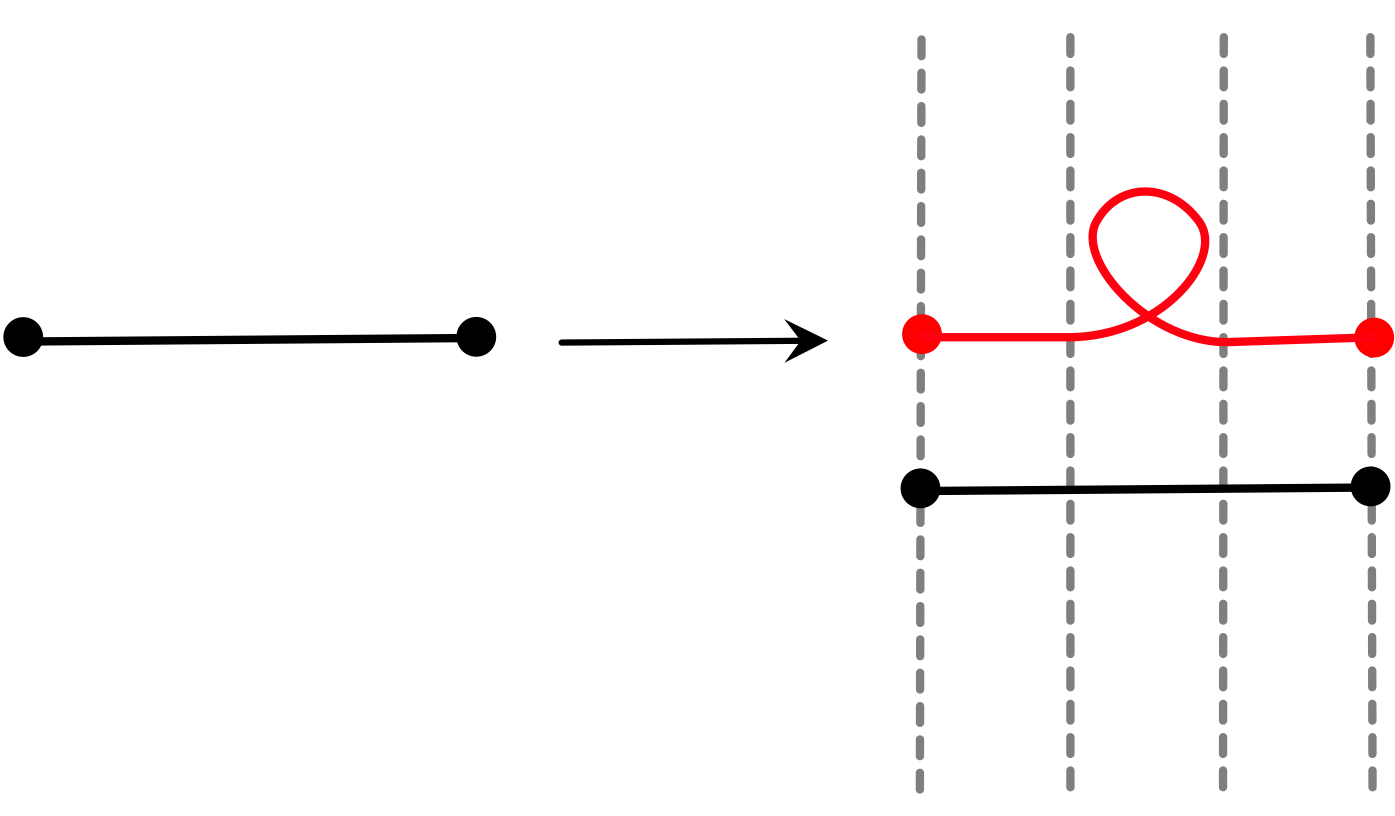}
		\put(5,23){\small $\overline{\NR}^4_{\base}\cong [0,1]$}
		\put(110,20){\small $\overline{\NR}^4\cong [0,1]$}
		\put(110,50){\small the total space of}
		\put(110, 40){\small $\overline{\V}^4\to \overline{\NR}^4$}
		\put(47,37){\small $\phi^\circ$} 
	\end{overpic}	
	\caption{The red curve is the image of $\phi^\circ$ where $n+1=4$.}
	\label{Pic23}
\end{figure}

Consider the smooth fiber bundle 
\begin{equation}\label{FSE.25}
\hat{S}^{|A|}_{\psi_A}\to \NR_{\base}^{|A|}
\end{equation}
whose fiber at $r\in \NR_{\base}^{|A|}$ is the completion of $(S_{g_A^\circ (r)},\psi^\circ_{A,r})$. The primary reason to introduce this framework is to bypass a technical difficulty that may arise otherwise in the proof of Theorem \ref{FS.T.3}; see Remark \ref{GF.R.13}. From now on we focus on the special case in Section \ref{SecFS.11}. For any triple $\Th_\Delta=(R, \Th_\pi, \Th_0)\in \Gamma_\Delta$, one can choose some $(\phi^\circ, g^\circ)$ which are  $(\epsilon,\CT)$-consistent, along with consistent cobordism data $(\fb_A)_{A\subset \Th}$, to construct a directed $A_\infty$-category
\begin{equation}\label{FSE.26}
\sE(\Th_\pi\sqcup \Th_0;\ \CT^{m, m}_R,\phi^\circ, (\fb_{A})_{A\subset \Th}),\ \Th\colonequals \Th_\pi\sqcup \Th_0.
\end{equation}

The cobordism data $\fb_A=(K_A, \Xi_A, \delta\kappa_A, \delta H_A)$ is defined similarly as in Definition \ref{FSD.6} except that each term is now associated to the bundle \eqref{FSE.25}. The moduli space \eqref{FSE.8} is then replaced by
\[
\M_{\NR^{2m}_{\base}}(\{p_k\}_{k=0}^{2m-1};\fb_{\Th})\colonequals \coprod_{r\in \NR^{2m}_{\base} } \M_{S_{g^\circ(r)}}(\{p_k\}_{k=0}^{2m-1};\fb_{\psi^\circ (r)}).
\] 

The generalization of Proposition \ref{FS.P.19} then states the following.

\begin{proposition}\label{FS.P.25} The directed $A_\infty$-category \eqref{FSE.26} is independent of the choice of $\Th_\Delta\in \Gamma_\Delta$, $(\phi^\circ, g^\circ)$ and $(\fb_A)$ up to canonical quasi-isomorphisms.
\end{proposition}
\begin{remark} The proof of Proposition \ref{FS.P.25} is identical to that of Proposition \ref{FS.P.19} and is omitted here. In the next section, we shall construct the smooth map $g^\circ: \overline{\NR}^{2m}_{\base}\to \overline{\NR}^{2m+2}$ and $\phi^\circ$ explicitly, for which the consistency conditions are easily verified. The advantage of Proposition \ref{FS.P.25} is in that $g^\circ$ need not be a diffeomorphism in order to define $_{\sA}\Delta_{\sB}$.
\end{remark}

\section{The Geometric Filtration}\label{SecGF}

\subsection{Introduction} From now on fix some admissible sets $\Th_0\in \Gamma_0$ and $\Th_\pi\in \Gamma_\pi$. Following the notations from Example \ref{FS.EX.2}, we write
\begin{align*}
\Th_0&=(S_1,\cdots, S_m), &S_j&=\Lambda_{x_j, \theta_j},\\
\Th_\pi &=(U_m,\cdots, U_1), & U_j&=\Lambda_{x_j,\eta_j},\ 1\leq j\leq m.
\end{align*}
The critical set $\Crit(W)=\{x_j\}_{j=1}^m$ is ordered as in \eqref{FSE.2}, and the angles $\theta_j, \eta_j, 1\leq j\leq m$ satisfy the relation \eqref{FSE.24}.  Take $0<\epsilon\ll 1$. Suppose that the $A_\infty$-categories 
\begin{align}\label{GF.E.18}
\sA&\colonequals \sE(\Th_0;\CT^m, \phi_*,(\fb_A)_{A\subset \Th_0}), & \sB&\colonequals \sE(\Th_\pi;\CT^m, \phi_*,(\fb_A)_{A\subset \Th_\pi})
\end{align}
are already defined using a smooth $(\epsilon, \CT^{m})$-consistent section $\phi_*$ of $\overline{\V}^{m}\to \overline{\NR}^{m}$ and some consistent cobordism data. As noted in Remark \ref{FS.R.4}, one may use some convenient quadratic differentials and cobordism data for the diagonal bimodule ${}_{\sA}\Delta_{\sB}$ to verify Theorem \ref{FS.T.3}. The goal of this section is to describe a suitable setup so that  ${}_{\sA}\Delta_{\sB}$ is filtered by a sequence of sub-bimodules
\begin{equation}\label{GF.E.24}
0={}_{\sA}\Delta^{(0)}_{\sB}\subset {}_{\sA}\Delta^{(1)}_{\sB}\subset \cdots \subset {}_{\sA}\Delta^{(m)}_{\sB}={}_{\sA}\Delta_{\sB}.
\end{equation}
This means that for each pair $(U_j, S_k)$, the complex $\Delta (U_j, S_k)$ comes with a filtration 
\[
0=\Delta^{(0)}(U_j, S_k)\subset \Delta^{(1)}(U_j, S_k)\subset \cdots\subset  \Delta^{(m)}(U_j, S_k)=\Delta(U_j, S_k)
\]
which is preserved under the $A_\infty$-actions of $\sA$ and $\sB$. In this paper, we will not talk about the invariance of this filtered bimodules but only take it as a trick to prove Theorem \ref{FS.T.3}. The construction of $\Delta^{(l)},0\leq l\leq n$ relies on the generalized framework Section \ref{SecFS.12} as well as a neck-stretching technique, so $\Delta^{(l)}=\Delta_R^{(l)}$ depends on an extra stretching parameter $R\gg \pi$. This idea can be generalized further to the case of Remark \ref{FS.R.5}, but we shall focus the admissible set $\Th_{\pi}\sqcup \Th_0$ here to simplify our exposition. 

\subsection{The energy filtration}\label{SecGF.2} We start with the filtered complex $\Delta_R(U_j, S_k)$ and specify its Floer datum. Fix some $0<\delta\ll 1$. For any $1\leq j,k\leq m$, choose smooth monotone functions $\alpha_j^{un}, \alpha_k^{st}:\R_s\to \R$ such that 
\begin{align}\label{GF.E.1}
\alpha_j^{un}(s)&= \left\{
\begin{array}{ll}
\pi &\text{if }s\geq \pi-\delta,\\
\eta_j-\pi&\text{if }s\leq \delta,
\end{array}
\right. &
\alpha_k^{st}(s)&= \left\{
\begin{array}{ll}
\theta_k &\text{if }s\geq \pi-\delta,\\
\pi &\text{if }s\leq \delta.
\end{array}
\right. 
\end{align}

Then for any $R\geq \pi$, define $\alpha_{jk}^R:\R_s\to \R$ by the formula
\begin{equation}\label{GF.E.26}
\alpha_{jk}^R(s)= \left\{
\begin{array}{llc}
\theta_k&\text{ if }R+\pi\leq s,\\
\alpha_k^{st}(s-R) & \text{ if }R\leq s\leq R+\pi,\\
\pi &\text{ if }\pi\leq s\leq R,\\
\alpha_j^{un}(s) & \text{ if } 0\leq s \leq \pi,\\
\eta_j-\pi & \text{ if }s\leq 0,
\end{array}
\right.
\end{equation}
which is ``a concatenation" of $\alpha_j^{un}$ and $\alpha_k^{st}$. Let $\gamma_{jk}^R:\R_s\to \C$ denote the characteristic curve of $\alpha^R_{jk}(s)$ which is normalized by the property that (see Figure \ref{Pic24} below)
\begin{itemize}
\item $\gamma_{jk}^R(s)\in \R^+\cdot e^{\eta_j-\pi}$ for all $s<0$ and $\in \R$ for all $s\in [\pi, R-\pi]$. 
\end{itemize}
If one takes $R_0\geq \pi$ as a reference parameter, then for any $R\geq R_0$:
\begin{equation}\label{GF.E.23}
\gamma_{jk}^R(s)=\left\{\begin{array}{ll}
\gamma_{jk}^{R_0}(s) & \text{ if }s\leq R_0,\\
\gamma_{jk}^{R_0}(R_0)+(s-R_0) & \text{ if } R_0\leq s\leq R,\\
 \gamma_{jk}^{R_0}(s)+(R-R_0) &\text{ if } R\leq s. 
\end{array}
\right.
\end{equation}
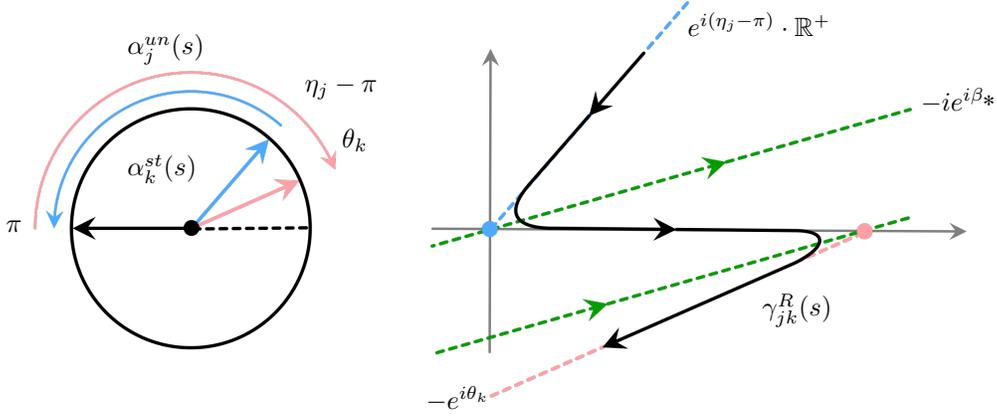
\begin{figure}[H]
	\centering
	\begin{overpic}[scale=.15]{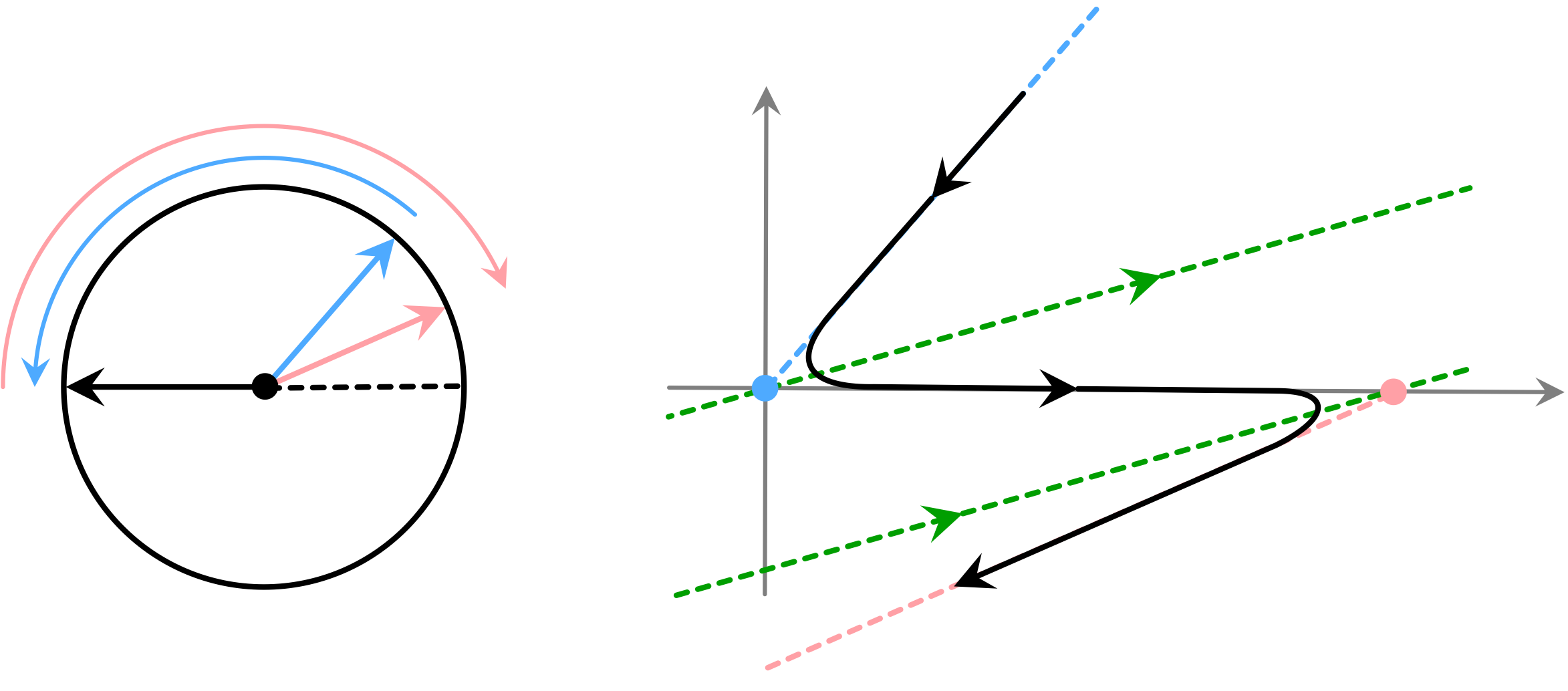}
		\put(-3,19){\small $\pi$}
		\put(29,34){\small $\eta_j-\pi$}
		\put(33,28){\small $\theta_k$}
		\put(10,25){\small $\alpha_k^{st}(s)$}
		\put(10,38){\small $\alpha_j^{un}(s)$}
		\put(70,40){\small $e^{i(\eta_j-\pi)}\cdot \R^+$}
		\put(42,0){\small$-e^{i\theta_k}$}
		\put(95,32){\small $-ie^{i\beta_*}$}
		\put(78,10){\small $\gamma_{jk}^R(s)$}
	\end{overpic}	
	\caption{The function $\alpha_{jk}^R(s)$ (left) and its characteristic curve $\gamma^{jk}_R(s)$ (right).}
	\label{Pic24}
\end{figure}

Fix some $\frac{\pi}{2}< \beta_*<\frac{\pi}{2}+ \min\{\theta_m,\eta_0-\pi\}$. By the relation \eqref{FSE.24}, there exists $\epsilon_*>0$ such that for all $1\leq j,k\leq m$ and $R\geq \pi$, we have 
\begin{equation}\label{GF.E.25}
\re(e^{i(\beta_*-\alpha_{jk}^R(s))})>\epsilon_*. 
\end{equation}

Finally, we require that the perturbation 1-form $\delta H_{jk}^R\in \Omega^1(\R_s, \SH)$ be supported on $[0,\pi]_s\cup [R,R+\pi]_s$, $\delta H^R_{jk}\equiv 0$ if $j=k$, and the estimate
\begin{equation}\label{GF.E.7}
\int_{\R_s}\|\delta H_{jk}^R\|_{L^\infty_1(M)}^2ds+\int_{\R_s}\|\delta H_{jk}^R\|_{L^\infty(M)}ds<\frac{1}{R}, 
\end{equation}
holds for all $j, k $ and $R\geq \pi.$; compare \eqref{E1.21}. If $\delta H_{jk}^R$ is chosen generically, then the Floer datum $\fa_{jk}^R=(R+\pi,\alpha_{jk}^R(s), \beta_*,\epsilon_*, \delta H_{jk}^R)$ is admissible, and the Floer complex 
\[
\Delta_R(U_j, S_k)\colonequals\Ch_\natural^*(U_j, S_k; \fa_{jk}^R)
\]
is well-defined. The key observation is that $\Delta_R(U_j, S_k)$ carries a natural energy filtration when $R\gg \pi$. For any $x_j\in \Crit(W)$, choose a neighborhood $\SO(x_j)\subset M$ of $x_j$ such that $\SO(x_j)\cap \SO(x_k)=\emptyset$ if $j\neq k$. The next two lemmas follow from the standard argument in Morse-Smale-Witten theory \cite[Section 2]{Bible}. In our case, they follow from \eqref{GF.E.7}, the energy estimate in Lemma \ref{L1.5}, the Local Compactness Hypothesis \ref{L1.6} and the fact that $W:M\to \C$ is Morse. 

\begin{lemma}\label{GF.L.1} Take a sequence $R_n\to +\infty$, and let $p_n(s)\in \FC(U_j, S_k;\fa_{jk}^{R_n})$ be any $\alpha_{jk}^{R_n}$-soliton. Then there exists a critical point $x_l\in \Crit(W)$, an $\alpha_j^{un}$-soliton $p^{un}_j:\R_s\to M$ connecting $x_j$ with $x_l$, and an $\alpha_k^{st}$-soliton $p^{st}_j: \R_s\to M$ connecting $x_l$ with $x_k$ (and with zero perturbation 1-forms) such that a subsequence of $\{p_n\}$ converges as broken flowlines to the concatenation of $p^{un}_j$ and $p^{st}_k$. This means that for  this subsequence 
	\[
	p_n(s)\to p_j^{un}(s) \text{ and } p_n(s-R_n)\to p_k^{st}(s) \text{ in } C^\infty_{loc}(\R_s, M). 
	\]
\end{lemma}

\begin{lemma}\label{GF.L.2} Under above assumptions, there exist constants $C,\zeta>0$ such that for any $1\leq j, k\leq m$, $R\geq \pi$ and any $\alpha_{jk}^{R}$-soliton $p(s)\in \FC(U_j, S_k;\fa_{jk}^R)$, we have estimates
	\begin{align}
	|\ps p(s)|^2+|\nabla H\circ p(s)|^2&<Ce^{-\zeta\min\{|s|,|s-(R+\pi)|\}},\ \forall s\in \R_s\label{GF.E.6}\\
	|W(p(s))-W(x_j)|&<Ce^{-\zeta |s|},\ \forall s\leq 0,\label{GF.E.2}\\
	|W(p(s))-W(x_k)|&<Ce^{-\zeta |s-(R+\pi)|},\ \forall s\geq R+\pi,\label{GF.E.3}
	\end{align}
	and for some $x_l \in \Crit(W)$, there holds
	\begin{equation}\label{GF.E.4}
	|W(p(s))-W(x_l)|<Ce^{-\zeta\min\{|s|,|s-(R+\pi)|\}},\ \forall s\in [\pi, R]_s. 
	\end{equation}
In particular, by \eqref{GF.E.6}, there exists a constant $R_0>0$ such that for any $R\geq 2R_0$ and $R_0\leq s\leq R-R_0$, $p(s)\in \SO(x_l)$.
\end{lemma}

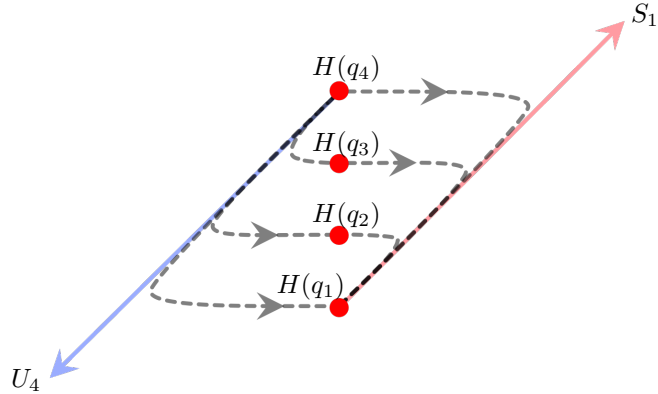
\begin{figure}[H]
	\centering
	\begin{overpic}[scale=.18]{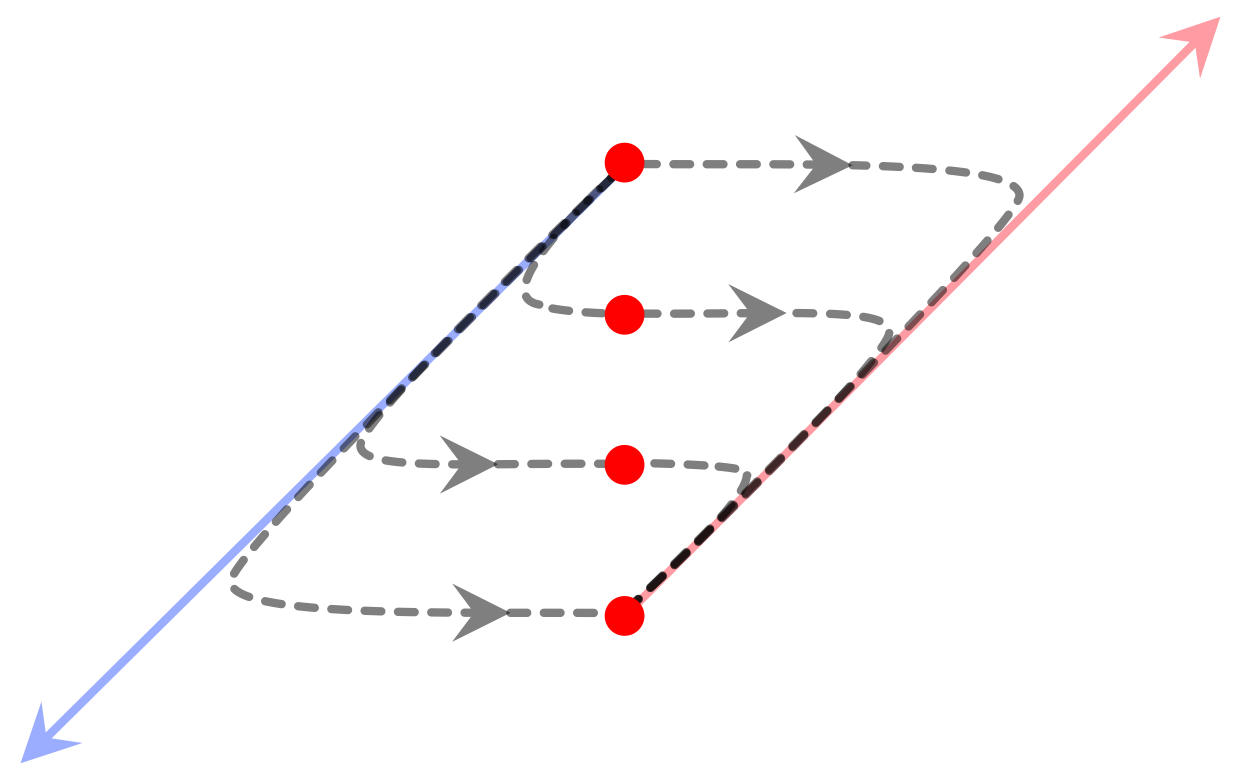}
		\put(100,62){\small $S_1$}
		\put(-5,0){\small $U_4$}
 \put(46,53){\small $H(q_4)$}
 \put(46,40){\small $H(q_3)$}
 \put(46,28){\small $H(q_2)$}
 \put(40,16){\small $H(q_1)$}
	\end{overpic}	
	\caption{The energy filtration on $\Delta(U_4, S_1)$.}
	\label{Pic25}
\end{figure}

For any $R\geq 2R_0$, Lemma \ref{GF.L.2} allows us to decompose the space of solitons as
\[
\FC(U_j, S_k;\fa_{jk}^{R})=\coprod_{l=0}^m \FC^{l}(U_j, S_k;\fa_{jk}^{R})
\]
such that $p(s)\in \FC^{l}$ only if the $\alpha_{jk}^R$-soliton $p(s)$ approximates $x_l\in \Crit(W)$ on the interval $[R_0, R-R_0]_s$. Let
\[
 \sG^l\Delta_R (U_j, S_k)= \bigoplus_{p\in  \FC^l(U_j, S_k;\fa_{jk}^{R}) } \BK\cdot p,
\]
denote the subspace of $\Delta_R(U_j, S_k)$ freely generated by solitons in $\FC^l$. Then 
\[
\Delta_R(U_j, S_k)=\bigoplus_{l=0}^m \sG^l\Delta_R (U_j, S_k).
\]

\begin{lemma}\label{GF.L.3} For all $R\gg \pi$, $\sG^l\Delta_R (U_j, S_k)$ is trivial unless $k\leq l\leq j$. 
\end{lemma}
\begin{proof}[Proof of Lemma \ref{GF.L.3}] It suffices to note that for the limiting solitons $p^{un}_j$ and $p^{st}_k$ in Lemma \ref{GF.L.1}, we have 
	\begin{align*}
\ps H(p^{un}_j(s))&=\langle \nabla H, \ps p^{un}_j(s)\rangle =\langle \nabla H, -\nabla \re(e^{-i\alpha_{j}^{un}(s)} W)\rangle\leq 0,\\
\ps H(p^{st}_k(s))&=\langle \nabla H, \ps p^{st}_k(s)\rangle =\langle \nabla H, -\nabla \re(e^{-i\alpha_{k}^{st}(s)} W)\rangle\leq 0,
	\end{align*}
because $\alpha_j^{un}(s)\in [\eta_j-\pi, \pi]$ and $\alpha_k^{st}(s)=[\theta_j,\pi]$. This implies that 
\[
H(x_j)\geq H(x_l)\geq H(x_k),
\]
or equivalently $k\leq l\leq j$. 
\end{proof}

Lemma \ref{GF.L.3} has an immediate corollary. 
\begin{corollary}\label{GF.C.4} For all $R\gg \pi$, $\Delta_R(U_j, S_k)$ is trivial if $j<k$. Since $\delta H_{jk}^R \equiv 0$ when $j=k$, $\Delta_R(U_j, S_j)=\sG^{j}\Delta(U_j, S_j)\cong \BK\cdot e_{x_j}$ is generated the quasi-unit $e_{x_j}$, i.e., the constant soliton at $x_j \in \Crit(W)$; cf. Proposition \ref{prop:Quasi-Units}.
\end{corollary}

\begin{lemma}\label{GF.L.4} There exists a constant $C$ independent of $R$ or the perturbation 1-form $ \delta H_{jk}^R$ such that for all $1\leq j, k\leq m$, $R\geq 2R_0$ and $p(s)\in \FC^{l}(U_j, S_k;\fa_{jk}^{R})$, we have 
	\begin{equation}\label{GF.E.5}
	|\CA_{W,\fa_{jk}^R}(p)+R\cdot H(x_l)|<C.
	\end{equation}
\end{lemma}

\begin{proof} The formula of $\CA_{W,\fa_{jk}^R}(p)$ is given by \eqref{ActionFunctional1}; then \eqref{GF.E.5} can be proved as follows. By Lemma \ref{GF.L.2}, the function
	\[
	\im(e^{-i\alpha_{jk}^R(s)}W(p(s)))=-H(p(s)),\ s\in [\pi, R]_s
	\]
	is ``almost constant" on this interval. In particular, \eqref{GF.E.4} implies that the difference
	\[
	|(R-\pi)\cdot H(x_l)-\int_\pi^R H(p(s))ds|<C_1
	\]
	is uniformly bounded by some constant $C_1>0$. Similarly, the differences 
	\[
	\int_{-\infty}^0 |\im (e^{-i\eta_j}(W(p(s))-W(x_j)))| \text{ and }	\int_{R+\pi}^{\infty} |\im (e^{-i\theta_k}(W(p(s))-W(x_k)))|.
	\]
	are controlled by \eqref{GF.E.2} and \eqref{GF.E.3}. Finally, the integral $\int_{\R_s} p^*\lambda_M$ is also bounded for all $R\gg \pi$. Indeed, if $p_n(s)\in \FC^l(U_j, S_k; \fa_{jk}^{R_n})$ is a converging subsequence in Lemma \ref{GF.L.1} as $R_n\to \infty$, then 
	\[
	\int_{\R_s}p_n^*\lambda_M\to 	\int_{\R_s}(p_j^{un})^*\lambda_M+\int_{\R_s}(p_k^{st})^*\lambda_M,
	\]
	as $n\to\infty$. This follows from the convergence in $C^\infty_{loc}$-topology in Lemma \ref{GF.L.1} and the decay estimate in Lemma \ref{GF.L.2}. In fact, one can show that 
	\[
	\CA_{W,\fa_{jk}^{R_n}}(p_n)+(R_n-\pi)\cdot H(x_l)\to \CA_{W,\fa_j^{un}}(p_j^{un})+\CA_{W,\fa_k^{st}}(p_k^{st})
	\]
	as $n\to\infty$, where $\fa_j^{un}=(\pi, \alpha_j^{un}(s), \beta, \epsilon_*,\delta H\equiv 0)$ and $\fa_k^{st}=(\pi, \alpha_k^{st}(s), \beta, \epsilon_*,\delta H\equiv 0)$. This proves \eqref{GF.E.5}.
\end{proof}

Note that the moduli space $\cM(p_-, p_+)$ in \eqref{E1.18} is alway empty if
\[
\CA_{W,\fa^R_{jk}}(p_-)< \CA_{W,\fa^R_{jk}}(p_+),
\]
so the Floer differential on $\Delta_R(U_j, S_k)$ can only increase the value of $\CA_{W,\fa_{jk}^R}$. Lemma \ref{GF.L.4} then implies that for all $R\gg \pi$ and $p_\pm\in \FC^{l_\pm}$, $\cM(p_-, p_+)$ is empty if $l_-> l_+$. Hence
\begin{equation}\label{GF.E.8}
\Delta^{(l)}_R(U_j, S_k)\colonequals \bigoplus_{n=1}^l \sG^n\Delta_R(U_j, S_k), 0\leq n\leq m
\end{equation}
is a subcomplex of $\Delta(U_j, S_k)$, and 
\[
\sG^l\Delta_R(U_j, S_k)=\Delta^{(l)}_R(U_j, S_k)/\Delta^{(l-1)}_R(U_j, S_k) , 1\leq l\leq m
\]
is the associated graded complex. By Lemma \ref{lemma:Vanishing}, for $j\neq k$, $H(\Delta_R(U_j, S_k))$ is independent of $R\geq \pi$ and is always trivial. However, since the Floer data are chosen differently here, the complex $\Delta_R(U_j, S_k)$ might be pretty large when $j>k$; see Figure \ref{Pic25}.

\subsection{The filtered bimodule} Having specified the Floer data for any pairs of objects in $\Th_\pi\sqcup \Th_0$, the next step is to define the $A_\infty$-operations 
\begin{align*}
\mu_{\Delta_R}^{r|1|s}:\hom_{\sA}(S_{k_{r-1}}, S_{k_r})\otimes \cdots&\otimes \hom_{\sA}(S_{k_0}, S_{k_1})\otimes \Delta_R(U_{j_s}, S_{k_0})\\
&\otimes \hom_{\sB}(U_{j_{s-1}}, U_{j})\otimes \cdots \otimes \hom_{\sB}(U_{j_0}, U_{j_1})\to \Delta_R(U_{j_0}, S_{k_r})
\end{align*}
for all $r+s\geq 1$ such that 
\begin{equation}\label{GF.E.21}
\mu_{\Delta_R}^{r|1|s}(a_r,\cdots, a_1, x, b_s,\cdots, b_1)\in \Delta_R^{(l)}(U_{j_0}, S_{k_r}).
\end{equation}
if $x\in \Delta^{(l)}_R(U_{j_s}, S_{k_0})$. This proves that each $\Delta^{(l)}_R$ is a sub-bimodule of $\Delta_R$. As usual we carry out the construction for the extended $A_\infty$-category $\sE_R\colonequals \sE_{\Delta_R}$ with $\Ob \sE_R=\Th\colonequals \Th_\pi\sqcup \Th_0$. \eqref{GF.E.21} will then follow from an energy inequality, but this is only possible if the quadratic differentials and cobordism data for $\sE_R$ are chosen carefully. We shall use the metric ribbon tree $\CT^{m,m}_R$ to define $\sE_R$. 

\medskip


For any $R\geq 4\pi$, denote by $Z_R=\R_t\times [2\pi, R-2\pi]_s$ the infinite strip of width $R-4\pi$ equipped with the product metric. Let $(S,\phi_R)$ be any $(d+1)$-pointed disk equipped with an $S$-compatible quadratic differential that is $\epsilon$-close to $\CT_R^{d_1,d_2}$ with $d_1+d_2=d+1$ and $d_1, d_2\geq 1$, and let $\{\iota_k\}_{k=0}^{d}$ be a set of strip-like ends adapted to $\phi_R$. For $0<\epsilon\ll 1$, there is an isometric embedding (see Figure \ref{Pic17})
\[
\tau: Z_R=\R_t\times [2\pi, R-2\pi]_s\to S
\]
such that 
\begin{equation}\label{GF.E.10}
\tau(t, s)=\iota_0(t,s),\ \forall t\leq 0
\end{equation}
and for some $\xi_\phi>0$ and $t_\phi>0$ 
\begin{equation}\label{GF.E.11}
\tau(t+t_\phi,s)=\iota_{d_1}(t,s+\xi_\phi),\ \forall t \geq 0.
\end{equation}

 The embedding $\tau$ is uniquely determined by the first property \eqref{GF.E.10}. Denote by $\CT_\phi$ the metric ribbon tree induced by $\phi$, and write $\partial \CT_\phi\cong \partial \CT_R^{d_1,d_2}=(\fo_0,\cdots, \fo_d)$. Then the constant $\xi_\phi$ in \eqref{GF.E.11} determined by the relation 
\[
\xi_\phi=d_{\CT_R^{d_1,d_2}}(\fo_0,\fo_{d_1})-d_{\CT_\phi}(\fo_0,\fo_{d_1});
\]
see Figure \ref{Pic27} below. Since $\CT_\phi$ is $\epsilon$-close to $\CT_R^{d_1,d_2}$ in the sense of Definition of \ref{QD.D.8}, $|\xi_\phi|\leq \epsilon\ll 1$. 

\begin{remark} The author did not figure out a short argument to show that there is a $(\epsilon, \CT^{d_1,d_2}_R)$-consistent section $\phi$ of $\overline{\V}^{d+1}\to \overline{\NR}^{d+1}$ such that  $\xi_{\phi_r}\equiv 0$ for all $r\in \NR^{d+1}$. However, this does not complexify our construction significantly. Readers may pretend that $\xi_\phi\equiv 0$ when first reading this proof. 
\end{remark}

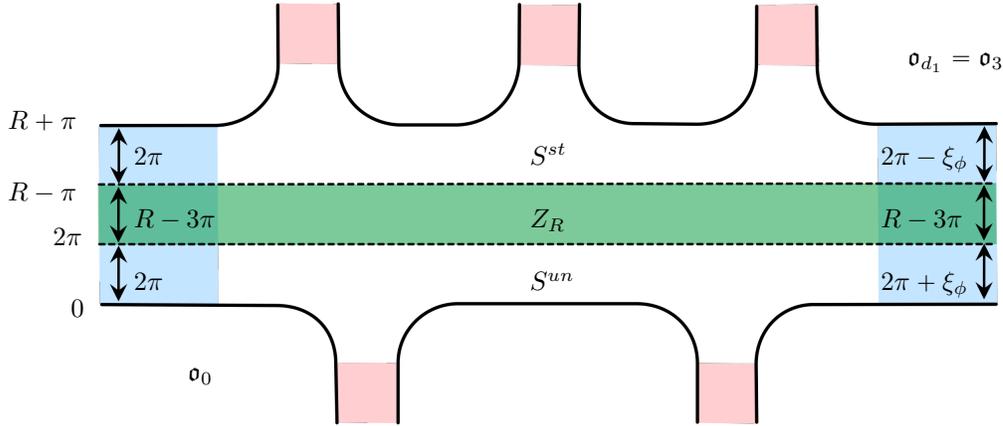
\begin{figure}[H]
	\centering
	\begin{overpic}[scale=.15]{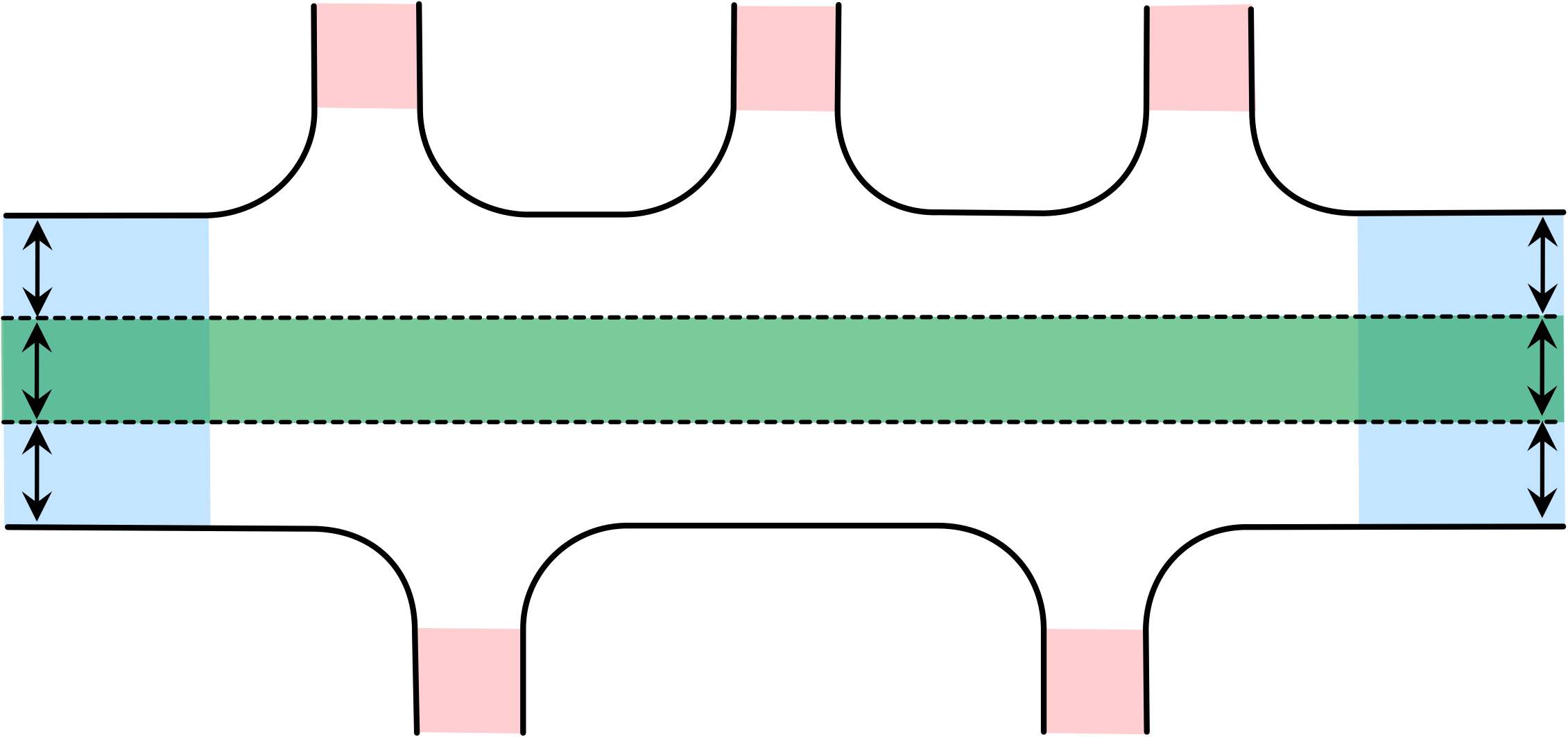}
\put(-3,12){\small $0$}
\put(-5,20){\small $2\pi$}
\put(-10,25){\small $R-\pi$}
\put(-10,33){\small $R+\pi$}
\put(4,15){\small $2\pi$}
\put(4,29){\small $2\pi$}
\put(4,22){\small $R-3\pi$}
\put(87,15){\small $2\pi+\xi_\phi$}
\put(87,29){\small $2\pi-\xi_\phi$}
\put(87,22){\small $R-3\pi$}
\put(10,5){\small $\fo_0$}
\put(90,40){\small $\fo_{d_1}=\fo_3$}
\put(48,29){\small $S^{st}$}
\put(48,22){\small $Z_R$}
\put(48,15){\small $S^{un}$}
	\end{overpic}	
	\caption{The image of $\tau$ is shaded by green. $d_1=3$ and $d_2=4$.}
	\label{Pic27}
\end{figure}

\begin{definition}\label{GF.D.8} Suppose that $S$ is labeled by a subset $A\subset \Th$ with $d_1=|A\cap \Th_\pi|$ and $d_2=|A\cap \Th_0|$. Then a smooth embedding $\Xi^\dagger:S\to \C$ is called \textit{rigid} if 
	\begin{equation}\label{GF.E.9}
	\Xi^\dagger \circ \tau(t,s)=-i\epsilon_*e^{i\beta}\cdot t+\gamma^{A}_R(s)
	\end{equation}
	for all $(t,s)\in \R_t\times [2\pi, R-2\pi]_s$, where $\gamma^{A}_R(s)$ denotes the characteristic curve associated to the unique incoming end of $S$. A cobordism datum  $\fb_{\phi}=(K,\Xi, \delta\kappa,\delta H)$ on the completion $\hat{S}$ is called \textit{rigid} if 
	\begin{itemize}
\item 	$\Xi$ is induced from some rigid embedding $\Xi^\dagger\in \Emb_{K}^S$;
\item  $\tau^*(\delta\kappa)=\im (i\epsilon_* e^{-i\beta} dt\cdot W)$;
\item  	$\delta H\equiv 0$ on the image of $\tau$ and $\|\delta H \|_\infty<\frac{1}{R}$. 
	\end{itemize}

This means that for a rigid cobordism datum the Floer equation \eqref{E1.12} takes the standard form 
\[
\pt P+J\ps P-\nabla H=0 \text{ on }\im \tau.  \qedhere
\]
\end{definition}
\begin{remark}\label{GF.R.9} Intuitively, the surface $S$ is decomposed into three pieces using $\tau$:
	\begin{equation}\label{GF.E.12}
	S=S^{un}\ \bigcup\ Z_R\ \bigcup\ S^{st},
	\end{equation}
so that $\tau: Z_R\to S$ is the inclusion map. $S^{un}$ (resp. $S^{st}$) is the surface that lies below (resp. above) $Z_R$. Each component of \eqref{GF.E.12} then carries a compatible quadratic differential by restricting $\phi$. Take any $q\in \Crit(W)$. Then we label  $S^{un}$ by the ordered set $\Th_\pi\sqcup \{\Lambda_{q, \pi}  \}$ and $S^{st}$ by $\{\Lambda_{q, 2\pi}\}\sqcup \Th_0$, where $\Lambda_{q, \pi}$ (resp. $\Lambda_{q, 2\pi}$) is attached to the boundary component
	\[
	S^{un}\cap Z_R=\tau(\R_t\times \{2\pi\})\ (\text{resp. } 
	S^{st}\cap Z_R=\tau(\R_t\times \{R-2\pi\})).
	\]
A rigid embedding $\Xi^\dagger \in \Emb_{K}^S$ is then obtained by patching some embedding $\Xi^\dagger_{un}\in \Emb_{K}^{S^{un}}$ of $S^{un}$, $\Xi^\dagger_{st}\in \Emb_{K}^{S^{st}}$ of $S^{st}$ (up to a translation) and the standard embedding \eqref{GF.E.9} of $Z_R$. 
\end{remark}

Fix some $R_0\geq 7\pi$, and let $\phi$ be $S$-compatible and $\epsilon$-close to $\CT^{d_1,d_2}_{R_0}$. For any $R\geq R_0$, we shall describe a stretching map that transform $(S,\phi)$ into a singular flat surface $(S_R, \phi_R^\circ)$ which is $\epsilon$-close to $\CT_R^{d_1,d_2}$,
\begin{equation}\label{GF.E.14}
(S,\phi)\mapsto (S_R, \phi_R^\circ).
\end{equation}

This is done by varying the width of $Z_R$ in the decomposition \eqref{GF.E.12} which will potentially change the conformal structure of $S$. Choose a smooth function $\chi:\R_s\to [0,1]$ such that $\supp \chi \subset [3\pi, R_0-3\pi]_s$ and $\int_{\R_s}\chi(s)ds=1$. For any $R\geq R_0$, let $h_R: \R_s\to \R_s$ denote the diffeomorphism such that $\ps h_R(s)=1+\chi(s)\cdot (R-R_0)$ and $h_R(s)=s$ for $s\leq 3\pi$, then 
\[
h_R\big([3\pi, R_0-3\pi]_s\big)=[3\pi, R-3\pi]_s.
\]

Consider a diffeomorphism of the form
\begin{align}\label{GF.E.20}
h_R^*: Z_{R_0}&\to Z_R\\
(t,s)&\mapsto (t, h^*_{R,t}(s))\nonumber
\end{align}
where $h^*_{R,t}: [3\pi, R_0-3\pi]_s\to [3\pi, R-3\pi]_s, t\in \R_t$ is a family of diffeomorphisms such that 
\begin{equation}\label{GF.E.13}
h^*_{R,t}(s)=\left\{\begin{array}{ll}
h_R(s) & \text{ if } t\leq 0,\\
h_R(s+\xi_\phi)-\xi_\phi & \text{ if }t\geq t_\phi,
\end{array}
\right.
\end{equation}
and that $\supp(\ps h^*_R(s)-1)\subset \inte(Z_{R_0})$. The last condition allows us to extend $h_R^*$ to be a diffeomorphism $\R_t\times [0,R_0+\pi]_s\to \R_t\times [0,R+\pi]_s$. The constants $\xi_\phi$ and $t_\phi$ are defined as in \eqref{GF.E.11}. As we will explain shortly, the condition \eqref{GF.E.13} is essential to Lemma \ref{GF.L.10} below, which  makes the gluing construction possible when considering the family version of $\eqref{GF.E.14}$. If $\xi_\phi=0$, then one can simply take $h^*_{R,t}=h_R$ for all $t$, so $h^*_R(t,s)=(t, h_R(s))$. In general, a convenient choice of $h^*_{R,t}$ is to conjugate $h_R$ by a translation, i.e.,
	\[
	h^*_{R,t}(s)=h_R(s+\xi(t))-\xi(t)
	\]
for some smooth function $\xi: \R_t\to \R$. 

\medskip

Suppose that such a diffeomorphism $h^*_R: Z_{R_0}\to Z_R$ in \eqref{GF.E.20} is chosen and that $S$ is decomposed as $S^{un}\cup Z_{R_0}\cup S^{st}$ using $\phi$. Then the $(d+1)$-pointed disk $S_R$ in \eqref{GF.E.14} is obtained by replacing $Z_{R_0}$ by $Z_{R}$:
\[
S_R\colonequals S^{un}\ \bigcup\ Z_R\ \bigcup\ S^{st}
\]
with the default complex structure on each piece. The quadratic differential $\phi_R^\circ$ is defined by patching $\phi|_{S^{un}}$ and $\phi|_{S^{st}}$ with the standard $(2,0)$-tensor $dz\otimes dz$ on $Z_R$. There is a diffeomorphism
\begin{equation}\label{GF.E.17}
h_{R,\phi}^\circ: S\to S_R
\end{equation}
which is identity on $S^{un}, S^{st}$ and is equal to $h_{R}^*$ on $Z_{R_0}$. Intuitively, the map $h^\circ_{R,\phi}$ is stretching the width of the infinite strip $Z_{R_0}$ and is clearly not conformal; see Figure \ref{Pic28}. In particular, $(h_{R,\phi}^\circ)^*(\phi_R^\circ)\neq \phi$. The next two lemmas are immediate from our construction. 

\begin{lemma}\label{GF.L.10} Let $\{\iota_k\}_{k=0}^d$ be the set of strip-like ends adapted to $\phi$ on $S$ which was used to define the embedding $\tau: Z_{R_0}\to S$ in \eqref{GF.E.10}. Then there exists a set of strip-like ends $\{\iota_k^R\}_{k=0}^d$ on $S_R$ adapted to $\phi_R^\circ$ such that the following diagrams commute: for $k=0, d_1$,
	\begin{equation}\label{GF.E.15}
\begin{tikzcd}
\R^\pm_t\times [0,R_0+\pi]_s\arrow[r,"\iota_k"]\arrow[d, "{\Id\times h_R^*}"] & S\arrow[d,"{h^\circ_{R,\phi}}"]\\
\R^\pm_t\times [0,R+\pi]_s\arrow[r,"\iota_k^R"] & S_R,
\end{tikzcd}
	\end{equation}
	and for all $k\neq 0,d_1$,
	\begin{equation}\label{GF.E.16}
\begin{tikzcd}
\R^+_t\times [0,\pi]_s\arrow[r,"\iota_k"]\arrow[d, equal] & S\arrow[d,"{h^\circ_{R,\phi}}"]\\
\R^+_t\times [0,\pi]_s\arrow[r, "{\iota_k^R}"]  & S_R.
\end{tikzcd}
	\end{equation}
In particular, we have $t_k^+(\phi_R^\circ)=t_k^+(\phi)$ for all $R\geq R_0$ and $0\leq k\leq d$, where $t_k^+$ is the length of the singular flat metric defined by \eqref{QD.E.18}. Finally, for all $0\leq j\leq d_1-1$ and $d_1\leq k\leq d$, we have 
\[
d_{\CT_\phi}(\fo_j, \fo_k)=d_{\CT_{\phi_R}}(\fo_j, \fo_k)+(R-R_0),
\]
while $d_{\CT_\phi}(\fo_{j'}, \fo_{k'})=d_{\CT_{\phi_R^\circ}}(\fo_{j'}, \fo_{k'})$ for any other pairs $(j',k')$. Hence, $\phi_R$ is $\epsilon$-close to $\CT^{d_1,d_2}_R$. 
\end{lemma}

\begin{figure}[H]
	\centering
	\begin{overpic}[scale=.123]{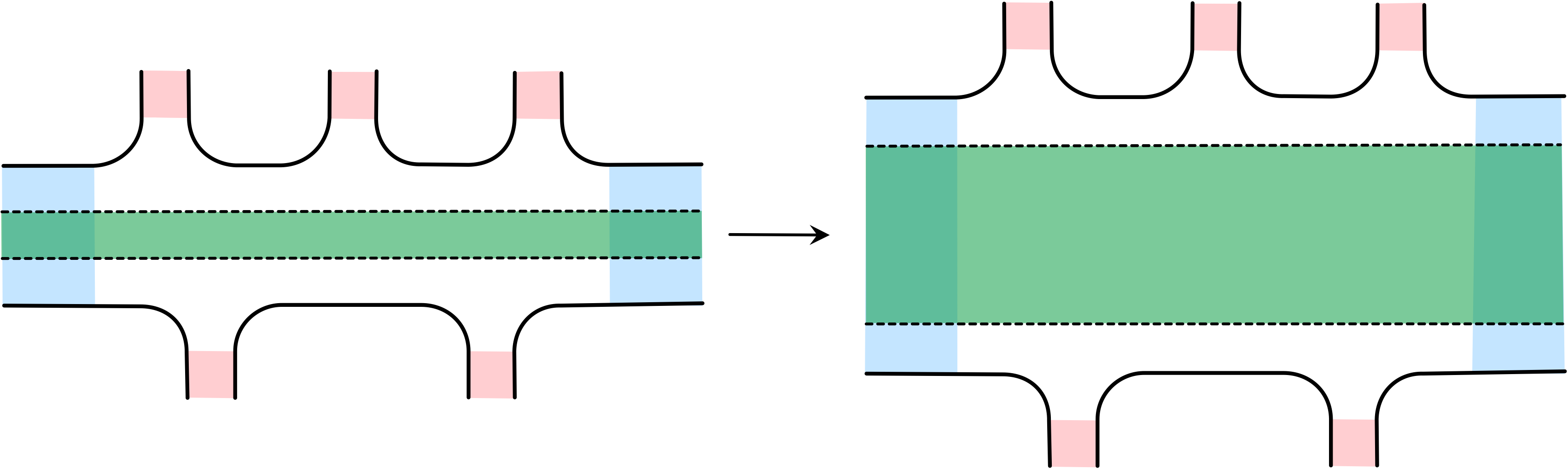}
		\put(47,17){$h^\circ_{R,\phi}$}
		\put(21,14.3){\small $Z_{R_0}$}
		\put(76,14.3){\small $Z_R$}
	\end{overpic}	
	\caption{The stretching map $h^\circ_{R,\phi}: S\to S_R$.}
	\label{Pic28}
\end{figure}

\begin{lemma}\label{GF.L.11} Suppose that $S$ is labeled by a subset $A\subset\Th$ with $d_1=|A\cap \Th_\pi|\geq 1$, $d_2=|A\cap \Th_0|\geq 1$. Then for any rigid cobordism datum $\fb_\phi=(K, \Xi,\delta\kappa, \delta H)$ in the sense of Definition \ref{GF.D.8}, one can construct a rigid cobordism datum $\fb_{\phi_R^\circ}=(K, \Xi_R, \delta\kappa_R, \delta H_R)$ such that $(h_{R,\phi}^\circ)^*\delta\kappa_R=\delta\kappa$ on $S^{un}\sqcup S^{st}$ and 
	\[
	\Xi_R^\dagger\circ h^\circ_{R,\phi}(z)=\left\{\begin{array}{ll}
	\Xi^\dagger(z) &\text{ if  }z\in S^{un},\\
\Xi^\dagger(z) + (R-R_0)&\text{ if } z\in S^{st}. 
	\end{array}
	\right.
	\]
	Moreover, on $Z_R\subset S_R$, $(\Xi^\dagger_R, \delta\kappa_R)$ is specified by Definition \ref{GF.D.8}; cf. the equation \eqref{GF.E.23} and Remark \ref{GF.R.9}. However, the perturbation 1-form $\delta H_R$ is not necessarily related to $\delta H$.
\end{lemma}

\subsection{The family version}\label{SecGF.4} Let us now consider the family version of Lemma \ref{GF.L.10} and \ref{GF.L.11}. For each $d\geq 2$, write $\NR^{d+1}_{\base}$ for another copy of $\NR^{d+1}$, which comes with smooth bundles $\Sch^{d+1}_{\base}\to \NR^{d+1}_{\base}$ and $\V^{d+1}_{\base}\to \NR^{d+1}_{\base}$. Let $\phi^{\base}=(\phi_{T}^{\base})$ be a $(\epsilon,\CT^{m,m}_{R_0})$-consistent section of $\overline{\V}^{2m}_{\base}\to \overline{\NR}^{2m}_{\base}$ equipped with a set of strip-like ends $(\iota_{T}^{\base})$ adapted to $\phi^{\base}$. This means that for any subset $A\subset \Th$ with $|A|\geq 3$, there is
\begin{itemize}
\item a smooth section $\psi_A^{\base}$ of $\V^{|A|}_{\base}\to \NR^{|A|}_{\base}$ which is $\epsilon$-close to the subtree $\CT_{R_0}^{d_1,d_2}$, where $d_1=|A\cap \Th_\pi|$ and $d_2=|A\cap \Th_0|$;
\item a set of strip-like ends $(\iota_{A, k}^{\base})_{k=0}^{|A|-1}$ of $\Sch^{|A|}_{\base}\to \NR^{|A|}_{\base}$ adapted to $\psi_A^{\base}$;
\end{itemize}
and they satisfy the conditions in Lemma \ref{QD.L.14}. When $A\subset \Th_\pi$ or $\Th_0$, we require that $\psi_{A}^{\base}$ is the restriction of  $\phi_*$ in the definition of $\sA$ and $\sB$ in \eqref{GF.E.18}, and so are the strip-like ends.

\medskip

Now for any subset  $A\subset \Th$ with $|A|\geq 3$ and $d_1,d_2\geq 1$, the constants in \eqref{GF.E.11} define smooth functions:
\[
 r\in \NR_{\base}^{|A|}\mapsto (\xi_{\psi_{A,r}^{\base}}, t_{\psi_{ A,r}^{\base}}),
\]
with $|\xi_{\psi_{A, r}^{\base}}|<\epsilon\ll 1$. For any $R\geq R_0$, by choosing a smooth family of diffeomorphisms \eqref{GF.E.20} 
\begin{equation}\label{GF.E.19}
Z_{R_0}\times  \NR_{\base}^{|A|}\to Z_R
\end{equation}
parametrized by $r\in \NR_{\base}^{|A|}$, we obtain smooth maps
\[
g^\circ_{R,A}: \NR^{|A|}_{\base}\to \NR^{|A|} \text{ and } \psi^\circ_{R,A}: \NR^{|A|}_{\base}\to \V^{|A|},
\]
which are pointwise defined by \eqref{GF.E.14}. $\psi^\circ_{R,A}$ is a smooth section of $(g^\circ_{R,A})^*\V^{|A|}$. By Lemma \ref{GF.L.10}, $\psi^\circ_{R,A}$ is $\epsilon$-close to $\CT^{d_1,d_2}_R$. Moreover, there is a smooth bundle map 
\[
h_{R,A}^\circ: \Sch^{|A|}_{\base}\to (g^\circ_{R,A})^*\Sch^{|A|}
\]
which is defined fiberwise by \eqref{GF.E.17}, along with a set of $\psi^\circ_{R,A}$-adapted strip-like ends $(\iota^\circ_{R,A,k})_{k=0}^{|A|-1}$ for the family $(g^\circ_{R,A})^*\Sch^{d+1}$ such that the diagrams \eqref{GF.E.15} and \eqref{GF.E.16} commute fiberwise. 

Finally, if $A\subset \Th_\pi$ or $\Th_0$, we simply set $g^\circ_{R,A}=\Id$ and $\psi^\circ_{R,A}=\psi^{\base}_A$. The next lemma then follows from the same line of arguments as in Lemma \ref{QD.L.14} and Lemma \ref{FSL.8}, and the proof is omitted. Note that the commutativity of \eqref{FSE.28} follows from \eqref{GF.E.15} and \eqref{GF.E.16}.
\begin{lemma} By choosing the family of diffeomorphisms \eqref{GF.E.19} inductively for any $A\subset \Th$, one can construct an $(\epsilon, \CT^{m,m}_R)$-consistent family $\phi^\circ_R$ of quadratic differentials $($in the sense of Definition \ref{FS.D.23}$)$ parametrized by a smooth degree one map $g^\circ_R: \overline{\NR}^{2m}_{\base}\to \overline{\NR}^{2m}$ such that for any subset $A\subset \Th$ with $|A|\geq 3$, the component $(\psi_{R,A}^\circ, g^\circ_{R,A})$ of $(\phi_R^\circ, g^\circ_R)$ associated to $A$ is given by the construction above. $g^\circ_R$ is called the stretching map.
\end{lemma}
\begin{remark}\label{GF.R.13} If $\xi_{\psi_{A, r}^{\base}}\equiv 0$ for all $A$ and $r\in \NR^{|A|}_{\base}$, then one can simply take \eqref{GF.E.19} to be a constant family of diffeomorphisms by taking $h_R^*(t,s)= (t, h_R(s))$ at each fiber . However, author did not figure out a simple criterion to make this possible.
\end{remark}
\begin{remark} Since the reference section $\phi^{\base}$ is only $\epsilon$-close to $\CT^{m,m}_{R_0}$, the smooth map $g^\circ_R$ might not be injective in general. If $\phi^{\base}$ is the section $\phi_{\CT^{m,m}_{R_0}}$ in Lemma \ref{QD.L.10}, then one can verify that $g^\circ_R$ is bijective, and $\phi^\circ_R$ is the pull-back of $\phi_{\CT^{m,m}_R}$. This heuristic motivates the stretching construction. As mentioned in Remark \ref{QD.R.15}, the section $\phi_{\CT^{m,m}_{R_0}}$ might not be smooth in general, so we have taken a detour to make this idea work.
\end{remark}

\begin{lemma}\label{GF.L.15} For any $R_0\geq 7\pi$, any $(\epsilon, \CT^{m,m}_{R_0})$-consistent section $\phi^{\base}=(\phi^{\base}_T)$ of $\overline{\V}^{2m}_{\base}\to \overline{\NR}^{2m}_{\base}$ and a set of adapted strip-like ends $(\iota^{\base}_{T})$, one can find a consistent family $(\fb_A^{\base})_{A\subset \Th}$ of cobordism data such that for all $A$ and $r\in \NR^{|A|}_{\base}$, $\fb_{A,r}^{\base}$ is rigid in the sense of Definition \ref{GF.D.8}. 
\end{lemma} 

\begin{proof} The proof follows the same line of arguments as in Lemma \ref{FSL.8}. We leave the details to the reader.
\end{proof}

We shall use the rigid cobordism data obtained in Lemma \ref{GF.L.15} to define the directed $A_\infty$-category
\[
\sE_{R_0}\colonequals \sE(\Th; \CT^{m,m}_{R_0}, \phi^{\base}, (\fb_A^{\base})_{A\subset \Th})
\]
and denote the associated $(\sA,\sB)$-bimodule by $\Delta_{R_0}$. In principle, one may increase the value of $R_0$ in Lemma \ref{GF.L.15} to define $\Delta_{R}$ for all $R\geq R_0$, but the lower bound $C_{\Th}$ in Lemma \ref{FS.L.12} may also change accordingly, as the choice of the cobordism data could be pretty random. However, this problem is remedied by Lemma \ref{GF.L.11}. 

\begin{lemma}\label{GF.L.16} Suppose that $R_0\geq 7\pi$ has been fixed, and a consistent family $(\fb^{\base}_A)_{A\subset \Th}$ of rigid cobordism data has been chosen as in Lemma \ref{GF.L.15}. Then for any $R\geq R_0$, one can construct a consistent family of rigid cobordism data $(\fb_{R,A})_{A\subset \Th}$ for the pair $(\phi^\circ_R, g^\circ_R)$ using Lemma \ref{GF.L.11}. This means that for any $A\subset \Th$ with $|A|\geq 3, d_1,d_2\geq 1$, and any $r\in \NR_{\base}^{|A|}$, 
	\[
	\fb_{R,A,r}=(K_A, \Xi_{R,A,r}, \delta\kappa_{R,A,r}, \delta H_{R,A,r})
	\]
is obtained from $\fb^{\base}_{A,r}=(K_A, \Xi^{\base}_{A,r},\delta\kappa^{\base}_{A,r}, \delta H_{A,r}^{\base})$ using Lemma \ref{GF.L.11}, and $\fb_{R,A,r}=\fb^{\base}_{A,r}$ are set equal if $A\subset \Th_\pi$ or $A\subset \Th_0$. In particular, the constant $K_A$ is independent of $R\geq R_0$, and the phase pair $(\Xi_{R,A,r}, \delta\kappa_{R,A,r})$ is determined canonically by $(\Xi^{\base}_{A,r},\delta\kappa^{\base}_{A,r})$. Then the perturbation 1-forms $\delta H_{R,A,r}$ are chosen generically to make moduli spaces regular. 
\end{lemma}

Following \eqref{FSE.26}, for any $R\geq R_0$, we construct the finite directed $A_\infty$-category
\[
\sE_R\colonequals \sE(\Th; \CT^{m,m}_R, \phi^\circ_R, (\fb_{R,A})),\ \Th=\Th_\pi\sqcup \Th_0.
\]
using the cobordism data $(\fb_{R,A})$ obtained from Lemma \ref{GF.L.16}. This $A_\infty$-category satisfies the following important property.

\begin{lemma}\label{GF.L.17} For any $A\subset\Th$ with $|A|\geq 3$, there exists a constant $C_A>0$ with the following property. For any $R\geq R_0$, if the moduli space 
	\[
	\M_{\NR^{|A|}_{\base}}(\{p_k\}_{k=0}^{|A|-1};\fb_{R,A})
	\]
	is nonempty, then we have an energy inequality:
	\begin{equation}\label{GF.E.22}
	\CA_{W,\fa_0}(p_0)-\sum_{k=1}^{|A|-1}\CA_{W,\fa_k}(p_k)\geq -C_A. 
	\end{equation}
	Here the parameter space $\NR^{|A|}_{\base}$ is labeled by $A$, and $\fa_k, p_k, 0\le k\leq |A|-1$ denote the Floer datum and respectively the soliton associated to each strip-like end of $  \Sch^{|A|}_{\base}\to \NR^{|A|}_{\base}$.
\end{lemma}

\begin{proof}[Proof of Lemma \ref{GF.L.17}] It suffices to verify that the estimate \eqref{GF.E.22} holds if the moduli space
	\[
		\M_{S_{g^\circ_{R,A}(s)}}(\{p_k\}_{k=0}^{|A|-1};\fb_{R,A,r})
	\]
	is nonempty for some $r\in \NR^{|A|}_{\base}$. For any $r$ in a fixed compact subset of $\NR^{|A|}_{\base}$, this lower bound is provided by \eqref{FSE.27}, because the constant $K_A$ and the length $t^+_k(\psi_{R,A,r}^\circ)=t^+_k(\psi^{\base}_{A,r}), 0\leq k\leq |A|-1$ are independent of $R\geq R_0$ by our construction. See Lemma \ref{GF.L.10} and \ref{GF.L.11}. If $r$ is close to a lower stratum of $\overline{\NR}^{|A|}_{\base}$, one may repeat the proof of Lemma \ref{FS.L.12} using the commutative diagram \eqref{FSE.28}. We leave details to the readers.
\end{proof}

If $d_1,d_2\geq 1$, then the $d_1$-th outgoing end $\iota_{A, d_1}$ is distinguished: for any $k\neq d_1$, the action functional $\CA_{W,\fa_k}$ and the soliton $p_k$ are independent of $R$. Hence by \eqref{GF.E.22}
\[
\CA_{W,\fa_0}(p_0)-\CA_{W,\fa_{d_1}}(p_{d_1})\geq -C_A'
\]
for a constant $C_A'>0$ independent of $R$, if the moduli space $	\M_{\NR^{|A|}_{\base}}(\{p_k\}_{k=0}^{|A|-1};\fb_{R,A})$ is non-empty. On the other hand, $\CA_{W,\fa_{d_1}}(p_{d_1})$ and $\CA_{W,\fa_0}(p_0)$ will blow up linearly as $R\to\infty$, and by Lemma \ref{GF.L.4}, the leading coefficients depend the filtration level to which $p_0, p_{d_1}$ belong. Thus we have proved

\begin{corollary}\label{GF.C.17} Let $\Delta_R$ denote  the $(\sA,\sB)$-bimodule induced by $\sE_R$. Then for all $R\gg R_0$, the relation \eqref{GF.E.21} holds for $\Delta_R$, and therefore $\Delta^{(l)}_R$ is a submodule of $\Delta_R$ for all $0\leq l\leq m$.
\end{corollary}

\section{Koszul Duality: Proof of Theorem \ref{FS.T.3}}\label{SecPT}

In this section, we prove Theorem \ref{FS.T.3} assuming some properties about the pairs-of-pants cobordism map, namely, Lemma \ref{PT.L.3} and \ref{PT.L.4}.  These properties will be proved later in Section \ref{SecVT} using the vertical gluing theorem. Recall from Example \ref{FS.EX.2} that a thimble $\Lambda_{q,\theta},\ q\in \Crit(W)$ is called \textit{stable} if $\theta\in (0, \theta_\star)$ and respectively \textit{unstable} if $\theta\in (\pi, \pi+\theta_\star)$.  Let $\Th_{st}$ (resp. $\Th_{un}$) be any admissible set of stable (resp. unstable) thimbles, and set $\widetilde{\sA}=\sE(\Th_{st})$ and $\widetilde{\sB}=\sE(\Th_{un})$. Then the construction from Section \ref{SecGF} can be generalized to define a filtered $(\widetilde{\sA},\widetilde{\sB})$-bimodule $\widetilde{\Delta}_R$ for any $R\gg \pi$. The critical points of $W$ associated to $\Th_{st}$ or $\Th_{un}$ might not be mutually distinct. This bimodule $\widetilde{\Delta}_R$ potentially depends on many auxiliary data, but its existence is already useful for the proof of Theorem \ref{FS.T.3}. 

\subsection{Some reductions}\label{SubsecPT.1} Recall that the diagonal $(\sA,\sB)$-bimodule $\Delta$ defines an $A_\infty$-functor $r_{\Delta}:\sA\to \sQ_r\colonequals \rfmod(\sB)$, and let $\sS_k$ denote the image of $S_k$ under $r_{\Delta}$. By Remark \ref{FS.R.4}, it remains to show that for all $k_1<k_2$, the chain map 
\begin{equation}\label{PT.E.1}
(r_\Delta)^1: \hom_{\sA}(S_{k_1}, S_{k_2})\to \hom_{\sQ_r}(\sS_{k_1},\sS_{k_2})
\end{equation}
is a quasi-isomorphism, and by Lemma \ref{AF.L7.3} and Proposition \ref{FS.P.25}, we may verify \eqref{PT.E.1} for one convenient choice of $(\phi^\circ, g^\circ)$ and $(\fb_A)$. We shall use the construction from Section \ref{SecGF}. Fix some $R\gg \pi$ such that Corollary \ref{GF.C.17} holds, so the diagonal bimodule $\Delta=\Delta_R$ carries an energy filtration:
\[
0\subset \Delta^{(0)}\subset \Delta^{(1)}\subset \Delta^{(2)}\subset \cdots\subset \Delta^{(m)}=\Delta,
\]

 Hence each $\sS_k, 1\leq k\leq m$ is filtered by a sequence of finite right $\sB$-modules, denoted by 
\[
0\subset \sS_k^{(0)}\subset \sS_k^{(1)}\subset \sS_k^{(2)}\subset \cdots \subset \sS_k^{(m)}=\sS_k.
\] 

By Lemma \ref{GF.L.3},  $\sS_k(U_j)=0$ if $j<k$, and for $j\geq k$, 
\[
0=\sS_k^{(0)}(U_j)=\cdots =\sS_k^{(k-1)}(U_j)\subset \sS_k^{(k)}(U_j)\subset \cdots \subset \sS_k^{(j)}(U_j)=\cdots= \sS_k^{(m)}(U_j).
\]
Since the $A_\infty$-category $\sB$ is directed, this means that the thimbles $U_j, 1\leq j\leq k_1-1$ are not involved in the definition of $\hom_{\sQ_r}(\sS_{k_1},\sS_{k_2})$. Thus one may focus on the last $(m-k_1+1)$ critical points of $W$ and carry out the proof instead for the $A_\infty$-subcategories $\widetilde{\sA}, \widetilde{\sB}$ with 
\[
\Ob \widetilde{\sB}\sqcup\Ob \widetilde{\sA}: U_m\prec \cdots \prec U_{k_1}\prec S_{k_1}\prec\cdots \prec S_m,
\]
and $\Delta$ restricts to the filtered $(\widetilde{\sA},\widetilde{\sB})$-bimodule $\widetilde{\Delta}$. With this in mind, we shall assume from now on that $k_1=1$ and $k_2=k>1$. The general case is then addressed using this reduction.

\subsection{Associated Graded Modules} Write $\sG^l \sS_k=\sS_k^{(l)}/\sS^{(l-1)}_k,\ 0\leq l\leq m$ for the associated graded submodule of $\sS_k$. Using the filtration on $\sS_1$, we obtain a ladder consisting of exact triangles on the right:
\begin{equation}\label{PT.E.2}
\begin{tikzcd}
\hom_{\sA} (S_1, S_k)\arrow[r,"(r_{\Delta})^1"]\arrow[rd]\arrow[rddd,bend right]\arrow[rdddd,bend right]& \hom_{\sQ_r}(\sS_1^{(m)},\sS_k) \arrow[d]&\hom_{\sQ_r}(\sG^m\sS_1, \sS_k)\arrow[l] \\
& \hom_{\sQ_r}(\sS_1^{(m-1)},\sS_k)\arrow[ru,"{[1]}"'] \arrow[d]&\hom_{\sQ_r}(\sG^{m-1}\sS_1, \sS_k)\arrow[l] \\
&\cdots\arrow[d]\arrow[ru,"{[1]}"'] \arrow[d]& \cdots\arrow[l]\\
&\hom_{\sQ_r}(\sS_1^{(1)},\sS_k)\arrow[ru,"{[1]}"']\arrow[d]  &\hom_{\sQ_r}(\sG^1\sS_1, \sS_k)\arrow[l]\\
&\hom_{\sQ_r}(\sS_1^{(0)},\sS_k)\arrow[ru,"{[1]}"']. & 
\end{tikzcd}
\end{equation}

To show that the top horizontal arrow $(r_{\Delta})^1$ is a quasi-isomorphism in \eqref{PT.E.2}, we have to understand how $H(\hom_{\sQ_r}(\sS^{(l)}_1, \sS_k))$ changes at each step. The next lemma is crucial to this proof.
\begin{lemma}\label{PT.L.1} For any $1\leq n\leq l\leq m$, the right $\sB$-module $\sG^l\sS_n$ is quasi-isomorphic to $D\hom_{\sA}(S_n, S_l)\otimes\sU_l$, where $D\hom_{\sA}(S_n, S_l)$ denotes the dual complex of $\hom_{\sA}(S_n, S_l)$ and $\sU_l$ is the Yoneda image of $U_l\in \Ob \sB$ in $\sQ_r$.
\end{lemma}

By taking $n=1$, Lemma \ref{PT.L.1} implies that
\begin{align}\label{PT.E.3}
H(\hom_{\sQ_r}(\sG^l\sS_1, \sS_k))\cong H(D\hom_{\sA}(S_1, S_l)\otimes\sU_l,\sS_k)&\cong \Hom_{H(\sA)}(S_1, S_k)\otimes H(\sS_k(U_l))\nonumber\\
&=\left\{\begin{array} {ll} 
\Hom_{H(\sA)}(S_1, S_k) & \text{ if }l=k,\\
0 &\text{ otherwise},
\end{array}
\right.
\end{align}
so  in the ladder \eqref{PT.E.2}, we have 
\begin{align}\label{PT.E.4}
H(\hom_{\sQ_r}(\sS^{(k-1)}_1, \sS_k))\cong \cdots \cong H(\hom_{\sQ_r}(\sS^{(0)}_1, \sS_k))&\cong 0,\nonumber\\
H(\hom_{\sQ_r}(\sS^{(m)}_1, \sS_k))\cong \cdots \cong H(\hom_{\sQ_r}(\sS^{(k)}_1, \sS_k))&\cong  H(\hom_{\sQ_r}(\sG^k\sS_1, \sS_k)).
\end{align}

We have to verify the $k$-th row of the ladder \eqref{PT.E.2} is a quasi-isomorphism:
\begin{equation}\label{PT.E.10}
\hom_{\sA} (S_1, S_k)\to \hom_{\sQ_r}(\sS_1^{(k)}, \sS_k)\xleftarrow{\cong} \hom_{\sQ_r} (\sG^{k} \sS_1,\sS_k),
\end{equation}
which fits into a diagram
\begin{equation}\label{PT.E.8}
\begin{tikzcd}
H(\hom_{\sA}(S_1,S_k))\arrow[r] \arrow[d]& H(\hom_{\sQ_r}(\sG^k \sS_1, \sG^k\sS_k))\arrow[d,"\cong"']\\
H(\hom_{\sQ_r}(\sS_1^{(k)},\sS_k))&H(\hom_{\sQ_r}(\sG^{k}\sS_1,\sS_k))\arrow[l,"\cong"'].
\end{tikzcd}
\end{equation}
The right vertical arrow of \eqref{PT.E.8} is induced by the inclusion $\sG^{k}\sS_k=\sS_k^{(k)}\to \sS_k$. Replacing $\sS_k$ by $\sG^k\sS_k$ in \eqref{PT.E.3} with $l=k$ shows that this arrow is an isomorphism. The bottom arrow of \eqref{PT.E.8} is an isomorphism by \eqref{PT.E.4}.  

For any $1\leq l \leq m$, the associated graded bimodule $\sG^l\Delta=\Delta^{(l)}/\Delta^{(l-1)}$ defines an $A_\infty$-functor $r_{\sG^l\Delta}:\sA\to \sQ_r$. In particular, for any $1\leq n\leq l$, there is a chain map
\begin{equation}\label{PT.E.9}
(r_{\sG^l\Delta})^1: \hom_{\sA}(S_n, S_l)\to hom_{\sQ_r}(\sG^l\sS_n, \sG^l\sS_l),
\end{equation}
which defines the top horizontal arrow in \eqref{PT.E.8}. 
\begin{lemma}\label{PT.L.2} For any $1\leq n\leq l\leq m$,  \eqref{PT.E.9} is quasi-isomorphism. 
\end{lemma}

The diagram \eqref{PT.E.8} is commutative: any cycle $a$ in $\hom_{\sA}(S_1, S_k)$ gives an $A_\infty$-homomorphism $\sG^k \sS_1\to \sG^k \sS_k$ which is composed to give 
\[
\sS_1^{(k)}\xrightarrow{\text{quotient}} \sG^k\sS_1\xrightarrow{(r_{\sG^k\Delta})^1(a)} \sG^k \sS_k=\sS_k^{(k)}\xrightarrow{\text{inclusion}} \sS_k. 
\]
This should be compared with 
\[
\sS^{(k)}_1\xrightarrow{\text{inclusion}} \sS_1\xrightarrow{(r_\Delta)^1(a)} \sS_k
\]
under the left vertical arrow of \eqref{PT.E.8}. They are equal because $(r_{\Delta})^1(a)$ preserves the filtration and sends $\sS_1^{(l)}$ to $\sS_k^{(l)}$ for all $0\leq l\leq m$; especially we take $l=k$. Lemma \ref{PT.L.2} implies that the top horizontal arrow in \eqref{PT.E.8} is an isomorphism, and so is the left vertical arrow. Thus to complete the proof of Theorem \ref{FS.T.3}, it remains to verify Lemma \ref{PT.L.1} and \ref{PT.L.2}, which dominates the rest of Section \ref{SecPT}.

\subsection{Proof of Lemma \ref{PT.L.1}: Step 1} Note that for any $1\leq n\leq l\leq m$ the chain map \eqref{PT.E.9} can dualized to obtain an $A_\infty$-homomorphism 
\begin{equation}\label{PT.E.5}
r_n: \sG^l\sS_n\to D\hom_{\sA}(S_n, S_l)\otimes\sG^l\sS_l. 
\end{equation}
More concretely, $r_n$ is defined by the formulae
\begin{align*}
(r_n)^d:\sG^l\sS_n(U_{j_d-1})\otimes \hom_{\sB}(U_{j_{d-2}},U_{j_{d-1}})\otimes\cdots\otimes \hom_{\sB}(U_{j_0},U_{j_1})&\to D\hom_{\sA}(S_n, S_l)\otimes\sG^l\sS_l(U_{j_0})\\
(x, b_{d-1},\cdots,b_1)&\mapsto\mu_{\sG^l\Delta}^{1|1|d-1}(\ \cdot\ , x, b_{d-1},\cdots, b_1)
\end{align*}
for every $d\geq 1$ and $m\geq j_0>j_1>\cdots>j_{d-1}\geq 1$. When $n=l$, the strict unitality of $\sG^l\Delta$ implies that $(r_l)^d=0$ when $d\geq 2$ and $(r_l)^1$ is an isomorphism between vector spaces. The next lemma shows that \eqref{PT.E.5} is also a quasi-isomorphism when $n<l$.
\begin{lemma}\label{PT.L.3} For any $R\gg \pi$, $1\leq j\leq m$ and $1\leq n<l$, the chain map
	\[
	\sG^l\Delta(U_j, S_n)\xrightarrow{x\mapsto \mu_{\sG^l\Delta}^{1|1|0}(\cdot, x) } D\hom_{\sA}(S_n, S_l)\otimes \sG^l\Delta(U_j, S_l)
	\]
	induced by the pair-of-pants cobordism is a quasi-isomorphism, and therefore $r_n$ is a quasi-isomorphism between right $\sB$-modules.
\end{lemma} 

\begin{figure}[H]
	\centering
	\begin{overpic}[scale=.15]{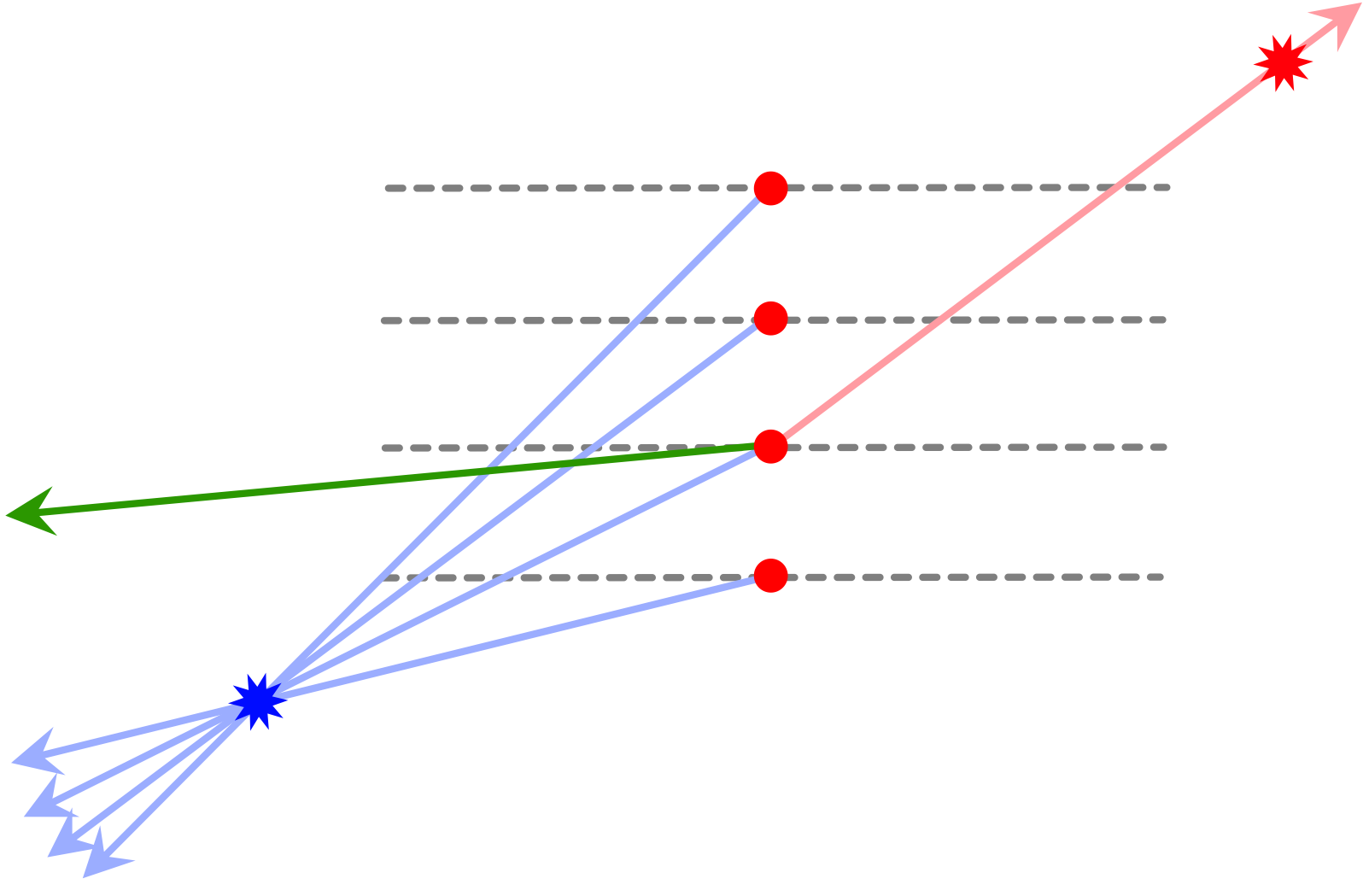}
		\put(103,67){\small $S_2$}
		\put(4,-4){\small $U_4$}
		\put(-1,-1){\small $U_3$}
		\put(-4,4){\small $U_2$}
		\put(-5,9){\small $U_1$}
		\put(55,55){\small $H(q_4)$}
		\put(55,45){\small $H(q_3)$}
		\put(55,28){\small $H(q_2)$}
		\put(55,18){\small $H(q_1)$}
		\put(-5,27){\small $V_2$}
	\end{overpic}	
	\caption{$V_2$ is a geometric representative of $\sU_2$ (with $l=2$).}
	\label{Pic29}
\end{figure}
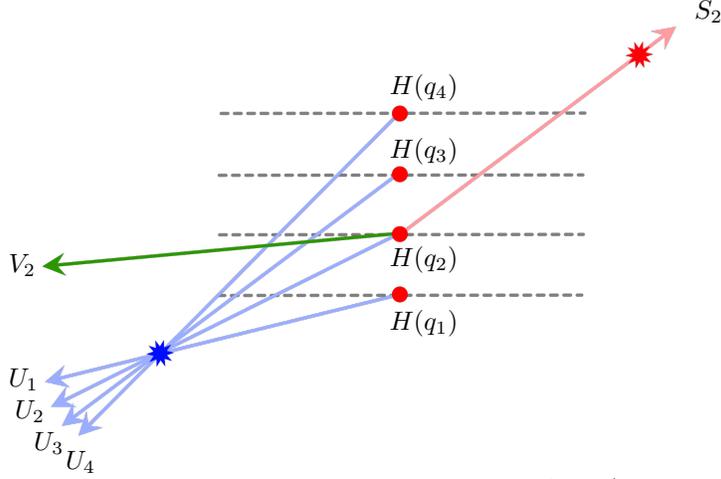

\subsection{Proof of Lemma \ref{PT.L.1}: Step 2} It remains to verify that $\sG^l\sS_l$ is quasi-isomorphic to $\sU_l$. To this end, we have to enlarge the $A_\infty$-category $\sB$ slightly. Let $x_l$ denote the $l$-th critical point of $W$. Now consider the set of unstable thimbles
$
\Th_{un}=( U_m,\cdots, U_2,U_1,V_l)
$
with $V_l=\Lambda_{x_l, \eta_l'}$ and $\eta_l'\in (0,\eta_1)$. See Figure \ref{Pic29}. Let $\Th_{st}=(S_l)$, and define
\[
\widetilde{\sA}=\sE(\Th_{st}) \text{ and }\widetilde{\sB}=\sE(\Th_{un}). 
\] 
Then the union $\Th_{un}\sqcup\Th_{st}$ is ordered as 
\[
 U_m\prec\cdots\prec U_2\prec U_1\prec V_l\prec S_l. 
\]

Following Section \ref{SecGF}, one can construct a directed $A_\infty$-category $\sE_{\widetilde{\Delta}}$ inducing a filtered $(\widetilde{\sA},\widetilde{\sB})$-bimodule $\widetilde{\Delta}$. This point of view is not so useful as $\widetilde{\sA}$ consists of a single object and is somewhat degenerate, and we will think of this differently this time: let $\widehat{\sA}$ denote the $A_\infty$-subcategory of $\sE_{\widetilde{\Delta}}$ with $\Ob\widehat{\sA}=(V_l, S_l)$. Then $\sE_{\widetilde{\Delta}}$ defines an $(\widehat{\sA},\sB)$-bimodule $\widehat{\Delta}$ together with an $A_\infty$-functor 
\[
r_{\widehat{\Delta}}: \widehat{\sA}\to \sQ_r=\rfmod(\sB).
\]
Let $\sV_l$ denote the image of $V_l$ under $r_{\widehat{\Delta}}$. Unlike the case for stable thimbles, $\sV_l$ does not carry a filtration, whereas the unique morphism space $\hom_{\widehat{\sA}}(V_l, S_l)=\widehat{\Delta}(V_l, S_l)$ of $\widehat{\sA}$ is filtered by the action functional:
\begin{equation}\label{PT.E.6}
0=\hom_{\widehat{\sA}}^{(0)}(V_l, S_l)\subset \hom_{\widehat{\sA}}^{(1)}(V_l, S_l)\subset\cdots \subset  \hom_{\widehat{\sA}}^{(m)}(V_l, S_l)=\hom_{\widehat{\sA}}(V_l, S_l).
\end{equation}
 
 This filtration is preserved by the chain map
\[
(r_{\widehat{\Delta}})^1: \hom_{\widehat{\sA}}^{(n)}(V_l, S_l)\to \hom_{\sQ_r}(\sV_l, \sS_l^{(n)}),\ 0\leq n\leq m,
\]
therefore giving rise to chain maps between the associated graded complexes/modules:
\[
(r_{\widehat{\Delta}})^1:\sG^n\hom_{\widehat{\sA}}(V_l, S_l)\colonequals\hom_{\widehat{\sA}}^{(n)}(V_l, S_l)/\hom_{\widehat{\sA}}^{(n-1)}(V_l, S_l)\to \hom_{\sQ_r}(\sV_l, \sG^n\sS_l),\ 0\leq n\leq m
\]
denoted also by $(r_{\widehat{\Delta}})^1$. Note that the filtration \eqref{PT.E.6} is rather boring, since by Lemma \ref{GF.L.3}, $\sG^n\hom_{\widehat{\sA}}(V_l, S_l)=0$ if $n\neq l$ and is generated by the quasi-unit $e_{x_l}$ if $n=l$ by Corollary \ref{GF.C.4}. We focus on this special filtration level $n=l$ in the next lemma. 

\begin{lemma}\label{PT.L.4} For any $1\leq j\leq m$, consider the product map
	\[
\mu^{1|1|0}_{\widehat{\Delta}}:\sG^l\hom_{\widehat{\sA}}(V_l, S_l) \otimes 	\widehat{\Delta}(U_j, V_l)\to 	\sG^l\widehat{\Delta}(U_j, S_l).
	\]
	Then $\mu^{1|1|0}(e_{x_l},\cdot): \widehat{\Delta}(U_j, V_l)\to 	\widehat{\Delta}(U_j, \sG^l\sS_l)$ is a quasi-isomorphism, and therefore $(r_{\widehat{\Delta}})^1(e_{x_l})\in \hom_{\sQ_r}(\sV_l, \sG^l\sS_l)$ is a quasi-isomorphism. 
\end{lemma}

\subsection{Proof of Lemma \ref{PT.L.1}: Step 3} As the final step, we verify that $\sV_l$ is quasi-isomorphic to $\sU_l$. In some senses, $V_l$ is the geometric representative of the Yoneda image $\sU_l$. At this point, there is no energy filtration involved, and the $A_\infty$-category $\widetilde{\sB}$ is an invariant of $\Th_{un}$. Consider the Yoneda embedding functor of $\widetilde{\sB}$ composed with the restriction functor to the subcategory $\sB$:
\[
\widetilde{\sB}\xrightarrow{r_{\widetilde{\sB}}} \rfmod(\widetilde{\sB})\xrightarrow{\text{restriction}} \sQ_r=\rfmod(\sB). 
\]
This defines a chain map 
\[
r_l: \hom_{\widetilde{\sB}}(U_l, V_l)\to \hom_{\sQ_r}(\sU_l, \sV_l). 
\]
By Lemma \ref{lemma:CanonicalGenerator}, $\hom_{\widetilde{\sB}}(U_l, V_l)$ is generated by the quasi-unit $e_{x_l}$. Note that for any $1\leq j\leq m$, the multiplication map 
\[
[\mu_{\widetilde{\sB}}^2(e_{x_l},\cdot )]: H(\hom_{\widetilde{\sB}}(U_j, U_l))\to H(\hom_{\widetilde{\sB}}(U_j, V_l))
\]
is always an isomorphism. For $j<l$, this follows from Proposition \ref{prop:Quasi-Units}. For  $j=l$, this follows from strict unitality. For $j>l$, the domain is trivial by definition, while $H(\hom_{\widetilde{\sB}}(U_j, V_l))$ is trivial by Lemma \ref{lemma:Vanishing}. This implies that $r_l(e_{x_l})\in \hom_{\sQ_r}(\sU_l, \sV_l)$ is quasi-isomorphism. Thus we have completed the proof of Lemma \ref{PT.L.1} assuming Lemma \ref{PT.L.3} and \ref{PT.L.4}.

\subsection{Proof of Lemma \ref{PT.L.2}} Consider the commutative diagram:
\begin{equation}\label{PT.E.7}
\begin{tikzcd}
\hom_{\sA}(S_n, S_l)\arrow[r,"(r_{\sG^l\Delta})^1"]\arrow[d,"a\mapsto a\otimes e_{\sG^l\sS_l}"] & \hom_{\sQ_r}(\sG^l \sS_n, \sG^l\sS_l)\\
\hom_{\sA}(S_n, S_l)\otimes \hom_{\sQ_r}(\sG^l\sS_l, \sG^l \sS_l)\arrow[r, equal]& \hom_{\sQ_r}(D\hom_{\sA}(S_n,S_l)\otimes \sG^l\sS_l, \sG^l\sS_l)\arrow[u,"{\mu^2_{\sQ_r}(\cdot ,r_n)}"']. 
\end{tikzcd}
\end{equation}
By Lemma \ref{PT.L.1}, the $A_\infty$-homomorphism $r_n$ defined by \eqref{PT.E.5} is a quasi-isomorphism, and so is the right vertical arrow in \eqref{PT.E.7} by Lemma \ref{AF.L7.8}. The left vertical arrow in \eqref{PT.E.7} is defined by the formula $a\mapsto a\otimes e_{\sG^l\sS_l}$ and is also a quasi-isomorphism. Indeed, by Lemma \ref{PT.L.1},
\[
H(\hom_{\sQ_r}(\sG^l\sS_l, \sG^l \sS_l))\cong H(\hom_{\sQ_r}(\sU_l, \sU_l)) \cong H(\hom_{\sB}(U_l, U_l))\cong \BK
\]
is generated by the identity homomorphism $e_{\sG^l\sS_l}$. This completes the proof of Lemma \ref{PT.L.2}. 

\begin{remark} The proof of Theorem \ref{FS.T.3} relies essentially only on a few formal properties of the filtered bimodule $\Delta=\Delta_R$, namely, Lemma \ref{GF.L.3}, Lemma \ref{PT.L.3} and Lemma \ref{PT.L.4}, whereas the explicit construction of $\Delta_R$ is somewhat irrelevant. In particular, Lemma \ref{PT.L.3} and \ref{PT.L.4} are statements about the cohomological categories. We have taken a detour in Section \ref{SecGF} to show that the operations $\mu^{r|1|s}_{\Delta_R}, r,s\geq 0$ preserve the geometric filtration on $\Delta_R(U_j, S_k)$. Readers should feel free to propose a more direct and less technical route as long as these properties can be verified.  
\end{remark}

\section{The Vertical Gluing Theorem}\label{SecVT}

\subsection{The idea of the vertical gluing theorem}\label{SecVT.1} To finish the proof of Theorem \ref{FS.T.3}, we prove Lemma \ref{PT.L.3}, \ref{PT.E.4} and Proposition \ref{prop:Quasi-Units} (on quasi-units) in this section using a vertical gluing theorem. To illustrate, we focus on the case of Lemma \ref{PT.L.3} and first explain the basic idea. Consider the $A_\infty$-subcategory $\sE_0$ of $\sE_R=\sE_{\Delta_R}$ with $\Ob\sE_0=(U_j, S_n ,S_l), 1\leq k<n\leq m$. One may apply the rotational operation $\sR$ in Remark \ref{FS.R.6} to $\sE_0$ and obtain a new $A_\infty$-category $\sE_1$ with $\Ob\sE_1=(S_n, S_l, U_j')$, $U_j'\colonequals\Lambda_{x_j, \eta_j-2\pi}$. $\sE_1$ is modeled on the metric ribbon tree $\CT^{2,1}_R$. Unwinding the definition, Lemma \ref{PT.L.3} is equivalent to the statement that
\begin{equation}\label{VT.E.0}
\mu^2_{\sE_1}: 
\sG^l\hom_{\sE_1}(S_l, U_j')\otimes \hom_{\sE_1}(S_n, S_l)\to \sG^l\hom_{\sE_1}(S_n, U_j'). 
\end{equation}
is a quasi-isomorphism for all $R\gg \pi$. The roles between stable and unstable thimbles are symmetric here. To ease our notation, we shall verify this instead for the original $A_\infty$-category. By rotating Figure \ref{Pic25} by $180^\circ$, it is clear that \eqref{VT.E.0} is equivalent to the following lemma.

\begin{lemma}\label{VT.L.0} For any triple of objects $ U_j\prec U_l\prec S_k$ in $\sE_R=\sE_{\Delta_R}$ and $R\gg \pi$, the map
\begin{equation}\label{VT.E.1}
\mu^{0|1|1}_{\Delta_R}: 
\sG^l\Delta_R (U_l, S_k)\otimes \hom_{\sB}(U_j, U_l) \to \sG^l\Delta_R (U_j, S_k). 
\end{equation}
is a quasi-isomorphism. This lemma is non-trivial only when $k\leq l\leq j$; see Figure \ref{Pic33}.
\end{lemma}

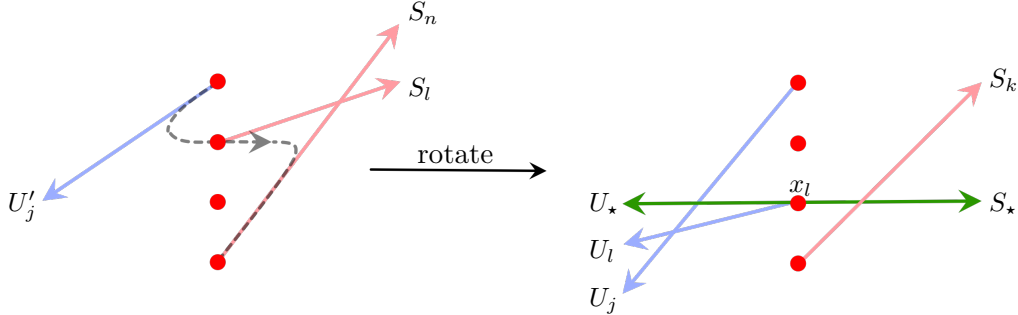
\begin{figure}[H]
	\centering
	\begin{overpic}[scale=.15]{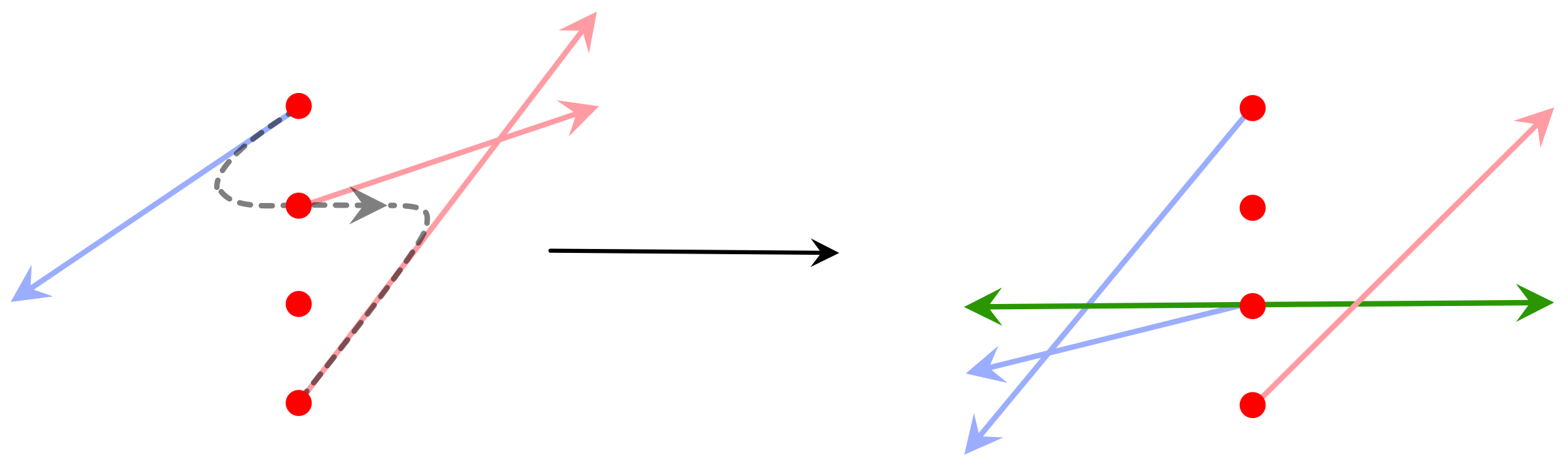}
		\put(-3,10){\small $U_j'$}
		\put(39, 30){\small $S_n$}
		\put(39,22){\small $S_l$}
\put(100,23){\small $S_k$}
\put(79,12){\small $x_l$}
\put(58,0){\small $U_j$}
\put(58, 5){\small $U_l$}
\put(58,10){\small $U_\star$}
\put(100,10){\small $S_\star$}
\put(40,15){rotate}
	\end{overpic}	
	\caption{The thimbles $U_\star$ and $S_\star$.}
	\label{Pic33}
\end{figure}

Lemma \ref{VT.L.0} is a property about the cohomological category, which depends only on a single $(2+1)$-pointed disk. In this case, the choice of the perturbation 1-forms $\delta H^R$ can be more flexible: the norms in \eqref{GF.E.7}  may not tend to zero as $R\to\infty$; it suffices to require that the left hand side of \eqref{GF.E.7} is bounded by a fixed constant $\epsilon'$. If $R$ is chosen sufficiently large depending on this $\epsilon'$, the continuation method can be used to show that $\eqref{VT.E.1}$ is independent of these perturbations up to chain homotopy (however, the phase pairs are remained fixed in this process). Lemma \ref{VT.L.0} is then proved by choosing some special $\delta H^R$ so that the complex $\sG^l\Delta (U_j, S_k)$ can be described rather concretely. To set the stage, consider the thimbles 
\[
U_\star=\Lambda_{x_l, \pi} \text{ and } S_\star\colonequals \Lambda_{x_l, 2\pi}. 
\]
 Choose admissible Floer data $\fa_j^{un}=(\pi,\alpha_{j}^{un}, \beta_*,\epsilon_*, \delta H_j^{st})$ and $\fa^{st}_k=(\pi, \alpha_{k}^{st}, \beta_*,\epsilon_*, \delta H_k^{st})$, one for each pair $(U_j, U_\star)$ and $(S_\star, S_k)$, where $\alpha_{j}^{un}(s),\alpha_k^{st}(s)$ are given as in \eqref{GF.E.26} and 
 $\beta_*,\epsilon_*$ as in \eqref{GF.E.25}. The perturbation 1-forms $\delta H_j^{st},\delta H_k^{un}$ are supported on $[0,\pi]_s$ as usual. Then these Floer data can be concatenated to give a Floer datum $\alpha_{jk}^R=(R+\pi, \alpha_{jk}^R, \beta_*,\epsilon_*,\delta H^{R}_{jk})$ for the pair $(U_j, S_k)$ with
 \begin{equation}\label{VT.E.31}
\delta H^R_{jk}=\left\{\begin{array}{ll}
\delta H_k^{st}& \text{ if }s\in [\pi, R+\pi]_s,\\
\delta H_j^{un}&\text{ if }s\in  [0, \pi]_s,\\
0 &\text{ otherwise.}
\end{array}
\right.
 \end{equation}
Under these assumptions, Lemma \ref{GF.L.1} can be refined as follows.

\begin{lemma}\label{VT.L.1} For all $k\leq l\leq j$ and $R\gg\pi$, there is a gluing bijection
	\begin{align*}
\FC(S_\star, S_k; \fa_{k}^{st})\times \FC(U_j, U_{\star}; \fa_{j}^{un})&\to 	\FC^l(U_j, S_k; \fa^R_{jk})\\
(p^{st}_k, p^{un}_j)&\mapsto  p^{st}_k\circ _Rp^{un}_j
	\end{align*}
	such that $p^{st}_k\circ_R p^{un}_j\to ( p^{st}_k, p^{un}_j)$ in the sense of Lemma \ref{GF.L.1} as $R\to\infty$ and $\gr(p^{st}_k\circ _Rp^{un}_j)=\gr(p^{st}_k)+\gr(p^{un}_j)$. Thus there is an isomorphism between graded vector spaces 
	\begin{equation}\label{VT.E.22}
	\widehat{\Phi}_{jk}:	 \Ch^*_\natural (S_\star, S_k; \fa_{k}^{st})\otimes \Ch^*_\natural (U_j, U_{\star}; \fa_{j}^{un})\to \sG^l \Delta(U_j, S_k).
	\end{equation}
\end{lemma}

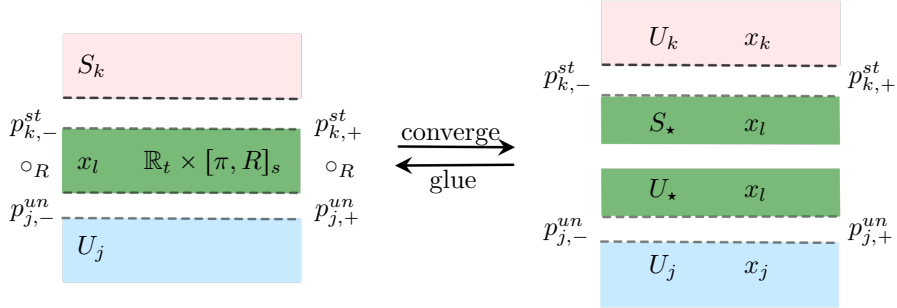
\begin{figure}[H]
	\centering
	\begin{overpic}[scale=.15]{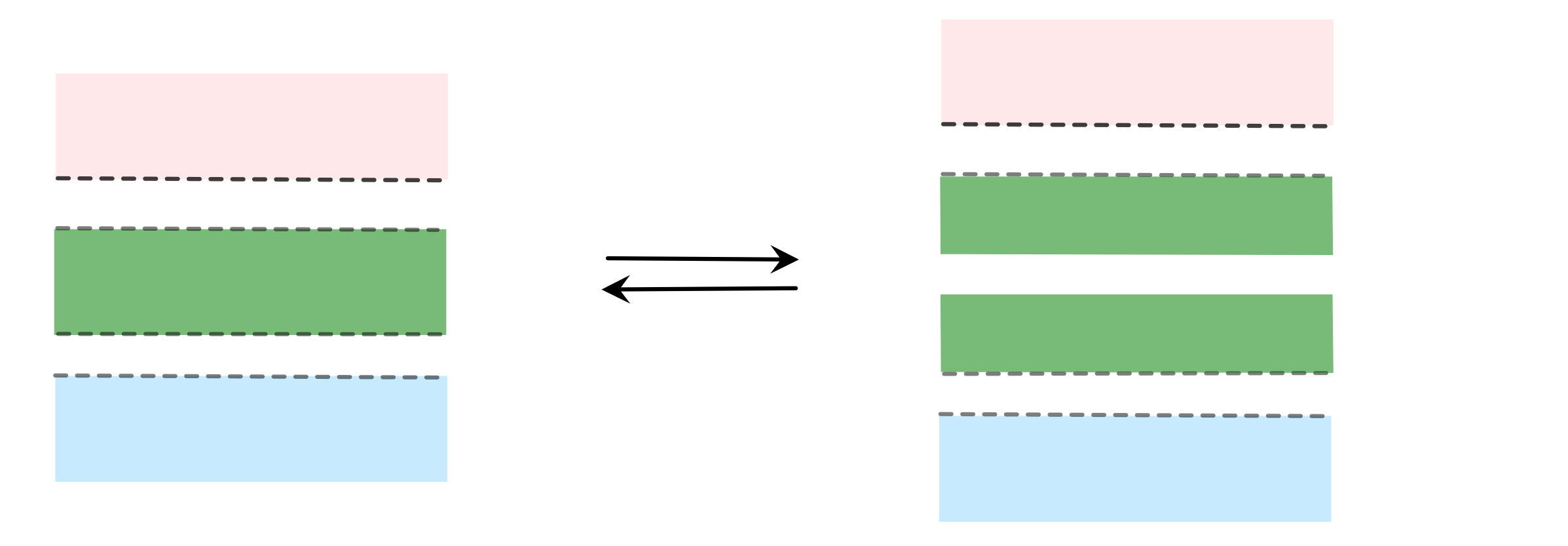}
		\put(12,16){\small$\R_t\times [\pi, R]_s$}
		\put(5,7){\small$U_j$}
		\put(5,16){\small$x_l$}
\put(5,26){\small$S_k$}
\put(65,5){\small$U_j$}
\put(65,13){\small $U_\star$}
\put(65, 20){\small $S_\star$}
\put(65, 29){\small $U_k$}
\put(75,5){\small$x_j$}
\put(75,13){\small $x_l$}
\put(75, 20){\small $x_l$}
\put(75, 29){\small $x_k$}
\put(54,25){\small $p^{st}_{k,-}$}
\put(86,25){\small $p^{st}_{k,+}$}
\put(54,9){\small $p^{un}_{j,-}$}
\put(86,9){\small $p^{un}_{j,+}$}
\put(-2,20){\small $p^{st}_{k,-}$}
\put(30,20){\small $p^{st}_{k,+}$}
\put(-2,11){\small $p^{un}_{j,-}$}
\put(30,11){\small $p^{un}_{j,+}$}
\put(-1,15.5){\small $\circ_R$}
\put(31,15.5){\small $\circ_R$}
\put(39,19){\small converge}
\put(42,14){\small glue}
	\end{overpic}	
	\caption{Gluing Floer differentials.}
	\label{Pic34}
\end{figure}

\begin{theorem} [The Vertical Gluing Theorem I]\label{VT.T.2} The map $\widehat{\Phi}_{jk}$ is also an isomorphism between complexes. This means that for any solitons $p_{k,\pm}^{st}\in \FC(S_\star, S_k; \fa_{k}^{st})$ and $p_{j,\pm}^{un}\in \FC(U_j, U_{\star}; \fa_{j}^{un})$, there is a gluing bijection between finite sets 
\begin{equation}\label{VT.E.9}
\cM(p_{k,-}^{st}\circ _Rp_{j,-}^{un}, p_{k,+}^{st}\circ_R p_{j,+}^{un};\fa_{jk}^R)\cong \left\{
\begin{array}{ll}
\cM(p_{k,-}^{st}, p_{k,+}^{st})& \text{ if }p_{j,-}^{un}=p_{j,+}^{un},\\
\cM(p_{j,-}^{un}, p_{j,+}^{un}) &\text{ if }p_{k,+}^{st}=p_{k,+}^{st},\\
\emptyset & \text{ otherwise }
\end{array}
\right.
\end{equation}
where $\cM(\cdots; \fa_{jk}^R)$ is the moduli space contributing to the differential map \eqref{E1.18} for the associated graded complex $\sG^l\Delta(U_j, S_k)$$($so the expected dimension of  $\cM=\M/\R$ $=0)$. 
\end{theorem}

Theorem \ref{VT.T.2} says that not only $\alpha$-solitons can be glued vertically in the direction of $s$ but so are $\alpha$-instantons. This is possible because the drop of the action functional along such an $\alpha$-instanton does not increase as $R\to\infty$, so we have uniform energy control. This implies that any $\alpha^{R}_{jk}$-instanton must decay exponentially towards the critical point $x_l\in \Crit(W)$ as $z\in \R_t\times [\pi, R]_s$ approaches the middle line $\R_t\times \{\frac{R+\pi}{2}\}$, and this convergence is uniform in the time variable; see Lemma \ref{VT.L.5} below. Thus one may establish a compactness theorem and compare the moduli space for a finite large $R$ with that of $``R=\infty"$, i.e., when the neck is completely stretched. 

\medskip

Now we explain the case for cobordism maps. The $(2+1)$-pointed disk $S$ carries a unique $S$-compatible quadratical differential $\phi$ with $\CT_\phi=\CT^{2,1}_R$. In terms of the decomposition $S=S^{st}\cup Z_R\cup S^{un}$ in Figure \ref{Pic27}, we label $S^{st}\cong \R_t\times [0,2\pi]_s$ by $(S_\star, S_k)$ and the $(2+1)$-pointed disk $S^{un}$ by $(U_j, U_l, U_\star)$. When restricted on $S^{un}$, the horizontal foliation of $\phi|_{S^{un}}$ is modeled on $\CT^{2,1}_{\pi}$. Fix an admissible continuation datum $\fc$ on $S^{st}$ and a cobordism datum $\fb$ on $S^{un}$ to define the maps
\begin{align*}
\Cont: & \Ch^*_\natural (S_\star, S_k; \fa_{k}^{st})\to  \Ch^*_\natural (S_\star, S_k; \fa_{k}^{st}),\\
\mu^2:& \Ch^*_\natural (U_l, U_{\star}; \fa_{l}^{un}) \otimes \hom_{\sB}(U_j, U_l)\to \Ch^*_\natural (U_j, U_{\star}; \fa_{j}^{un}).
\end{align*}
One may simply take $\fc$ be the identity cobordism, so $\Cont=\Id$; but this is unnecessary for the next gluing theorem. Concatenate $\fc$ and $\fb$ to obtain a rigid cobordism datum on $(S,\phi)$, denoted by $\fc\circ_R\fb$. The next theorem says that any moduli space contributing to \eqref{VT.E.1} is also obtained by a gluing construction when $R\gg \pi$.
\begin{theorem}[The Vertical Gluing Theorem II]\label{VT.T.3} There is a gluing bijection between finite sets:
	\[
\M(p_{k,0}^{st}\circ _Rp_{j,0}^{un},\ p_1,\ p_{k,2}^{st}\circ _Rp_{l,2}^{un}: \fc\circ_R\fb)\xrightarrow{\cong}\M(p_{k,0}^{st},p_{k,2}^{st};\fc)\times  \M(p_{j,0}^{un}, p_1, p_{l,2}^{un};\fb)
	\]
	where $\M(\cdots; \fc\circ_R\fb)$ $($resp. $\M(\cdots; \fc)$ and $\M(\cdots;\fb))$ is any moduli space contributing to the map $\mu_{\Delta}^{0|1|1}$ $($resp. to $\Cont$ and $\mu^2$$)$; see Figure \ref{Pic35} below. Thus the chain map \eqref{VT.E.1} fits into a commutative diagram 
	\begin{equation}
\begin{tikzcd}[column sep=-4em]
\Ch^*_\natural (S_\star, S_k; \fa_{k}^{st})\otimes\Ch^*_\natural (U_l, U_{\star}; \fa_{l}^{un}) \otimes \hom_{\sB}(U_j, U_l) \arrow[rd,"\Cont\otimes \mu^2"]\arrow[dd,"\widehat{\Phi}_{lk}\otimes \Id"]& \\
& \Ch^*_\natural (S_\star, S_k; \fa_{k}^{st})\otimes\Ch^*_\natural (U_j, U_{\star}; \fa_{j}^{un}) \arrow[d,"\widehat{\Phi}_{jk}"]\\
\sG^l\Delta (U_l, S_k)\otimes \hom_{\sB}(U_j, U_l) \arrow[r,"\mu^{0|1|1}_{\Delta}"]&\sG^l\Delta (U_j, S_k). 
\end{tikzcd}\label{VT.E.2}
	\end{equation}
\end{theorem}

Theorem \ref{VT.T.3} implies $[\mu_{\Delta}^{0|1|1}]=[\Cont]\otimes [\mu^2]$. Since $[\Cont]=\Id$, and $\HFF_\natural^*(U_l, U_{\star})=H(\Ch^*_\natural (U_l, U_{\star}; \fa_{l}^{un}))$ is generated by the quasi-unit $e_{x_l}$ at $x_l$, $[\mu^2]$ is an isomorphism by Proposition \ref{prop:Quasi-Units}, and so is $[\mu_{\Delta}^{0|1|1}]$. This proves Lemma \ref{VT.L.0} and so Lemma \ref{PT.L.3}.

 Lemma \ref{PT.L.4} is in fact the special case of Lemma \ref{VT.L.0} with $j=l$. Finally, Proposition \ref{prop:Quasi-Units} (the property about quasi-units) is proved by another application of the vertical gluing theorem and is reduced the easier case, Lemma \ref{L3.7}, which has been verified directly.

 The proof of Theorem \ref{VT.T.3} is almost identical to that of Theorem \ref{VT.T.2} and is omitted in this paper. In fact, Theorem \ref{VT.T.3} is slightly simpler, since we are gluing moduli spaces of index $0$, in which case the equation does not carry a translation symmetry. 
 
 The rest of this section is therefore devoted to the technical proof of Theorem \ref{VT.T.2}. Readers who feel comfortable with this heuristic may proceed directly to the next part.

\begin{figure}[H]
	\centering
	\begin{overpic}[scale=.15]{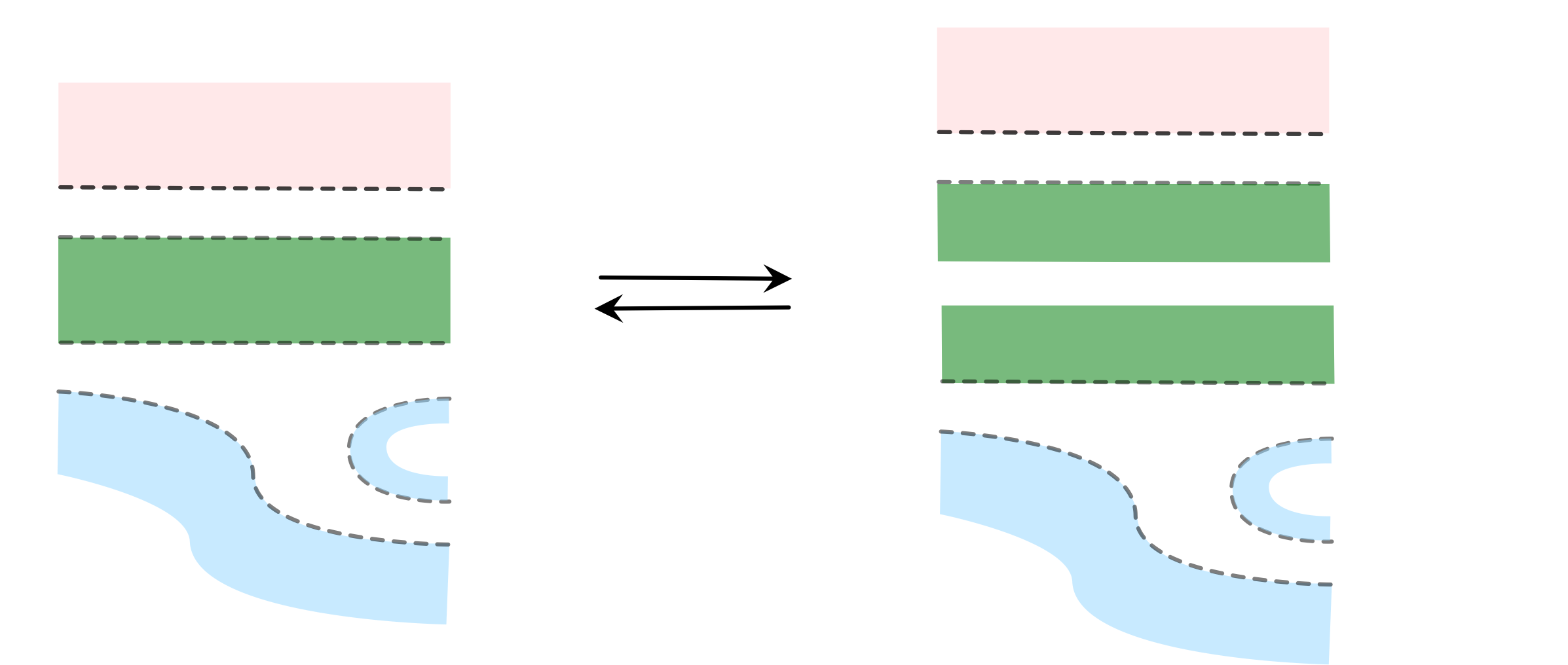}
		\put(12,23){\small$\R_t\times [\pi, R]_s$}
		\put(5,13){\small$U_j$}
		\put(5,23){\small$x_l$}
		\put(5,32){\small$S_k$}
		\put(65,9){\small$U_j$}
		\put(65,20){\small $U_\star$}
		\put(65, 27){\small $S_\star$}
		\put(65, 36){\small $S_k$}
		\put(75,2){\small$x_j$}
		\put(75,20){\small $x_l$}
		\put(75, 27){\small $x_l$}
		\put(75, 36){\small $x_k$}
		\put(54,32){\small $p^{st}_{k,0}$}
		\put(86,32){\small $p^{st}_{k,2}$}
		\put(54,16){\small $p^{un}_{j,0}$}
		\put(86,16){\small $p^{un}_{l,2}$}
		\put(-2,27){\small $p^{st}_{k,0}$}
		\put(30,27){\small $p^{st}_{k,2}$}
		\put(-2,18){\small $p^{un}_{j,0}$}
		\put(30,18){\small $p^{un}_{l,2}$}
		\put(30,8){\small $p_1$}
		\put(86,6){\small $p_1$}
		\put(-1,22.5){\small $\circ_R$}
		\put(31,22.5){\small $\circ_R$}
		\put(39,26){\small converge}
		\put(42,20){\small glue}
	\end{overpic}	
	\caption{Gluing the cobordism map.}
	\label{Pic35}
\end{figure}
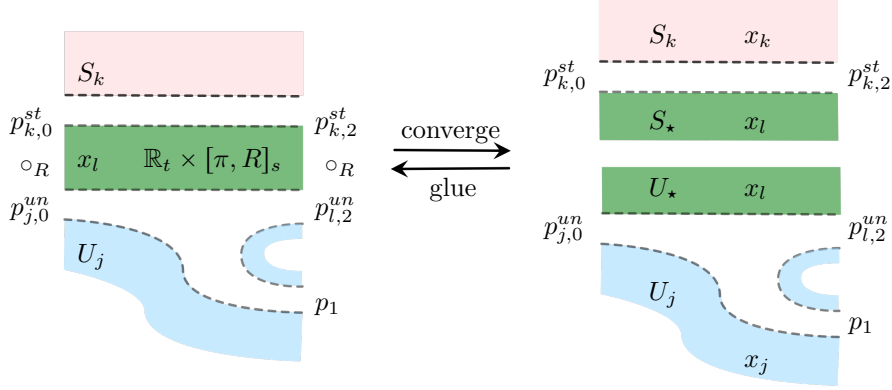
\subsection{Proof of Lemma \ref{VT.L.1}: some estimates about solitons} As usual, these gluing theorems are proved using the implicit function theorem \cite[Lemma 9.4.4]{AD14} or \cite[Propsoiton A.3.4]{MS12}, and the key ingredient is to estimate the right inverse of the linearized operator at a pre-gluing configuration. Although Lemma \ref{VT.L.1} is standard and follows immediately from the argument of \cite[Section 14]{AD14}, we summarize some key estimates as a warmup for the more sophisticated case of Theorem \ref{VT.T.2}. The upshot is that all such estimates must be uniform in the stretching parameter $R$.  From now on we assume that $k=1, j=m$ to simplify our notations; the general case will follow by the same argument. The subscripts $j,k$ are saved later for different purposes. For simplicity, write 
\[
\FC^{st}=\FC(S_\star, S_1;\fa^{st}_1),\ \FC^{un}=\FC(U_m, U_\star;\fa^{un}_m) \text{ and }\FC^R=\FC^l(U_m, S_1;\fa^R_{m1}). 
\]

To start, choose a cutoff function $\chi:\R_s\to [0,1]$ such that $\chi(s)\equiv 0$ if $s\leq 0$ and $\equiv 1$ if $s\geq 1$, and define
\begin{align*}
\chi^{un}_R(s)&\colonequals\chi\big(3-\frac{8s}{R+\pi}\big), & 
\chi^{st}_R(s)&\colonequals\chi\big(\frac{8s}{R+\pi}-5\big).
\end{align*}
Then $(\chi^{un}_R,\chi^{st}_R)(s)\equiv (1,0)$ if $s\leq \frac{R+\pi}{4}$ and $\equiv (0,1)$ if $s\geq \frac{3(R+\pi)}{4}$ and $\equiv (0,0)$ for $s\in [\frac{3(R+\pi)}{8}, \frac{5(R+\pi)}{8}]_s$. For any pair of thimbles $(\Lambda_0, \Lambda_1),\ \Lambda_n=\Lambda_{q_n, \theta_n}$, $n=1,2$, let $\Pa_r(\Lambda_0, \Lambda_0), r\geq 2$ denote the space of smooths paths $\R_s\to M$ which connect $q_0, q_1$ and have finite $L^r_1$-distance to a model path; in contrast to \eqref{E1.19}, we use the $L^r_1$-norm instead of $L^2_k$. Let $R_0>0$. If $p^{un}\in \Pa_r(U_j, U_\star)$ and $p^{st}\in \Pa_r(S_\star, S_k)$ satisfy that $p^{un}(s), p^{st}(s)\in \SO(x_l)$ for all $s\geq R_0$, where $\SO(x_l)$ is a normal neighborhood of $x_l$, then for $R\geq 4R_0$, there is a pre-gluing map 
\begin{equation}\label{VT.E.15}
\Phi_R:  \Pa_r(S_\star, S_1)\times \Pa_r(U_m, U_\star)\dashrightarrow\Pa_r(U_m, S_1)
\end{equation}
such that $(p^{st}, p^{un})$ is sent to the path
\[
\Phi_R(p^{st}, p^{un})(s)\colonequals \left\{
\begin{array}{ll}
p^{st}(s-R) & \text{ if } s\geq \frac{3(\pi+R)}{4},\\
\exp_{x_l}\big(\chi^{un}_R(s)\cdot \exp_{x_l}^{-1}(p^{un}(s))+\chi^{st}_R(s)\cdot \exp_{x_l}^{-1}(p^{st}(s))\big) &\text{ otherwise, }\\
p^{un}(s) & \text{ if }s\leq \frac{\pi+R}{4}.
\end{array}
\right.
\]
By construction, $\Phi_R(p^{st}, p^{un})(s)\equiv x_l$ for $s\in [\frac{3(R+\pi)}{8}, \frac{5(R+\pi)}{8}]_s$. If $p^{st}, p^{un}$ are $\alpha^{st}_{1}$-and $\alpha^{un}_m$-solitons respectively, then $p^{st}(-s), p^{un}(s)\to x_l$ exponentially as $|s|\to\infty$. We use $p^{\pre}_R\colonequals\Phi_R(p^{st}, p^{un})$ as the pre-gluing configuration and construct an actual  $\alpha_{m1}^R$-soliton which is close to $p^{\pre}_R$ using Newton-Picard iteration when $R\gg \pi$. The key ingredient in this argument is the following estimate.

\begin{lemma}\label{VT.L.4} For any $r\geq 2$, there exists a constant $C_1=C_1(r)>0$ with the following property. For any $R\gg \pi$, $p^{st}\in \FC^{st}$ and $p^{un}\in \FC^{un}$, the operator 
	\[
	\Hess \CA_{W,\fa^R_{m1}}(p^{\pre}_R): L^r_1(\R_s; (p^{\pre}_R)^*TM)\to L^r(\R_s; (p^{\pre}_R)^*TM)
	\]
	is invertible for $p^{\pre}_R=\Phi_R(p^{st}, p^{un})$. Moreover, its inverse is bounded by $C_1$ in the operator norm.
\end{lemma}
\begin{proof}[Sketch of Proof] Consider the following truncations of $p^{st}, p^{un}$:
\begin{align*}
p^{un}_{R}(s)&=\left\{\begin{array}{ll}
\Phi_R(p^{st},p^{un})(s) &\text{ if }s\leq \frac{3(R+\pi)}{8},\\
x_l& \text{ otherwise, }
\end{array}
\right.
\\
p^{st}_{R}(s)&=\left\{\begin{array}{ll}
\Phi_R(p^{st},p^{un})(s+R) &\text{ if }s+R\geq \frac{5(R+\pi)}{8},\\
x_l& \text{ otherwise. }
\end{array}
\right.
\end{align*}
For any $R\gg \pi$, one can think of $p^{un}_{R}(s)$ as a section of $L^r_1(\R_s; (p^{un})^*TM)$ using the exponential map along $p^{un}$, and $p^{un}_{R}\to p^{un}$ as $R\to \infty$ in $L^r_1(\R_s; (p^{un})^*TM)$ for all $r\geq 2$. The same holds also for $p^{st}_{R}$. This shows that $\Hess \CA_{W,\fa_{1}^{st}}(p^{st}_{R})$ and  $\Hess \CA_{W,\fa_m^{un}}(p^{un}_{R})$ are invertible with uniformly bounded inverses. 

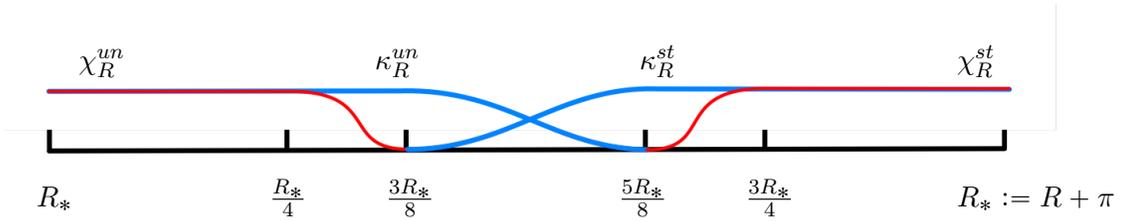
\begin{figure}[H]
	\centering
	\begin{overpic}[scale=.30]{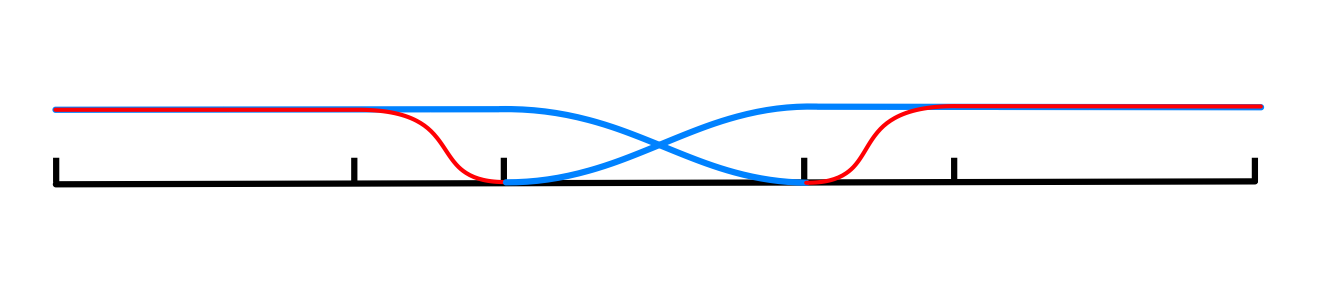}
		\put(3,2){$R_*$}
				\put(25,2){$\frac{R_*}{4}$}
								\put(36,2){$\frac{3R_*}{8}$}
				\put(58,2){$\frac{5R_*}{8}$}	
				\put(70,2){$\frac{3R_*}{4}$}
				\put(90,2){$R_*\colonequals R+\pi$}	
				\put(7,15){$\chi_R^{un}$}		
					\put(90,15){$\chi_R^{st}$}			
						\put(35,15){$\kappa_R^{un}$}	
							\put(60,15){$\kappa_R^{st}$}					
	\end{overpic}	
	\caption{The cutoff functions.}
	\label{Pic38}
\end{figure}

The pre-gluing configuration $p^{\pre}_R=\Phi_R(p^{st}, p^{un})$ is a ``concatenation" of $p^{un}_R$ and $p^{st}_R$. To deduce the invertibility of  $\Hess\CA_{W,\fa^R_{m1}}$ at $p^{\pre}_R=\Phi_R(p^{st}, p^{un})$, consider the cutoff functions 
\[
\kappa^{un}(s)\colonequals \cos\bigg(\frac{\pi}{2}\cdot \chi(s)\bigg) \text{ and }\kappa^{st}(s)\colonequals\sin\bigg(\frac{\pi}{2}\cdot \chi(s)\bigg).
\]
Then $(\kappa^{un})^2+(\kappa^{st})^2=1$, $(\kappa^{un}, \kappa^{st})(s)=(1,0)$ if $s\leq 0$ and $=(0,1)$ if $s\geq 1$. 
\[
\kappa^{un}_R(s)\colonequals \kappa^{un}(\frac{4s}{R+\pi}-\frac{3}{2})\text{ and }\kappa^{st}_R(s)\colonequals\kappa^{st}(\frac{4s}{R+\pi}-\frac{3}{2}). 
\]
Then $(\kappa^{un}_R,\kappa^{st}_R)(s)\equiv (1,0)$ when $s\leq \frac{3(R+\pi)}{8}$ and $\equiv (0,1)$ when $s\geq \frac{5(R+\pi)}{8}$. Now consider the transformation
\begin{align}
U_R: L^r_1(\R_s; (p^{st}_R)^*TM)\oplus L^r_1(\R_s; (p^{un}_R)^*TM)&\to L^r_1(\R_s; T_{x_l}M)\oplus L^r_1(\R_s; (p^{\pre}_R)^*TM)\nonumber\\
\begin{pmatrix}
v^{st}\\
v^{un }
\end{pmatrix}&\mapsto  
\begin{pmatrix}
\kappa^{un}_R&-\kappa^{st}_R \\
\kappa^{st}_R & \kappa^{un}_R\\
\end{pmatrix}
\begin{pmatrix}
v^{st}(s-R)\\
v^{un}(s)
\end{pmatrix}\label{VT.E.5}
\end{align}
with the inverse $U_R^{-1}$ defined by a similar matrix operator. Let 
\[
D_{x_l}=J\ps-\Hess_{x_l}: L^r_1(\R_s; T_{x_l}M)\to L^r_1(\R_s; T_{x_l}M),\ r\geq 2
\]
denote the self-adjoint operator associated to the constant path at $x_l$. Then on the interval $ [\frac{3(R+\pi)}{8},\frac{5(R+\pi)}{8}]_s$, the operators
\[
D_{x_l}= \Hess \CA_{W,\fa_m^{un}}(p^{un}_{R})=\Hess\CA_{W,\fa^R_{m1}}(\Phi_R(p^{st}, p^{un}))
\]
are the same as $\Hess \CA_{W,\fa_{1}^{st}}(p^{st}_{R})$ on $[\frac{3(R+\pi)}{8}-R,\frac{5(R+\pi)}{8}-R]_s$. With this in mind, one verifies that 
\begin{align}\label{VT.E.6}
U^{-1}_R\circ \big(D_{x_l}\oplus \Hess\CA_{W,\fa^R_{m1}}(\Phi_R(p^{st}, p^{un}))\big)&\circ U_R\nonumber\\
&=\Hess \CA_{W,\fa_{1}^{st}}(p^{st}_{R})\oplus \Hess \CA_{W,\fa_m^{un}}(p^{un}_{R})+ \delta_R
\end{align}
where the error term $\delta_R: L^r_1\to L^r_1$ involves only the derivatives of $\kappa_R^{un}, \kappa_R^{st}$ and has norm bounded by $C/R$ for some constant $C>0$. Since the operator norms of $U_R, U_R^{-1}$ are uniformly bounded, this proves that $\Hess\CA_{W,\fa^R_{m1}}(\Phi_R(p^{st}, p^{un}))$ is invertible, and its inverse is bounded in the operator norm by some uniform $C_1>0$ in dependent of $R\gg \pi$.
\end{proof}

Using \cite[Lemma 9.4.4]{AD14} we can now construct a genuine soliton $p^{st}\circ_R p^{un}\in \FC^R$ which is $L^r_1$-close to $p^{\pre}_R=\Phi_R(p^{st}, p^{un})$ for $R\gg \pi$. To conclude the proof of Lemma \ref{VT.L.1}, one has to verify an additional compactness property: suppose that $p_n\in \FC^{R_n}$ is any sequence of solitons with $R_n\to\infty$ and $p_n\to (p^{st},p^{un})$ in the sense of Lemma \ref{GF.L.1}, then for any $n\gg 1$,  then $p_n$ is $C^0$-close to $p^{\pre}_n\colonequals \Phi_{R_n}(p^{st},p^{un})$ and so can be viewed as a section of $(p^{\pre}_n)^*TM$ using the exponential map along. With this understood, we have 
\begin{equation}\label{VT.E.3}
\|p_n-p_n^{\pre}\|_{L^r_1(\R_s; (p^{\pre}_n)^*TM)}\to 0
\end{equation}
as $n\to\infty$ for all $r\geq 2$. This is a stronger statement than the $C^\infty_{loc}$-convergence in Lemma \ref{GF.L.1} and implies that  $p_n=p^{st}\circ_{R_n} p^{un}$ for all $n\gg 1$ in this sequence, by the uniqueness part of \cite[Lemma 9.4.4]{AD14}. In order to verify that $\gr(p^{st}\circ _Rp^{un})=\gr(p^{un})+\gr(p^{st})$, note that replacing $p^{st}$ by $p^{st}_R$, $p^{un}$ by $p^{un}_R$ and $p^{st}\circ_R p^{un}$ by $p^{\pre}_R$ does not introduce any spectral flow to the Hessians for all $R\gg \pi$. One can perturb these paths further to reduce this grading computation to the operators with Lagrangian boundary conditions using the Axiom \ref{A=II}. Finally, by Axiom \ref{A-II}, it suffices to verify that if $\Pi^{un}, \Pi^{st}$ are any graded linear Lagrangian subspaces of $T_{x_l}M$ such that $\Pi^{un}\pitchfork \Pi^{st}$, $\Pi^{un}\pitchfork T_{x_l}U_\star$ and  $T_{x_l}S_\star\pitchfork \Pi^{st}$, then
\begin{equation}\label{VT.E.23}
i(T_{x_l}S_\star,\Pi^{st})+i(\Pi^{un},T_{x_l}U_\star)=i(\Pi^{un},\Pi^{st}).
\end{equation} 
Note that $i(T_{x_l}S_\star,T_{x_l}U_\star)=i(T_{x_l}\Lambda_{x_l,2\pi},T_{x_l}\Lambda_{x_l, \pi})=0$ by our grading convention. Then \eqref{VT.E.23} follows from the general fact that 
\[
i(\Pi_0^{\#},\Pi_1^{\#})+i(\Pi_2^\#,\Pi_3^\#)=i(\Pi_0^\#,\Pi_3^\#)+i(\Pi_2^\#,\Pi_1^\#)
\]
for any graded linear Lagrangian subspaces $\Pi_j^\#, 0\leq j\leq 3$ of $T_{x_l}M$ whenever these indices are defined; this follows from the defining property of the Maslov index; see \cite[Section (11h)]{S08}. This completes the proof of Lemma \ref{VT.L.1}.

\subsection{Proof of Theorem \ref{VT.T.2} Part I: Some estimates about instantons} From now on fix some solitons 
\[
p^{st}_\pm\in \FC^{st}, p^{un}_\pm\in \FC^{un}, 
\]
and let 
\[
p^{\pre}_{R,\pm}=\Phi_R(p^{st}_\pm, p^{un}_\pm),\ p_{R,\pm}=p^{st}_\pm \circ_R p^{un}_\pm\in \FC^R.
\]
$p_{R,\pm}$ are non-degenerate critical points of $\CA_{W,\fa^R_{m1}}$ (solitons), while $p^{\pre}_{R,\pm}$ are the pre-gluing configuration approximating $p_{R,\pm}$. By Lemma \ref{VT.L.4} and \eqref{VT.E.3}, 
\begin{equation}\label{VT.E.4}
\|\Hess^{-1}\CA_{W,\fa^R_{m1}}(p_{R,\pm})\|_{L^r\to L^r_1}<C_1'
\end{equation}
for some constant $C_1'=C_1'(r)$. Then by Lemma \ref{GF.L.4}, the drop of the action functional $\CA_{W,\fa^R_{m1}}$
\[
0\leq \CA_{W,\fa_{m1}^R}(p_{R,-})- \CA_{W,\fa_{m1}^R}(p_{R,+})<C_2
\]
is bounded by a uniform constant $C_2>0$. This property is special to the associated graded complex $\sG^l\Delta(U_m, S_1)$. Combined with Lemma \ref{L1.5}, this implies that for some $C_3>0$, we have 
\begin{equation}\label{VT.E.7}
\E_{an}(P_R; I_t\times \R_s)<C_3,\ I_t\colonequals [t-1,t+1]_t
\end{equation}
for any $\alpha_{m1}^R$-instanton $P_R\in \M(p_R^-,p_R^+;\fa^{R}_{m1})$ and $t\in \R_t$. The next lemma allows us to control the solution $P_R$ over the strip $Z_R=\R_t\times [2\pi, R-\pi]_s$ on which the $\alpha_{m1}^R$-instanton equation \eqref{E1.6} takes the standard form 
\[
\pt P_R-J\ps P_R-\nabla H=0. 
\]
\begin{lemma}\label{VT.L.5} There exists constants $C_4, R_4,\zeta>0$ with the following property. For any $R+\pi>2R_4$ and any $\alpha_{m1}^R$-instanton $P_R\in \M(p_{R,-},p_{R,+};\fa_{m1}^R)$, we have $P_R(t,s)\in \SO(x_l)$ and 
	\[
	0\leq u_R(t,s)\leq C_4\cdot \frac{\cosh(\zeta\cdot (\frac{R+\pi}{2}-s))}{\cosh(\zeta \cdot \frac{R+\pi}{2})}
		\]
	if $s\in [R_4,(R+\pi)-R_4]_s$, where $u_R$ is the energy density function of $P_R$. 
\end{lemma}
\begin{proof}[Proof of Lemma \ref{VT.L.5}] The proof follows the same line of arguments as in Proposition \ref{P1.8}. To start, we have to show a uniform energy decay as in Lemma \ref{L1.9}, but this is easy: the estimate \eqref{VT.E.7} combined with our assumption that the angle $\theta=0$ is admissible implies the analogue of Lemma \ref{L1.10}, then we deduce the uniform energy decay using the triviality of point-like solutions, i.e., Lemma \ref{L1.15}. To verify this exponential decay, we exploit the inequality
	\[
	0\geq (\Delta+\zeta^2)|dP_R|^2
	\]
	as in the proof of Proposition \ref{P1.8} but now on the strip $\R_t\times[R_4, (R+\pi)-R_4]_s$ for some $R_4\gg \pi$. Then we use the maximal principle from \cite[Corollary A.4]{Wang202} to conclude. 
\end{proof}

The next step is to  control the behavior of $P$ when the drop of the action functional $\CA_{W,\fa^R_{m1}}$ is small on some $I_t\times \R_s$, $I_t=[t-1,t+1]_t$; cf. \cite[Proposition 13.4.7]{Bible}. Chose $r_0>0$ such that $r_0$ is less then the injective radius of $(M, g_M)$ and 
\begin{equation}\label{VT.E.29}
r_0<\half\cdot\min\{ \dist_{C^0(\R_s)}(p^{st}_0, p^{st}_1) ,\dist_{C^0(\R_s)}(p^{un}_0, p^{un}_1): p^{st}_j\in \FC^{st},\ p^{un}_j\in \FC^{un},\ j=0,1\}.
\end{equation}

\begin{lemma}\label{VT.L.6} Let $I_t'\colonequals [t-\half, t+\half]_t$. For any $r\geq 2$, there exist constants $\epsilon_5, C_5$ and $C_6=C_6(r)>0$ with the following properties. For any solitons $p_{R,\pm}\in \FC^R$, any instanton $P_R\in \M(p_{R,-}, p_{R,+};\fa^{R}_{m1})$ and $t\in \R_t$, suppose that 
	\begin{equation}\label{VT.E.12}
	\CA_{W,\fa^R_{m1}}(P_R(t-1,\cdot))-	\CA_{W,\fa^R_{m1}}(P_R(t+1,\cdot))<\epsilon_5,
	\end{equation}
	then for some soliton $p_R\in \FC^R$, $P_R(t',\cdot)$ has $C^0(\R_s)$-distance $<r_0$ with $p_R$ for all $t'\in I_t'$. This $p_R$ is unique due to our choice \eqref{VT.E.29} of $r_0$. Then $P_R|_{I_t'\times \R_s}$ can be viewed as a section of $(p_R)^*TM$ using the exponential map along $p_R$, and we have the estimate
	\begin{equation}\label{VT.E.13}
\|	P_R|_{I_t'\times \R_s}-p_R\|_{L^r_1(I_t'\times \R_s)}\leq C_6(r)\bigg(
\CA_{W,\fa^R_{m1}}\big(P_R(t-1,\cdot)\big)-	\CA_{W,\fa^R_{m1}}\big(P_R(t+1,\cdot)\big)
\bigg),
	\end{equation}
	where $p_R$ is viewed as a constant trajectory on $I_t'$. Moreover, for all $t'\in I'_t$,
	\begin{equation}\label{VT.E.14}
\big|\CA_{W,\fa^R_{m1}}\big(P_R(t',\cdot))-\CA_{W,\fa^R_{m1}}(p_R)\big|\leq C_5\|\grad \CA_{W,\fa^R_{m1}}\big(P_R(t',\cdot)\big)\|^2_{L^2(\R_s)}. 
	\end{equation}
\end{lemma}
\begin{proof}This lemma follows from the standard analysis of Floer trajectories; see for instance \cite[Section 13.4]{Bible}. The reason why $\epsilon_5, C_5, C_6$ can be made uniform in $R$ is \eqref{VT.E.4}: the Hessians of $\CA_{W,\fa^R_{m1}}$ can be inverted uniformly at all solitons in $\FC^R$ for $R\gg \pi$.
\end{proof}

\subsection{Proof of Theorem \ref{VT.T.3} Part II: Compactness of instantons} With Lemma \ref{VT.L.5} and Lemma \ref{VT.L.6} at hands, we are ready to analyze the limit of $P_R$ as $R\to\infty$. Recall that the moduli space $\M(p_{R,-}, p_{R,+};\fa_{m1}^R)$ carries an $\R_t$-action: let $\tau_t P=P(\cdot-t, \cdot), t\in \R_t$. 

\begin{lemma}\label{VT.L.7} For any sequence of $\alpha^{R_n}_{m1}$-instantons $P_n\in \M(p_{R_n,-}, p_{R_n,+};\fa_{m1}^{R_n})$ with $R_n\to\infty$, there exists a subsequence converging to a broken instanton:
	\begin{equation}\label{VT.E.24}
(p^{st}_0,p^{un}_0)\xrightarrow{(P^{st}_0, P^{un}_0)} (p^{st}_1,p^{st}_1)\to \cdots\to	(p^{st}_k,p^{un}_k)\xrightarrow{(P^{st}_{k_1}, P^{un}_{k_1})}(p^{st}_{k_1+1},p^{un}_{k_1+1}),
	\end{equation}
	in the following sense:
	\begin{itemize}
\item  for every $0\leq k\leq k_1$, $p^{st}_k\in \FC^{st}$ and $p^{un}_k\in \FC^{un}$ are some solitons;
\item $(p^{st}_0,p^{un}_0)=(p^{st}_-, p^{un}_-)$ and $(p^{st}_{k_1+1},p^{un}_{k_1+1})=(p^{st}_+, p^{un}_+)$;
\item for every $0\leq k\leq k_1$, $P_k^{st}\in \M(p_k^{st}, p_{k+1}^{st};\fa_1^{st})$ and $P_k^{un}\in \M(p_k^{un}, p_{k+1}^{un};\fa_m^{un})$ are some instantons; one of them must be non-constant in time;
\item for any $n$, there exists a sequence of numbers:
\[
t^0_n<t^1_n<\cdots <t^{k_1}_n
\]
such that for all $0\leq k\leq k_1$, we have 
\[
\tau_{t^k_n}P_n\to P_k^{un} \text{ and }(\tau_{t^k_n}P_n)(\cdot,\cdot-R_n)\to P_k^{st}
\]
in $C^\infty_{loc}$-topology, and in addition $ |t^{k+1}_n-t^k_n|\to +\infty$ as $n\to\infty$. 
	\end{itemize}
\end{lemma}

\begin{proof} This lemma follows from Lemma \ref{VT.L.5} \& \ref{VT.L.6} and the argument in \cite[Section 16.1\& 16.2]{Bible}. Note that the divergence $|t^{k+1}_n-t^k_n|\to\infty$ may be slower or faster than $R_n\to\infty$; they are not necessarily related. 
\end{proof}

\begin{figure}[H]
	\centering
	\begin{overpic}[scale=.25]{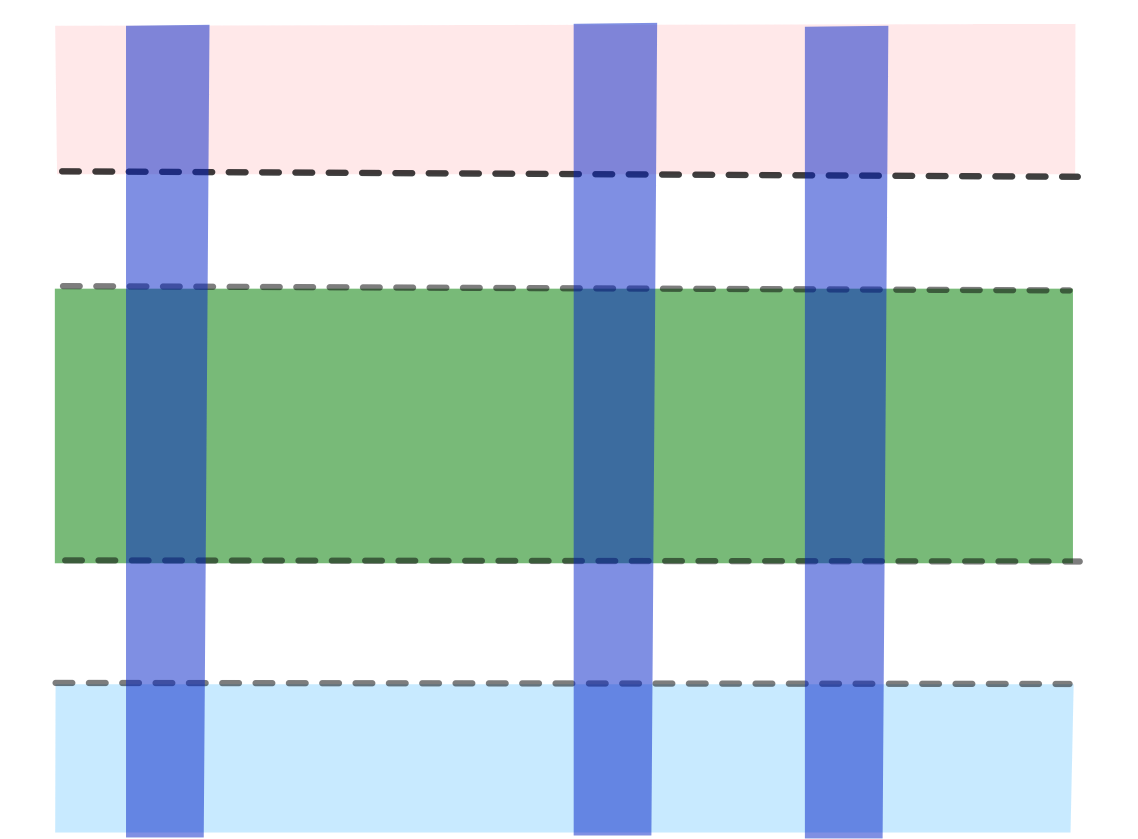}	
		\put(12,73){$t_n^0$}	
				\put(52,73){$t_n^1$}
						\put(72,73){$t_n^2$}
\put(11,17){$P_0^{un}$}
\put(11,53){$P_0^{st}$}	
\put(51,17){$P_1^{un}$}
\put(51,53){$P_1^{st}$}	
\put(71,17){$P_2^{un}$}
\put(71,53){$P_2^{st}$}	
\put(30,65){$x_1$}	
	\put(30,35){$x_l$}	
	\put(30,5){$x_m$}
	\put(0,17){$p_0^{un}$}
	\put(0,53){$p_0^{st}$}
		\put(30,17){$p_1^{un}$}
	\put(30,53){$p_1^{st}$}
		\put(62,17){$p_2^{un}$}
	\put(62,53){$p_2^{st}$}	
		\put(85,17){$p_3^{un}$}
	\put(85,53){$p_3^{st}$}								
	\end{overpic}	
	\caption{Converging to a broken instanton ($k_1=2$).}
	\label{Pic39}
\end{figure}
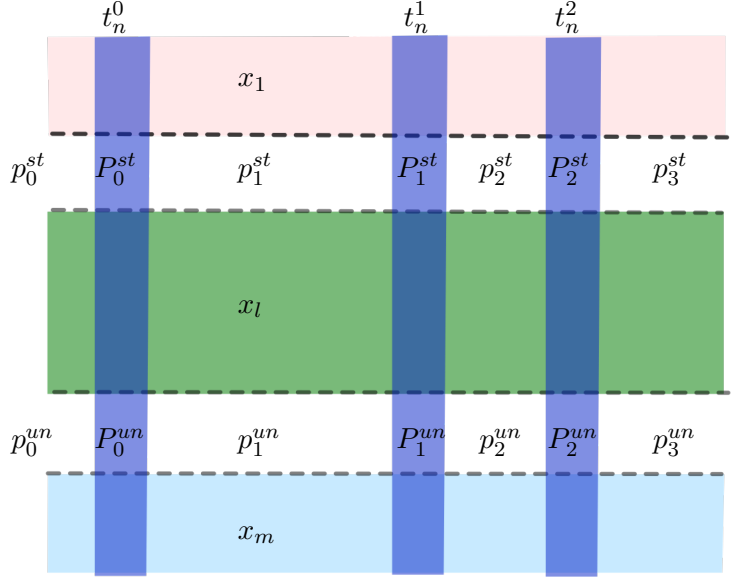

By Lemma \ref{VT.L.1}, for such a broken instanton \eqref{VT.E.24}, we must have 
\begin{align}\label{VT.E.8}
\gr(p_{R_n,-})-\gr(p_{R_n,+})&=\big(\gr(p^{st}_{0})-\gr(p^{st}_{k_1+1})\big)+ \big(\gr(p^{un}_{0})-\gr(p^{un}_{k_1+1})\big)\nonumber\\
&=\sum_{k=0}^{k_1} \big(\gr(p^{st}_{k})-\gr(p^{st}_{k+1})\big)+ \big(\gr(p^{un}_k)-\gr(p^{un}_{k+1})\big).
\end{align}
In the second line, each term in the summation is either the index of $P_k^{st}$ or that of $P_k^{un}$. Since the Floer data $\fa_1^{st}$ and $\fa_m^{un}$ are assumed to be admissible, each of them must be $\geq 0$ and $=0$ only if the associated instanton is a constant trajectory in time. Thus in the case that $\gr(p_{R,-})-\gr(p_{R,+})=1$, we must have $k_1=0$ and
\begin{equation}\label{VT.E.11}
\big(\gr(p_{-}^{st})-\gr(p_{+}^{st}),\ \gr(p_{-}^{un})-\gr(p_{+}^{un})\big)=(0,1) \text{ or }(1,0)
\end{equation}
depending on whether $P_0^{st}$ or $P_0^{un}$ is a constant trajectory; this explains the trichotomy in \eqref{VT.E.9}. In particular, we have proved that

\begin{corollary}\label{VT.C.8} Suppose that $\big(\gr(p_{-}^{st})-\gr(p_{+}^{st})\big)+\big( \gr(p_{-}^{un})-\gr(p_{+}^{un})\big)=1$, then the moduli space $\M(p_{R,-}, p_{R,+};\fa^R_{m1})=\emptyset$ for all $R\gg \pi$ unless 
	\begin{enumerate}[label=$($\roman*$)$]
\item $\big(\gr(p_{-}^{st})-\gr(p_{+}^{st}),\ \gr(p_{-}^{un})-\gr(p_{+}^{un})\big)=(0,1)$ and $p_-^{st}=p_+^{st}$; or 
\item $\big(\gr(p_{-}^{st})-\gr(p_{+}^{st}),\ \gr(p_{-}^{un})-\gr(p_{+}^{un})\big)=(1,0)$ and $p_-^{un}=p_+^{un}$.
	\end{enumerate}
\end{corollary}

Since the two cases in Corollary \ref{VT.C.8} are symmetric, we shall focus on the first one from now on. In this special case, Lemma \ref{VT.L.7} is refined as follows.
 
 \begin{corollary}\label{VT.C.9} Suppose that the first case of Corollary \ref{VT.C.8} happens. For any sequence of instantons $P_n\in \M(p_{R_n,-}, p_{R_n,+};\fa_{m1}^{R_n})$ with $R_n\to\infty$, there exists a $($non-constant$)$ $\alpha^{un}_m$-instanton $P^{un}\in \M(p_-^{un}, p_+^{un};\fa^{un}_m)$ and a sequence of real numbers $t^n$ such that 
 	\[
 	\tau_{t_n}P_n\to P^{un}\text{ and } (\tau_{t_n}P_n)(\cdot, \cdot-R_n)\to P^{st}
 	\]
 	in $C^\infty_{loc}(\R_t\times \R_s)$-topology, where $P^{st}$ is the constant trajectory at $p^{st}_-=p^{st}_+$. 
 \end{corollary}

The $C^\infty_{loc}$-convergence in Corollary \ref{VT.C.9} is not quite enough for the proof of Theorem \ref{VT.T.2}. A global estimate in the spirit of \eqref{VT.E.3} is more relevant. The proof of Lemma \ref{VT.L.7} implies the following stronger result.

\begin{corollary}\label{VT.C.10} Then there exists a constant $C_7, K_7>0$ with the following properties. For the converging subsequence in Corollary \ref{VT.C.9}, we have 
	\[
	\|\grad \CA_{W,\fa^R_{m1}}(P_R(t,\cdot))\|_{L^2(\R_s)}^2<\frac{\epsilon_5}{4},
	\]
	for all $|t|\geq K_7$ and $R\gg \pi$; so the condition \eqref{VT.E.12} holds on $I_t$ for all $t>K_7+1$. Then \eqref{VT.E.14} implies that 
	\[
	0\leq  \mp\CA_{W,\fa^R_{m1}}(p_{R,\pm})\pm\CA_{W,\fa^R_{m1}}(P_R(t,\cdot))\leq C_7 e^{\pm C_5 t},\ \forall \pm t\geq K_7+1.
	\]
Combined with \eqref{VT.E.13}, this means that the $L^r_1$-norm of $P_R-p_{R,\pm}$ on $I_t'\times \R_s$ decays exponentially to zero as $\pm t\to\infty$. These estimates are uniform in the stretching parameter $R$ and hold also for the limiting instanton $P^{un}\in \M(p^{un}_-, p^{un}_+;\fa^{un}_m)$ in Corollary \ref{VT.C.9}.
\end{corollary}

\subsection{Proof of Theorem \ref{VT.T.2} Part III: A gluing scheme} First recall the Implicit Function Theorem we shall use later:

\begin{theorem} [{\cite[Theorem A.3.3]{MS12}}]\label{VT.T.11} Let $X$ and $Y$ be Banach spaces and $U\subset X$ be an open subset. If $F: U\to Y$ is Fredholm $($i.e. the linear map $dF(x): X\to Y$ is bounded and Fredholm for all $x\in X)$ and continuously differentiable, and $y$ is a regular value of $F$, then $\CN= F^{-1}(y)\subset U$ is a $C^1$-manifold with $T_x\CN=\ker dF(x)$ for all $x\in \CN$. 
\end{theorem}

In what follows, we shall apply this theorem to a family of Fredholm maps $F:X\to Y$ which depend on the stretching parameter $R$. The next proposition, which is used in the proof of Theorem \ref{VT.T.11}, becomes more useful in order to obtain estimates uniform in $R$.

\begin{proposition}[{\cite[Proposition A.3.4]{MS12}}]\label{VT.P.12} Under the assumptions of Theorem \ref{VT.T.11}, suppose that the differential $D=dF(x_0):X\to Y$ is surjective at some $x_0\in U$ and has a right inverse $Q:Y\to X$. Choose $c,\delta>0$ such that $\|Q\|\leq c$ and $\|dF(x)-D\|\leq \frac{1}{2c}$ for all $x\in B(x_0,\delta)\subset U$. If $x_1\in X$ satisfies
	\[
	\|F(x_1)\|<\frac{\delta}{4c},\ \|x_1-x_0\|<\frac{\delta}{8},
	\]
	 then there exists a unique $x\in B(x_0,\delta)$ such that $F(x)=0$ and $x-x_1\in \im Q$. Moreover, $\|x-x_1\|\leq 2c\|f(x_1)\|$. 
\end{proposition}


In Proposition \ref{VT.P.12}, one should think of $x_0$ as a reference point and  $x_1$ as the approximate solution. 

\medskip

We begin with a linear problem to explain our gluing scheme. Suppose that $V$ is a finite dimensional Euclidean space, $\sigma_\pm:V\to V$ are self-adjoint and invertible, and $ \sigma_t:V\to V, t\in \R_t$ is a smooth family of self-adjoint operators such that $\|\sigma_t-\sigma_\pm\|_{L^r(I_t, \End(V))}\to 0$, $I_t\colonequals [t-1,t+1]$ exponentially as $t\to\pm\infty$ for some $r>2$. Then the linear map:
\begin{align*}
D: C^\infty(\R_s;V)&\to C^\infty(\R_s;V)\\
v(t)&\mapsto \pt v(t)+\sigma_t(v(t)). 
\end{align*}
is Fredholm from $L^r_1(\R_t; V)$ to $L^r(\R_t; V)$. If $D: L^r_1\to L^r$ is invertible, then any $\xi\in L^r$ has a unique solution $v\in L^r_1$ such that $D(v)=\xi$.

However, we would like to solve this equation for $\xi$ in a larger subspace of $C^\infty(\R_s;\V)$. Let $\chi_\pm:\R_t\to [0,1]$ be a pair of cutoff functions such that for some $K>1$, 
\begin{equation}\label{VT.E.16}
\chi_-(t)=\chi_+(-t),\ \chi_+(t)\equiv 1 \text{ if }t\geq K+1 \text{ and }\equiv 0 \text{ if }t\leq K. 
\end{equation}
Let $\V_j, j=0,1$ denote the subspace of $C^\infty(\R_s; V)$ whose elements take the form
\begin{equation}\label{VT.E.17}
v=v^-\chi_-(s)+v^+\chi_+(s)+v_0 \text{ with } v_\pm\in V_\pm, v_0\in L^r_j(\R_s; V)
\end{equation}
where $V^\pm$ is another copy of $V$, and $\V_j$ is equipped with the norm on $V_-\oplus V_+\oplus L^r_j(\R_s; V)$. Then $D$ extends to a bounded operator $\V_1\to \V_0$ which takes the form
\begin{equation}\label{VT.E.18}
\begin{pmatrix}
v^-\\
v^+\\
v_0
\end{pmatrix}\mapsto 
\begin{pmatrix}
\sigma_-& 0 &0\\
0 & \sigma_+  & 0\\
\pt\chi_-+\chi_-(\sigma_t-\sigma_-)  & \pt\chi_++\chi_+(\sigma_t-\sigma_+)& \pt+ \sigma_t
\end{pmatrix}
\begin{pmatrix}
v^-\\
v^+\\
v_0
\end{pmatrix}.
\end{equation}
Since all diagonal entries are invertible, so is the operator $D:\V_1 \to \V_0$. Thus any $\xi \in \V_0$ also admit a unique solution $D(v)=\xi$ with $v\in \V_1$. One should think of this $D$ as the differential map $dF(x_1)$ in Proposition \ref{VT.P.12} and $\xi=F(x_1)$ is the error term to be corrected.

\medskip

Returning to the proof of Theorem \ref{VT.T.2}, we fix some $r>2$ from now on. suppose that we are in the first case of Corollary \ref{VT.C.8}, $P^{un}\in \M(p_-^{un}, p_+^{un};\fa^{un}_m)$ is any $\alpha^{un}_m$-instanton, and $P^{st}$ is the constant trajectory at $p^{st}\colonequals p^{st}_-=p^{un}_+$. For any $R\gg \pi$, we have to construct an instanton $P_R\in \M(p^-_{R}, p^+_R;\fa^R_{m1})$ such that $P_R\to (P^{un}, P^{st})$ as $R\to\infty$ in the sense of Corollary \ref{VT.C.9}. 

\medskip

The approximate solution $P^{\pre}_R$ is simply obtained by applying the pre-gluing map \eqref{VT.E.15} to $P^{st}$ and $P^{un}$ at each time slice, i.e.,
\begin{equation}\label{VT.E.21}
P^{\pre}_{R}(t,\cdot)=\Phi_R( P^{st}(t,\cdot),P^{un}(t,\cdot))=\Phi_R(p^{st}, P^{un}(t,\cdot) ), \forall t\in \R_t.
\end{equation}

It is also convenient to choose a reference trajectory which is constant when $|t|\gg 1$ so that our gluing scheme looks more similar to one above. Consider an approximation $P^{un}_K$ of $P^{un}$ with $K\gg 1$ such that
\begin{itemize}
\item $P^{un}_K(t,s)=P^{un}(t,s)$ when $|t|\leq K-1$;
\item $P^{un}_K(t,s)=p^{un}_\pm(s)$ when $\pm t\geq K$;
\item $\dist_{C^0}(P^{un}, P^{un}_K)<r_0$; identify $P^{un}$ as a section of $L^r_1(\R_t\times \R_s; (P^{un}_K)^*TM)$, then
\begin{equation}\label{VT.E.20}
\delta_K\colonequals\|P^{un}-P^{un}_K\|_{L^r_1(\R_t\times \R_s; (P^{un}_K)^*TM)}<Ce^{-C_5K}
\end{equation}  
for some $C>0$, and $C_5$ is the constant in Lemma \ref{VT.L.6} and Corollary \ref{VT.C.10}. 
\end{itemize}

This approximation $P^{un}_K$ is constructed using the exponential maps along $p^{un}_\pm$ and some cutoff functions as in the case of the pre-gluing map \eqref{VT.E.15}, and the required exponential decay follows from Corollary \ref{VT.C.10}. Now we replace $P^{un}$ by $P^{un}_K$ in \eqref{VT.E.21} to obtain a reference trajectory denoted by $P^{\refe}_R$, which satisfies the following properties:
\begin{itemize}
\item $P^{\refe}_R(t,\cdot)\equiv p_{R,\pm}^{\pre}=\Phi_R(p^{st}, p^{un}_\pm)$ when $\pm t\geq K$;
\item $P^{\refe}_R=P^{\pre}_R\equiv x_l$  when $s\in [\frac{3(R+\pi)}{8}, \frac{5(R+\pi)}{8}]_s$;
\item $P^{\refe}_R=P^{\pre}_R$ when $|t|\leq K-1$.
\end{itemize}

Neither of $P^{\pre}_R$ and $P^{\refe}_R$ satisfy the boundary condition \eqref{E1.7} yet, and they must be corrected as in our model problem above. Let $V^\pm_j=L^r_1(\R_s; (p^{\pre}_{R,\pm})^*TM)$ and $V_j^0=L^r_1(\R_t\times \R_s; (P^{\refe}_R)^*TM), j=0,1.$ Then define 
\[
\V_j\colonequals V^-_j\oplus V^+_j\oplus V^0_j,\ j=0,1.
\]
Since  $P^{\refe}_R$ is a constant trajectory when $|t|\geq K$, we can identify $\V_j$ as a subspace of $L^r_{loc}(\R_t\times \R_s;(P^{\refe}_R)^*TM)$ using the cutoff functions \eqref{VT.E.16} and the formula \eqref{VT.E.17}.

\medskip

For any $x\in M$ and $w\in T_xM$, let $\Psi_x(w): T_xM\to T_{\exp_x(w)}M$ denote the parallel transportation along the geodesic $t\mapsto \exp_x(t w), t\in [0,1]$. For any $v\in \V_1$, consider the map $P_v(z)=\exp_{P^{\refe}_R(z)}(v(z)), z\in \R^2$. Using the Floer datum $\fa_{m1}^R$, the formula \eqref{E1.12} defines a section $\F(P_v(z))\in L^r_{loc}(\R_t\times \R_s; (P_v)^*TM\otimes \Lambda^{0,1}\C)$, $\C_z=\R_t\times \R_s$ , which is turned into a non-linear map using the parallel transportation $\Psi$: 
\begin{align}\label{VT.E.25}
\CF: \V_1&\to L^r_{loc}(\R_t\times \R_s;  (P^{\refe}_R)^*TM)\\
v&\mapsto \Psi_{P^{\refe}_R(z)}(v(z))^{-1}\big(\F(P_v(z))\big). \nonumber
\end{align}
where the bundle $\Lambda^{0,1}\C$ has been trivialized using  $d\overline{z}=dt-is$. Although we have dropped the subscript, $\CF$ depends on the stretching parameter $R\gg \pi$. 

\medskip

Suppose that under the exponential map, the approximate solution $P^{\pre}_{R}$ is represented by the section $v_1\in  \V_1$, so $P^{\pre}_{R}(z)=P_{v_1}(z)=\exp_{P^{\refe}_R(z)}(v_1(z))$. Then in fact $v_1\in L^r_1$ and $v_1\equiv 0$ when $|t|\leq K-1$. To apply Proposition \ref{VT.P.12} to the non-linear map \eqref{VT.E.25}, we take $x_0=0\in \V_1$ as the reference point and $x_1=v_1$ as the approximate solution. The conditions of Proposition \ref{VT.P.12} are verified by the next lemma. 

\begin{lemma}\label{VT.L.13} There exists constants $\delta_1, c, c_1>0$ independent of $R,K$ so that the following holds for $R\gg \pi$:
	\begin{enumerate}[label=$($\roman*$)$]
\item\label{IFT1} $\CF(0)\in \V_0$ and $(d \CF)(0): \V_1\to \V_0$ is bounded linear and Fredholm of index $1$;
\item\label{IFT2} for all $v_2, v_3\in B(0,\delta_1)\subset \V_1$, 
\[
\|d\CF(v_2)-d\CF(v_3)\|_{\V_1\to\V_0}\leq c_1\|v_2-v_3\|_{\V_1}.
\]
Combined with the first property, this implies that the image of $\CF(B(0,\delta_1))$ lies in the smaller space $\V_0\subset L^r_{loc}$, and $\CF:\V_1\to \V_0$ is continuously differentiable;
\item\label{IFT3} $\|v_1\|_{L^r_1}=\|v_1\|_{\V_1}\leq 2\delta_K$ and $\|\CF(v_1)\|_{\V_0}\to 0$ as $R\to\infty$, with $\delta_K$ defined as in \eqref{VT.E.20};
\item\label{IFT4} $d\CF(v_1):\V_1\to \V_0$ is surjective and has a right inverse $Q_R$ with norm bounded by $c$. 
	\end{enumerate}
\end{lemma}

Take $F=\CF$, $x_0=0$, $x_1=v_1$ and and $\delta=\min\{ \delta_1, \frac{1}{2cc_1}\}$ in Proposition \ref{VT.P.12}. We have to verify in addition that 
\[
\|\CF(v_1)\|_{\V_0}<\frac{\delta}{4c},\ \|v_1\|_{\V_1}<\frac{\delta}{8}. 
\]
By Lemma \ref{VT.L.13}, this can be arranged by taking a fixed $K\gg 1$ (use \eqref{VT.E.20}) then by letting $R\to \infty$. Thus we obtain an $\alpha^R_{m1}$-instanton $P_R^\star(z)=\exp_{P^{\pre}_R(z)}(v_\star(z))\in \M(p^R_{-}, p^R_+;\fa^R_{m1})$ with $v_\star-v_1\in \im Q_R$ and $\|v_\star-v_1\|_{\V_1}\leq 2c\|\CF(v_1)\|_{\V_0}\to 0$ as $R\to\infty$. 

\medskip

By Theorem \ref{VT.T.11}, $v_R$ fits into a 1-parameter family of perturbed $\alpha^R_{m1}$-instantons. The uniqueness part of Proposition \ref{VT.P.12} then says that in a small neighborhood of $P_R^{\pre}$, this family is obtained by translating $P_R^*$. In the meantime, the $C^\infty_{loc}$-convergence in Corollary \ref{VT.C.9} can be improved as follows.
\begin{lemma}\label{VT.L.14} For the converging subsequence in Corollary \ref{VT.C.9}, we have 
	\[
\|	\tau_{t_n} P_n- P^{\pre}_{R_n}\|_{\V_1}\to 0
	\]
	as $n\to\infty$, where $\tau_{t_n} P_n$, $P^{\pre}_{R_n}$ are viewed as sections in $\V_1$ using the exponential maps. This implies that $\tau_{t_n} P_n$ is a translated copy of $P_{R_n}^\star$. 
\end{lemma}

Thus for $R\gg \pi$, any instanton in $\M(p^R_-, p^R_+;\fa^R_{m1})$ is obtained  up to a translation by gluing some $P^{un}$ vertically with $P^{st}$. To complete the proof of Theorem \ref{VT.T.2}, it remains to prove Lemma \ref{VT.L.13} and \ref{VT.L.14}.

\begin{proof}[Proof of Lemma \ref{VT.L.13} \ref{IFT1}] For the reference trajectory $P^{\refe}_R$, the section $\CF(0)$ is defined by the formula
	\[
	\pt P^{\refe}_R+J\ps P^{\refe}_R+\nabla \im (e^{-i\alpha_{m1}^R(s)}W)+\nabla\delta H^{m1}_R=0,
	\]
	so $\CF(0)(t,\cdot)\equiv \grad p^{\pre}_{R,\pm}\in L^r_1(\R_s; (p^{\pre}_{R,\pm})^*TM)$ for $\pm t\geq K$. The gradient vector $\grad p^{\pre}_{R,\pm}$ is supported on $[\frac{R+\pi}{4},\frac{3(R+\pi)}{4}]_s$. By construction the difference 
	\[
	\CF(0)-\chi_-(s)\grad p^{\pre}_{R,-}-\chi_+(s)\grad p^{\pre}_{R,+}. 
	\]
	is supported on 
	\[
\{K-1\leq |t|\leq K+1\}\ \bigcup\ [-K-1, -K+1]_t\times [\frac{R+\pi}{4},\frac{3(R+\pi)}{4}]_s
	\]
	and lies in $L^1_r(\R_t\times \R_s)$. The linearization of $\CF$ at $0$ takes a lower triangular form as in \eqref{VT.E.18}:
	\begin{equation}\label{VT.E.19}
	\begin{pmatrix}
	\Hess \CA_{W,\fa^R_{m1}}(p_{R,-}^{\pre})  & 0 & 0\\
	0 & \Hess \CA_{W,\fa^R_{m1}}(p_{R,+}^{\pre}) & 0\\
	\pt \chi_- & \pt\chi_+ & D_0
	\end{pmatrix}
	\end{equation}
	where $D_0$ is $d\CF(0)$ as a map between $L^r_1\to L^r$. These entries are bounded linear, and so is $d\CF(0):\V_1\to \V_0$. The Fredholm index will be computed later using $d\CF(v_1)$. \end{proof}

\begin{proof}[Proof of Lemma \ref{IFT2}] The second statement \ref{IFT2} requires some work. Its predecessor can be found in \cite[Theorem 3a]{	F88},\cite[Lemma 9.4.8]{AD14} and \cite[Proposition 3.5.3]{MS12}, and we follow the same argument.The key ingredient is the following pointwise estimate.
	\begin{lemma}\label{VT.L.15} Let $S$ be any Riemann surface, $M_0\subset M$ any compact subset, and $P: S\to M_0\subset M$ any smooth map. Following the construction of \eqref{VT.E.25}, for any smooth vector field $\X\in \Omega^1(S; C^\infty(M;TM))$, one can define a non-linear map 
		\begin{align*}
	\CF_0: C^\infty(S; P^*(TM))&\to C^\infty(S; P^*(TM)\otimes \Lambda^{0,1} S )\\
	\xi&\mapsto \Psi_{P} (\xi)^{-1}\big(\F_0(\exp_P(\xi))\big)
		\end{align*}
	with $\F_0(P_\xi)\colonequals (dP_\xi-\X)^{0,1}$, $P_\xi=\exp_P(\xi)$. Then for any $c_2>0$, there exist some constant $C>0$ such that for any smooth $\delta\xi_1, \delta\xi_2, \xi_0,\xi\in C^\infty(B(0,1); P^*TM)$ with $L^\infty$-norms $<c_2$, there are pointwise bounds
	\begin{equation}\label{VT.E.27}
	|(d\CF_0)(\xi_0)\xi-(d\CF_0)(\xi_0+\delta\xi_1)\xi|<C(|\delta\xi_1||\nabla \xi|+|\nabla \delta\xi_1||\xi|+|\delta\xi_1||\xi|),
	\end{equation}
	and 
	\begin{align}\label{VT.E.28}
&\big|(d\CF_0)(\xi_0)\xi-(d\CF_0)(\xi_0+\delta\xi_1)\xi-(d\CF_0)(\xi_0+\delta\xi_2)\xi+(d\CF_0)(\xi_0+\delta\xi_2+\delta\xi_1)\xi\big|\\
\leq & C\big(|\delta\xi_1||\delta\xi_2||\nabla\xi|+|\nabla\delta\xi_1||\delta\xi_2||\xi|+|\nabla\delta\xi_1||\delta\xi_2||\xi|+|\delta\xi_1||\delta\xi_2||\xi|\big).\nonumber
	\end{align}
on $B(0,\half)$. This constant $C$ depends on the geometry of $M$, $c_2>0$, $\|dP\|_\infty$ and the $C^2$-norm of $\X$ on a slightly larger compact subset of $M$ depending on $c_2$. 
	\end{lemma}
\begin{proof} [Proof of Lemma \ref{VT.L.15}] If one of $\xi_0, \xi_1$ is zero, then \eqref{VT.E.27} follows from the computation in \cite[Proposition 3.5.3]{MS12}. In general, $\eqref{VT.E.26}$ and \eqref{VT.E.27} follow from a (well-known) estimate on the second and the third derivatives of $\CF_0$. For the sake of completeness, this estimate is summarized and proved in Lemma \ref{TC.L.1}.
\end{proof}
	
To apply Lemma \ref{VT.L.15}, let $P=P^{\pre}_R$ and $\X$ is the Hamiltonian vector field associated to $\im(e^{-\alpha_{m1}^R(s)}W)ds+\delta H_{m1}^R\in \Omega^1(S; C^\infty(M;\R))$. The compactness theorem implies that $P^{\pre}_R$ lies a fixed compact subset of $M$. The $L^\infty_1$-norm of $P^{\pre}_R$ and the $C^2$-norm of $\X$ on this compact subset are uniformly bounded. 
	
	Let $\delta v=v_3-v_2$. We have to show that for all $v_2,v_3\in B(0,\delta_1)\subset \V_1$ and $v\in \V_1$, the estimate 
	\begin{equation}\label{VT.E.26}
	\|(d\CF)(v_2+\delta v)(v)-(d\CF)(v_2)(v)\|_{\V_0}\leq c_1\|\delta v\|_{\V_1}\|v\|_{\V_1}
	\end{equation}
holds for some $c_1>0$. By the nature of $\V_1$, decompose $\delta v$ as $\chi_-\delta v^-+\chi_+\delta v^++\delta v^0$. It suffices to verify \eqref{VT.E.26} when only one of $\delta v^\pm, \delta v^0$ is non-zero. Similarly write $v=\chi_-v^-+\chi_+v^++v^0$ with $v^\pm\in V_1^\pm$ an $v^0\in V_1^0.$

\medskip

\Case 1. If $\delta v=\delta v^0\in L^r_1(\R_t\times\R_s)$, then we can prove a stronger result:
\[
\|(d\CF)(v_2+\delta v)(v)-(d\CF)(v_2)(v)\|_{L^r(\R^2)}\leq c_1\|\delta v\|_{L^r_1}\|v\|_{\V_1}.
\] 
This follows immediately from \eqref{VT.E.27} by noting that $\V_1\subset L^\infty$, and 
\[
\nabla (\chi_-v^-+\chi_+v^+)\in L^\infty(\R_t; L^r(\R_s))  \text{ and }L^r_1\embed L^r(\R_t; L^r_1(\R_s))\embed L^r(\R_t; L^\infty(\R_s)).
\]
The multiplication $L^\infty(\R_t; L^r(\R_s))\times L^r(\R_t; L^\infty(\R_s))\to L^r(\R^2)$ is continuous. Thus 
\begin{align*}
&\|(d\CF)(v_2+\delta v)(v)-(d\CF)(v_2)(v)\|_r\\
\lesssim &\||\delta v||\nabla v^0|\|_r+\||\delta v||\nabla (\chi_-v^-+\chi_+v^+)|\|_r+\||\nabla \delta v||v|\|_r+\||\delta v||v|\|_r\\
\lesssim&\|\delta v\|_\infty \|\nabla v^0\|_{r}+\|\delta v\|_{L^r(L^\infty)}\|\nabla (\chi_-v^-+\chi_+v^+)\|_{L^\infty(L^r)}+\|\nabla \delta v\|_r\|v\|_\infty+\||\delta v\|_r\|v\|_\infty\\
\lesssim &\|\delta v\|_{L^r_1}\|v\|_{\V_1}.
\end{align*}

\medskip

\Case 2. If $\delta v=\chi_-\delta v^-$, then we think of $(d\CF)(v_2+\delta v)-(d\CF)(v_2)$ as a lower triangular matrix as in \eqref{VT.E.18} and verify that each component is bounded. The same argument as in \Case 1 shows that this operator is $L^r_1\to L^r$ and it acts trivially on $v=\chi_+v^+$, $v^+\in V_1^+$. 

We focus on the case when $v=\chi_-v^-$ for some $v^-\in V^-_1$. The $(1,1)$-entry of \eqref{VT.E.18} is bounded using Lemma \ref{VT.L.15} (this is a property about the real line $\R_s$). Since the section $v$ is supported on $\{t\leq -K\}$, we may pretend that $P^{\refe}_R$ is the constant trajectory at $p^{\pre}_{R,-}$ and $v_2=\chi_-v_2^-+v_2^0$ for this computation. With this understood, the $(3,1)$-entry of \eqref{VT.E.18} for $(d\CF)(v_2+\delta v)-(d\CF)(v_2)$ applied to $v=\chi_-v^-$ is given by
\[
(d\CF)(v_2+\chi_-\delta v^-)(\chi_- v^-)-(d\CF)(v_2)(\chi_- v^{-})-\chi_-\big((d\CF)(v_2^-+\delta v^-)(v^-)-(d\CF)(v_2^-)(v^-)\big)
\]
which is bounded pointwise by
\[
\chi_-(|\delta v^{-}||v^-||v_2^0|+|\nabla\delta v^{-}||v^-||v_2^0|+|\delta v^{-}||\nabla v^-||v_2^0|+|\delta v^{-}||v^-||\nabla v_2^0|)
\]
up to a uniform constant. This is due to \eqref{VT.E.28}. Since $\nabla \delta v^-,\nabla v^-\in L^\infty(\R_t; L^r(\R^s))$ and $v_2^0\in L^r_1(\R_t\times \R_s)\embed L^r(\R_t;L^\infty(\R^s))$, we conclude that this $(3,1)$-entry applied to $v=\chi_-v^-$ has $L^r$-norm bounded by 
\[
\|v^0_2\|_{L^r_1(\R_t\times\R_s)}\|\delta v^-\|_{L^r_1(\R_s)}\| v^-\|_{L^r_1(\R_s)}. 
\]

\Case 3. The case that $\delta v=\chi_-\delta^+$ is identical to \Case 2. 

This completes the proof of Lemma \ref{VT.L.13} \ref{IFT2}.
\end{proof}

\begin{proof}[Proof of Lemma \ref{VT.L.13} \ref{IFT3}] Note that if $s\leq \frac{(R+\pi)}{4}$, then $P^{\pre}_{R}=P^{un}$ and $P^{\refe}_R=P^{un}_K$, so $v_1$ is equal to the section $P^{un}-P^{un}_K$ in \eqref{VT.E.20}. Moreover, $v_1\equiv 0$ if $s\geq \frac{3(R+\pi)}{8}$. For $s\in  [\frac{R+\pi}{4},\frac{3(R+\pi)}{8}]_s$, since the instanton $P^{un}$ decays exponentially to $x_l$ as $s\to\infty$, one verifies that 
	\[
\|v_1\|_{L^r_1(\R_t\times [\frac{R+\pi}{4},\frac{3(R+\pi)}{8}]_s)}\lesssim \|P^{un}-P^{un}_K\|_{L^r_1(\R_t\times [\frac{R+\pi}{4},\frac{3(R+\pi)}{8}]_s)}\to 0
	\]
	as $R\to\infty$. This implies that $\|v_1\|_{L^r_1}\leq 2\delta_K$ for $R\gg \pi$.
	
	\medskip 
	
To estimate $\|\CF(v_1)\|_{\V_0}$, since $\Psi_{P^{\refe}_R}(v_1)$ is an isometry, it suffices to estimate $\F(P^{\pre}_R)$. This section is supported on $\R_t\times [\frac{R+\pi}{4},\frac{3(R+\pi)}{4}]_s$. Apply Proposition \ref{P1.8} to $P^{un}$ and $P^{st}$; we conclude that the $L^r_1$-norm of $\F(P^{\pre}_R)$ on $[-R,R]_t\times [\frac{R+\pi}{4},\frac{3(R+\pi)}{4}]_s$ decays exponentially as $R\to\infty$. The limits of $\F(P^{\pre}_R)$ at $\pm\infty$ are given precisely by $\grad \CA_{W,\fa^R_{m1}}(p^{\pre}_{R,\pm})$ which also decays exponentially as $R\to \infty$ in $L^r_1(\R_s)$. Finally, the convergence $P^{\pre}_R(t,\cdot)\to p^{\pre}_{R,\pm}$ is also exponentially in $L^r_1(I_{t}'\times \R_s)$ as $t\to \pm\infty$, so $\F(P^{\pre}_R)-\grad \CA_{W,\fa^R_{m1}}( p^{\pre}_{R,\pm})\to 0 $ exponentially in $L^r(I_t'\times \R_s)$. We conclude that $\|\CF(v_1)\|_{\V_0}\to 0$ as $R\to\infty$.
\end{proof}
\begin{proof}[Proof of Lemma \ref{VT.L.13}\ref{IFT4}] Let $D^{un}, D^{st}$ denote the linearization of  the $\alpha$-instanton equation at $P^{st}$ and $P^{un}$. Then $D^{st}$ is invertible, and $\ker D^{un}$ is 1-dimensional and is generated by $\pt P^{st}$. Choose a smooth section $v'$ of $(P^{un})^*TM\to \R^2$ supported on $[-1,1]_t\times [-1,1]_s$ such that the $L^2$-pairing $f=\langle \cdot, v'\rangle$ is non-zero on $\ker D^{un}$, so 
	\[
	D^{un}\oplus f: L^r_1(\R^2; (P^{un})^*TM)\to L^r(\R^2; (P^{un})^*TM\oplus \R.
	\]
is invertible. Such a section $v'$ exists due to a unique continuation property of $\pt P^{st}$. Since $P^{\refe}_R=P^{\pre}_R=P^{un}$ on $[-1,1]_t\times [-1,1]_s$, $f$ can be viewed also a map $\V_1\to \R$. We first consider the linearization of $\F$ at $P^{\pre}_R$:
	\[
D^{\pre}_R:	L^r_1(\R^2; (P^{\pre}_R)^*TM)\to L^r(\R^2; (P^{\pre}_R)^*TM).
	\]
	Since $P^{\pre}_R\equiv x_l$ when $s\in [\frac{3(R+\pi)}{8},\frac{5(R+\pi)}{8}]_s$, the same excision argument using the transformation \eqref{VT.E.5} on $\R_t\times \R_s$ allows us to compare $D^{st}\oplus D^{un}\oplus f$ with the operator 
	\[
(\pt+D_{x_l})\oplus (D^{\pre}_R\oplus f).
	\]
	and obtain an equation similar to \eqref{VT.E.6}. Again the error term $\delta_R$ has operator norm $\to 0$ as $R\to\infty$, so $D^{\pre}_R\oplus f: L^r_1\to L^r$ is also invertible, and its inverse is bounded uniformly. Returning to the reference trajectory $P^{\refe}_R$, this implies that $d\CF(v_1)\oplus f: L^r_1\to L^r$ is invertible. As a map from $\V_1\to \V_0$, $d\CF(v_1)\oplus f$ is cast into a lower triangular matrix as in \eqref{VT.E.18}, and we just verified that the last diagonal entry is invertible. The first two are given by $\Hess \CA_{W,\fa^R_{m1}}(p^{\pre}_{R,\pm})$, which are invertible by Lemma \ref{VT.L.4}. Off-diagonal entries are bounded by \ref{IFT2}. This proves that $d\CF(v_1)\oplus f: \V_1\to \V_0\oplus \R$ is invertible. We take $Q_R=(d\CF(v_1)\oplus f)^{-1}|_{\V_0}$. Then the operator norm of $Q_R$ is bounded uniformly and $\im Q_R=\ker f$.  
\end{proof}
\begin{proof}[Proof of Lemma \ref{VT.L.14}] For any $\epsilon>0$, we prove that $\|\tau_{t_n}P_n-P^{\pre}_{R_n}\|_{\V_1}<100\epsilon$ for $R_n\gg \pi$. The limits of $\tau_{t_n}P_n$ and $P^{\pre}_{R_n}$ at $\pm\infty $ are given respectively by $p_{R_n,\pm}\colonequals p^{st}_\pm\circ_{R_n}p^{un}$ and $p^{\pre}_{R_n,\pm}$. Write
	\[
	p^{st}_\pm\circ_{R_n} p^{un}=\exp_{p^{\pre}_{R_n,\pm}}(v_{R_n}^\pm)
	\]
with $v_{R_n}^\pm \in L^r_1(\R_s; (p^{\pre}_{R_n,\pm})^*TM)=V_1^\pm$, then by  \eqref{VT.E.3}
\[
\|v_{R_n}^\pm\|_{L^r_1(\R_s)}<\epsilon
\]
for $R_n\gg \pi$. Write 
\[
\tau_{t_n}P_n=\exp_{P^{\refe}_R}(v_n),\ v_n=\chi_-v_{R_n}^-+\chi_+v_{R_n}^++v_n^0\in \V_1. 
\]
We have to show that $\|v_n^0-v_1\|_{L^r_1}\to 0$ for $R_n\gg \pi$. Recall that 
\[
P^{\pre}_{R_n}=\exp_{P^{\refe}_R}(v_1),\ v_1\in L^r_1=V_1^0. 
\]

Since the convergences $P^{\pre}_{R_n}\to p^{\pre}_{R_n,\pm}$ and $\tau_{t_n}P_{n}\to p_{R_n,\pm}$ are exponential as $t\to\pm\infty$ (and uniform in $R_n$), for some $K_\epsilon\gg 0$, we have 
\[
\|v_n^0\|_{L^r_1(\{|t|\geq K_\epsilon\})}<\epsilon,\ \|v_1\|_{L^r_1(\{|t|\geq K_\epsilon\})}<\epsilon.
\]
By Lemma \ref{VT.L.5} and Proposition \ref{P1.8}, for some $R_\epsilon>0$, 
\[
\|v_n^0\|_{L^r_1(\{|s|>R_\epsilon\}\cap \{|s-R|\geq R_\epsilon\})}<\epsilon,\ \|v_1\|_{L^r_1(\{|s|>R_\epsilon\}\cap \{|s-R|\geq R_\epsilon\})}<\epsilon.
\]
Thus it suffices to estimate the $L^r_1$-norm of $|v_n^0-v_1|$ on the rectangles
\[
\Omega_\epsilon=[-K_\epsilon, K_\epsilon]_t\times [-R_\epsilon, R_\epsilon]_s\ \bigcup\ [-K_\epsilon, K_\epsilon]_t\times [R-R_\epsilon, R+R_\epsilon]_s.
\]
However, $\|v_n^0-v_1\|_{L^r_1(\Omega_\epsilon)}\to 0$ as $R_n\to \infty$ by the $C^\infty_{loc}$-convergence in Corollary \ref{VT.C.10}. This completes the proof of Lemma \ref{VT.L.14}.
\end{proof}

\subsection{Quasi-units revisited}\label{SecVT.6} Finally we prove Proposition \ref{prop:Quasi-Units} using the vertical gluing theorem. We focus on the second case when $q_0=q_1$ and set up the neck-stretching picture. As we will see shortly, this case is identical to the one addressed in Theorem \ref{VT.T.3}. Since $l_{q_2,\theta_2}$ intersects with $l_{q_0,\theta_0}$ at $y_0'$, we must have $\theta_0-\pi<\theta_2$. Choose another angle $\theta_{-1}=\theta_0+\epsilon$ for some $0<\epsilon\ll 1$ such that 
\begin{equation}\label{VT.E.32}
\theta_0-\pi<\theta_{-1}-\pi<\theta_2<\theta_1<\theta_0<\theta_{-1}.
\end{equation}
Moreover, if $l_{q_2,\theta_2}$ intersects with the ray $l_{q,\theta_{-1}}$ at $y_{-1}'$, then we require that the triangle $(y'_{-1}, y_1', W(q))$ contains no critical values of $W$ in the interior (think of $l_{q,\theta_{-1}}$ as a slight perturbation of $l_{q,\theta_0})$. Now consider the thimbles:
\[
U_\star=\Lambda_{q,\theta_{-1}-\pi} \text{ and }S_\star=\Lambda_{q,\theta_{-1}}. 
\]
and write for convenience  $\Lambda_j=\Lambda_{q_j, \theta_j}, j=0,1,2$. 
 For some $0<\delta\ll 1$, fix smooth monotone functions $\alpha_0,\alpha_1, \alpha_2:\R_s\to \R$ such that 
\begin{align}\label{VT.E.30}
\alpha_j(s)&= \left\{
\begin{array}{ll}
\theta_{-1}-\pi&\text{if }s\geq \pi-\delta,\\
\theta_j-\pi&\text{if }s\leq \delta,
\end{array}
\right. &
\alpha_2(s)&= \left\{
\begin{array}{ll}
\theta_2 &\text{if }s\geq \pi-\delta,\\
\theta_{-1}-\pi&\text{if }s\leq \delta.
\end{array}
\right. 
\end{align}

Choose admissible Floer data $\fa_{j}=(\pi, \alpha_j(s),\beta,\epsilon,\delta H_j\equiv 0)$ for the pair $(\Lambda_j, U_\star)$, $j=0,1$ and $\fa_2=(\pi,\alpha_2(s),\beta,\epsilon,\delta H_2)$ for $(S_\star, \Lambda_2)$. Then they can be concatenated as in Section \ref{SecVT.1} to give a Floer datum $\fa_{j2}^R=(R+\pi, \alpha_{j2}^R(s),\beta,\epsilon,\delta H_{j2}^R)$ for the pair $(\Lambda_j, \Lambda_2)$ with $\alpha_{j2}^R:\R_s \to \R$ defined similarly as in \eqref{GF.E.26} and $\delta H_{j2}^R$ as in \eqref{VT.E.31} for all  $R\geq \pi$ and $j=0,1$.  Finally, choose any admissible Floer data $\fa_{01}=(\pi, \alpha_{01},\beta_{01},\epsilon,\delta H_{01}\equiv0 )$ for the pair $(\Lambda_0,\Lambda_1)$.

\begin{figure}[H]
	\centering
	\begin{overpic}[scale=.15]{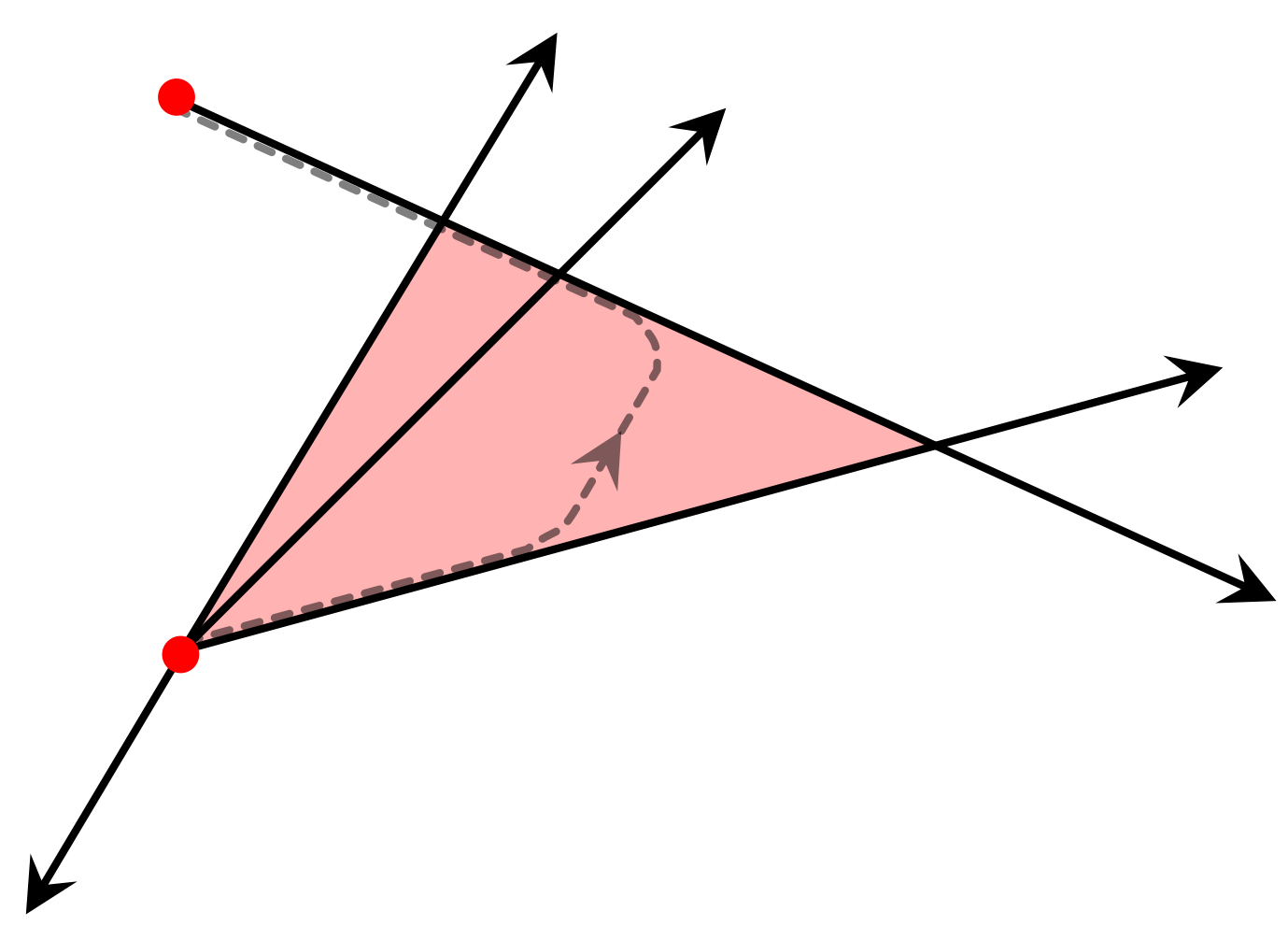}		
		\put(-20,20){$q=q_0=q_1$}	
		\put(5,65){$q_2$}
		\put(96,45){$\Lambda_1$}
		\put(60,65){$\Lambda_0$}
		\put(100,25){$\Lambda_2$}
		\put(45,70){$S_\star$}
		\put(-5,0){$U_\star$}
		\put(23,54){$y_{-1}'$}
				\put(49,52){$y_0'$}
						\put(70,42){$y_1'$}
	\end{overpic}	
	\caption{The dashed curve is the projection of an $\alpha_{12}^R$-soliton. As $R\to\infty$, it converges to the constant soliton at $q$ followed by an $\alpha_2$-soliton. }
	\label{Pic40}
\end{figure}
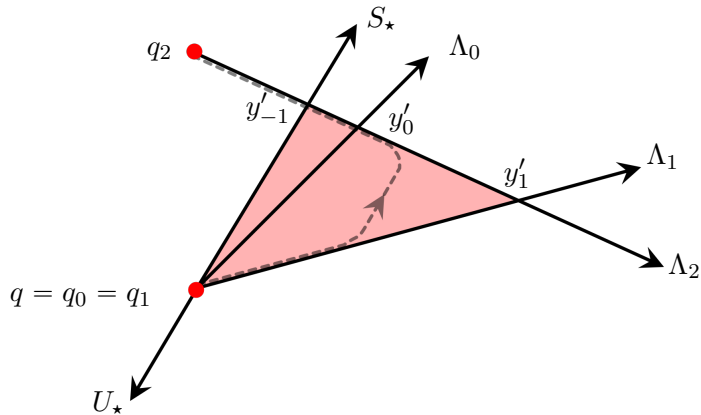

By \eqref{VT.E.32}, the function $\alpha_{j2}^R(s):\R_s\to \R$ is also monotone for all $R\geq \pi$ and $j=0,1$. Combined with our condition on the critical values of $W$, this implies that the complex 
\[
\Ch_\natural^*(\Lambda_j,\Lambda_2; \fa_{j2}^R), j=0,1
\]
only has one filtration level when $R\gg \pi$, and in fact 
\[
\Ch_\natural^*(\Lambda_0,\Lambda_2; \fa_{j2}^R)\cong \Ch_\natural^*(S_\star, \Lambda_2;\fa_2)\otimes \Ch_\natural^*(\Lambda_j, U_\star; \fa_j).
\]
This means any soliton $p_R\in \FC(\Lambda_0, \Lambda_2;\fa_{j2}^R)$ is obtained by gluing some $p_2\in \FC(S_\star. \Lambda_2;\fa_2)$ with the constant soliton $e_q\in \FC(\Lambda_j ,U_\star; \fa_j)$. As in the case of Theorem \ref{VT.T.3}, the vertical gluing theorem in this case identifies the cobordism map 
\begin{equation}\label{VT.E.34}
\mu^2:  \Ch_\natural^*(\Lambda_1,\Lambda_2; \fa_{12}^R)\otimes \Ch_\natural^*(\Lambda_0,\Lambda_1,\fa_{01})\to \Ch_\natural^*(\Lambda_0,\Lambda_2; \fa_{02}^R)
\end{equation}
with the tensor product of the identity continuation map
\[
\Id: \Ch_\natural^*(S_\star, \Lambda_2;\fa_2)\to \Ch_\natural^*(S_\star, \Lambda_2;\fa_2)
\]
and the cobordism map 
\[
\mu^2_q:  \Ch_\natural^*(\Lambda_1,U_\star; \fa_1)\otimes \Ch_\natural^*(\Lambda_0,\Lambda_1,\fa_{01})\to \Ch_\natural^*(\Lambda_0,U_\star; \fa_0).
\]
For the very last map $\mu^2_q$, all thimbles are associated to the same $q\in \Crit(W)$; so each Floer complex has rank 1 and is generated by the quasi-unit $e_q$. Then $\mu^2_q$ is an isomorphism by Lemma \ref{L3.7}. This implies that the cobordism map $\mu^2$ in\eqref{VT.E.34} is a chain isomorphism for $R\gg \pi$. This completes the proof of Proposition \ref{prop:Quasi-Units}. One may compare Figure \ref{Pic41} below with Figure \ref{Pic35}.

\begin{figure}[H]
	\centering
	\begin{overpic}[scale=.15]{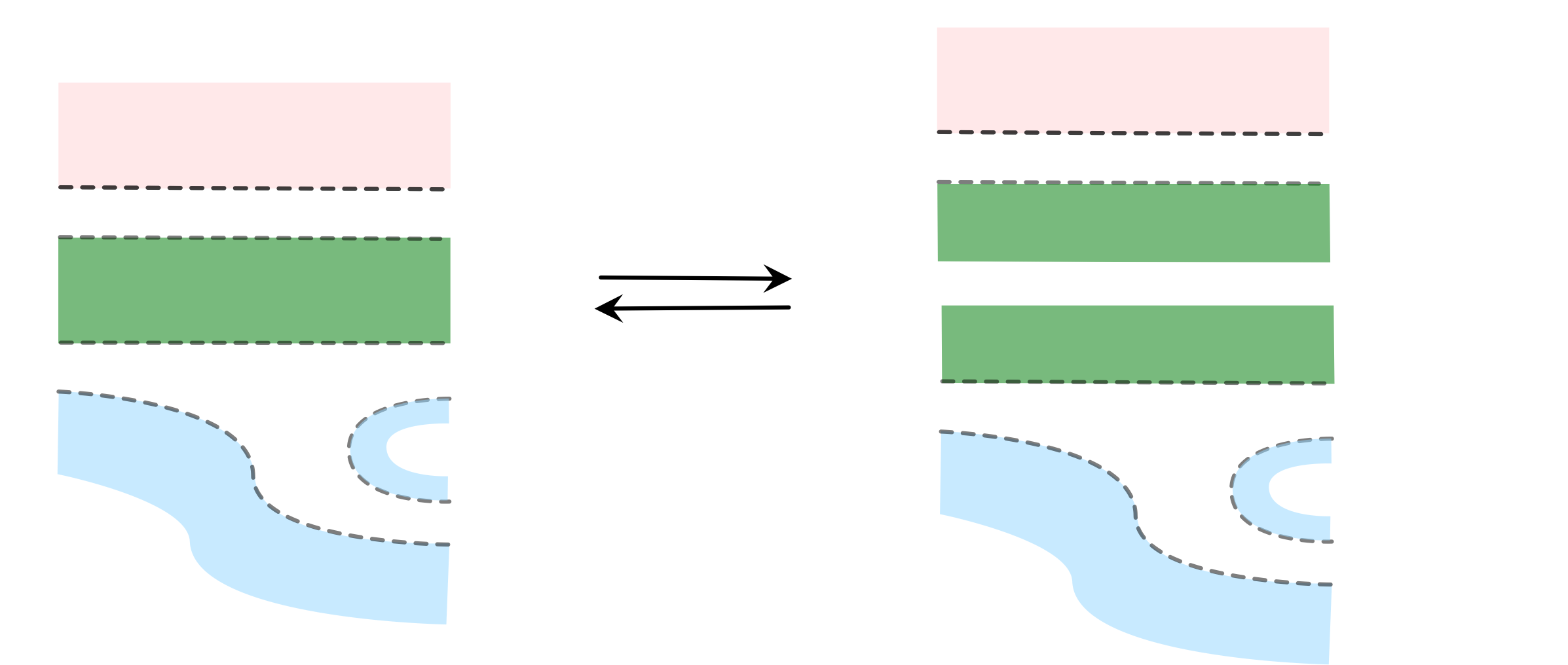}
		\put(12,23){\small$\R_t\times [\pi, R]_s$}
		\put(5,13){\small$\Lambda_0$}
		\put(5,23){\small$q$}
		\put(5,32){\small$\Lambda_2$}
		\put(65,9){\small$\Lambda_0$}
		\put(65,20){\small $U_\star$}
		\put(65, 27){\small $S_\star$}
		\put(65, 36){\small $\Lambda_2$}
		\put(75,2){\small$q$}
		\put(75,20){\small $q$}
		\put(75, 27){\small $q$}
		\put(75, 36){\small $q_2$}
		\put(54,32){\small $p_2$}
		\put(86,32){\small $p_2$}
		\put(54,16){\small $e_q$}
		\put(86,16){\small $e_q$}
		\put(-2,27){\small $p_2$}
		\put(30,27){\small $p_2$}
		\put(-2,18){\small $e_q$}
		\put(30,18){\small $e_q$}
		\put(30,8){\small $e_q$}
		\put(86,6){\small $e_q$}
		\put(-1.5,22.5){\small $\circ_R$}
		\put(30.5,22.5){\small $\circ_R$}
		\put(39,26){\small converge}
		\put(42,20){\small glue}
	\end{overpic}	
	\caption{Quasi-Units.}
	\label{Pic41}
\end{figure}

\subsection{Seidel's long exact sequence} Proposition \ref{prop:Quasi-Units} is concerned with the case when there is no critical values in interior of $(y_0', y_1', W(q))$, so the thimble $\Lambda_{q,\theta_0}$ can be deformed continuous into $\Lambda_{q,\theta_1}$ without violating the condition \eqref{E1.5} or the weaker version in Remark \ref{R1.5}. But what if there is such a critical value? 

\begin{figure}[H]
	\centering
	\begin{overpic}[scale=.15]{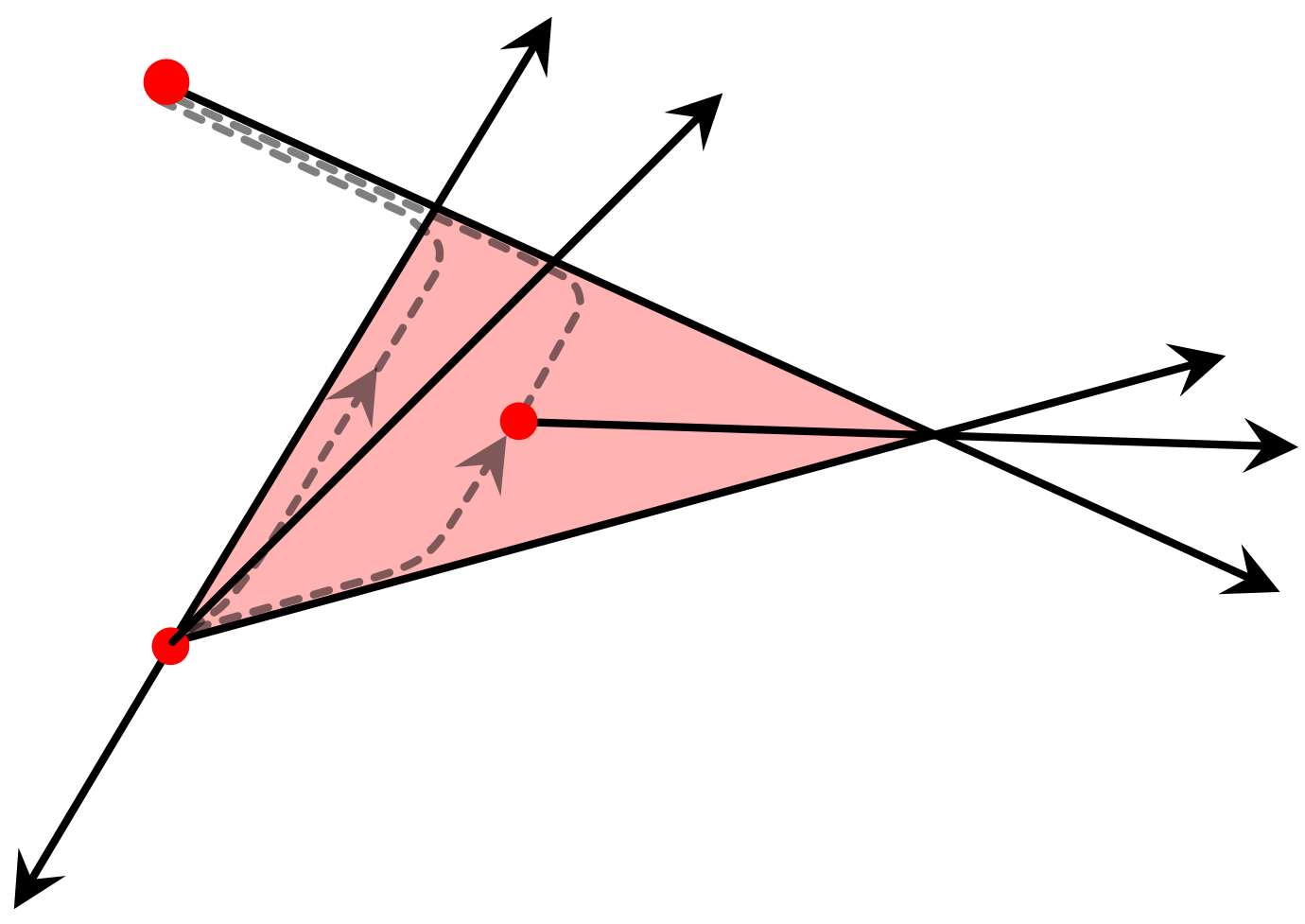}		
		\put(-20,20){$q=q_0=q_1$}	
		\put(5,65){$q_2$}
		\put(96,45){$\Lambda_1$}
		\put(60,65){$\Lambda_0$}
		\put(100,25){$\Lambda_2$}
				\put(103,35){$\Lambda_3$}
		\put(45,70){$S_\star$}
		\put(-5,0){$U_\star$}
		\put(42,33){$q_3$}
	\end{overpic}	
	\caption{The dashed curves are the projection of some $\alpha_{12}^R$-solitons. There are two filtration levels now.}
	\label{Pic42}
\end{figure}
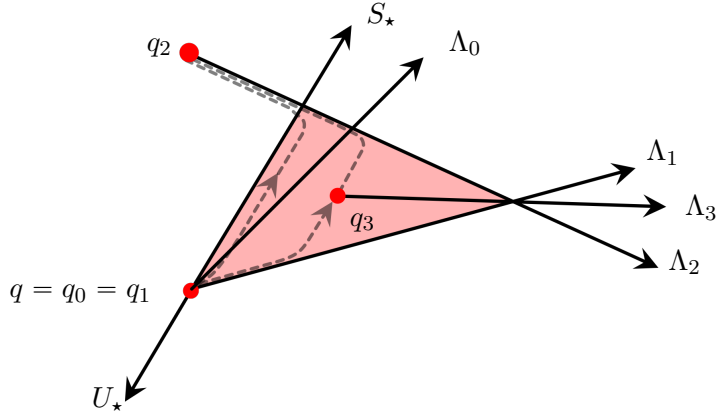

Suppose that this critical value is associated to $q_3\in \Crit(W)$. Choose an angle $\theta_3$ such that 
\[
\theta_{-1}-\pi<\theta_2<\theta_3<\theta_1<\theta_0<\theta_{-1}
\]
and let $\Lambda_3\colonequals \Lambda_{q_3,\theta_3}$. Figure \ref{Pic40} is then replaced by Figure \ref{Pic42} above. One can also choose a function $\alpha_3:\R_s\to\R$ such that for some $0<\delta\ll 1$, 
\begin{align*}
\alpha_3(s)&= \left\{
\begin{array}{ll}
\theta_{-1}-\pi&\text{if }s\geq \pi-\delta,\\
\theta_3-\pi&\text{if }s\leq \delta,
\end{array}
\right. 
\end{align*}
to define a filtered complex $\Ch^*_\natural(\Lambda_3,\Lambda_2;\fa_{32}^R)$ for $R\gg \pi$ (replace $\Lambda_0$ by $\Lambda_3$ throughout the construction of Section \ref{SecVT.6} ). In this case, there are two filtration levels for $\Ch^*_\natural(\Lambda_1, \Lambda_2;\fa_{12}^R)$: 
\[
\Ch^*_\natural(\Lambda_1, \Lambda_2;\fa_{12}^R)=\sG^0\Ch^*_\natural(\Lambda_1, \Lambda_2;\fa_{12}^R)\oplus \sG^1\Ch^*_\natural(\Lambda_1, \Lambda_2;\fa_{12}^R)
\]
where $\sG^0$ (resp. $\sG^1$) is associated to $q_3$ (resp. $q$). The summand $\sG^0$ has higher energy level. The geometric filtration on $\Ch^*_\natural(\Lambda_0, \Lambda_2;\fa_{02}^R)$ then gives a short sequence sequence of chain complexes:
\begin{equation}\label{VT.E.35}
0\to \sG^0\to \Ch^*_\natural(\Lambda_0, \Lambda_2;\fa_{02}^R)\to \sG^1\to 0.
\end{equation}
Meanwhile,
\[
\Ch^*_\natural(\Lambda_0,\Lambda_2;\fa_{32}^R)=\sG^1\Ch^*_\natural(\Lambda_0,\Lambda_2;\fa_{02}^R)  \text{ and }\Ch^*_\natural(\Lambda_3,\Lambda_2;\fa_{32}^R)=\sG^0\Ch^*_\natural(\Lambda_3,\Lambda_2;\fa_{32}^R).
\]
The chain map induced by the $(2+1)$-disk preserves the filtration, so we arrive at a diagram 
\begin{equation}\label{VT.E.33}
\begin{tikzcd}[column sep=3em]
\sG^0\Ch^*_\natural(\Lambda_3,\Lambda_2;\fa_{32}^R)\otimes \Ch^*_{\natural}(\Lambda_1,\Lambda_3;\fa_{13})\arrow[r,"\mu^2"]\arrow[d,"\mu^2"]& \sG^0\Ch^*_\natural(\Lambda_1, \Lambda_2;\fa_{12}^R)\arrow[d,hook]\\
\Ch^*_\natural(\Lambda_1,\Lambda_2;\fa_{32}^R)\arrow[r,equal]\arrow[d,->>, "{\mu^2(\cdot, e_q)}"] &\Ch^*_\natural(\Lambda_1,\Lambda_2;\fa_{32}^R)\arrow[d,->>]\\
\sG^1\Ch^*_\natural(\Lambda_0,\Lambda_2;\fa_{32}^R)\arrow[r,"{\mu^2(\cdot,e_q)}","\cong"']&\sG^1\Ch^*_\natural(\Lambda_1, \Lambda_2;\fa_{12}^R)
\end{tikzcd}
\end{equation}
where the right column is given by \eqref{VT.E.35}. As in the case of \eqref{VT.E.34}, the bottom horizontal arrow of \eqref{VT.E.33} is a chain isomorphism, whereas the top arrow is only a quasi-isomorphism by the vertical gluing theorem and Proposition \ref{prop:Quasi-Units}. Nevertheless, this proves that the left column forms an exact triangle and hence recovers a special case of Seidel's long exact sequence \cite{S03}. 
\begin{proposition}\label{VT.P.17} Suppose we are in the scenario of Figure \ref{Pic42}, then there is a long exact sequence:
	\[
\begin{tikzcd}
	\HFF_\natural^*(\Lambda_3,\Lambda_2)\otimes 	\HFF_\natural^*(\Lambda_1,\Lambda_3)\arrow[rr,"m_*"]& & 	\HFF_\natural^*(\Lambda_1,\Lambda_2)\arrow[dl,"{m_*(e_q,\cdot)}"]\\
	&  \HFF_\natural(\Lambda_0,\Lambda_2)\arrow[lu,"{[1]}"].&
\end{tikzcd}
	\]
	Thus the group $\HFF_\natural^*(\Lambda_3,\Lambda_2)\otimes 	\HFF_\natural^*(\Lambda_1,\Lambda_3)$ measures the failure of $m_*(e_q,\cdot)$ being an isomorphism. 
\end{proposition}
\begin{remark} The fact that the composition $m_*(e_q, m_*(\cdot, \cdot)): \HFF_\natural^*(\Lambda_3,\Lambda_2)\otimes 	\HFF_\natural^*(\Lambda_1,\Lambda_3)\to \HFF_\natural(\Lambda_0,\Lambda_2)$ is trivial be also seen as follows: since $m_*$ is associative, this map factorizes through the group $\HFF_\natural^*(\Lambda_0,\Lambda_3)$, which is trivial, since $\Lambda_0\cap \Lambda_3=\emptyset$; see Figure \ref{Pic42}.
\end{remark}

\part{Lagrangian Submanifolds}\label{Part3}

Having verified the Koszul duality pattern between $\sA$ and $\sB$, our next step is to explore its implication for a pair of $(X, Y)$ of Lagrangian submanifolds. Section \ref{SecFL} is devoted to a generalization of $\HFF_\natural^*$ where one entry is replaced by an exact graded Lagrangian submanifold, which allows us to define the $A_\infty$-modules ${}_{\sA}X$ and $Y_{\sB}$. In Section \ref{SecSS}, we prove our main result Theorem\ref{Intro.T.8} by verifying the geometric filtration on the Floer complex of $(X,Y)$ is isomorphic to the algebraic filtration of Seidel \cite[Corollary 18.27]{S08}. This proof will follow the same argument as in Theorem \ref{FS.T.3}. 

\section{Floer Cohomology for Lagrangian Submanifolds}\label{SecFL}

\subsection{The setup} Our proof of the compactness theorem so far is based on the energy method; in this regard, it is not realistic to work with arbitrary \textbf{non-compact} Lagrangian submanifolds. We have to make some mild conditions on the projection $W(X), W(Y)\subset \C$  in order to make this energy method work, which are spelled out as follows.
\begin{assumpt}\label{assumption:2} In this paper, we shall only consider Lagrangian submanifolds $X$ of $M$ satisfying the following properties: 
	\begin{enumerate}[label=$($L\arabic*$)$]
\item\label{L-1} $X$ is exact, i.e., there exists a smooth function $h_X: X\to \R$ such that $dh_X=\lambda_M|_X$;
\item\label{L-2} $X$ is graded, i.e., the Maslov class $\mu_X\in H^1(X,\Z)$ vanishes, and a grading $(X,\xi_X^\#)$ is already fixed. 
\item\label{L-3} there exist some $\beta_X\in \R$ and $C_X>0$ such that 
\begin{equation}\label{FL.E.6}
\re(e^{-i\beta_X}W)|_X<C_X,
\end{equation}
i.e., the projection $W(X)\subset \C$ is bounded above in the direction of $e^{i\beta_X}$. If $W(X)$ is compact, then the angle $e^{i\beta_X}$ in \eqref{FL.E.6} can be arbitrary. But we may also see the case when $e^{i\beta_X}$ is unique. We assume that a choice of such $\beta_X\in \R$  has been chosen, once and for all, for this $X$. Unlike the case of thimbles, the grading of $X$ is not related to $\beta_X$. 
\item\label{L-4} $X$ has bounded geometry as a totally real submanifold of $(M, J_M, g_M,\omega_M)$. \qedhere
	\end{enumerate}
\end{assumpt}

Let $\Lambda_{q,\theta}, q\in \Crit(W)$ be any thimble satisfying \eqref{E1.5} and with
\begin{equation}\label{FL.E.8}
e^{i(\theta-\beta_X)}>0.
\end{equation}
Under this assumption, one can construct the Floer cohomology $\HFF_\natural^*(X, \Lambda_{q,\theta})$ following the scheme in Section \ref{Sec1}. In this case, a Floer datum $\fa_X= (R,\alpha_X(s),\beta_X, \epsilon_X, \delta H)$ consists of a constant $R\geq\pi$, a real-valued function $\alpha_X: \R^+_s\to \R$, an angle $\beta_X\in \R$ as in \ref{L-3}, $\epsilon_X>0$, and a perturbation 1-form $\delta H=\delta H_s ds\in \Omega^1(\R_s^+; \SH)$, which satisfy the following properties:
\begin{itemize}
\item for some $0<\delta\ll 1$, $\alpha_X(s)=\theta$ for all $s\geq R-\delta$;
\item $\re(e^{i(\alpha_X(s)-\beta_X)})\geq \epsilon_X$ for all $s\in\R^+_s$ (this condition requires \eqref{FL.E.8});
\item $\delta H$ is supported on $[0,R]_s$ and satisfies \eqref{E1.21}. 
\end{itemize}
In practice, one may simply take $\alpha_X(s)\equiv \theta$ to be constant, but it is harmless to make the setup a bit more flexible. Let $p_{\model}: \R_s^+\to M$ be a reference map such that $p_{\model}(s)\equiv q$ when $s\geq R$.  Consider the path space 
\[
\Pa_k(X,\Lambda_{q,\theta})\colonequals \{p\in C^\infty(\R_s^+; M): p(0)\in X,\ p \text{ has finite $L^2_k$-distance with } p_{\model}\}, k\geq 1.
\]
Then any Floer datum $\fa_X$ defines an action functional $\CA_{X,\fa_X}$ on  $\Pa_k(X, \Lambda_{q, \theta})$ by the formula:
\begin{align}\label{FL.E.3}
\CA_{X,\fa_X}(p)&\colonequals -h_X(p(0))+\int_{\R_s^+}-p^*\lambda_M+\delta H_s(p(s))ds\\
&+\int_{\R_s^+}\big( \im(e^{-i\alpha_X}(W\circ p(s)))-\chi_{[R,+\infty)} \im(e^{-i\theta}W(q))\big) ds \nonumber
\end{align}
whose gradient vector at $p\in \Pa_k(X, \Lambda_{q, \theta})$ is an $L^2$-section of $p^*TM$:
\[
\grad \CA_{X,\fa_X}(p)=J\ps p(s)+\nabla \im (e^{-i\alpha_X(s)}W)\circ p(s)+\nabla (\delta H_s)\in L^2\cap C^\infty(\R_s^+, p^*TM). 
\]

\begin{figure}[H]
	\centering
	\begin{overpic}[scale=.15]{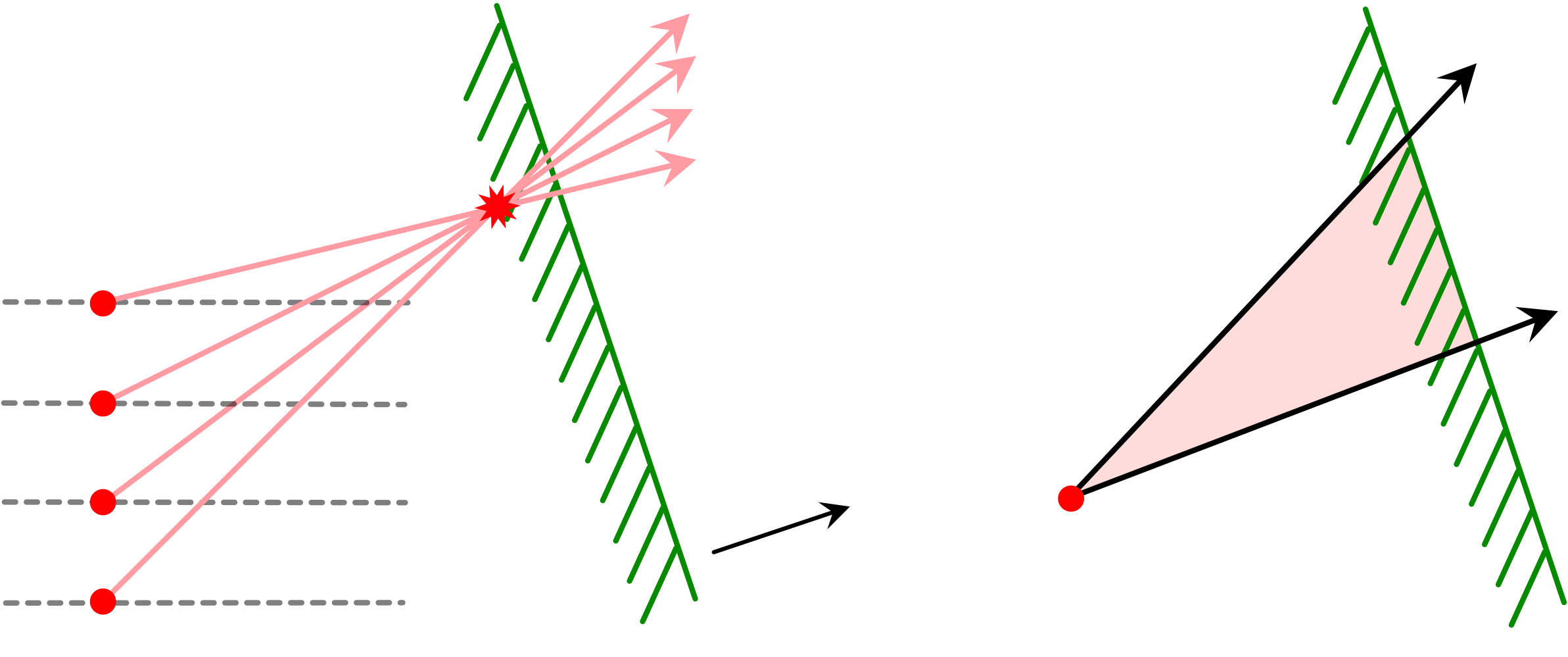}
		\put(45,31){\small $S_4$}
		\put(45,34){\small $S_3$}
		\put(45,37){\small $S_2$}
		\put(45,40){\small $S_1$}
		\put(-10,22){\small $H(x_4)$}
		\put(-10,16){\small $H(x_3)$}
		\put(-10,9){\small $H(x_2)$}
		\put(-10,3){\small $H(x_1)$}
		\put(35,5){\small $X$}
		\put(48,10){\small $e^{i\beta_X}$}
		\put(60,35){\small $\re\langle w, e^{i\beta_X}\rangle=C_X$}
		\put(95,40){\small $l_{q,\theta_0}$}
		\put(102,24){\small $l_{q,\theta_1}$}
		\put(66,6){\small $W(q)$}
	\end{overpic}	
	\caption{The left $\sA$-module ${}_{\sA}X$ (left) and quasi-units (right).}
	\label{Pic30}
\end{figure}
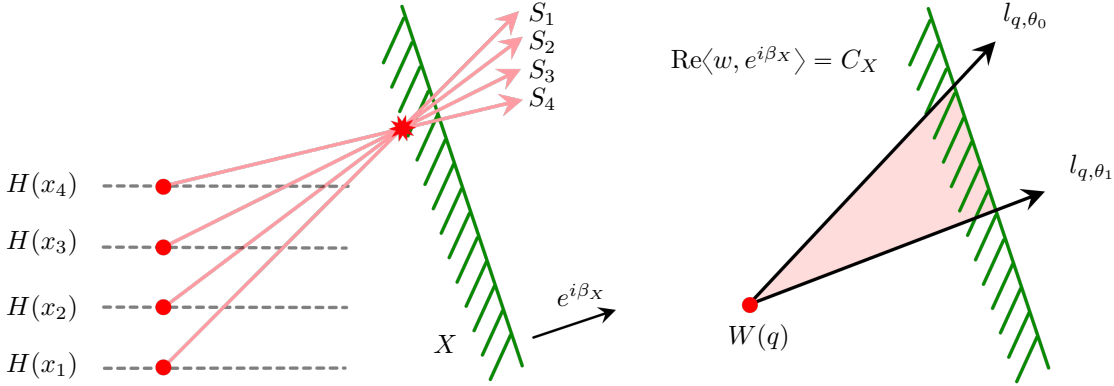

Let $\FC(X,\Lambda_{q,\theta};\fa_X)$ denote the space of critical points of $\CA_{X,\fa_X}$, also called $\alpha_X$-solitons (this is an abuse of notations, because they also rely on the perturbation 1-form $\delta H$). The Floer cohomology  $\HFF_\natural^*(X,\Lambda_{q,\theta};\fa_X)$ is defined by counting the $\alpha_X$-instantons on the upper half plane $\R_t\times \R_s^+$ with a boundary condition along $\R_t\times \{0\}$.
\begin{equation}\label{FL.E.4}
\left\{\begin{array}{rl}
P: \R_t\times \R_s^+&\to M,\\
\pt P+J(\ps P+\nabla\re(e^{-i\alpha_X(s)}W))+\nabla (\delta H)_s&=0,\\
\lim_{s\to \infty} P(t,s)&=q,\\
\lim_{t\to\pm \infty}(P,s)&=p_\pm (s),\\
P(t,0)&\in X. 
\end{array}
\right.
\end{equation}
with $p_\pm\in \FC(X,\Lambda_{q,\theta};\fa_X)$. The boundary condition should be interpreted as in \eqref{E1.7}. A solution of \eqref{FL.E.4} is formally a downward gradient flowline of $\CA_{X,\fa_X}$.

\subsection{Compactness} We explain quickly the necessary adaption in order to prove the compactness theorem in this context.  Let $\gamma_X:\R^+_s\to \C$, $\ps \gamma_X(s)=-e^{i\alpha_X(s)}$ be a characteristic curve of $\alpha_X(s)$. To drive the energy estimate, consider the phase function $\Xi$
\[
\Xi(s,t)=-i\epsilon_X e^{i\beta_X}\cdot t+\gamma_X(s),\ (t,s)\in\HH^+=\R_t\times \R_s^+
\]
and set $\delta\kappa=\im (i\epsilon_X e^{-i\beta_X}dt\cdot W)$. For any $\alpha_X$-instanton $P$, apply the energy estimate of Lemma \ref{L1.12} first to the rectangle $[t_0, t_1]_t\times [0,R_1]_s$, $R_1\gg R$ and then let $R_1\to\infty$ to obtain that 
\begin{align}\label{FL.E.5}
&\CA_{X,\fa_X}\big(P(t_0,\cdot)\big)-\CA_{X,\fa_X}\big(P(t_1,\cdot)\big)+\text{ some terms involving }\delta H\\
&+\epsilon_{X}\int_{t_0}^{t_1}\re\big(e^{-i\beta_X}\big(W(p(t,0))-W(q)\big)\big)\geq \epsilon_X'\E_{an}(P; [t_0, t_1]\times \R_s^+). \nonumber
\end{align}
for a constant $\epsilon_X'>0$. This is the analogue of Lemma \ref{L1.5} in the presence of a Lagrangian boundary condition. The condition \eqref{FL.E.6} is used here to upper bound $\re(e^{-i\beta_X}W(p(t,0)))$ on the left hand side of \eqref{FL.E.5}. At this point, we need a boundary version of Lemma \ref{L1.6}. For future reference, we record this as follows. 

\begin{lemma}[Local Compactness II: the Boundary Case]\label{FL.L.1} Let $\Omega\subset \HH^+$ denote any bounded open subset and $C>0$. Suppose that $\xi: \Omega\to \C$ is any smooth function such that $|\xi(z)|<C$, and $H^t_z, H^s_z: M\to \R,\ z\in \Omega$ are families of Hamiltonian functions on $M$ with
	\[
	\|\nabla H^s_z\|_{L^\infty(M)}, \|\nabla H^t_z\|_{L^\infty(M)}<C
	\]
	for all $z\in \Omega$. Then any sequence of solutions $P_n:\Omega\to M$ to the Floer equation 
	\begin{equation}
	(\pt P_n-J\nabla H_z^t)+J(\ps P_n-\nabla \re (\overline{\xi(z)} W)-J\nabla H_z^s)=0
	\end{equation}
	subject to the boundary condition $P(\Omega\cap \R_t\times \{0\})\subset X$ and the energy bound
	\begin{equation}
	\E_{an}(P_n;\Omega)=\int_\Omega |\nabla P_n|^2+|\nabla H\circ P_n|^2<C,
	\end{equation}
	has a subsequence that converges in $C^\infty_{loc}$-topology on $\Omega$.  
\end{lemma}

This lemma is deduced again using a diameter estimate for ``almost" $J$-holomorphic curves in an almost Hermitian manifold with boundary lying in a totally real submanifold. 
\begin{proposition}\label{FL.P.1} Suppose that $(M, J_M, g_M,\omega_M)$ is an almost Hermitian manifold with bounded geometry $($$\omega_M$ is not necessarily closed$)$, and $X\subset M$ is a totally real submanifold with $2\dim_\R X=\dim_\R M$. Denote by $\bpartial_J$ the operator $\pt+J\ps$ and by $B^+(0,R)=B(0,R)\cap \HH^+, R>0$ the upper half disk. If $X$ also has bounded geometry, then for any $r>2$, there exists a function $\varphi:\R_+\to \R_+$ such that for any $C>0$ and for any map $P: B^+(0,1)\to M$ if 
	\[
	\int_{B^+(0,1)} |\nabla P|^2+|\bpartial_J P|^r\leq C \text{ and }P(t,0)\in X \text{ for all }t\in (-1,1)_t,  
	\]
	then $\diam (P(B^+(0,\half)))\leq \varphi(C)$.
\end{proposition}

The proof of this proposition is by repeating the argument in Appendix \ref{SecDE} and using the reflection principle; the complete argument is omitted in this paper. With this diameter estimate in hand, one can deduce the compactness theorem of \eqref{FL.E.4}, both the local and global versions, following the argument in Section \ref{Sec2}. Hence the Floer cohomology $\HFF_\natural^*(X,\Lambda_{q, \theta};\fa_X)$ is well-defined if the perturbation 1-form $\delta H$ is chosen generically, in which case we say that $\fa_X$ is admissible. Finally, the grading function $\gr$ introduced in Section \ref{SecLG.4} is generalized rather trivially to this case, so $\HFF_\natural^*(X,\Lambda_{q, \theta};\fa_X)$  is canonically $\Z$-graded.

\subsection{Continuation Maps} Given any pair of admissible Floer data $$\fa_X^\pm=(R^\pm,\alpha_X^\pm, \beta_X, \epsilon_X^\pm, \delta H^\pm),$$
the continuation method in Section \ref{Subsec:Continuation} shows that $\HFF_\natural^*(X,\Lambda_{q,\theta};\fa_X^\pm)$ are canonically isomorphic, but there is a caveat about the proof: the energy inequality \eqref{E2.13} is replaced in this case by 
\begin{align}\label{FL.E.7}
&\CA_{X,\fa^-_X}(P(0,\cdot ))-\CA_{X,\fa^+_X}(P(K,\cdot ))+ \text{ some terms involving }\delta H\\
+
&\int_{0}^{K} \im(W(P(t,0))\cdot\overline{\pt g_0(t)})dt-\int_{0}^{K} \im(W(q)\cdot\overline{\pt g_1(t)})dt
>\epsilon_\fc'\int_{[0,K]\times \R_s} u.\nonumber
\end{align}
Thus for the embedding $\Xi^\dagger:\R_t\times [0,\pi]_s\to \C$ considered in Section \ref{Subsec:Continuation} (again we have assumed that $R^\pm= \pi$), we have to insist that $\pt g_0(t)$ is always in the direction of $-ie^{i\beta_X}$, i.e., 
\begin{equation}\label{FL.E.9}
g_0(t)\colonequals \Xi^\dagger(t,0)=-ie^{i\beta_X}\cdot h(t)+\gamma^-(0),
\end{equation}
for a real-valued function $h: \R_t\to \R$ such that $\pt h(t)\equiv \epsilon_X^-$ for $t\leq 0$ and $\equiv \epsilon_X^+$ for $t\gg 0$. Otherwise, we can not estimate the term
\[
\int_{0}^{K} \im(W(P(t,0))\cdot\overline{\pt g_0(t)})dt
\]
on the left hand side of \eqref{FL.E.7} using \eqref{FL.E.6}. It is for the same reason that $\HFF_\natural^*(X,\Lambda_{q,\theta};\fa_X)$ is defined only when the condition \eqref{FL.E.8} holds. This limitation may be get rid of using more sophisticated analytic tools in the future, but in that case the invariance of this Floer cohomology can not be verified using the continuation method. From now on we shall use $\HFF_\natural^*(X,\Lambda_{q,\theta})$ to denote the isomorphism class of $\HFF_\natural^*(X,\Lambda_{q,\theta};\fa_X)$ for any admissible $\fa_X$. 

\subsection{Stable- and Unstable-type Lagrangian submanifolds} We say that a Lagrangian submanifold $X$ is of \textit{unstable-type} (resp. \textit{stable-type}) if $\re(e^{i\beta_X})>0$ (resp. $<0$), in which case we shall always assume that $\beta_X\in (-\frac{\pi}{2}, \frac{\pi}{2})$ (resp. $\in (\frac{\pi}{2},\frac{3\pi}{2})$). Let $\sA$ and $\sB$ denote the Fukaya-Seidel category of $(M,W)$ and its dual category respectively. If $X$ is of unstable-type, then we require further that $\sA=\sE(\Th_0)$ is defined using an admissible set of thimbles $\Th_0=(S_1,\cdots, S_m)\in \Gamma_0,\ S_{k}=\Lambda_{x_k,\theta_k}, 1\leq k\leq m$ such that
\begin{equation}\label{FL.E.10}
0<\theta_m<\cdots<\theta_2<\theta_1<\min\{\theta_\star,\beta_X+\frac{\pi}{2}\}.
\end{equation}

Under this condition the Floer cohomology $\HFF_\natural^*(X,S_k)$ is well-defined for all $S_k\in \Ob \sA$. Choose admissible Floer data $\fa_{X,k}$, one for each pair $(X, S_k)$. Then there is a finite directed $A_\infty$-category $\sE_X$ with 
\[
\Ob\sE_X: X\prec S_1\prec S_2\prec\cdots \prec S_m, 
\]
containing $\sA$ as a full $A_\infty$-subcategory and with
\[
\hom_{\sE_X}(X, S_k)=\Ch^*_\natural(X, S_k;\fa_{X,k}),
\]
which is the Floer complex associated to $(X, S_k;\fa_{X,k})$ for all $k$. The construction of $\sE_X$ is identical to the one in Section \ref{SecFS}: we shall use $\CT^{m+1}$ as the metric ribbon tree; if a $(d+1)$-pointed disk $S$ is equipped with a compatible quadratic differential $\phi$ and labeled by $(X, S_{k_1},\cdots, S_{k_d})$, then we remove the lower half plane $\HH^-_0$ from $\hat{S}$ and consider the partial completion:
\[
\hat{S}_X\colonequals \hat{S}\setminus \sigma_0(\HH_0^-)=S\ \bigcup\ \HH^-_1\cup \cdots \cup \HH^-_d.
\]

To specify the phase pair on $\hat{S}_X$, pretend that $X$ is a thimble $\Lambda_{q, \theta_{0}}$ with $\theta_{0}=\pi$  and use $\Lambda_{q,\theta_{0}}\sqcup \Th_0$ as the set of labels to determine a phase pair $(\Xi, \delta\kappa)$ on $\hat{S}$ (as in Section \ref{SecFS.2}); then restricts this pair to the partial completion $\hat{S}_X\subset \hat{S}$. As we already noted, in order to make the energy method work when replacing the planar end $\HH^-_0$ in $\hat{S}$ by a Lagrangian  boundary condition along $\partial \hat{S}_X$, we have to require that 
\begin{equation}\label{FL.E.12}
g_0\colonequals \Xi^\dagger \circ \sigma_0(\cdot,0):\R_t\to\C
\end{equation}
be a straight line in the direction of $-ie^{i\beta_X}$; compare \eqref{FL.E.9}. Figure \ref{Pic31} below illustrates such a phase function for a cobordism map 
\[
\HFF_\natural^*(\Lambda_{q_0,\theta_0}, \Lambda_{q_1,\theta_1})\otimes \HFF_\natural^*(X, \Lambda_{q_0,\theta_0})\to  \HFF_\natural^*(X, \Lambda_{q_1,\theta_1})
\]
with $\theta_1<\theta_0<\theta_1+2\pi$. With this convention understood, one can prove the energy estimate as usual. Finally, suppose that this construction is carried out for another $\Th_0'\in \Gamma_0$ which satisfies \eqref{FL.E.10} and $\Th_0'\sto \Th_0'$, then the $A_\infty$-category defined using $X\sqcup \Th_0'$, which we denote by $\sE_X'$,  is quasi-isomorphic to  $\sE_X$. This is seen by building up a larger category $\tilde{\sE}_X$ with
\[
\Ob \tilde{\sE}_X: X\prec S_1\prec S_2\prec\cdots \prec S_m\prec S_1'\prec S_2'\prec\cdots\prec S_m'.
\]
and by repeating the proof of Proposition \ref{FS.P.18}. Thus the left $\sA$-module ${}_\sA X$ associated to $\sE_X$, which assigns to each $S_k\in \Ob\sA$ the Floer complex $\hom_{\sE_X}(X, S_k)$, is well-defined up to canonical quasi-isomorphism. The following property about quasi-units is used in this proof.

\begin{figure}[H]
	\centering
	\begin{overpic}[scale=.15]{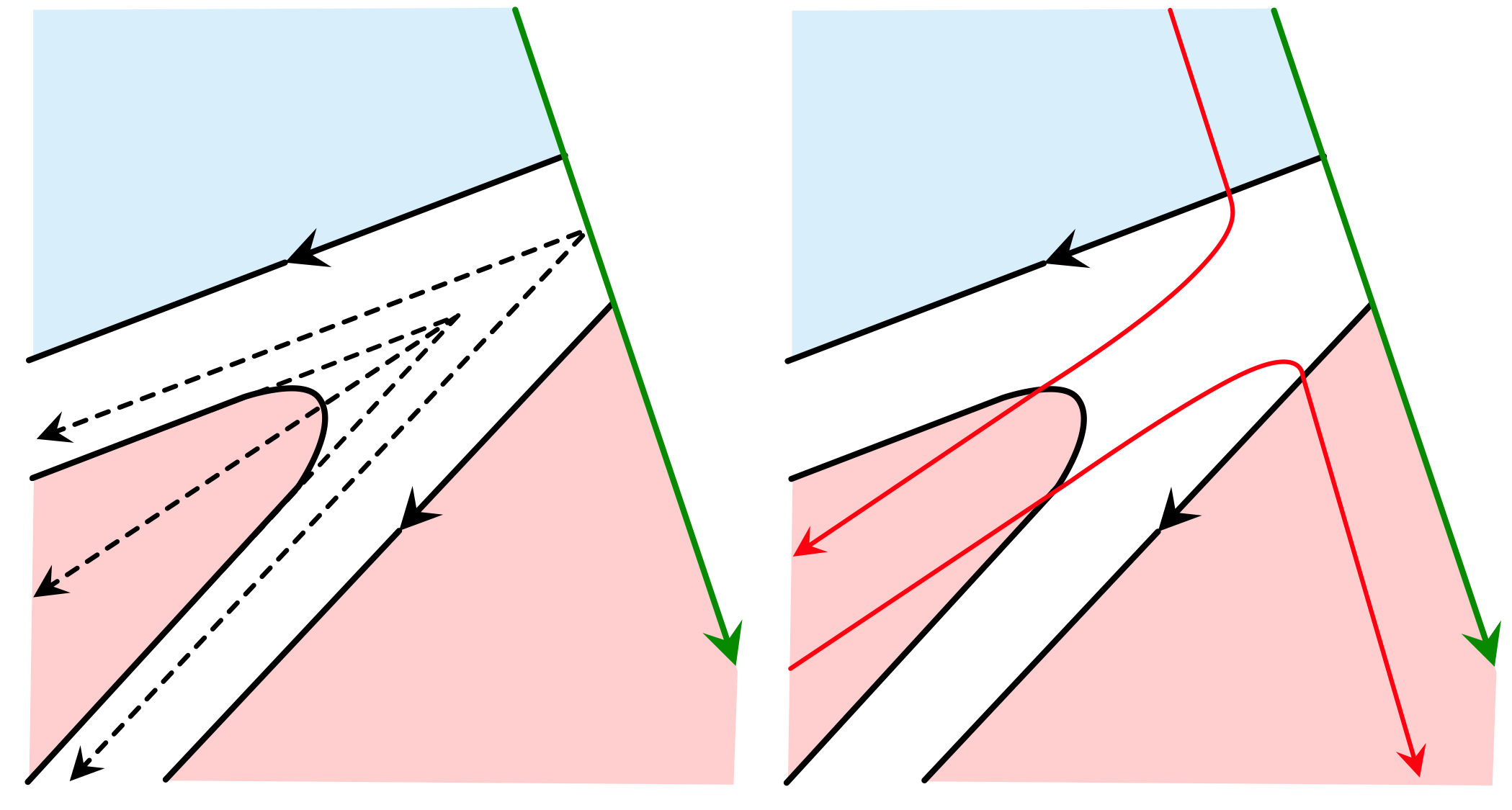}
		\put(-5,22){\small $-e^{i\theta_1}$}
		\put(2,-2){\small $-e^{i\theta_0}$}
		\put(101,10){\small $-ie^{i\beta_X}$}
	\end{overpic}	
	\caption{The phase function on a pair-of-pants cobordism map. The green curve is $g_0(t)$, while red curves are the images of $g_1(t)$ and $g_2(t)$. The blue region corresponds to the incoming end, while the red ones correspond to outgoing ends.}
	\label{Pic31}
\end{figure}

\begin{proposition}\label{FL.P.4} Let $(\Lambda_{q,\theta_0}, \Lambda_{q, \theta_1})$ be any thimbles with $\theta_1<\theta_0<\theta_1+2\pi$ and such that \eqref{FL.E.8} holds with $\theta=\theta_0,\theta_1$. Suppose in addition that the triangle bounded by $l_{q,\theta_0},\ l_{q,\theta_1}$ and the equation $\re\langle w, e^{i\beta_X}\rangle\leq C_X$ contains no critical values of $W$ other than the vertex $W(q)$, then the multiplication map
	\[
	m_*(e_q,\cdot):\HFF_\natural^*(X, \Lambda_{q, \theta_0})\to \HFF_\natural^*(X, \Lambda_{q, \theta_1})
	\] 
	is an isomorphism. 
\end{proposition}

The proof of this result follows the same line of arguments in Section \ref{SecVT.6} (using the vertical gluing theorem) and is omitted here.

\medskip

In the same vein, for any stable-type Lagrangian submanifold $Y=(Y, h_Y, \xi_Y^{\#},\beta_Y)$, one may construct a finite directed $A_\infty$-category $\sE_Y$ with 
\[
\Ob\sE_Y: Y\prec U_m\prec U_{m-1}\cdots \prec U_1.
\]
if the admissible set of thimbles $\Th_\pi\in \Gamma_\pi$ satisfies the relation
\begin{equation}\label{FL.E.11}
\pi<\eta_1<\eta_2<\cdots<\eta_m<\min\{\pi+\theta_\star, \beta_Y+\frac{\pi}{2}\}, \beta_Y\in (\frac{\pi}{2},\frac{3\pi}{2}),
\end{equation}
and the left $\sB$-module $_{\sB}Y$ associated to $\sE_Y$ is well-defined up to canonical quasi-isomorphism.

\subsection{Lower half planes}\label{SecFL.5} As an alternative route, one may switch the role of $X$ and $\Lambda_{q, \theta}$ and define the Floer cohomology $\HFF_\natural^*(\Lambda_{q, \theta}, X)$ by counting instantons on the lower upper plane $\HH^-$. As the theory is identical to the previous one, we only explain the basic setup. In this case, a Floer datum $\fa'_X$ is a quintuple $(R,\alpha_X'(s),\beta_X-\pi, \epsilon_X, \delta H')$ as usual, but the function $\alpha_X'$ is now defined on $(-\infty, R]_s$ such that for some $0<\delta \ll 1$,
\[
\alpha_X(s)\equiv \theta-\pi \text{ for all }s\leq \delta,
\]
and $\beta_X$ is replaced by $\beta_X-\pi$. The reference path $p_{\model}: (-\infty, R]_s\to M$ satisfies that $p_{\model}(s)\equiv q$ when $s\leq 0$. Consider the path space 
\[
\Pa_k(\Lambda_{q,\theta}, X)\colonequals \{p\in C^\infty((-\infty, R]_s; M): p(R)\in X,\ p \text{ has finite $L^2_k$-distance with } p_{\model}\}, k\geq 1.
\]
Then any a Floer datum $\fa'_X$ defines an action functional on $\Pa_k(\Lambda_{q, \theta}, X)$ by the formula:
\begin{align}
\CA_{X,\fa'_X}(p)&\colonequals h_X(p(R))+\int_{-\infty}^R-p^*\lambda_M+\delta H_s(p(s))ds\\
&+\int_{-\infty}^R\big( \im(e^{-i\alpha_X(s)}(W\circ p(s)))+\chi_{(-\infty, 0]} \im(e^{-i\theta}W(q))\big) ds. \nonumber
\end{align}

If $\fa'_X$ is admissible, then the Floer cohomology $\HFF_\natural^*(\Lambda_{q, \theta}, X;\fa'_X)$ is defined as the Morse cohomology of $\CA_{X,\fa'_X}$. This construction is not new. Given any such $\fa'_X$, one can construct a Floer datum $\fa_X=(R, \alpha_X, \beta_X, \epsilon_X, \delta H)$ for the pair $(X, \Lambda_{q, \theta})$ by taking
\[
\alpha_X=\tau^*\alpha_X'+\pi, \delta H=\tau^*\delta H'
\]
with $\tau:\R^+_s\to (-\infty, R]_s$ defined by $\tau(s)=R-s$. Thus $\HFF_\natural^*(\Lambda_{q, \theta}, X;\fa'_X)$ is isomorphic to the Morse homology of $\CA_{X,\fa_X}$ up to a degree shifting
\[
\HFF_\natural^*(\Lambda_{q, \theta}, X;\fa_X')\cong \HFF^{\natural}_{\fn-*}(X, \Lambda_{q, \theta};\fa_X),
\]
or equivalently, $\HFF_\natural^*(\Lambda_{q, \theta}, X;\fa_X')\cong D\HFF_\natural^*(X,\Lambda_{q, \theta};\fa_X)[-\fn]$, where $DV$ stands for the dual of a finite graded vector space $V$ and $\fn=\dim_\C(TM, J_M)$. If $X$ is of unstable-type, then one can construct a right $\sA$-module $X_{\sA}$ using $\Ch^*_\natural(S_k, X;\fa_{k,X}')$, $1\leq k\leq m$, which is obtained from $_{\sA}X$ by applying the duality functor $\sD: \sP_l^{\opp}\to \sP_r$ in \eqref{AF.E.24}.

\section{Seidel's Spectral Sequence}\label{SecSS}

\subsection{Filtered bimodules} In the most ideal case, one would like to establish Seidel's algebraic spectral sequence for any pair of Lagrangian submanifolds $(X,Y)$, where $X=(X, h_X, \beta_X,\xi_X^{\#})$ is of unstable-type, and $Y=(Y, h_Y,\beta_Y, \xi_Y^{\#})$ of stable-type. In this section, we shall assume something stronger to make our method work:
\begin{enumerate}[label=(S\arabic*)]
\item\label{SX} $X$ is of unstable-type, so $\beta_X\in (-\frac{\pi}{2},\frac{\pi}{2})$, and $H|_X\geq C_X$;
\item\label{SY} $Y$ is of stable-type, so $\beta_Y\in (\frac{\pi}{2},\frac{3\pi}{2})$, and $H|_Y\leq C_Y$. 
\end{enumerate}
See Figure \ref{Pic45} below. Choose admissible sets of thimbles $\Th_0\in \Gamma_0$ and $\Th_\pi\in \Gamma_\pi$, which satisfy the conditions \eqref{FL.E.10} \eqref{FL.E.11}, and denote by $\sA=\sE(\Th_0)$ and $\sB=\sE(\Th_\pi)$ the Fukaya-Seidel category and its dual category respectively. Under \ref{SX} and \ref{SY},  for some $C'>0$, we have 
 \begin{align}\label{SS.E.15}
\re(e^{-\beta_X'}W)|_{X}&\leq C' \text{ for all }\beta_X'\in [-\frac{\pi}{2},\beta_X] \\
\re(e^{-\beta_Y'}W)|_{Y}&\leq C' \text{ for all }\beta_Y'\in [\frac{\pi}{2},\beta_Y].\nonumber
 \end{align}
 
 \begin{figure}[H]
 	\centering
 	\begin{overpic}[scale=.12]{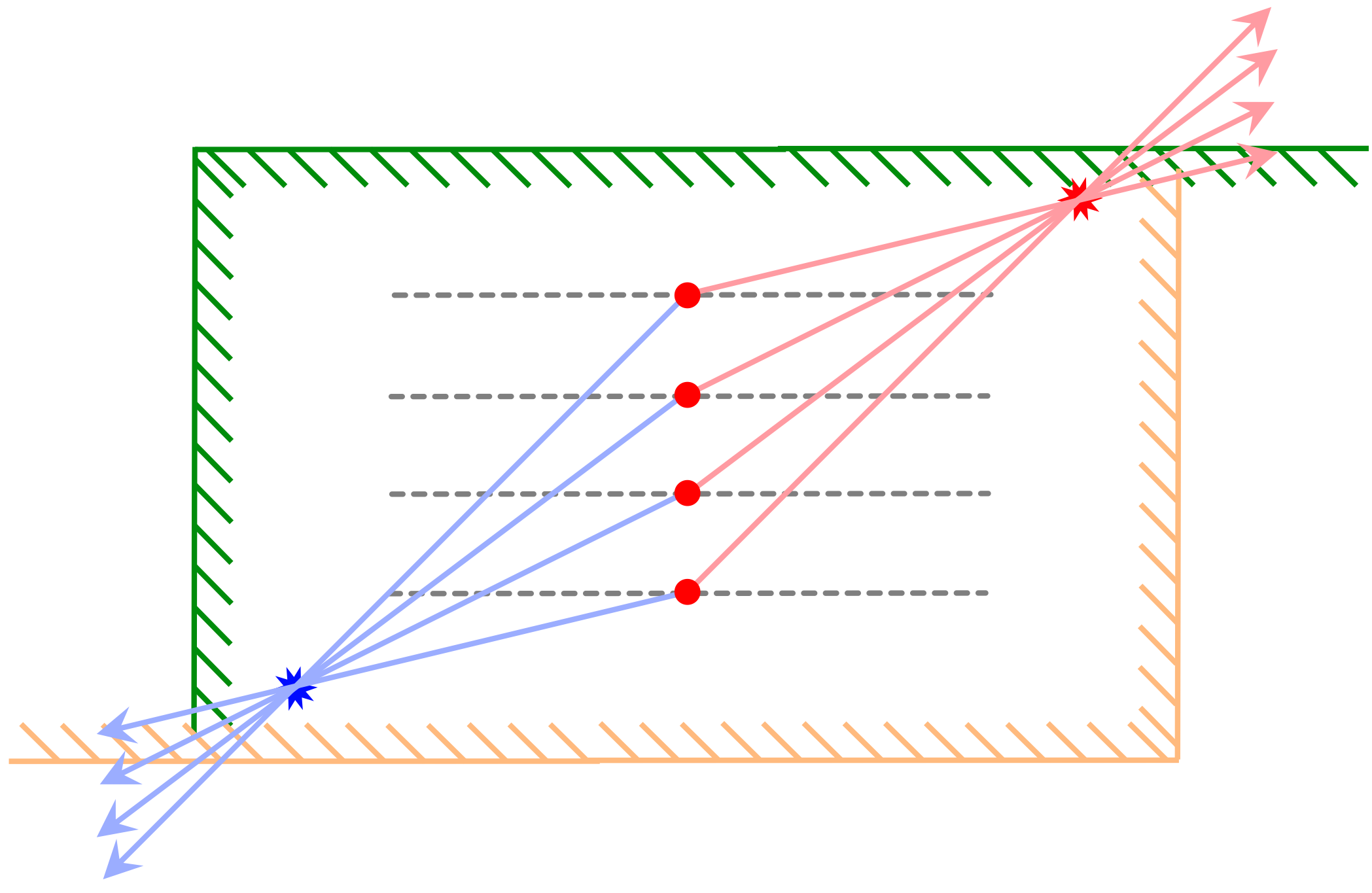}
 		\put(95,60){\small $S_k$}
 		\put(0,0){\small $U_j$}
 		\put(20,45){\small $Y$}
 		\put(75,15){\small $X$}
 	\end{overpic}	
 	\caption{The projection of thimbles, $X$ and $Y$ onto $\C$ (with $\beta_X=0$ and $\beta_Y=\pi$).}
 	\label{Pic45}
 \end{figure}
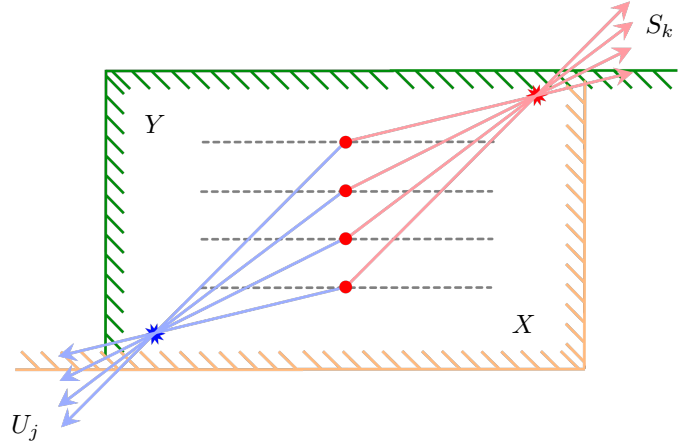

Thus one can also define a left $\sA$-module ${}_{\sA}Y$ for $Y$ and a $\sB$-module $X_{\sB}$ for $X$. In fact, there is a finite directed $A_\infty$-category $\sE_{\Delta_*}$ with 
\[
\Ob \sE_{\Delta_*}: \underbrace{Y\prec U_m\prec \cdots U_2\prec U_1}_{\Ob \sE_Y}\prec \underbrace{X\prec S_1\prec S_2\prec\cdots \prec S_m}_{\Ob\sE_X},\ S_k\in \Ob\sA, U_j\in \Ob\sB,
\]
and which contains $\sE_X, \sE_Y$ and $\sE_{\Delta}$ as full $A_\infty$-subcategories. This $A_\infty$-category $\sE_{\Delta_*}$ depends on an extra stretching parameter $R\gg \pi$ (which is omitted from our notations) and is defined using the metric ribbon tree $\CT^{m+1, m+1}_R$. For $R\gg \pi$, the induced $(\sE_X, \sE_Y)$-bimodule $\Delta_*=\Delta_{*,R}$ carries an increasing filtration:
\[
0=\Delta^{(0)}_*\subset \Delta^{(1)}_*\subset \cdots \subset \Delta^{(m)}_*=\Delta_*,
\]
which is induced by the action functionals and agrees with the increasing filtration \eqref{GF.E.24} on the bimodule $\Delta$ when restricted to $\sA, \sB$. 

\medskip 

We elaborate on the construction of $\Delta_*$. Suppose that the functions $\alpha_j^{un}, \alpha_k^{st}$ have been chosen as in \eqref{GF.E.1}; take $\frac{\pi}{2}<\beta_*<\frac{\pi}{2}+\min\{\theta_m, \eta_1-\pi\}$ and $\epsilon_*>0$ as in \eqref{GF.E.25}. Combined with \eqref{FL.E.10}\eqref{FL.E.11}, this implies that 
\[
0<\beta_*-\frac{\pi}{2}<\theta_m<\cdots<\theta_1<\beta_X+\frac{\pi}{2}<\pi<\beta_*+\frac{\pi}{2}<\eta_1<\cdots<\eta_m<\beta_Y+\frac{\pi}{2}<2\pi.
\]
with $S_k=\Lambda_{x_k, \theta_k}$ and $U_j=\Lambda_{x_j, \eta_j}$. We require further that 
\begin{equation}
0<\epsilon_*<\min\{\re(e^{i\beta_X}), -\re(e^{i\beta_Y})\}.
\end{equation}

\begin{figure}[H]
	\centering
	\begin{overpic}[scale=.15]{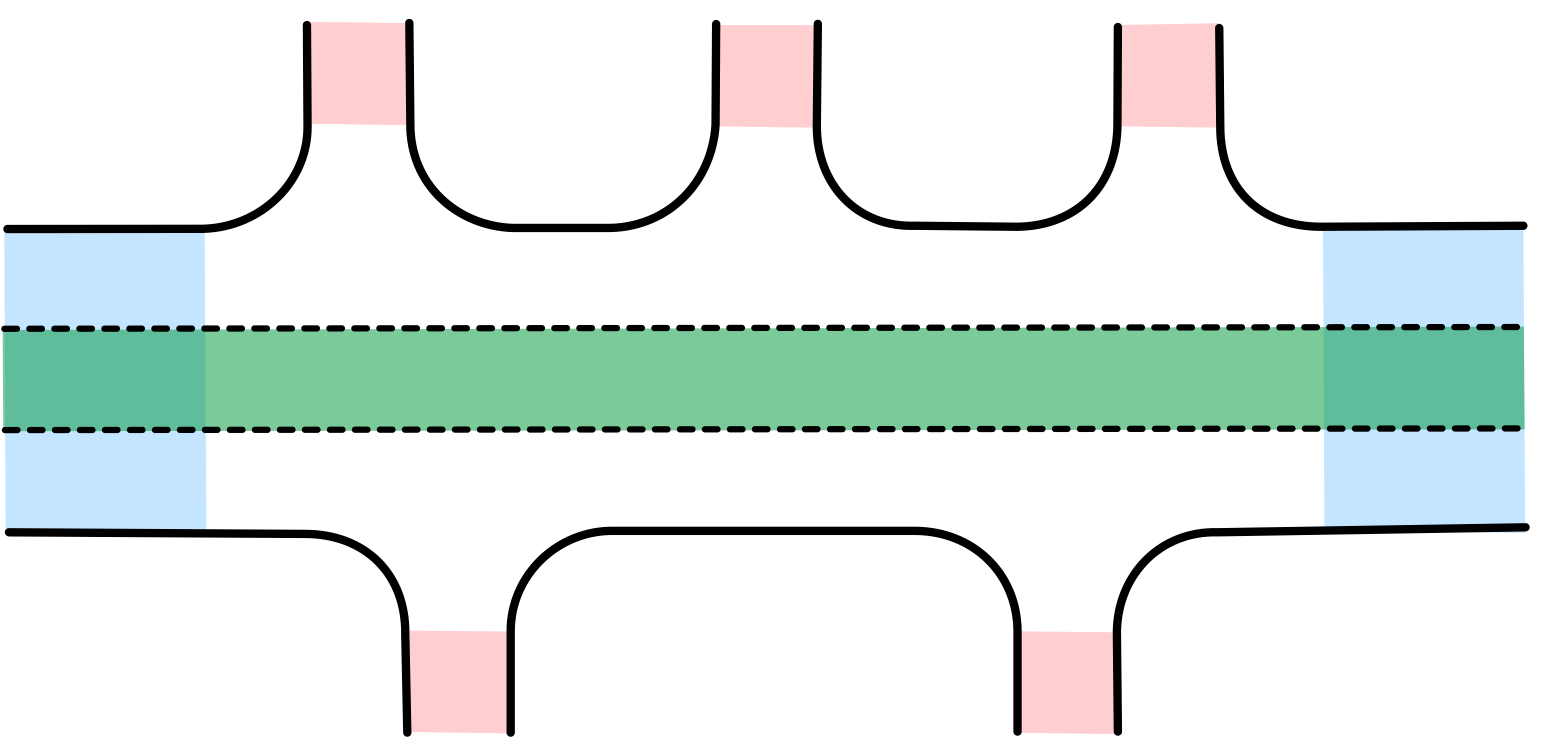}
		\put(10,5){$Y$}
				\put(45,5){$U_{j_2}$}
				\put(85,5){$U_{j_1}$}
						\put(85,40){$X$}
					\put(5,40){$S_{k_3}$}
									\put(35,40){$S_{k_2}$}
													\put(60,40){$S_{k_1}$}
	\end{overpic}	
	\caption{The diagonal bimodule $\Delta_*.$ The green region is stretched vertically.}
	\label{Pic44}
\end{figure}
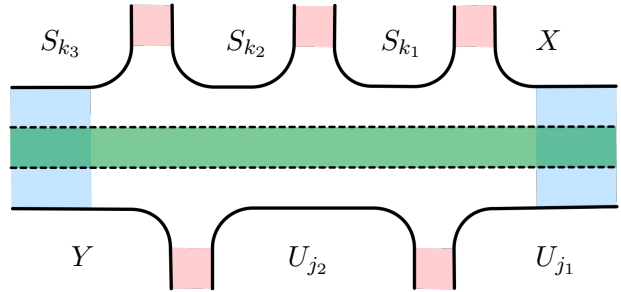

As a label, we think of $X$ as $\Lambda_{q, \pi}$ and $Y$ as $\Lambda_{q,2\pi}$ for any choice of $q\in \Crit(W)$. This helps us determine Floer data as well as cobordism data on pointed disks. Now for any $R\geq \pi$, consider the functions
 \[
\alpha^R_{Y, k}(s)=\alpha_k^{st}(s-R),\ s\in [0,+\infty)_s
\text{ and }\alpha^R_{j, X}(s)=\alpha_j^{un}(s),\ s\in (-\infty, R+\pi]_s.
\]
Then for any $s\in [\pi, R]_s$, $\alpha^R_{Y,k}(s)=\alpha^R_{j,X}(s)=\pi$. The Floer data are then defined as 
\begin{align*}
\fa_{Y,k}^R&=(R+\pi,\alpha^R_{Y, k}(s), \beta_*,\epsilon_*, \delta H^R_{Y, k})& &\text{ for }(Y,S_k), 1\leq k\leq m, \\
\fa_{j, X}^R&=(R+\pi,\alpha^R_{j, X}(s), \beta_*,\epsilon_*, \delta H^R_{j, X})&&\text{ for }(U_j,X), 1\leq j\leq m,\\
\fa^R_{Y,X} &=(R+\pi, \alpha^R_{Y,X}(s), \beta_*, \epsilon_*,\delta H_{Y,X}) &&\text{ for }(Y,X)
\end{align*}
with the usual assumptions on perturbation 1-forms; cf. Section \ref{SecGF.2}. For simplicity we shall take the function $\alpha^R_{Y,X}(s)\equiv \pi$ to be constant.

Let $S$ be any $(d+1)$-pointed disk labeled by a subset $A\subset \Ob \sE_{\Delta_*}$ with $|A|= d+1\geq 3$, and let $\phi$ be any compatible quadratic differential which is $\epsilon$-close to $\CT^{d_1,d_2}_R$ with $d_1=|A\cap \Ob\sE_Y|$ and $d_2=|A\cap\Ob\sE_X|$. If $Y\in A$, then $Y$ is attached the first component $C_0\subset \partial S$. Similar to \eqref{FL.E.12}, the curve $g_0(t)\colonequals \Xi^{\dagger}\circ \sigma_0(t,0):\R_t\to\C$ must satisfy some additional properties:
\begin{itemize}
\item If $d_1=1$, then $A\setminus\{Y\}\subset \Ob \sE_X$, and $\pt g_0(t)\equiv -i\epsilon_*e^{i\beta_*}$;
\item If $d_2=0$, then $A\subset \Ob\sE_Y$, and  $\pt g_0(t)\equiv -i\epsilon_*e^{i\beta_X}$;
\item If $d_1\geq 2$ and $d_2\geq 1$, then $\pt g_0(t)$ is in the direction of $-ie^{i\beta_*}$ for $t\leq 0$ and $-ie^{i\beta_Y}$ for $t\gg 0$. In general, this direction lies in the interval $[\beta_*-\frac{\pi}{2},\beta_Y-\frac{\pi}{2}]$ for all $t\in \R_t$. The condition \ref{SY} ensures that $\re(e^{-i\beta}W)|_Y$ is bounded above for any $\beta\in [\beta_*,\beta_Y]\subset [\frac{\pi}{2}, \beta_Y]$.
\end{itemize}

If $X\in A$, then $X$ is attached to the $d_1$-th component $C_{d_1}\subset \partial S$, and for the curve $g_{d_1}(t)\equiv \Xi^{\dagger}\circ \sigma_{d_1}(t,0): \R_t\to \C$, we require that 

\begin{itemize}
	\item If $d_2=1$, then $A\setminus\{X\}\subset \Ob \sE_Y$, and $\pt g_{d_1}(t)\equiv i\epsilon_*e^{i\beta_*}$; 
	\item If $d_1=0$, then $A\subset \Ob \sE_X$, and $\pt g_{d_1}(t)\equiv -i\epsilon_*e^{i\beta_X}$; 
	\item If $d_2\geq 2$ and $d_1\geq 1$, then $\pt g_{d_1}(t)$ is in the direction of $i\epsilon_*e^{i\beta_*}$ for $t\ll 0$ and $-i\epsilon_*e^{i\beta_X}$ for $t\gg 0$. In general, this direction lies in the interval  $[\beta_*-\frac{3\pi}{2}, \beta_X-\frac{\pi}{2}]$. The condition \ref{SX} ensures that $\re(e^{-i\beta}W)|_X$ is bounded above for any $\beta\in [\beta_*-\pi, \beta_X]\subset [-\frac{\pi}{2}, \beta_X]$.
\end{itemize}
With these conventions understood, one can repeat the argument in Section \ref{SecGF} to construct the filtered bimodule $\Delta_*=\Delta_{*,R}$. See Figure \ref{Pic46} below for a typical phase function on the partial completion of a $(6+1)$ pointed disk (compare Figure \ref{Pic44}). The green region is the image of the strip $Z_R=\R_t\times [2\pi, R-\pi]_s$, which is stretched as $R\to\infty$. The blue (resp. pink) regions are associated to the incoming end and the $d_1$-th outgoing end (resp. other outing ends). Each gray curve is the translated copy of some characteristic curves. The two red curves are the images of $g_0(t)$ and $g_{d_1}(t)$ respectively. 

\begin{figure}[H]
	\centering
	\begin{overpic}[scale=.15]{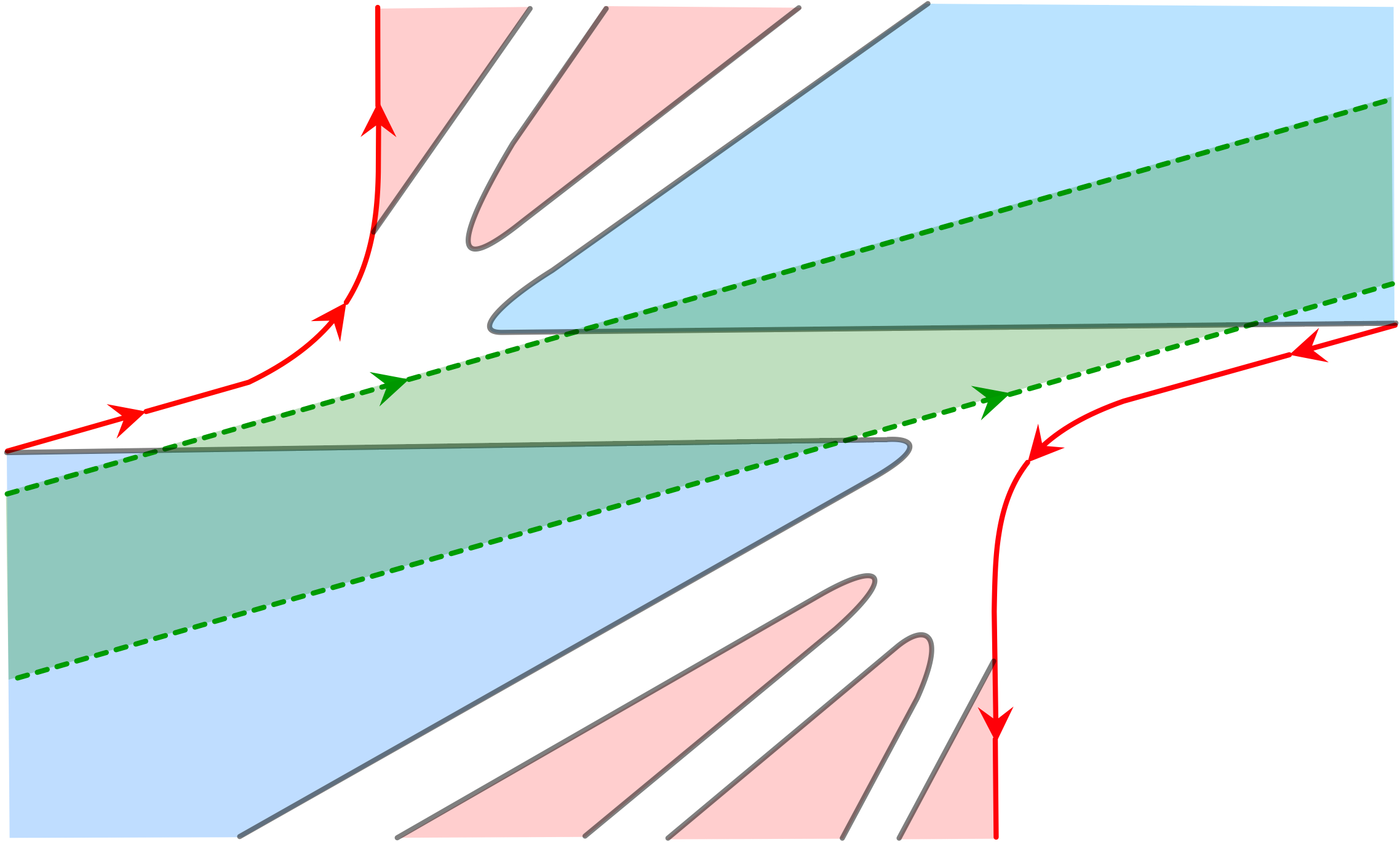}
		\put(10,50){$Y$}
		\put(38,60){$U_{j_2}$}
		\put(60,60){$U_{j_1}$}
		\put(90,10){$X$}
		\put(20,-2){$S_{k_3}$}
		\put(43,-2){$S_{k_2}$}
		\put(60,-2){$S_{k_1}$}
		\put(18,40){$g_0(t)$}
		\put(75,25){$g_{d_1}(t)$}
		\put(17,58){$-ie^{i\beta_Y}$}
				\put(72,5){$-ie^{i\beta_X}$}
				\put(0,30){-$ie^{i\beta_*}$}
					\put(95,30){$ie^{i\beta_*}$}
						\put(101,52){$-ie^{i\beta_*}$}
							\put(101,40){$-ie^{i\beta_*}$}
	\end{overpic}	
	\caption{A typical phase function on a $(6+1)$-pointed disk with $d_1=3$ and $d_2=4$ ($\beta_X=0$ and $\beta_Y=\pi$).}
	\label{Pic46}
\end{figure}
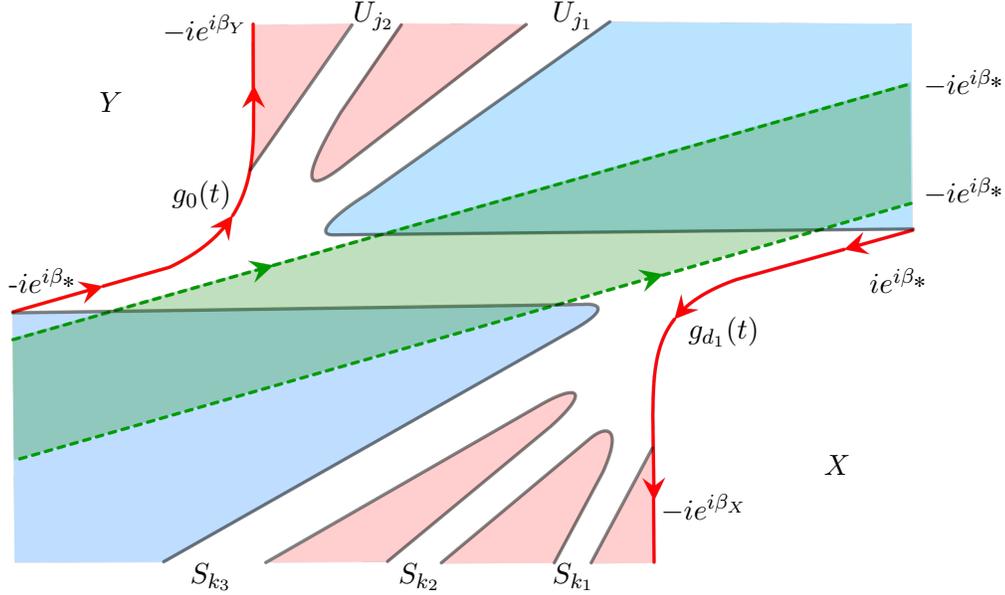

\begin{remark} Readers may wonder how these properties about $g_0$ and $g_{d_1}$ are certified when we construct the family of phase functions on $\Sch^{|A|}\to \NR^{|A|}$ using Lemma \ref{GF.L.15}, or Lemma \ref{FSL.8}. The problem may occur in the proof of Lemma \ref{FSL.4} as we ``push" these curves outwards. Since one may always increase the constant $K$ to construct the null-homotopy, these properties can be preserved if this push-out is designed carefully enough (the proof outlined in \Step 2 of Lemma \ref{L2.5} is just one possible way for this push-out). We leave the details to interested readers. 
\end{remark}

\subsection{Statements} Now we are already to state the algebraic formula to rebuild the Floer cohomology $\HFF^*_\natural(X, Y)=H(D\Delta_*(Y, X)[-\fn])$ in terms of the $A_\infty$-modules ${}_{\sA}X$, $Y_{\sB}$ and the bimodule ${}_{\sA}\Delta_{\sB}$. Applying the cyclic operation $\sC$ defined in \eqref{AF.E.25}, one can transform $\sE_{\Delta_*}$ into another category with 
\[
\Ob \sC\sE_{\Delta_*}:  U_m\prec \cdots \prec U_2\prec U_1\prec X\prec S_1\prec S_2\prec\cdots \prec S_m\prec Y.
\]
This defines an $A_\infty$-functor into $\sQ_r\colonequals\rfmod(\sB)$:
\begin{equation}\label{SS.E.6}
r_{X,Y}:\sC\sE_{\Delta_*}\xrightarrow{r_{\sC\sE_{\Delta_*}}} \rfmod(\sC\sE_{\Delta_*})\to \sQ_r. 
\end{equation}
such that $r_{X,Y}|_{\sA}$ is the functor $r_{\Delta}$ in Theorem \ref{FS.T.3}. While $r_{X,Y}$ is not cohomologically full and faithful in general, one can still prove the following result. 
\begin{theorem}\label{SS.T.1} For any $V\prec V'\in\Ob \sC\sE_{\Delta_*}$, the first order map of $r_{X,Y}$
	\[
	r^1_{X,Y}: \sC\sE_{\Delta_*}(V, V')\to \hom_{\sQ_r}(r_{X,Y}(V),r_{X,Y}(V'))
	\]
	is a quasi-isomorphism. 
\end{theorem}

If $V\in \Ob\sB$, then this statement follows from the Yoneda embedding theorem. If $V,V'\in \Ob\sA$, then this follows from Theorem \ref{FS.T.3}. Let $X_{\sB}, Y_{\sB}, \sS_k\in \sQ_r$ denote the images of $X,Y, S_k$ under $r_{X,Y}$ respectively. Thus the non-trivial part of Theorem \ref{SS.T.1} is when the map $r^1_{X,Y}$ is between
\begin{align}
\sC\sE_{\Delta_*}(X,Y)=\sD\Delta_*(X, Y)\colonequals D(\Delta_*(X,Y))[-\fn]&\to \sC\sE_{\Delta_*}(S_k,Y)=\hom_{\sQ_r}(X_{\sB},Y_{\sB}),\label{SS.E.1}\\
\sD\Delta_*(S_k, Y)\colonequals D(\Delta_*(S_k,Y))[-\fn]&\to \hom_{\sQ_r}(\sS_k,Y_{\sB}),\label{SS.E.2}\\
\Delta_*(X, S_k)&\to \hom_{\sQ_r}(X_{\sB},\sS_k),\label{SS.E.3}
\end{align}
Note that for $R\gg \pi$, the right $\sB$-module $X_{\sB}$ inherits an increasing filtration from $\Delta_*$:
\[
0=X_{\sB}^{(0)}\subset \cdots X_{\sB}^{(m)}= X_{\sB},
\]
while the dual complex $\sD\Delta_*(X, Y)$ inherits a decreasing filtration
\begin{equation}\label{SS.E.13}
0=(\sD\Delta_*)^{(m+1)}(X, Y)\subset \cdots \subset (\sD\Delta_*)^{(1)}(X, Y)=(\sD\Delta_*)(X, Y)
\end{equation}
with 
\[
(\sD\Delta_*)^{(n)}(X, Y)\colonequals D(\Delta_*(Y,X)/\Delta_*^{(n-1)}(Y,X))[-\fn],\ 1\leq n\leq m+1.  
\]
While $X_{\sB}^{(n)}$ involves the first $n$-critical points $x_1,\cdots, x_n$ of $W$, $(\sD\Delta_*)^{(n)}(X, Y)$ involves $x_n,\cdots, x_m$. Theorem \ref{SS.T.1} can be refined as follows. 
\begin{theorem}\label{SS.T.2} Since the geometric filtration is preserved under the bimodule action of $\Delta_*$, $r^1_{X,Y}$ restricts to a chain map for all $0\leq n\leq m$:
	\[
	(\sD\Delta_*)^{(n)}(X, Y)\to \hom_{\sQ_r}(X_{\sB}/X_{\sB}^{(n-1)}, Y_{\sB}). 
	\]
	This map is a quasi-isomorphism. The same holds if $X$ is replaced by any stable thimble $S_k\in \Ob\sA$.
\end{theorem}

The setup we have developed so far is symmetric between $\sA$ and $\sB$, so one may work instead with the $(\sE_Y, \sE_X)$-bimodule $\sD \Delta_*$, which is the dual of $\Delta_*$; see for instance Lemma \ref{AF.L7.9}. $\sD \Delta_*$ is defined by the relation $\sE_{\sD\Delta^*}\colonequals\sC^{m+1}\sE_{\Delta^*}$ and is filtered as in \eqref{SS.E.13} where $X$ is replaced by any $S_k$ and $Y$ by $U_j$ in general. Thus $\sC\sE_{\sD\Delta_*}=\sC^{m+2}\sE_{\Delta_*}$ has objects ordered by 
\[
\Ob \sC\sE_{\sD\Delta_*}:  S_1\prec S_2\prec\cdots \prec S_m\prec Y\prec U_m\prec \cdots \prec U_2\prec U_1\prec X.
\]
Let  $r_{Y,X}: \sC\sE_{\sD\Delta_*}\to \sP_r\colonequals \rfmod(\sA)$  denote the induced functor. At this point, there are four $A_\infty$-modules associated to $X$ (and resp. to $Y$)
\begin{align*}
{}_{\sA}X=\sD X_{\sA} &\text{ and }{}_{\sB}X=\sD(X_{\sB}) \text{ (with filtrations)},\\
{}_{\sB}Y=\sD Y_{\sB} &\text{ and }{}_{\sA}Y=\sD(Y_{\sA}) \text{ (with filtrations)},
\end{align*}
where $\sD$ denotes the duality functor defined by \eqref{AF.E.24}. Then Theorem \ref{SS.T.1} and \ref{SS.T.2} implies the following. 

\begin{corollary}\label{SS.C.4} By \eqref{SS.E.3}, the $A_\infty$-homomorphism 
	\[
	t^l_X: {}_{\sA} X\to \hom_{\sQ_r} (X_{\sB}, {}_{\sA}\Delta_{\sB})
	\]
	induced by $\Delta_*$ is a quasi-isomorphism in $\sP_l\colonequals\lfmod(\sA)$. Indeed, its first order map is given precisely \eqref{SS.E.3}. Similarly, Theorem \ref{SS.T.2} (with $X$ replaced by $S_k$) implies that for any $0\leq n\leq m$, 
	\[
t^r_Y:Y_{\sA}^{(n)}\to \hom_{\sQ_r}({}_{\sA}\Delta_{\sB}/{}_{\sA}\Delta^{(n-1)}_{\sB}, Y_{\sB})\xrightarrow[\cong]{\text{ duality }}\hom_{\sQ_l}({}_{\sB}Y,{}_\sB(\sD\Delta)_{\sA}^{(n)} )
	\]
	is a quasi-isomorphism in $\sP_r\colonequals \rfmod(\sA)$. 
\end{corollary}

\begin{corollary}\label{SS.C.3} Switching the roles of $\Delta$ and  $\sD\Delta$ in Corollary \ref{SS.C.4}, we obtain that 
	\[
t_X^r: X_{\sB}^{(n)}\to \hom_{\sP_l}({}_{\sA}X, {}_{\sA}\Delta_{\sB}^{(n)}),
	\]
	is a quasi-isomorphism between finite right $\sB$-modules for all $0\leq n\leq m$.
\end{corollary}

\begin{corollary}\label{SS.C.5} By Lemma \ref{AF.L.9}, $t_X^r$ admits a quasi-inverse $(t_X^r)^{-1}$ in $\sQ_r$. By composing $r^1_{X,Y}$ with $\mu^2_{\sQ_r}(\cdot, (t_X^r)^{-1})$, we obtain a quasi-isomorphism:
	\[
	\sD\Delta_*(X,Y)\xrightarrow{r^1_{X,Y}}\hom_{\sQ_r}(X_{\sB}, Y_{\sB})\xrightarrow{\mu^2_{\sQ_r}(\cdot, ( t_X^r)^{-1})}\hom_{\sQ_r}\big(\hom_{\sP_l}({}_{\sA}X, {}_{\sA}\Delta_{\sB}),Y_{\sB}\big).
	\]
	Switching the roles of $X$ and $Y$, we also have 
	\[
\Delta_*(Y,X)\xrightarrow{r^1_{Y,X}}\hom_{\sP_r}(Y_{\sA}, X_{\sA})\xrightarrow{\mu^2_{\sP_r}(\cdot, ( t_Y^r)^{-1})}\hom_{\sP_r}\big(\hom_{\sQ_l}({}_{\sB}Y, {}_{\sB}(\sD\Delta)_{\sA}),X_{\sA}\big).
\]
\end{corollary}

\subsection{Proof of Theorem \ref{SS.T.1} and \ref{SS.T.2}: Step 1} Theorem \ref{SS.T.2} is proved by a few quick reductions. First, since the chain map $r^1_{X,Y}$ preserves the filtrations, it induces a map between two spectral sequences. To show that $[r^1_{X,Y}]$ is an isomorphism between $E_\infty$-pages, one can work instead with the $E_1$-page. To put it another way, for any $0\leq n\leq m-1$, we have a map between exact triangles, which goes vertically in the diagram below,

\begin{equation}\label{SS.E.5}
\begin{tikzcd}[column sep={11em,between origins},
row sep={3em,between origins}]
H\big((\sD\Delta_*)^{(n+1)}(X, Y)\big)\arrow[rr] \arrow[dd]&& H\big((\sD\Delta_*)^{(n)}(X, Y)\big)\arrow[ld]\arrow[dd]\\
& 	H\big(\sG^{n}(\sD\Delta_*)(X,Y)\big)\arrow[lu] \arrow[dd]& \\
H\big(\hom_{\sQ_r}(X_{\sB}/X_{\sB}^{(n)}, Y_{\sB})\big)\arrow[rr]& &H\big(\hom_{\sQ_r}(X_{\sB}/X_{\sB}^{(n-1)}, Y_{\sB})\big) \arrow[ld]\\
& H\big(\hom_{\sQ_r}(\sG^{n}X_{\sB}, Y_{\sB})\big)\arrow[lu]&  
\end{tikzcd}
\end{equation}
 where $\sG^n$ denotes the associated graded complexes/modules. Thus it suffices to verify the quasi-isomorphism for each associated graded component:
\begin{equation}\label{SS.E.8}
r^1_n: \sG^{n}(\sD\Delta_*)(X,Y) \to \hom_{\sQ_r}(\sG^{n}X_{\sB}, Y_{\sB}),\ 1\leq n\leq m.
\end{equation}

\subsection{Proof of Theorem \ref{SS.T.1} and \ref{SS.T.2}: Step 2} For the next reduction, we apply Lemma \ref{AF.L.1} to the $A_\infty$-functor $r_{X,Y}: \sC\sE_{\Delta_*}\to \sQ_r$ and to the triple $(X, S_n, Y)\subset \Ob\sC \sE_{\Delta_*}$. Then \eqref{AF.E.3} says that the following diagram is commutative up to chain homotopy (take the $n$-th graded component $\sG^n$):
\begin{equation}\label{SS.E.7}
\begin{tikzcd}
\sG^n(\sD\Delta_*)(S_n,Y)\otimes \hom_{\sE_X}(X, S_n)\arrow[r]\arrow[d]& \hom_{\sQ_r}(\sG^n \sS_n, Y_{\sB}) \otimes \hom_{\sE_X}(X, S_n)\arrow[d]\\
\sG^n(\sD\Delta_*)(X,Y)\arrow[r]&\hom_{\sQ_r}(\sG^n X_{\sB}, Y_{\sB}).
\end{tikzcd}
\end{equation}

The next lemma implies that the vertical arrows in \eqref{SS.E.7} are quasi-isomorphisms.
\begin{lemma}\label{SS.L.5} The  $A_\infty$-homomorphism 
	\[
	\sG^nX_{\sB}\to	\sG^n \sS_n\otimes D\hom_{\sE_X}(X, S_n)
	\]
	induced by $r_{X,Y}^1: \hom_{\sE_X}(X, S_n)\to \hom_{\sQ_r}(\sG^nX_{\sB}, \sG^n\sS_n )$ is a quasi-isomorphism meaning that for any $U_j\in \Ob\sB$, the cobordism map induced by a pair-of-pants
	\[
	\sG^n\Delta_*(U_j, X)\to 	\sG^n\Delta_*(U_j, S_n)\otimes D\hom_{\sE_X}(X, S_n)
	\]
	is a quasi-isomorphism. The same holds if $U_j$ is replaced by $Y$.
\end{lemma}

\begin{remark} Lemma \ref{SS.L.5} is the analogue of Lemma \ref{PT.L.3} where $S_n$ is replaced by $X$ and $S_l$ by $S_n$. 
\end{remark}

 Thus it suffices to show that the top row of \eqref{SS.E.7} is a quasi-isomorphism, so we may replace $X$ by $S_n$ in  \eqref{SS.E.8} and verify instead that 
\begin{equation}\label{SS.E.9}
r^1_n: \sG^{n}(\sD\Delta_*)(S_n,Y) \to \hom_{\sQ_r}(\sG^{n}\sS_n, Y_{\sB}).
\end{equation}
is a quasi-isomorphism. From now on the problem is not related to $X$ any longer. 

\subsection{Proof of Theorem \ref{SS.T.1} and \ref{SS.T.2}: Step 3} For the third reduction, let $V_n=\Lambda_{x_n, \eta_n'}$ denote the thimble in the direction of $\eta_n'\in (0,\eta_1)$. As in \Step 2 of the proof of Lemma \ref{PT.L.1}, we consider an enlargement of $\sE'$ of $\sE_{\Delta_*}$ whose objects are ordered as 
\[
Y\prec U_1\prec U_2\prec\cdots\prec U_m\prec V_n\prec S_n.
\]
where we have omitted $X$ and other stable thimbles for simplicity. This defines a filtered $(\sE_X',\sE_Y')$-bimodule $\Delta_*'$, where $\Ob\sE_X'=(S_n)$ and $\sE_Y'$ consists of all other objects of $\sE_{\Delta_*'}$. Then the cyclic rotation $\sC\sE'$ of $\sE'$ has objects ordered by
\[
 U_1\prec U_2\prec\cdots\prec U_m\prec V_n\prec S_n\prec Y
\]
and defines an $A_\infty$-functor 
\[
r': \sC\sE'\xrightarrow{r_{\sC\sE'}}\rfmod(\sC\sE')\xrightarrow{\text{restriction}} \sQ_r=\rfmod(\sB).
\]
Apply Lemma \ref{AF.L.1} to $r'$ and the triple $(V_n, S_n, Y)$; then \eqref{AF.E.4} says that the following diagram is commutative up to a chain-homotopy (take the $n$-th graded component $\sG^n$):
\begin{equation}\label{SS.E.10}
\begin{tikzcd}[column sep=1 em]
\sG^n\hom_{\sC\sE'}(S_n,Y)\arrow[r] \arrow[d]& \hom_{\sQ_r}(\sG^n \sS_n, Y_{\sB})\arrow[d]\\
\hom_{\sC\sE'}(V_n, Y)\otimes D \sG^n \hom_{\sC\sE'}(V_n,S_n)\arrow[r]&\hom_{\sQ_r}(\sV_n, Y_{\sB})\otimes D\sG^n \hom_{\sC\sE'}(V_n,S_n).
\end{tikzcd}
\end{equation}

The next lemma says that the vertical arrows in \eqref{SS.E.10} are quasi-isomorphisms. 
\begin{lemma}\label{SS.L.7} The $A_\infty$-homomorphism 
	\[
\sG^n\hom_{\sE'}(V_n, S_n)\otimes\sV_n \to \sG^n\sS_n
	\]
	induced by $(r')^1: \sG^n\hom_{\sE'}(V_n, S_n)\to \hom_{\sQ_r}(\sV_n, \sG^n\sS_n)$ is a quasi-isomorphism meaning that for any $U_j\in \Ob \sB$, the cobordism map induced by a pair-of-pants 
	\[
	 \sG^n\hom_{\sE'}(V_n, S_n)\otimes \hom_{\sE'}(U_j, V_n)\to \sG^n\hom_{\sE'}(U_j, S_n)
	\]
	is a quasi-isomorphism. The same holds if $U_j$ is replaced by $Y$.
\end{lemma}

\begin{remark} The first part of Lemma \ref{SS.L.7} is exactly Lemma \ref{PT.L.4}. Only the case when $U_j$ is replaced by $Y$ is new. 
\end{remark}
\subsection{Proof of Theorem \ref{SS.T.1} and \ref{SS.T.2}: Step 4}  Finally, it remains to verify the second row of \eqref{SS.E.10} 
\[
r': \hom_{\sC\sE'}(V_n, Y)\to \hom_{\sQ_r}(\sV_n, Y_{\sB})
\]
is a quasi-isomorphism. This follows from the fact that $\sV_n$ is quasi-isomorphic to the Yoneda image $\sU_n$ of $U_n$ in $\sQ_r$. At this point, there is no energy filtration involved in the chain map. Apply Lemma \ref{AF.L.1} again to the functor $r':\sC\sE'\to \sQ_r$ but now to the triple $(U_n, V_n, Y)$. Then \eqref{AF.E.3} implies that the following diagram is commutative up to a chain-homotopy:

\begin{equation}\label{SS.E.12}
\begin{tikzcd}
\hom_{\sE'}(V_n,Y)\otimes \hom_{\sE'}(U_n,V_n)\arrow[r]\arrow[d]& \hom_{\sQ_r}(\sV_n, Y_{\sB}) \otimes \hom_{\sE'}(U_n, V_n)\arrow[d]\\
\hom_{\sE'}(U_n, Y)\arrow[r]&\hom_{\sQ_r}(\sU_n, Y_{\sB}).
\end{tikzcd}
\end{equation}

Note that $ \hom_{\sE'}(U_n,V_n)$ is 1-dimensional and is generated by the quasi-unit $e_{x_n}$ (the constant soliton at $x_n$). By Proposition \ref{FL.P.4} and \Step 3 in the proof of Lemma \ref{PT.L.1}, the vertical arrows of \eqref{SS.E.12} are quasi-isomorphisms. By the Yoneda embedding theorem, the second row of \eqref{SS.E.12} is a quasi-isomorphism, and so is the first row.

Lemma \ref{SS.L.5} and \ref{SS.L.7} are proved in the same way as Lemma \ref{PT.L.3} and Lemma \ref{PT.L.4} using the vertical gluing theorem (the proof is identical and hence omitted). This completes the proof of Theorem \ref{SS.T.2}. Note that \eqref{SS.E.1} and \eqref{SS.E.2} follow from Theorem \ref{SS.T.2} by taking $n=0$. 

Finally, \eqref{SS.E.3} is proved by repeating the argument in Theorem \ref{FS.T.3} with $S_1$ replaced by $X$ and Lemma \ref{PT.L.1} by Lemma \ref{SS.L.5}. This completes the proof of Theorem \ref{SS.T.1}. 

\subsection{Identifying the geometric/algebraic filtrations} We would like to compare the geometric  filtration on $(\sD\Delta_*)(X,Y)$ with the algebraic filtration on $\hom_{\sQ_r}(X_{\sB}, Y_{\sB})$ as defined in \eqref{AF.E.26}. By Theorem \ref{SS.T.2} and the argument in Section \ref{SecAF.5}, it suffices to show that the ladder 
\begin{equation}\label{SS.E.14}
\begin{tikzcd}[row sep=1.5em, column sep=5em]
X_{\sB}/X_{\sB}^{(m)}\arrow[r,equal]\arrow[rd,"{[1]}"]& 0\arrow[d]\\
X_{\sB}/X_{\sB}^{(m-1)}\arrow[u]\arrow[rd,"{[1]}"]& \sG^mX_{\sB}\arrow[l] \arrow[d]\\
\cdots \arrow[u]\arrow[rd,"{[1]}"]& \cdots \arrow[d]\\
X_{\sB}/X_{\sB}^{(2)} \arrow[u]\arrow[rd,"{[1]}"]& \sG^3X_{\sB}\arrow[d]\arrow[l]\\
X_{\sB}/X_{\sB}^{(1)} \arrow[u]\arrow[rd,"{[1]}"]& \sG^2X_{\sB}\arrow[d]\arrow[l]\\
X_{\sB}=X_{\sB}/X_{\sB}^{(0)}\arrow[u] & \sG^1X_{\sB}\arrow[l]\\
\end{tikzcd}
\end{equation}
defines a Postnikov decomposition of $X_{\sB}$; cf. \eqref{AF.E.21}$(L)$. Combining Lemma \ref{SS.L.5} and Lemma \ref{SS.L.7}, we have $\sG^nX_{\sB}\cong \sU_n\otimes D\hom_{\sS_X}(X,S_n)$. An inductive argument using the exact triangles in \eqref{SS.E.14} shows that $X_{\sB}/X_{\sB}^{(n)}\in \sQ_{n,r}$ and the property \eqref{AF.E.19} holds for every $\sN\in \sQ_{n,r}, 0\leq k\leq m$. Although we have worked constantly with chain complexes in Section \ref{SecAF.5}, this spectral sequence relies essentially on the cohomological category $H(\sQ_r)$, especially if one thinks in terms of exact couples \cite[P.263]{GM03}. Thus we have proved that 

\begin{theorem}\label{SS.T.8} The spectral sequence induced by the geometric filtration on $\sD\Delta_*(X,Y)$ is isomorphic to Seidel's algebraic spectral sequence defined in Theorem \ref{AF.T.12}. This means that there is an isomorphism between each page of the spectral sequences, starting from $E_1$, which commutes with the differential maps. 
\end{theorem}

\subsection{Seidel's long exact sequence: revisited} We end this section with some applications of the vertical gluing theorem. Proposition \ref{FL.P.4} is concerned with the case when the thimble $\Lambda_{q, \theta_1}$ can be continuously deformed into $\Lambda_{q, \theta_0}$ without violating the condition \eqref{E1.5}. As in the case of thimbles, if the triangle in Proposition \ref{FL.P.4} (compare Figure \ref{Pic30} with Figure \ref{Pic43} below) contains a single critical point, then the multiplication map $m_*(e_q,\ \cdot\ ): \HFF_\natural^*(X,\Lambda_{q, \theta_{0}})\to  \HFF_\natural^*(X,\Lambda_{q, \theta_{1}})$ might not be an isomorphism and will fit into a long exact sequence, as first proved by Seidel \cite{S03}. The next proposition generalizes Proposition \ref{VT.P.17} to the case of Lagrangian submanifolds. To see their resemblance, we now put $X$ in the second entry. 

\begin{proposition}\label{SS.P.11} Choose angles $\theta_3<\theta_1<\theta_0<\theta_3+\pi$ such that $e^{i(\theta_j-\beta_X)}>0$ for all $j=0,1,3$. Let $\Lambda_j=\Lambda_{q,\theta_j}, j=0,1$, $q\in \Crit(W)$, and suppose that the closed triangle bounded by $l_{q,\theta_0},l_{q,\theta_1}$ and the equation $\re\langle w, e^{i\beta_X}\rangle\leq C_X$ only contains the critical value $W(q)$ on the boundary and a critical value $W(q_3)$ in the interior. Let $\Lambda_3=\Lambda_{q_3,\theta_3}$. Then we have a long exact sequence :
		\[
	\begin{tikzcd}
	\HFF_\natural^*(\Lambda_3,X)\otimes 	\HFF_\natural^*(\Lambda_1,\Lambda_3)\arrow[rr,"m_*"]& & 	\HFF_\natural^*(\Lambda_1,X)\arrow[dl,"{m_*(\cdot,e_q)}"]\\
	&  \HFF_\natural(\Lambda_0,X)\arrow[lu,"{[1]}"].&
	\end{tikzcd}
	\]

\end{proposition}

The proof of Proposition \ref{SS.P.11} follows the same line of arguments as in Proposition \ref{VT.P.17} with $\Lambda_2$ replaced by $X$ and hence is omitted here.

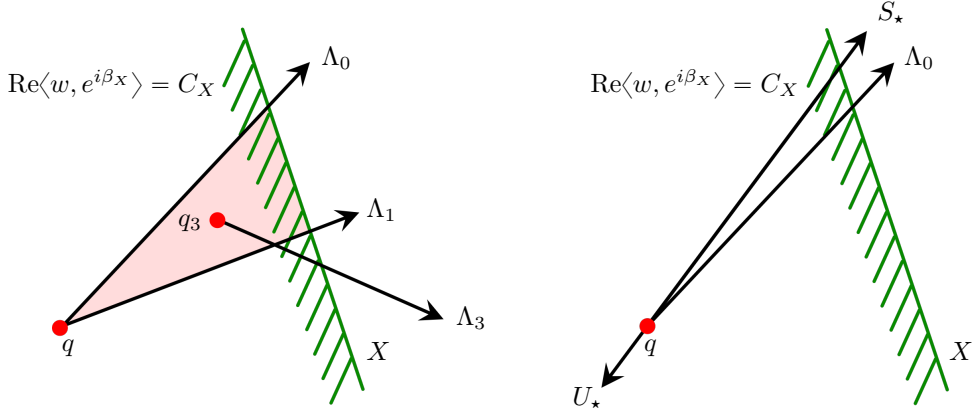
\begin{figure}[H]
	\centering
	\begin{overpic}[scale=.15]{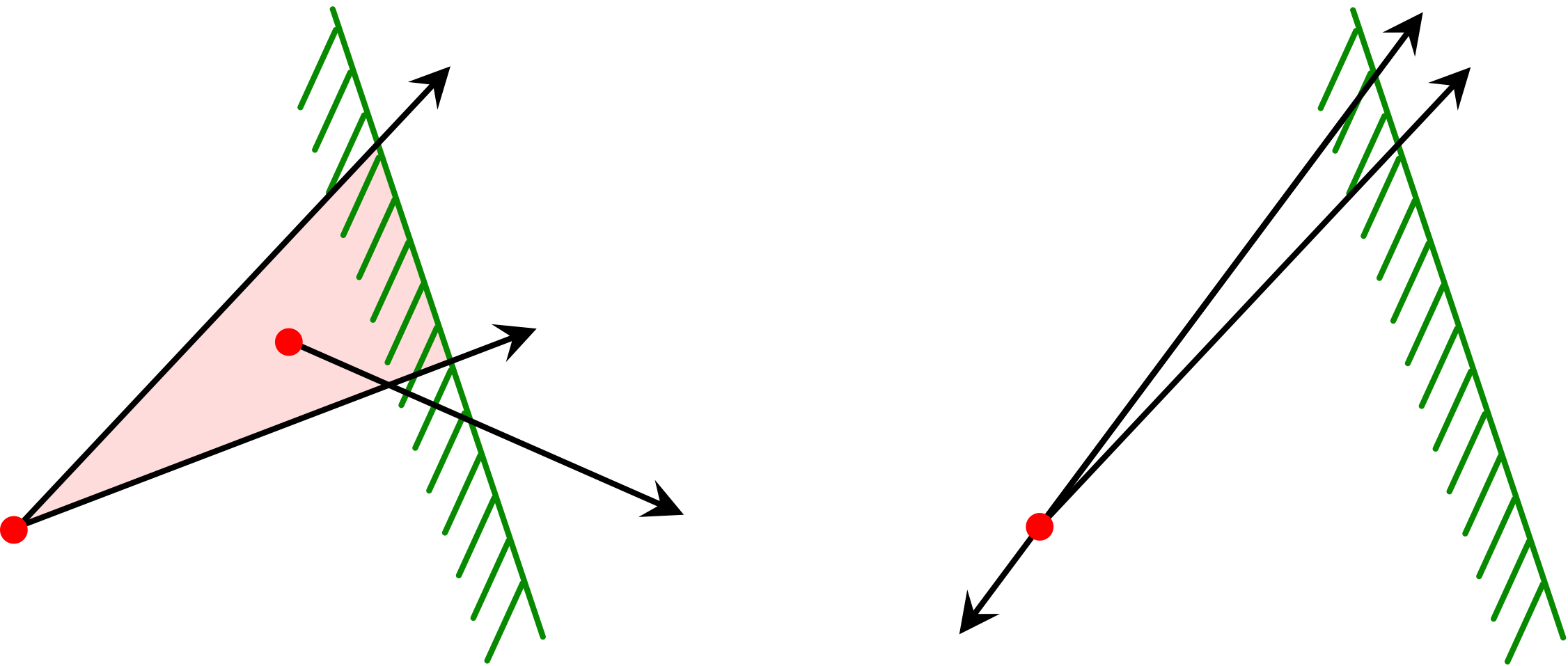}
		\put(35,5){\small $X$}
		\put(100,5){\small $X$}
		\put(60,35){\small $\re\langle w, e^{i\beta_X}\rangle=C_X$}
		\put(-5,35){\small $\re\langle w, e^{i\beta_X}\rangle=C_X$}
		\put(95,38){\small $\Lambda_0$}
		\put(92,43){\small $S_\star$}
		\put(58,0){\small $U_\star$}
		\put(66,6){\small $q$}
		\put(1,6){\small $q$}
		\put(14,20){\small $q_3$}
		\put(30,38){\small $\Lambda_0$}
		\put(35,21){\small $\Lambda_1$}
		\put(45,9){\small $\Lambda_3$}
	\end{overpic}	
	\caption{Seidel's long exact sequence (left) and the isomorphism between of $\HFF_\natural^*$ and $\HFF^*$ (right).}
	\label{Pic43}
\end{figure}

\subsection{The isomorphism between $\HFF_\natural^*$ and $\HFF^*$}\label{SecSS.9} Let $\Lambda_0=\Lambda_{q, \theta_{0}}$. Suppose that $W(X)\subset \C$ is bounded above in the direction of $e^{i\beta_X}$ and $e^{i(\theta_0-\beta_X)}>0$. Another application of the vertical gluing theorem is to show that $\HFF_\natural^*(\Lambda_0,X)$ is isomorphic to the  $\HFF^*( \underline{\Lambda}_0, X)$, the Floer cohomology defined using Lagrangian boundary condition (counting solutions on strips). We used an underline here to emphasis that the Lagrangian boundary condition along $\underline{\Lambda}_0$ is used here to define this group. At this point, we have to assume that the thimble $\Lambda_0$ has bounded geometry as a totally real submanifold of $(M,J_M,\omega_M, g_M)$. This is a technical point which we will not address in this paper.

Suppose that we are in the scenario of Figure \ref{Pic43}, and $\theta_{-1}=\theta_0+\epsilon, 0<\epsilon\ll 1$ is a slight perturbation of $\theta_0$. Let $U_\star=\Lambda_{q,\theta_{-1}-\pi}$ and $S_\star=\Lambda_{q,\theta_{-1}}$. Choose a function $\alpha_0(s):\R_s\to \R$ as in \eqref{VT.E.30} to define an admissible Floer datum $\fa_0$ for the pair $(\Lambda_0, U_\star)$. If $\fa_X$ is any admissible Floer datum for $(S_\star,X)$, then they can be concatenated to give a Floer datum $\fa_{0,X}^R$ for $(\Lambda_0, X)$ (cf. Section \ref{SecVT.1}). There are two ways to define the Floer cohomology of $(\Lambda_0, X)$, either by counting instantons on the lower half planes, $\R_t\times (-\infty, R+\pi]$, which gives $\Ch^*_\natural ( \Lambda_0, X; \fa_{0,X}^R)$, or by restricting the Floer datum $\fa_{0,X}^R$ on $[0,R+\pi]_s\times \R_t$ and counting instantons on strips, which gives $\Ch^*(\underline{\Lambda}_0,X; \fa_{0,X}^R)$. For any $R\gg \pi$, the Floer complex $\Ch^*(\Lambda_0,X; \fa_{0,X}^R)$ has only one filtration level. The proof of Theorem \ref{VT.T.2} is adapted to this case to show that 
\[
\Ch^*(\underline{\Lambda}_0, X; \fa_{0,X}^R)\cong\Ch^*_\natural(S_\star, X;\fa_X)\otimes  \Ch^*_\natural( \underline{\Lambda}_0, U_\star; \fa_0), 
\]
that is,  they are isomorphic as chain complexes). If $\fa_0$ is chosen properly, $\Ch^*_\natural( \underline{\Lambda}_0, U_\star; \fa_0)$ has rank 1 and is generated by the constant soliton $e_q$. This proves 
\[
H(\Ch^*( \underline{\Lambda}_0; \fa_{0,X}^R)))\cong \HFF_\natural^*(S_\star, X)\cong \HFF_\natural^*(\Lambda_0, X)\ (\text{by Proposition }\ref{FL.P.4})
\]
Finally, the cohomology group $H(\Ch^*(\underline{\Lambda}_0, X; \fa_{0,X}^R))$ is independent of $R>\pi$. We can always pull back the solution (cf. \eqref{Intro.E.13}) and think of the Floer equation as defined on the same domain $\R_t\times [0,1]_s$, but the Floer datum may change as $R$ varies. Even in the worst case that $W(X)$ is only bounded in the direction of $e^{i\beta_X}$, since $W(\Lambda_0)$ is the ray $l_{q,\theta_0}$, one may still use a cobordism datum on the strip $\R_t\times [0,1]_s$ to construct continuation maps (cf. Section \ref{Subsec:Continuation}) and verify this invariance.

In this paper the construction of Floer cohomology always relies on a perturbation by the superpotential $W$. If one wishes to compare $H(\Ch^*_\natural( \underline{\Lambda}_0,X; \fa_{0,X}^R))$ with the version without such a perturbation, extra assumptions on $(M,W)$ are required to verify the compactness theorem, and we leave this task to interested readers. One may also replace $X$ by a thimble $\Lambda_1$ to obtain the following result. 

\begin{proposition} Suppose that $\Lambda_j=\Lambda_{q, \theta_{j}}, j=0,1$ has bounded geometry and  $\theta_1<\theta_0<\theta_1+2\pi$, then $\HFF_\natural^*(\Lambda_0, \Lambda_1)\cong \HFF_\natural^*(\Lambda_0, \underline{\Lambda}_1)\cong\HFF_\natural^*(\underline{\Lambda}_0, \Lambda_1)\cong  \HFF^*(\underline{\Lambda}_0, \underline{\Lambda}_1)$. 
	
	Moreover, for any Lagrangian submanifold $X=(X, h_X, \beta_X,\xi_X^\#)$ satisfying the conditions in Assumption \ref{assumption:2} and with $e^{i(\beta_X-\theta_0)}>0$, we have $\HFF_\natural^*(\Lambda_0, X)\cong \HFF^*(\underline{\Lambda}_0, X)$. Here $\HFF^*$ is defined by counting instantons on infinite strips $\R_t\times [0,1]_s$, but the equation is still perturbed by the superpotential $W$ using a suitable Floer datum. 
\end{proposition}

\newpage
\appendix

\section{The Diameter Estimate}\label{SecDE}

This appendix is devoted to the proof of the diameter estimate in Proposition \ref{P1.1.3}. This result concerns only the almost Hermitian geometry of $(M, J_M,\omega_M, g_M)$ and has played an essential role in the proof of the Local Compactness Lemma \ref{L1.6}. When $P: B(0,1)\to M$ is $J$-holomorphic, this estimate follows from \cite{IS00}; we shall only explain the necessary adaptation. In what follows, we shall always assume that $(M, J_M,\omega_M, g_M)$ has bounded geometry. 

\subsection{The first step} For any $z\in \C$, let $B(z,R)$ denote the disk centered at $z$ of radius $t$. Fix some $r>2$. We begin with the small energy estimate.

\begin{lemma}\label{App.L.1} There exist $\delta_1, C_1>0$ with the following property. Let $B(z,R)\subset B(0,1)$. If $P: B(0,1)\to M$ is any smooth map such that $\diam P(B(z,R))\leq\delta_1$, then 
	\[
	\diam\big(P\big(B(z,\frac{R}{2})\big)\big)<C_1\big(\|dP\|_{L^2(B(z,R))}+\|\bpartial_J P\|_{L^r(B(z,R))}\big). 
	\]
\end{lemma}
\begin{proof}[Proof of \ref{App.L.1}] This follows from the proof of \cite[Lemma 1.1, Step 2]{IS00}. For $0<\delta_1 \ll 1$, we may assume that the image $P(B(z,R))$ lies in a geodesic ball, then the estimate follows from the case when $M=\R^{2n}$ and $J=J_M$ is $C^0$-close to the standard complex structure. See \cite[Lemma 1.2]{IS00}.
\end{proof}
\begin{lemma}\label{App.L.2}There exist $\epsilon_2, C_2>0$ with the following property. If $P: B(0,1)\to M$ is any smooth map such that 
	\begin{equation}\label{App.E.2}
	\int_{B(0,1)} |dP|^2+|\bpartial_J P|^r<\epsilon_2,
	\end{equation}
	then 
	\begin{equation}\label{App.E.3}
	\diam\big(P\big(B(0,\frac{1}{2})\big)\big)<C_2\big(\|dP\|_{L^2(B(0,1))}+\|\bpartial_J P\|_{L^r(B(0,1))}\big). 
	\end{equation}
\end{lemma}
\begin{proof}[Proof of \ref{App.L.2}] This follows from Lemma \ref{App.L.1} and by repeating the argument in \cite[Lemma 1.1, Step 3]{IS00}. Alternatively, it can be organized as follows. Let $\delta_1$ be the constant from Lemma \ref{App.L.1}. Consider the continuous function 
	\begin{equation}\label{App.E.1}
	f: B(0,1)\to \R^+,\ f(z)=\max\{0\leq R\leq  1-|z|: \diam(P(B(z,R)))\leq\delta_1\},
	\end{equation}
and let
\[
\alpha_0\colonequals\min_{|z|<\frac{3}{4}} \frac{f(z)}{\frac{3}{4}-|z|}.
\]
We claim that if $0<\epsilon_2\ll 1$ is chosen sufficiently small, then $\alpha_0\geq \frac{1}{2}$. If not, suppose that this minimum is attained at $z_0\in B(0,\frac{3}{4})$, set $r_0\colonequals \frac{3}{4}-|z_0|$, and consider the map 
\[
g: B(0,1)\to B(z_0,r_0),\ w\mapsto z_0+r_0w.
\]
Define $P_0(w)\colonequals P(g(w)): B(0,1)\to M$, and denote by $f_0$ the function \eqref{App.E.1} associated to $P_0$. Then $f_0(0)=\alpha_0$. In general, since 
\[
f(g(w))\geq \alpha_0\bigg(\frac{3}{4}-|g(w)|\bigg)\geq \alpha_0\bigg(r_0-|g(w)-z_0|\bigg),
\]
we have $f_0(w)=\min\{f(g(w))/r_0, 1-|w| \}\geq \min\{\alpha_0(1-|w|), 1-|w| \}\geq \alpha_0(1-|w|)$. Hence, for all $|w|\leq \alpha_0$, $f_0(w)\geq \alpha_0/2$. Moreover, under this rescaling, the quantity \eqref{App.E.2} associated to $P_0$ gets smaller. Since $B(0,\alpha_0)$ can be covered by finitely many disks $B(w_i, \alpha_0/4)$ with each $|w_i|\leq \alpha_0$, the diameter estimate from Lemma \ref{App.L.1} applied to each $B(w_i, \alpha_0/2)$ shows that $\diam(P_0(B(0,\alpha_0)))<\delta_1$ if $\epsilon_2$ is chosen sufficiently small. This contradicts our choice of $\alpha_0$.

 We have verified that $\alpha_0\geq \half$, and so $f(z)\geq \frac{1}{8}$ for all $|z|\leq \half$. Now we cover $B(0,\half)$ by finitely many disks $B(z_i,\frac{1}{16})$ with each $|z_i|\leq \half$. The estimate \eqref{App.E.3} now follows from Lemma \ref{App.L.1}.
\end{proof}

\subsection{The second step} Lemma \ref{App.L.2} is almost enough to deduce Proposition \ref{P1.1.3}. But when the energy is highly concentrated near a point, we need another lemma to control the diameter over an annulus. Let $A(R_1, R_2)$ denote the annulus $B(0,R_2)\setminus B(0,R_1)$ for $R_1<R_2$.

\begin{lemma}\label{App.L.3} There exists $\epsilon_3, C_3>0$ with the following property.  Suppose that $P: B(0,1)\to M$ is any smooth map such that for some $0<R<\frac{1}{16}$, 
	\begin{equation}\label{App.E.4}
	\int_{A(R,1)} |dP|^2+|\bpartial_J P|^r<\epsilon_3,
	\end{equation}
	then 
	\begin{equation}\label{App.E.5}
	\diam\big(P\big(A(2R,1/2)\big)\big)<C_3\big(\|dP\|_{L^2(A(R,1))}+\|\bpartial_J P\|_{L^r(A(R,1))}\big). 
	\end{equation}
\end{lemma}

The best way to understand Lemma \ref{App.L.3} is to think of cylinders instead of annuli. Let $Z(a,b)$ denote the cylinder $[a,b]\times S^1$ equipped with the product metric for $a<b$, and $Z_n\colonequals Z(n-1,n)$. The proof of \cite[Lemma 1.4]{IS00} can be adapted to show the following.

\begin{lemma}\label{App.L.4} There exists $\epsilon_4, C_4>0$ and $\gamma\in (0,1)$ with the following property. Suppose that $P: Z(0,5)\to M$ is any smooth map with
	\[
	\|dP\|_{L^2(Z_n)}+\|\bpartial_J P\|_{L^r(Z_n)}< \epsilon_4
	\]
	for all $1\leq n\leq 5$, then 
	\[
	\|dP\|_{L^2(Z_3)}^2\leq \frac{\gamma}{2}\bigg(\|dP\|_{L^2(Z_2)}^2+\|dP\|_{L^2(Z_2)}^2\bigg)+C_4\|\bpartial_JP\|_{L^r(Z(1,4))}^2.
	\]
\end{lemma}

We prove Lemma \ref{App.L.3} given Lemma \ref{App.L.4}.
\begin{proof}[Proof of Lemma \ref{App.L.3}] Consider $P$ as a map on $(-\infty, 0]_\rho\times S^1_\theta$ with $z=\exp(\rho+i\theta)$. Define 
	\begin{align*}
a_k&=\|dP\|_{L^2(Z_{-k})}^2, &b_k&=\|\bpartial_J P\|_{L^r(Z_{-k})}^2.
	\end{align*}
	If $\epsilon_3$ is sufficiently small, then Lemma \ref{App.L.4} implies that 
	\[
	a_k\leq \frac{\gamma}{2}(a_{k+1}+a_{k-1})+C_4\cdot b_k
	\]
	for all $1\leq k\leq n-1$ with $n=\lfloor -\ln R\rfloor$. Let $c_0\colonequals\|dP\|_{L^2(A(R,1))}+\|\bpartial_J P\|_{L^r(A(R,1))}$. Then $a_k\leq c_0^2$ and 
	\begin{align*}
	c_0^r&\geq 	\int_{A(R,1)} |\bpartial_J P|^r dt ds=\int_{[\ln R, 0]\times S^1} |\bpartial_JP|^r e^{(2-r)\rho}d\rho d\theta\geq \int_{Z_{-k}}e^{(r-2)k}|\bpartial_J P|^r d\rho d\theta,
	\end{align*}
	which implies that $b_k=\|\bpartial_J P\|_{L^r(Z_{-k})}^2\leq c_0^2e^{-2k(r-2)/r}$, so $b_k$ decays exponentially as $k\to\infty$. Now we can repeat the argument of \cite[Corollary 1.5]{IS00} to show that both $\sum a_k^{1/2}$ and $\sum b_k^{1/2}$ are bounded by a constant multiple of $c_0^2$. Now apply Lemma \ref{App.L.2} to each $1\leq k\leq n-1$ to obtain the diameter bound.
\end{proof}

\subsection{The third step}Finally, we prove Proposition \ref{P1.1.3}. Let $\epsilon_\star$ be the minimum of the constants $\epsilon_2$ and $\epsilon_3$ in the Lemma \ref{App.L.2}. Then the proof proceeds by induction for $C\leq n\epsilon_\star$. The base case when $n=1$ is Lemma \ref{App.L.2}. For any $n\geq 2$, consider the continuous function $h:B(0,1)\to \R^+$
\[
h(z)=\max\{ 0\leq R\leq 1-|z|:\int_{B(z, R)}|dP|^2+|\bpartial_JP|^2\leq (n-1)\epsilon_\star \},
\]
and let 
\[
\beta_0=\min_{|z|<\frac{3}{4}}\frac{h(z)}{\frac{3}{4}-|z|}.
\]

If $\beta_0\geq \frac{1}{16}$, then for all $z\in B(0,\half)$, $h(z)\geq 1/64$. Then one can cover $B(0,\half)$ using finitely many disks $B(z_k, \frac{1}{32})$ with each $|z_k|\leq \half$ and use the induction hypothesis to estimate  $\diam(P(B(0,\half)))$.

If $\beta_0<\frac{1}{16}$, then the energy concentrates at a smaller scale. Suppose that this minimum is attained at some $z_0\in B(0,\frac{3}{4})$. Let $r_1=\frac{3}{4}-|z|$ and $r_0\colonequals h(z_0)=\beta_0\cdot r_1<r_1/16$. Consider the inclusion
\[
B(z_0, r_0)\subset B(z_0, 2r_0)\subset B(z_0,\frac{r_1}{2})\subset B(z_0,r_1).
\]

First, the diameter of $P$ on $B(z_0,\frac{r_1}{2})\setminus B(z_0, 2r_0)$ can be estimated using Lemma \ref{App.L.3}. Second, the region $B(0,\half)\setminus B(z_0, \frac{r_1}{2})$ can be covered using finitely many disks of radius $\frac{1}{32}$ which are all disjoint from $B(z_0,r_0)$, then the diameter of $P$ on $B(0,\half)\setminus B(z_0, \frac{r_1}{2})$ is estimated using Lemma \ref{App.L.2}. Finally, for any $z\in B(z_0, 2r_0)$, we have 
\[
h(z)\geq \beta_0\bigg(\frac{3}{4}-|z|\bigg)\geq r_0\cdot \frac{r_1-2r_0}{r_1}\geq \frac{r_0}{2}.
\]

Now we cover  $B(z_0, 2r_0)$ by finitely many disks of radius $\frac{r_0}{4}$, then the diameter of $P$ on $B(z_0, 2r_0)$ can be estimated using the induction hypothesis. This finishes the proof of Proposition \ref{P1.1.3}.

\section{A Well-Known but Technical Lemma}

This appendix is devoted to a well-known but technical lemma used in the proof of Lemma \ref{VT.L.15}. 

\begin{lemma}\label{TC.L.1} Under the assumption of Lemma \ref{VT.L.15}, we have the following pointwise estimate for the $n$-th derivative of $\CF_0$, $n\geq 1$. For any $c_2>0$, there exists a constant $C>0$ such that for any smooth $\xi_0, \xi_j\in C^\infty(B(0,1); P^*TM)$, $1\leq j\leq n$ with $\|\xi_0\|_\infty<c_2$ , we have 
	\begin{equation}\label{TC.E.6}
	|(d_n \CF_0)(\xi_0; \xi_1,\cdots,\xi_l)|<C(1+|dP|+|\nabla\xi_0|+\sum_{j=1}^n \frac{|\nabla\xi_j|}{|\xi_j|})\prod_{j=1}^n |\xi_j|. 
	\end{equation}
	on $B(0,\frac{1}{2})$. This constant $C$ depends on $c_2$ and the $C^{n}$-norm of $\X$ on a slightly larger compact subset of $M$ $($which depends on $c_2)$. 
\end{lemma}

Let $\CB^r_1(S; M)$ denote the space of $L^r_1$-maps $P: S\to M$, and consider the Banach bundle:
\[
\E^r\to \CB^r_1(S;M), r>2, 
\]
whose fiber at $P$ is $L^r(S;\Lambda^{0,1}S\otimes TM)$. If $M$ and $S$ are compact, then Lemma \ref{TC.L.1} implies that $\F_0(P)=(dP-\X)^{0,1}$ defines a smooth section of $\E^r$, and $\CF_0$ is the expression of $\F_0$ in a local trivialization of $\E^r$.

\begin{proof}[Proof of Lemma \ref{TC.L.1}] We decompose $\F_0(P)$ as 
	\[
	\half (d P+J_{P}\circ dP\circ j_S)+\X^{0,1}
	\]
	where $j_S$ is the complex structure on $S$. The estimate \eqref{TC.E.6} is linear in $\F_0$, so one may prove \eqref{TC.E.6} for each of the three pieces separately. To start, suppose that $\F_0(P)$ takes the form $\CG(P)\circ dP$, where $\CG:TM\to TM$ is any smooth bundle map. We are mostly interested in the case when $\CG=\Id$ or $J_M$, but the proof below works for a general $\CG$. 
	
Let $\exp:TM\to M$ denote the exponential map and $\pi: TM\to M$ the project map. Then the parallel transportation $\Psi$ can be viewed as a smooth bundle isometry:
	\[
	\begin{tikzcd}
\pi^*TM\arrow[rd] \arrow[rr,"\Psi"]& &  \exp^*TM\arrow[ld]\\
& TM & 
	\end{tikzcd}
	\]
	Given any smooth sections $\xi_0, \xi_j\in C^\infty(B(0,1); P^*TM)$, $1\leq j\leq n$, consider a bundle map $\hat{P}$ which covers $P: B(0,1)\to M$:
	\begin{equation}\label{TC.E.1}
\begin{tikzcd}
S\times \R^{n+1}\arrow[r, "\hat{P}"] \arrow[d]& TM\arrow[d,"\pi"]\arrow[r,"\exp"]& M\\
S\arrow[r,"P"]  &M, & 
\end{tikzcd}
\qquad
\begin{array}{l}
	\hat{P}: S\times \R^{n+1}\to TM, \\
	\\
(z, t_0,\cdots, t_n)\mapsto (z,\sum_{j=0}^n t_j \xi_j).
\end{array}
	\end{equation}
 Now we consider the differential of the top row of \eqref{TC.E.1} along $S$ then apply $\CG$:
 \begin{equation}\label{TC.E.2}
\begin{tikzcd}
TS\times \R^{n+1}\arrow[r, "d\hat{P}"] \arrow[d]& T(TM)\arrow[d,"\pi"]\arrow[r,"d\exp"]& TM\arrow[d,"\pi"]\arrow[r,"\CG"]& TM\arrow[d,"\pi"]\\
S\times \R^{n+1}\arrow[r, "\hat{P}"] & TM\arrow[r,"\exp"]& M \arrow[r,equal] &M,
\end{tikzcd}
 \end{equation}
 where $d\hat{P}$ is given explicitly by the formula:
 \begin{align}\label{TC.E.3}
 d\hat{P}: TS\times \R^{n+1}&\to T(TM)=T^h(TM)\oplus T^v(TM)\\
(Y, z, t_0,\cdots, t_n)&\mapsto (dP(Y),\sum_{j=0}^n t_j \nabla_Y\xi_j),\ Y\in T_z S\nonumber
 \end{align}
Here the vertical and horizontal sub-bundles of $T(TM)$ are defined using the Levi-Civita connection. One may view $d\exp$ as a bundle map $T(TM)\to (\exp)^*TM$ over $TM$ and compose with $\Psi^{-1}$ to obtain another digram
 \begin{equation}\label{TC.E.4}
\begin{tikzcd}
TS\times \R^{n+1}\arrow[r, "d\hat{P}"] \arrow[d]& T(TM)\arrow[dr,"\pi"]\arrow[r,"d\exp"]&\exp^*(TM)\arrow[r,"\CG"]\arrow[d]& \exp^*(TM)\arrow[dl]\arrow[r,"\Psi^{-1}"] &\pi^*TM\arrow[dll]\\
S\times \R^{n+1}\arrow[rr, "\hat{P}"] && TM.& & 
\end{tikzcd}
\end{equation}

The composition of the first row is the definition of $\CF_0(\sum_{j=0}^nt_j\xi_j)$, whose output is a smooth section of $P^*TM\to S$. Now fix some $(Y,z)\in TS$ and let $t_0=1$. We shall differentiate $\CF_0(\sum_{j=0}^nt_j\xi_j)$ in the variables $t_1,\cdots,t_n$ and set them to be zero. In this case, the image of $d\hat{P}$ lies in 
\[
T_{\hat{P}(z,t_0,\cdots, t_n)}(TM)\cong T_{\hat{P}(z,t_0,\cdots, t_n)}^h(TM)\oplus T_{\hat{P}(z,t_0,\cdots, t_n)}^v(TM)
\]
which is isomorphic to  $T_{P(z)}M\oplus T_{P(z)}M$ canonically, so $\Psi^{-1}\circ d\exp$ is a smooth family of endomorphism $T_zM\oplus T_z M\to T_zM$ parametrized by $T_{\hat{P}(z,t_0,\cdots, t_n)}^v(TM)\cong T_zM$. Thus  $\CF_0(\sum_{j=0}^nt_j\xi_j)$ takes the form 
\begin{equation}\label{TC.E.5}
(\Psi^{-1}\circ\CG\circ d\exp)(t_0\xi_0+\sum_{j=1}^n t_j\xi_j)(dP(Y), t_0\nabla_Y\xi_0+\sum_{j=1}^n t_j\nabla_Y\xi_j ).
\end{equation}
Since $\|\xi_0\|_\infty<c_2$, all $n$-th derivatives of $(\Psi^{-1}\circ\CG\circ d\exp)$ are bounded uniformly. It is clear that if  \eqref{TC.E.5} is differentiated along $t_j, 1\leq j\leq n$ and evaluated at the origin, the final result is bounded by 
\[
C_1|Y|(|dP(z)|+|\nabla\xi_0(z)|+\sum_{j=1}^n \frac{|\nabla\xi_j(z)|}{|\xi_j(z)|})\prod_{j=1}^n |\xi_j(z)|. 
\]
for some $C_1>0$. By taking $\CG=\Id$ and $J_M$, we obtain the desired estimate for $\F_0(P)=dP$ and $\F_0(P)=J_M\circ dP\circ j_S$ respectively (the complex rotation $j_S$ is irrelevant for this estimate). 

Finally, if $\CF=\X^{0,1}$, then the diagram \eqref{TC.E.2} is replaced by 
\begin{equation}\label{TC.E.7}
\begin{tikzcd}[column sep=5em]
\pi_1^*TS\arrow[d] \arrow[rr]& &\pi_1^*TS\arrow[d]\arrow[r,"\X^{0,1}"]& \pi_2^*TM\arrow[ld]\\
S\times \R^{n+1}\arrow[r, "{(\pi_1, \hat{P})}"]&  S\times TM\arrow[r, "{(\pi_1, \exp\circ \pi_2)}"] & S\times M &{}
\end{tikzcd}
\end{equation}
where $\pi_1: S\times M\to S$, $\pi_2:S\times M\to M$ are projection maps. The diagram \eqref{TC.E.4} is replaced by
\begin{equation}\label{TC.E.8}
\begin{tikzcd}
\pi_1^*TS\arrow[d] \arrow[r,equal]& \pi_1^*TS\arrow[dr]\arrow[r,"\X^{0,1}"]& (\exp\circ \pi_2)^*TM\arrow[r, "\Psi^{-1}"]\arrow[d]& (\pi\circ \pi_2)^*TM\arrow[dl]\\
S\times \R^{n+1}\arrow[rr, "{(\pi_1, \hat{P})}"]& & S\times TM& &
\end{tikzcd}
\end{equation}

For any fixed $(Y,z)\in TS$, this diagram defines a smooth function $\R^{n+1}\to T_zM$ which can be differentiated along $t_j, 1\leq j\leq n$ and evaluated at the origin. This derivative is  then bounded by 
\[
C_2\prod_{j=1}^n |\xi_j(z)|.
\]
and this $C_2$ depends at most on the $C^n$-norm of $\X$ on a slightly larger compact subset of $M$ (which depends on $c_2$). This completes the proof of Lemma \ref{TC.L.1}.  
\end{proof}

\bibliographystyle{alpha}
\bibliography{sample}

\end{document}